\documentclass[11pt,twoside,a4paper]{article}  

\usepackage[inner=3cm,outer=2cm]{geometry}
\usepackage{amsmath}
\usepackage{amsfonts}
\usepackage{amssymb}
\usepackage{amsthm}
\usepackage{wasysym}
\usepackage{graphicx}
\usepackage{rotating}
\usepackage{diagrams}
\usepackage{multirow}
\usepackage{tikz}

 \makeatletter
 \def\namedlabel#1#2{\begingroup
     #2%
     \def\@currentlabel{#2}%
     \phantomsection\label{#1}\endgroup
 }  
 \makeatother

\numberwithin{equation}{section}

\theoremstyle{definition}
\newtheorem{defn}{Definition}[subsection]

\theoremstyle{plain}
\newtheorem{thm}[defn]{Theorem}
\newtheorem{propn}[defn]{Proposition}
\newtheorem{lem}[defn]{Lemma}
\newtheorem{cor}[defn]{Corollary}

\theoremstyle{remark}
\newtheorem{rmk}[defn]{Remark}
\newtheorem{eg}[defn]{Example}

\newarrow{Equal}=====
\newarrow{Hook}C----
\newarrow{Empty}{}{}{}{}{}

\DeclareMathOperator{\GL}{GL}
\DeclareMathOperator{\SL}{SL}

\DeclareMathOperator{\Sp}{Sp}
\DeclareMathOperator{\SO}{SO}

\DeclareMathOperator{\proj}{Proj}
\DeclareMathOperator{\spec}{Spec}
\DeclareMathOperator{\Hom}{Hom}

\DeclareMathOperator{\sym}{Sym}

\DeclareMathOperator{\stab}{Stab}

\DeclareMathOperator{\id}{id}
\DeclareMathOperator{\pr}{pr}
\DeclareMathOperator{\ev}{ev}

\DeclareMathOperator{\Pic}{Pic}

\newcommand{\nocontentsline}[3]{}
\newcommand{\tocless}[2]{\bgroup\let\addcontentsline=\nocontentsline#1{#2}\egroup}

\newcommand{\act}{\curvearrowright}
\newcommand{\fns}{\footnotesize}
\newcommand{\CC}{\mathbb{C}}
\newcommand{\PP}{\mathbb{P}}
\newcommand{\A}{\mathbb{A}}
\newcommand{\GGm}{\mathbb{C}^{\times}}
\newcommand{\GG}{\mathbb{G}}
\newcommand{\UU}{\mathbb{U}}
\newcommand{\OO}{\mathcal{O}}

\newcommand{\ten}{\otimes}
\newcommand{\mc}{\mathcal}
\newcommand{\mf}{\mathfrak}
\newcommand{\mb}{\mathbb}

\newcommand{\ol}{\overline}
\newcommand{\pt}{\mathrm{pt}}
\newcommand{\kk}{\Bbbk}
\newcommand{\symdot}{\sym^{\bullet}}

\newcommand{\dblslash}{/\! \!/}
\newcommand{\env}{\!
\mathbin{\text{\rotatebox[origin=c]{70}{\scalebox{1.2}{$\approx$}}}} \!}
\newcommand{\nss}{\mathrm{nss}}
\newcommand{\ssfg}{\mathrm{ss,fg}}
\newcommand{\ssUfg}{\mathrm{ss,} H_u-\mathrm{fg}}
\newcommand{\rms}{\mathrm{s}}
\newcommand{\rmss}{\mathrm{ss}}

\newcommand{\aff}{\mathrm{aff}}
\newcommand{\dashfg}{\mathrm{-fg}}

\newcommand{\olGUX}{\overline{G \times^{H_u} X}}

\newcommand{\GUX}{G\times^{H_u} X} 
\newcommand{\GHX}{G\times^{H} X}

\title{Constructing quotients of algebraic varieties by linear algebraic group
  actions}  

\author{Gergely B\'{e}rczi, Thomas Hawes, Frances Kirwan \thanks{Much of the
    work behind the writing of this article was supported by 
    funding from the Engineering and Physical Sciences Research Council.} \\
  {\fns Mathematical 
    Institute, University of Oxford, Andrew Wiles Building, Oxford. OX2
    6GG. UK.} \\
{\fns Emails: berczi@maths.ox.ac.uk, thomas.hawes01@gmail.com, kirwan@maths.ox.ac.uk} 
  \and Brent Doran\thanks{Brent Doran was partially supported by Swiss National
    Science Foundation Award 200021\_138071.} \\ 
{\fns Departement Mathematik,
    Eidgen\"{o}ssische Technische Hochschule, 8092 Z\"{u}rich, Switzerland.} \\
{\fns Email: brent.doran@math.ethz.ch}}
     
\date{}

\begin{document}

\setcounter{secnumdepth}{3}
\setcounter{tocdepth}{3}

\maketitle                  


\begin{abstract}

In this article we review the question of constructing geometric quotients of
actions of linear algebraic groups on irreducible varieties over algebraically
closed fields of characteristic zero, in the spirit of Mumford's geometric
invariant theory (GIT). The article surveys some recent work on geometric
invariant theory and quotients of varieties by linear algebraic group actions,
as well as background material on linear algebraic groups, Mumford's GIT and
some of the challenges that the non-reductive setting presents. The earlier work
of two of the authors in the setting of unipotent group actions is extended to
deal with actions of any linear algebraic group. Given the data of a
linearisation for an action of a linear algebraic group $H$ on an irreducible
variety $X$, an open subset of stable points $X^{\rms}$ is defined which admits a
geometric quotient variety $X^{\rms}/H$. We construct projective completions of the
quotient $X^{\rms}/H$ by considering a suitable extension of the group action to an
action of a reductive group on a reductive envelope and using Mumford's GIT. In
good cases one can also compute the stable locus $X^{\rms}$ in terms of
stability (in 
the sense of Mumford for reductive groups) for the reductive envelope. 

\end{abstract}

\tableofcontents


\section{Introduction}
\label{sec:introduction}

Group actions are ubiquitous within algebraic geometry. Many spaces that one
might want to understand arise naturally as the quotient of a variety by a group
action, with moduli spaces giving some of the most prominent examples
\cite{mfk94,new78,gie83,KPT01}. Given a variety $X$ over an algebraically closed
field $\kk$ of characteristic zero\footnote{We assume characteristic zero throughout this article, although Mumford's GIT has been extended to algebraically closed fields of arbitrary characteristic; cf. \cite{mfk94} Appendices A and C.}  and a linear algebraic group $H$ acting on
$X$, a basic question to ask therefore is how one
can construct the quotient $X/H$ and study it. By
`quotient' here we mean more precisely a \emph{geometric quotient}, in the sense
of \cite{mfk94}: this is a variety
$X/H$ with an $H$-invariant morphism $X \to X/H$ that, amongst other properties,
is universal with respect to $H$-invariant morphisms from
$X$, and whose fibres are the orbits of the action on $X$. As is well known, there are 
lots of cases of
interest where a geometric quotient for an action cannot possibly exist in the
category of varieties. One way to address this is to enlarge one's category and
work with more general geometric objects, such as algebraic spaces
\cite{art71,knu71} or even stacks \cite{dm69,lmb00}. Another approach is to instead 
look
for nonempty open subsets of $X$ that admit a geometric quotient variety; such
open subsets are guaranteed to exist by a theorem of Rosenlicht \cite{ros63}. It
is this second approach with which we will be concerned in this article.  

In the case where $H=G$ is a reductive linear algebraic group this second
approach was studied by Mumford in the
first edition of \cite{mfk94}, resulting in his \emph{geometric invariant
  theory} (GIT). (In this context see also the work of Seshadri
\cite{ses62,ses72} and also \cite{ses77} which is valid over an arbitrary base.) Mumford's GIT works particularly well in the case
where $X$ is projective. Given the additional choice of an ample linearisation $L \to
X$ (that is, an ample line bundle $L \to X$ with a lift of the action of $G$ to
$L$) Mumford defines a $G$-invariant open subset of stable points $X^{\rms}$ in $X$,
which has a 
geometric quotient variety $X^{\rms}/G$. This is contained in the $G$-invariant
open subset of semistable 
points $X^{\rmss}$ in $X$ and there is a natural surjective $G$-invariant map from
$X^{\rmss}$ onto a projective variety $X \dblslash G$ (canonical to the choice of
$L$) which can be described as $\proj(S^G)$, where $S=\bigoplus_{r \geq 0}
H^0(X,L^{\ten r})$ and $S^G$ is the ring of invariant sections of non-negative tensor 
powers of $L\to X$. Thus (for projective $X$ and ample $L$) there is a diagram
\[
\begin{diagram}
X^{\rms} & \rEmpty~\subseteq & X^{\rmss} & \rEmpty~\subseteq & X  \\
\dTo^{geo} & & \dOnto & & \dDashto\\
X^{\rms}/G & \rEmpty~\subseteq & X \dblslash G & \rEmpty~\subseteq &
\proj(S^G)\\
\end{diagram}
\]
The variety $X \dblslash G$, which is often referred to as the \emph{GIT
  quotient}, provides a natural projective completion of the quotient of the stable
set. Tools abound for studying the spaces in this 
diagram. In \cite{mfk94} Mumford gave numerical criteria for computing the sets
of stable and semistable points in terms of the actions of one-parameter
subgroups of $G$. When $X$ is smooth the local geometry of the orbits in $X^{\rmss}$ 
can be studied
with the slice theorem of Luna \cite{lun73} and, in the case where the ground field is 
the complex numbers
$\CC$, links with symplectic geometry yield ways to compute the (rational
intersection) cohomology of $X \dblslash G$
\cite{kir84,kir85,kir86,kir87,jk95,jkkw03}. Moreover, as studied in the work of
`variation of GIT' (VGIT) by
Thaddeus \cite{tha96}, Dolgachev and Hu \cite{dh98} and Ressayre \cite{res00},
the GIT quotients undergo birational transformations when the linearisation
varies, which can be described explicitly as certain kinds of flips. 

Various authors have considered the question of finding open subsets of
`stable' points that admit geometric quotients in the case where $H$ is not
reductive. This problem is very much more challenging, in essence because the
representation theory of a non-reductive group is not as complete or
well-behaved as in the reductive case. Fauntleroy \cite{fau83,fau85} and Dixmier
and Raynaud \cite{dr81} give geometric 
descriptions of open subsets that admit geometric quotients, but these
are typically difficult to find in practice (requiring knowledge of which points
are separated by invariant functions) and often some extra condition on $X$, such as
normality or quasi-factoriality, needs to be imposed. More algorithmic
approaches have been taken in \cite{gp93,gp98,vde93,san00}, though here
the geometric picture is somewhat more obscure in favour of computation. Other 
progress, this time in the algebraic side of the subject, involves the search
for \emph{separating sets of invariants} to construct quotient morphisms of
affine varieties $X$, made
popular by Derksen and Kemper in \cite{dk02} and pursued recently in the work of
Dufresne and others, \cite{duf13,des14,dj15} and \cite[Section 4]{dk15}. (The use of
separating sets of (rational) invariants seems to in fact go back to Rosenlicht
\cite{ros56} and is used in the proof of his aforementioned theorem; see \cite{vp94}.) 
A key ingredient here is the observation that one can find a \emph{finite} set of
invariants $\mc{S} \subseteq \OO(X)^H$ such that two points in $X$ get
separated by the natural map $X \to \spec(\OO(X)^H)$ if, and only if, there is
an element of $\mc{S}$ that separates the points. Therefore one does not need to
find a full generating set for $\OO(X)^H$ to describe quotient maps.

Any linear algebraic group $H$ has a
canonical normal unipotent subgroup $H_u$, called the unipotent radical of $H$, such 
that $H/H_u$ is reductive, thus constructing quotients for $H$ can, in
principle, be reduced to studying the actions of unipotent groups. The case where $H$ 
is a
unipotent group acting on an irreducible projective variety $X$ with ample
linearisation $L \to X$ was studied in \cite{dk07}, building on the work in
\cite{fau83,fau85,gp93,gp98,win03}. The overarching idea in that paper was to
consider various notions of `stability',  
`semistability' and `quotient' that are intrinsic to the linearisation $L \to X$, and 
relate
these to the GIT of \cite{mfk94} of certain
reductive linearisations associated to $L \to X$. The main appeal of this
approach is that it gives ways to use the  
tools available in the reductive setting to study quotients of unipotent group
actions on (open subsets of) projective varieties. A summary of the main results
and definitions from \cite{dk07} will be given in the upcoming Section
\ref{sec:background}. These techniques have been used in practice in a number of
settings where unipotent actions arise naturally
\cite{ad07,ad08,ad09,ki09, ki11, wdgc13,dg14,dh15,dfs13,dfs14}. 

The purpose of the present article is to generalise the
material of \cite{dk07} to the case where $H$ is any
linear algebraic group, not necessarily unipotent or reductive. Thus we develop a 
theoretical
framework for studying non-reductive group actions that is in the same spirit as
Mumford's GIT, with the basic guiding goal of obtaining
results that are 
as close as possible to the earlier diagram relating the stable and semistable sets 
and GIT
quotient in the reductive setting. Indeed, our constructions reduce to Mumford's 
theory when
$H$ is a reductive group and $L \to X$ is an ample linearisation over a
projective variety. The way we extend the work in \cite{dk07} is to make use of
natural residual actions of the reductive group $H_r:=H/H_u$ and take quotients in
stages---first by $H_u$, then by $H_r$.        

Let us now give a summary of the contents of the rest of this article. We begin in
Section \ref{sec:background} by recalling background material on linear
algebraic groups and their quotients. We discuss various notions of `quotient' in
the category of algebraic varieties and recall the concept of a linearisation, as
introduced by Mumford in the first edition of \cite{mfk94}. Some
of the main differences between actions of reductive groups and unipotent
groups are also highlighted. We then recall the main theorems of GIT for
reductive groups in \cite{mfk94}, paying particular attention to
the case of ample linearisations over projective varieties. A summary
of the work on GIT for unipotent groups in \cite{dk07} is given, 
describing more fully the various 
notions of `stability' and `semistability' considered there, as well as the
definition of the \emph{enveloping quotient} and \emph{enveloped
quotient} and the construction of \emph{reductive envelopes}. 

In Section \ref{sec:gitNonReductive} we begin the work of extending the
theory of \cite{dk07} to more general linear algebraic groups $H$, focussing on
constructing objects from the data of a linearisation $L \to X$. Unlike in
\cite{dk07}, we do not assume $X$ is projective or irreducible, or that the
linearisation $L$ is ample. We start by considering the
natural $H$-invariant rational map  
\[
q:X \dashrightarrow \proj(S^H)
\]
to a scheme $\proj(S^H)$ which is not necessarily noetherian. The maximal
domain of definition contains the open subsets $X_f$ given by the non-vanishing of 
invariant
sections $f$ of positive tensor powers of $L \to X$, and imposing various conditions
on the sections $f$ yields different $H$-invariant open subsets of $X$. In this
way, a subset $X^{\ssfg}$ of \emph{finitely generated semistable} points of $X$ is
defined in Definition \ref{def:GiFi1} such that $q$ maps $X^{\ssfg}$ into an
open subscheme $X \env H$ of $\proj(S^H)$, locally of
finite type over $\kk$, called the
\emph{enveloping quotient} (Definition \ref{def:GiFi3}). While this
looks similar to Mumford's GIT quotient in the reductive setting, in general
there are two key differences to note. Firstly, the enveloping quotient $X \env
H$ is not a quasi-projective variety in 
general, although when $X$ is projective, $L \to X$ is ample and $S^H$ is
finitely generated then $X \env H =
\proj(S^H)$ is in fact the projective variety associated to the graded algebra
$S^H$. Secondly, the map $q:X^{\ssfg} \to X \env H$
is not surjective in general; instead the image $q(X^{\ssfg})$ is a dense
constructible subset of $X \env H$ called the \emph{enveloped quotient}. To
address the fact that the enveloping quotient 
is only a scheme \emph{locally} of finite type in general, we introduce
\emph{inner enveloping quotients} in Definition \ref{def:GiFi3.1} as
quasi-compact open subschemes of 
the enveloping quotient that contain the enveloped quotient
$q(X^{\ssfg})$. Inner enveloping quotients are not canonical to the
linearisation, but have the advantage of being quasi-projective varieties; this is
shown in Proposition \ref{prop:GiFi3.2}. A way in which the collection of
inner enveloping quotients can be thought of as `universal' with respect to
$H$-invariant morphisms from the finitely generated semistable locus is
discussed. We also compare the framework developed here with Mumford's reductive
GIT in the case where $H=G$ is reductive. Building on the notion of stability in
\cite{dk07}, we define the
\emph{stable locus} $X^{\rms}$ for a general linearisation of a linear algebraic group 
over
an irreducible variety in Definition \ref{def:GiSt1.1}. When the
group $H$ is reductive or unipotent our definition reduces to that of Mumford
or \cite{dk07}, respectively. For any choice of inner enveloping quotient
$\mc{U} \subseteq X \env H$, we show that the natural map $q:X^{\ssfg} \to \mc{U}$ 
restricts to
define a geometric quotient on the stable locus, thus obtaining a
diagram 
\[
\begin{diagram}
X^{\rms} & \rEmpty~\subseteq & X^{\rmss} & \rEmpty~\subseteq & X  \\
\dTo^{geo} & & \dOnto & & \dDashto\\
X^{\rms}/G & \rEmpty~\subseteq & X \dblslash G & \rEmpty~\subseteq &
\proj(S^G)\\
\end{diagram}
\]
(see Theorem \ref{thm:GiSu2}). This is analogous to Mumford's in the reductive
case, but in contrast the map 
$q:X^{\ssfg} \to \mc{U}$ is not necessarily surjective, while there are 
many choices of inner
enveloping quotient $\mc{U}$ containing $X^{\rms}/H$ and such a $\mc{U}$ is not
necessarily a projective variety. 

The possible lack of projectivity of an inner enveloping quotient $\mc{U}$
naturally motivates the construction of their projective completions
$\ol{\mc{U}}$. Any such completion contains the enveloped quotient
$q(X^{\ssfg})$ as a dense constructible subset, so we refer to $\ol{\mc{U}}$ as
a \emph{projective completion of the enveloped quotient} (Definition
\ref{def:Co10}). In Section
\ref{sec:compactifications1} we extend the theory of reductive envelopes from
\cite{dk07} to give ways of constructing projective completions of the enveloped 
quotient. We consider the formation of fibre spaces $\GUX$ defined by homomorphisms $H 
\to
G$, with $G$ a reductive group, which restrict to give embeddings of the unipotent 
radical $H_u
\hookrightarrow G$. Such homomorphisms give a diagonal action of the reductive
group $H_r=H/H_u$ on $\GUX$ that commutes with the $G$-action, so $\GUX$ is
a $G \times H_r$-variety that comes equipped with a canonical $G \times
H_r$-linearisation. Various kinds of \emph{reductive
  envelope} $(\olGUX,L^{\prime})$ are defined in Definitions \ref{def:Co1Re5} and
\ref{def:Co1St1}, where $\olGUX$ is an equivariant
completion of $\GUX$ and $L^{\prime} \to \olGUX$ is an extension
of the $G \times H_r$-linearisation over $\GUX$, by requiring that invariant
sections of certain choices of linear systems over $X$ extend to linear
systems over the reductive envelope, satisfying assumptions of varying
strength. In Theorem \ref{thm:Co1Re11} we show that when the line bundle
$L^{\prime} \to \olGUX$ in the
reductive envelope is ample, then the reductive GIT quotient $\olGUX
\dblslash_{L^{\prime}} (G \times H_r)$ gives a projective completion of the
enveloped quotient, and there is a chain of inclusions 
\[
X \cap (\ol{G \times^U X})^{\rms(L^{\prime})} \subseteq X^{\rms} \subseteq
X^{\ssfg} \subseteq X \cap (\ol{G \times^U X})^{\rmss(L^{\prime})}.
\]
When the reductive envelope $(\olGUX,L^{\prime})$ is \emph{strong} (Definition
\ref{def:Co1St1}) then in Proposition \ref{prop:Co1St2} we
obtain equalities
\[
X^{\rms} = X \cap (\ol{G \times^U X})^{\rms(L^{\prime})}, \quad  X^{\ssfg}= X
\cap (\ol{G \times^U X})^{\rmss(L^{\prime})},
\]
which thus provides a way to compute the
intrinsically defined stable locus and 
finitely generated semistable locus using methods from reductive GIT. The
existence of strong reductive envelopes with ample $L^{\prime}$ 
is therefore especially good for the purposes of computation in our non-reductive
geometric invariant theory. In relation to this, we show that when the $H$-
linearisation over $X$ extends to one of the reductive
group $G \times H_r$ in an appropriate way, the arguments in \cite{dk07} can be 
extended to reduce the 
construction of strong reductive envelopes with ample $L^{\prime}$ to a study of
the homogeneous space $G/H_u$. (Such homogeneous spaces were first considered by 
\cite{bbhm63} and studied by
Grosshans in \cite{gro73,gro97}.) This set-up 
works out particularly well when $H_u$ is a \emph{Grosshans subgroup} of $G$,
for then $S^H$ is a finitely generated algebra and
an explicit descriptions of $X^{\rms}$, $X^{\ssfg}$ and $X \env H=\proj(S^H)$
can be obtained in terms of the reductive GIT of
$L^{\prime} \to \olGUX$ (Corollary \ref{cor:Co1St7}). 

In Section \ref{sec:example} we study the space of $n$ unordered points
on $\PP^1$ under the action of a Borel subgroup in $\SL(2,\kk)$. This serves
both to illustrate the use of strong reductive envelopes for performing 
computations and to give an informal look at the 
potential for studying the variation of non-reductive quotients in certain good
cases. The final section contains a brief outline of ongoing research applying the theory developed in the article to moduli spaces occurring naturally in algebraic geometry.             

To supplement the main text, and for the convenience of the reader, we have also
included a proof of a well-known result concerning the GIT 
quotient of a product of reductive groups in Appendix
\ref{subsec:ApLiProducts}.  
                                    
We would like to thank the referee for some very helpful suggestions (including the inclusion of an index) and corrections, and to apologise for any remaining errors, misprints and sources of confusion which are entirely our responsibility. 

\subsection{Notation and Conventions}
\label{subsec:notation}

We work over a ground field $\kk$ that is algebraically closed and of
characteristic zero. By 
`variety' we mean a reduced, separated scheme of finite type over $\kk$; note 
that we do not assume varieties are irreducible unless otherwise
stated, but do insist they are separated. By a `point' 
in a scheme we will always mean a closed $\kk$-valued 
point. A projective completion $X \hookrightarrow \ol{X}$ of a variety $X$ is
a dominant open immersion into a projective variety $\ol{X}$.  If a topological space satisfies
the condition that every cover of it by open sets admits a finite subcover then
we say it is `quasi-compact'. A scheme is quasi-compact if and only if it is the union of finitely many open affine subschemes; any scheme which is complete or quasi-projective is quasi-compact. 

When we talk about actions of groups on varieties or vector spaces, we always
mean a \emph{left action}, unless stated otherwise.   

When talking about line bundles we will usually be referring to the
total space of an invertible sheaf of modules; the sheaf of
sections of a line bundle $L$ is denoted by $\underline{L}$, so that
$L=\mathrm{\mathbf{Spec}}(\underline{L}^*)$. An exception to
this is when talking about twisting sheaves $\OO(n)$ on varieties---here we
don't make any notational distinction between the sheaf and its total
space. Given a linearisation $L \to X$ of a group $H$ and a character $\chi$ of
$H$, the \emph{twist} $L^{(\chi)} \to X$ of $L \to X$ by $\chi$ is the
linearisation obtained by multiplying the fibres of the linearisation $L \to X$
by the character $\chi^{-1}$; this will also be emphasised in
Section \ref{subsec:BgGrLinearisations} of the main text. 

If $\phi:X \to Y$ is a morphism of schemes, then the natural pullback morphism
of structure sheaves is denoted $\phi^{\#}:\OO_Y \to \phi_*\OO_X$. On the other
hand, if $L \to Y$ is a line bundle then we use the notation 
$\phi^*$ to denote pull-back $\underline{L} \to \phi_*(\underline{\phi^* L})$. Given 
an
$\OO_X$-module $\mc{F}$ and an $\OO_Y$-module  
$\mc{G}$, then $\mc{F} \boxtimes \mc{G}:=(\pr_X^*\mc{F}) \ten_{\OO_{X \times Y}}
(\pr_Y^*\mc{G})$, where $\pr_X:X \times Y \to X$ and $\pr_Y:X \times Y \to Y$
are the projections.

Unless indicated otherwise, graded rings $R$ are always non-negatively $\mb{Z}$-
graded,
i.e. if $R$ is a graded ring then the degree $d \in \mb{Z}$ piece $S_d$ is trivial 
whenever
$d<0$. If $f \in R$ a non-zero homogeneous element then $R_{(f)}$ is
the subring of the localisation $R_f$ consisting of degree 0
elements. Similarly, if $M$ a graded $R$-module and $f \in R$ a
non-zero homogeneous element, then $M_{(f)}$ is
the $R_{(f)}$-submodule of the localisation $M_f$ consisting of degree 0
elements. If $r \in \mb{Z}$ is positive then $R^{(r)}$ denotes
the Veronese subring of $R$ whose degree $m$ piece is $(R^{(r)})_m=R_{mr}$.
 
Associated to a vector space $V$
we understand the projective space $\PP(V)$ to be the space whose
points correspond to one-dimensional subspaces of $V$. Another way to say this
is that $\PP(V) = \proj (\sym^{\bullet}(V^*))$, where $\symdot(V^*)$ is the
symmetric algebra $\bigoplus_{m\geq 0} \sym^m(V^*)$. With these conventions, if
$L \to X$ is 
a line bundle on a scheme $X$ with a basepoint-free linear system $V
\subseteq H^0(X,L)$, then there is a canonical morphism $X \to \PP(V^*)$.  

Finally, for basic facts in algebraic geometry we refer the reader to \cite{har77}
and \cite{stacks-project}. The latter is particularly useful for results
regarding schemes that are not necessarily noetherian.


\section{Background: Quotients of Varieties and Geometric Invariant Theory}
\label{sec:background}

In this section we recall some background material that will be used in
subsequent sections. We begin in Section \ref{sec:LiBasic} by recalling basic
definitions concerning linear algebraic groups, then discuss various kinds of quotient 
in
the category of varieties and review the concept of a linearisation of an
action. We also recall the definitions of reductive groups and unipotent groups
and compare them from the point of view of the geometry of their actions and
their invariant theory. In Section \ref{sec:BgReductive} we give a summary of
the main results from Mumford's GIT for reductive groups,
paying particular attention to the case of ample linearisations over projective
varieties. Finally, in Section \ref{sec:BgUnipotent} we recall the main
definitions and results of \cite{dk07}, which will form the basis for our development
of a geometric invariant theoretic approach to studying actions of more general linear
algebraic groups in later sections.       

The material on linear algebraic groups is taken from
\cite{bor91,spr94} and our main references for the material on quotients are
\cite{ser58,ses72,mfk94}. For reductive GIT we have mostly used \cite{mfk94};
see also \cite{new78,dol03,muk03,sch08}. 


\subsection{Basics of Group Actions and Quotients}
\label{sec:LiBasic}


\subsubsection{Linear Algebraic Groups and Quotients}
\label{subsec:BgLinear}

We begin by recalling some of the basic theory of algebraic groups. Following
\cite[Chapter 1]{bor91} we define an
\emph{algebraic group}  as a variety $H$ equipped with a group
structure such that the multiplication map $H \times H \to H$ and inversion map
$H \to H$ are morphisms of varieties. We write $e \in H$ for the identity
element of $H$. A \emph{homomorphism} of algebraic groups
$H_1 \to 
H_2$ is a morphism of varieties that is also a homomorphism of the group
structures. If $H_1 \hookrightarrow H_2$ is a homomorphism that is also a closed
immersion then we say that $H_1$ is a \emph{closed subgroup} of $H_2$. A first
example of an algebraic group is the general linear group $\GL(n,\kk)$,
for any integer $n\geq 0$. A \emph{linear algebraic group}  \index{linear algebraic group} is an 
algebraic group that is a closed subgroup of $\GL(n,\kk)$, for some $n \geq
0$. It is these groups we will be concerned with in this article; see
\cite{bri09,bri11,bsu13,bri15} for work on the structure and geometry of more
general algebraic groups. A basic result in the theory of algebraic groups says that 
every
\emph{affine} algebraic group (i.e. an algebraic group that is also an affine
variety) is isomorphic to a linear algebraic group \cite[Proposition
1.10]{bor91}, thus one can identify and work with linear algebraic groups in a more
intrinsic fashion.  

\begin{eg} \label{ex:TrRe0.1}
Any finite group is a linear algebraic group. The group $\GG_m:=(\kk \setminus
\{0\},\times)$ of non-zero elements of $\kk$ 
under multiplication is a linear algebraic group (indeed, it is just
$\GL(1,\kk)$). The group $\GG_a:=(\kk,+)$ of elements of $\kk$ under addition is
also a linear algebraic group: it is isomorphic to the
group $\mb{U}_2 \subseteq \GL(2,\kk)$ of upper-triangular matrices via $a
\mapsto \left(\begin{smallmatrix} 1 & a \\ 0 & 1 \end{smallmatrix} \right)$.
\end{eg}

\begin{eg} \label{ex:TrRe0.2} (Operations on linear algebraic groups
  \cite[Chapters 1 and 6]{bor91}.)
A \emph{normal subgroup} $N$ of an affine algebraic group $H$ is a closed
subgroup that is normal as an abstract group. If $H$ is a linear algebraic
group and $N$ a normal subgroup, then $H/N$ is a linear algebraic
group. Products of linear algebraic groups are also linear algebraic groups.
\end{eg}

Now let $X$ be a variety over $\kk$ and $H$ a linear algebraic group. An
\emph{action} of $H$ on $X$ is a (left)
action $H \times X \to X$ that is also a morphism of varieties. In this case we
will often refer to $X$ as an  \index{$H$-variety}
``$H$\emph{-variety}'' and
sometimes write $H \act X$ to indicate the given action. We usually
write the morphism of an action as   
\[
H \times X \to X, \quad (h,x) \mapsto h \cdot x \quad \text{(or $(h,x) \mapsto hx$)}  
\]
if no confusion is likely to occur. Given a point $x \in X$, we write 
\[
H \cdot x := \{hx \mid h \in H\}
\]
for the \emph{orbit} of $x$ under the action of $H$ and 
\[
\stab_H(x):=\{h \in H \mid hx=x \}
\]
for the \emph{stabiliser} of $x$. Given a subset $Z \subseteq X$, we say $Z$ is
\emph{$H$-stable}, or \emph{$H$-invariant}, if $H \cdot z \subseteq Z$ for all
$z \in Z$.   

Note that in the case where $X=V$ is a finite-dimensional vector space with an
action of $H$ that is algebraic, $V$ is also called a \emph{rational
  $H$-module}.     \index{rational $H$-module}
\begin{rmk} \label{rmk:TrRe0}  
Throughout this article we shall assume for simplicity that all actions are such
that stabilisers of general points are finite.      
\end{rmk}

Given a homomorphism of linear algebraic
groups $\rho:H_1 \to H_2$, an $H_1$-variety $X$ and an $H_2$-variety $Y$, we say
a 
morphism of $\phi:X \to Y$ is  \index{equivariant}
\emph{equivariant (with respect to $\rho:H_1 \to H_2$)} if
$\phi(hx)=\rho(h)\phi(x)$ for all $h \in H_1$ and all $x \in X$. If $H=H_1=H_2$
and $\rho$ is 
the identity homomorphism, then we simply say $\phi$ is \emph{$H$-equivariant},
and if furthermore $H$ acts trivially on $Y$ (so that $\phi(hx)=\phi(x)$ for all
$h \in H$ and all $x\in X$) then we say $\phi$ is \emph{$H$-invariant}.    

If a linear algebraic group $H$ acts on a variety $X$ then a fundamental
question, if vaguely stated, is to ask: does there exist a variety $Y$ that is a 
`quotient' of $X$
by the action of $H$? There are various definitions to make the term
`quotient' more precise, with varying agreement with one's geometric
intuition. We will recall the kinds of `quotient' we shall be concerned with 
momentarily. Before doing so, note that given an $H$-variety
$X$ and an $H$-stable open subset $U \subseteq X$, there is a canonically
induced action of $H$ on the ring of regular functions: 
\begin{equation} \label{eq:TrRe1} 
(h\cdot f)(x):=f(h^{-1}x) \quad \text{for all } \
x\in U, \ f \in \OO(U), \ h\in H,
\end{equation}
and one can consider the subring of invariant functions:
\[
\OO(U)^H = \{ f \in \OO(U) \mid h \cdot f = f \text{ for all } h \in H \}.
\]  
\begin{defn} \label{def:TrRe1}  
Let $H$ be a linear algebraic group acting on a variety $X$.
\begin{enumerate}
\item \label{itm:TrRe1-1} A \emph{categorical quotient}  \index{categorical quotient} is a variety $Y$ together with 
an 
$H$-invariant morphism $\phi:X \to Y$ satisfying the following universal
property: any other $H$-invariant morphism $X \to Z$ admits a unique
factorisation through $\phi$. 
 \item \label{itm:TrRe1-2} A \emph{good quotient} \index{good quotient} is an $H$-invariant morphism
   $\phi:X \to Y$ satisfying the following properties:
\begin{enumerate}
\item \label{itm:TrRe1-2a} the morphism $\phi$ is surjective and affine;
\item \label{itm:TrRe1-2b} the pull-back map $\phi^{\#}:\OO_Y \to \phi_*\OO_X$ induces 
an isomorphism of
  sheaves $\OO_Y \cong (\phi_*\OO_X)^H$, where
  $(\phi_*\OO_X)^H(U)=\OO_X(\phi^{-1}(U))^H$ for each open subset $U \subseteq
  Y$; and 
\item \label{itm:TrRe1-2c} if $W_1,W_2$ are disjoint $H$-invariant closed
  subsets of $X$, then
  $\phi(W_1)$ and $\phi(W_2)$ are disjoint closed subsets of $Y$. (Note this
  implies $\phi:X \to Y$ is a submersion \cite[Chapter 0, \S 2, Remark
6]{mfk94}.)    
\end{enumerate}
\item \label{itm:TrRe1-3} A \emph{geometric quotient}  \index{geometric quotient} is a good quotient $\phi:X
  \to Y$ that is also an \emph{orbit
space}; i.e. $\phi^{-1}(y)$ is a single $H$-orbit for each $y\in Y$. In this
case we write $Y=X/H$.
\item \label{itm:TrRe1-4} A \emph{principal $H$-bundle}  \index{principal $H$-bundle} (or a \emph{locally
    isotrivial quotient})  \index{locally isotrivial quotient} is an $H$-invariant
  morphism $\phi:X \to Y$ such that, for every point $y \in Y$, there is a
  Zariski-open neighbourhood 
  $U_y \subseteq Y$ of $y$ and a finite \'{e}tale 
 morphism 
  $\widetilde{U_y} \to U_y$ such that there exists an $H$-equivariant
  isomorphism $H \times \widetilde{U_y} \cong \widetilde{U_y} \times_Y X$, where the
  fibred product $\widetilde{U_y} \times_Y X$ has
  the canonical $H$-action and $H \times \widetilde{U_y}$ has the
  \emph{trivial $H$-bundle} action, induced by left multiplication by $H$ on itself: 
\[
H \times (H \times \widetilde{U_y}) \to H \times \widetilde{U_y}, \quad
(h,h_0,u) \mapsto (hh_0,u). 
\] 
\end{enumerate}
\end{defn}

Definition \ref{def:TrRe1}, \ref{itm:TrRe1-1} is taken from \cite[Definition
0.5]{mfk94}, while \ref{itm:TrRe1-2}--\ref{itm:TrRe1-3} are from
\cite[Definitions 1.4 and 1.5]{ses72}
and \ref{itm:TrRe1-4} is \cite[Definition 2.2]{ser58}. 

\begin{rmk} \label{rmk:BgLi1}
Because
we work exclusively with linear algebraic groups, by a result of Grothend-ieck \cite[Page 326]{gro59} we may equivalently work with
quotients that are locally trivial in the fppf-topology in Definition \ref{def:TrRe1}, \ref{itm:TrRe1-4} (the reader may also
consult \cite[Remark 2.1.1.6]{sch08} for a justification of this). \end{rmk}

In general we have the
following chain of implications: principal bundle $\implies$ geometric quotient
$\implies$ good quotient $\implies$ categorical quotient.
(The main non-trivial implication is the last one, whose proof may be found in
\cite[Chapter 0, \S 2, Remark 6]{mfk94}. Accordingly, one often refers to a good
quotient as a ``good categorical quotient''.) However, none of the reverse
implications hold. 
\begin{eg} \label{ex:TrRe1.1}
(Good quotient $\nRightarrow$ geometric quotient.) Let $\GG_m$ act on $X=\kk^n$
by the usual scaling action, $t \cdot (x_1,\dots,x_n)=(tx_1,\dots,tx_n)$. Then
the unique map $\kk^n \to \pt:=\spec \kk$ is a good quotient
for this action, but is 
clearly not a geometric quotient: the preimage of $\pt$ consists of many
orbits.  
\end{eg}
       
\begin{eg} \label{ex:TrRe1.2}
(Geometric quotient $\nRightarrow$ principal bundle.) Let $\GG_m$ act on $\kk^n
\setminus \{0\}$ ($n>1$) by the action $t \cdot (x_1,\dots,x_n) =
(t^rx_1,\dots,t^rx_n)$, where $r\geq 2$ is an integer. Then the usual projection
$\kk^n \setminus \{0\} \to \PP^{n-1}$ is a geometric quotient which is not a
principal $\GG_m$-bundle, because the action is not set-theoretically
free. (Thus, it is possible 
for geometric quotients to exist for actions where some stabilisers are
non-trivial.) 
\end{eg}

\begin{eg} \label{ex:TrRe1.3}
(Categorical quotient $\nRightarrow$ good quotient.) Such examples are more
difficult to come by, but do exist. The interested reader can refer to
\cite{ah99,ah01}.     
\end{eg}

A very useful property of good and geometric quotients is that they are
determined locally on the base variety. That is \cite[Proposition
3.10]{new78},
\begin{itemize}
\item an $H$-invariant morphism $\phi:X \to Y$ is a good (respectively,
geometric) quotient if, and only if, there is is an open cover $\{U_i\}$ of $Y$
such that each restriction $\phi:\phi^{-1}(U_i) \to U_i$ is a good
(respectively, geometric) quotient of $H \act \phi^{-1}(U_i)$; and
\item if $\phi:X \to Y$ is a good (respectively, geometric) quotient, then for
each open subset $U \subseteq Y$ the restriction $\phi:\phi^{-1}(U) \to U$ is a good
(respectively, geometric) quotient of $H \act \psi^{-1}(U)$.
\end{itemize} 

Given an action of $H$ on $X$, there are certain topological restrictions on the
action that must be fulfilled if a 
geometric quotient or a principal bundle structure is to exist. If
a geometric quotient for $H \act X$ exists then the action must be  \index{closed action}
\emph{closed}: that is, for each point $x \in X$ the orbit $H \cdot x$ is a
closed subset of $X$. Furthermore, by \cite[Proposition 0.9]{mfk94} and
\cite[Theorem 6.1]{ses72} a geometric quotient $X \to X/H$
has the structure of a principal $H$-bundle if, and only if, the action of $H$
on $X$ is  \index{free action}\emph{free}: that is, the graph morphism
\[
H \times X \to X \times X, \quad (h,x) \mapsto (hx,x)
\]
of the action morphism is a closed immersion. Checking freeness of an action can
be made simpler by the following lemma.
\begin{lem} \label{lem:TrRe1} 
\cite[\S 6.3, Lemma 8]{eg98} An action of a linear algebraic group $H$ on $X$
  is free (in the above sense) if, and only if, it is set-theoretically free and
  \emph{proper}  \index{proper action} (that is, the graph morphism $H \times X \to X \times X$ of the action 
is
  a proper morphism).
\end{lem}

Example \ref{ex:TrRe1.1} shows that not every action of a linear
algebraic groups on a variety need admit a geometric quotient. More generally,
there are actions which do not admit even a categorical quotient; see
\cite{ah00} for examples in the context of toric varieties under actions of
subtori. From here, there
are a couple of possible ways to proceed if one wants to construct a quotient
for the action. One way is to try to enlarge
the category in which one works so that it contains a quotient object
for the action. For example, any proper action with finite stabilisers has an
\emph{algebraic space} that is a geometric quotient
\cite{art71,knu71,kol97,km97}. More generally, one can use the category of stacks
\cite{dm69,lmb00}, where every action of a linear algebraic group on a variety has a 
quotient
stack. An alternative approach, which we adopt in this article, is to look
for nonempty invariant open subsets
that admit a geometric quotient. This approach is validated by the
following result of Rosenlicht.
\begin{thm} \label{thm:TrRe1}
\cite{ros63}\footnote{For slightly more modern treatments, see also
  \cite[Theorem 4.4]{vp94} and \cite[Theorem 6.2]{dol03}.} Let $H$ be a linear
algebraic group acting on an irreducible variety $X$. Then 
there is a nonempty $H$-invariant open subset $U \subseteq X$ admitting a
quasi-projective geometric quotient $U/H$. 
\end{thm}
Rosenlicht's proof of Theorem \ref{thm:TrRe1} is non-constructive, so the 
question remains of how to explicitly find nonempty open subsets---ideally as
large as possible---that admit geometric
quotients. This is the basic task of \emph{geometric invariant theory}. We will
discuss ways in which this has been done for certain kinds of linear algebraic group 
in the
upcoming Sections \ref{sec:BgReductive} and \ref{sec:BgUnipotent}.   


\subsubsection{Linearisations of Actions}
\label{subsec:BgGrLinearisations}

A natural way to try and construct open subsets of $X$ that admit geometric quotients 
is to
glue together quotients of smaller open subsets which are easier to
understand (for example, affine open subsets) and appeal to the fact that
geometric quotients are local on the base. Such a strategy in general runs
the risk of resulting in non-separated quotient schemes. This can be avoided by 
considering open subsets $X_f$ defined by the non-vanishing of
some invariant rational function $f$ and gluing the maps $X_f \to
\spec(\OO(X_f)^H)$, for then orbits in $X_f$ are separated by orbits outside of
$X_f$ by $f$. This is essentially what a
\emph{linearisation} achieves for us, a notion due to Mumford \cite[Definition
1.6]{mfk94} that  is fundamental for what follows. 
\begin{defn} \label{def:TrRe2} Let $H$ be a linear algebraic group acting on a
variety $X$. A \emph{linearisation}  \index{linearisation} of the action is a line bundle
$L \to X$ together with a choice of $H$-action on $L$ such that
\begin{enumerate} 
\item the bundle projection $L \to X$ is $H$-equivariant; and
\item for each $h \in H$ and $x \in X$, the induced map between the fibres 
\[  
L|_x \to L|_{hx}, \quad l \mapsto hl
\]
is linear. 
\end{enumerate} 
\end{defn}

\begin{rmk}
If $L \to X$ is a linearisation for the action of $H$ on $X$, we will often
represent this using the notation $H \act L \to X$, or say that $L \to X$ is an
``$H$\emph{-linearisation}''  \index{$H$-linearisation} for short. In general we will not distinguish
between the line bundle and the linearisation in our notation, unless this
is likely to lead to confusion.
\end{rmk}

For practical purposes (e.g. the study of moduli problems) the following two
classes of examples frequently arise.
\begin{eg} \label{ex:TrRe2.1}
Consider the case where $X=\spec A$ affine and $L=\OO_X=X \times \kk$ is the trivial
line bundle. Then a linearisation
of $H$ on $\OO_X$ corresponds to a choice of character $\chi:H \to
  \mb{G}_m$ \cite[Theorem 7.1 and Corollary 7.1]{dol03} via
\[
H \times (X \times \kk) \to X \times \kk, \quad (h,x,t) \mapsto (hx,\chi(h)t).
\]  
\end{eg}

\begin{eg} \label{ex:TrRe2.2}
Let $V$ be a finite dimensional vector space over
$\kk$ and $\rho:H \to \GL(V)$ a homomorphism. Then $H$ acts on $\PP(V)$ in the
obvious manner, and $\rho$ defines a canonical choice of linearisation on the 
tautological line
bundle $\OO(-1) \to \PP(V)$. This dually defines a linearisation on $\OO(1) \to
\PP(V)$ (see below). 
\end{eg}

A linearisation $H \act L \to X$ gives us a natural action on the
sections of $L \to X$ over invariant open subsets $U \subseteq X$, by the
formula
\begin{equation} 
\label{eq:TrRe2} (h \cdot f)(x)=hf(h^{-1}x) \quad \text{for all } x\in
U, \ f\in H^0(U,L), \ h\in H.
\end{equation}
Given any invariant open subset $U \subseteq X$ we write
\[
H^0(U,L)^H:=\{f \in H^0(U,L) \mid h \cdot f = f \text{ for all } h\in H\}
\]
for the sections invariant under the action \eqref{eq:TrRe2}. Elements of
$H^0(U,L)^H$ are called \emph{invariants} for the linearisation $L|_U \to U$. 

\begin{rmk} \label{rmk:TrRe2.3}
We should point out here that saying a \index{invariant} \index{equivariant}
  section $f \in H^0(U,L)$ is invariant in this sense is the same as
  saying, in the 
  terminology of Section \ref{subsec:BgLinear}, it is \emph{$H$-equivariant} as a 
  morphism $f:U \to L$, rather than necessarily invariant. When
  talking about sections of line bundles we always take invariance to be with
  respect to the action \eqref{eq:TrRe2}. Unfortunately both uses of the term
  `invariant' are commonplace.    
\end{rmk}

There are also various natural operations on linearisations over an $H$-variety
arising from the standard operations on line bundles. Given an $H$-linearisation $L
\to X$, the dual line bundle $L^* \to X$ has a canonical linearisation, defined
fibre-wise by pulling back linear maps along the action of $H$ i.e. for any $x
\in X$, an element $h \in H$ acts via  
\[
(L|_x)^* \to (L|_{hx})^*, \quad (h,\alpha) \mapsto \alpha \circ
h^{-1}:L|_{hx} \to \kk.
\]
(Note the use of $h^{-1}$ is to ensure the resulting $H$-action is a left
action.) Also, given two $H$-linearisations $L_1 \to X$ and $L_2 \to X$, there is a
canonical linearisation on the tensor product $L_1 \ten L_2 \to X$,
induced by the map on fibres
\[
( (L_1)|_x \ten (L_2)|_x) \to (L_1)|_{hx} \ten (L_2)|_{hx}, \quad
(h,l_1 \ten l_2) \mapsto (hl_1) \ten (hl_2),
\]         
for any $x \in X$ and $h \in H$.

A \emph{character} of $H$ is simply a group homomorphism $H \to \GG_m$. Given a
linearisation $L \to X$  
and a character $\chi \in \Hom(H,\GG_m)$, we define a linearisation
$L^{(\chi)} \to X$, which is said to be the result of 
\emph{twisting $L$ by the character}  \index{twisting a linearisation by a character} $\chi$, as the linearisation $L
\ten \OO_X^{(\chi)}$, where $\OO_X^{(\chi)}$ is the linearisation of the trivial
bundle $\OO_X=X \times \kk$ defined by $\chi^{-1}$ (see Example
\ref{ex:TrRe2.1}). In other words, $L^{(\chi)} \to X$ is obtained by multiplying
the fibres of the linearisation $L \to X$ by the character $\chi^{-1}$.  

Finally, if $\phi:X \to Y$ is an equivariant morphism between two $H$-varieties and $H
\act L \to Y$ is a linearisation, then there is a unique linearisation on the
pullback line bundle $\phi^*L \to X$ making the natural bundle map $\phi^*L \to L$
equivariant. This linearisation makes the pullback map $\phi^*:H^0(Y,L) \to
H^0(X,\phi^*L)$ an $H$-equivariant linear map with respect to the actions defined
as in \eqref{eq:TrRe2}.

Given a line bundle $L \to X$, define the \emph{section ring}  \index{section ring} of $L \to X$ to be the
commutative graded ring
\[ 
S:=\kk[X,L]:=\bigoplus_{r\geq 0}H^0(X,L^{\ten r}),
\] 
where the multiplication is induced by the natural maps
\[ 
H^0(X,L^{\ten r_1}) \otimes H^0(X,L^{\ten r_2}) \to H^0(X,L^{\ten(r_1 +
r_2)}).
\] 
The action in \eqref{eq:TrRe2} defines a linear action of
$H$ on $\kk[X,L]$ that respects the grading and distributes over the
multiplicative structure. Given $r >0$ and an invariant global section $f \in 
H^0(X,L^{\ten
  r})^H$, the open set 
\[  
X_f:=\{x\in X \mid f(x) \neq 0\}
\] 
is $H$-invariant. There is a naturally induced action of $H$ on $S_{(f)}$, and
the corresponding action on $\OO(X_f)$ under the canonical isomorphism $S_{(f)} \cong 
\OO(X_f)
$ is the one defined by the formula in \eqref{eq:TrRe1}. A linearisation thus gives a 
way of
studying the invariant functions on certain open subsets of $X$, which is an
important consideration for the construction of geometric quotients
(cf. Definition \ref{def:TrRe1}).
\begin{eg} \label{ex:TrRe3.1}
In the case of Example \ref{ex:TrRe2.1}, where $X=\spec A$ is affine, $L=\OO_X$ and 
the action of $H$ on
  $\OO_X$ is defined by a character $\chi:H \to
  \mb{G}_m$, then $S$ is the graded
  ring $\bigoplus_{r \geq 0} A$ (with the grading corresponding to $r$) and the ring 
of 
  invariants $S^H$ is the graded subring of \emph{semi-invariants},  \index{semi-invariants}
\[
\bigoplus_{r \geq 0} A^H_{\chi^r}, \quad A^H_{\chi^r}:=\{f \in A \mid
\text{$f(hx)=\chi(h)^rf(x)$ for all $x \in X$, $h \in H$}\};
\]
see \cite[Chapter 6]{muk03} or \cite[\S 3]{vp94}.  
\end{eg}

\begin{eg} \label{ex:TrRe3.2} 
Suppose now $X$ is a projective $H$-variety and $L \to X$ a very ample
linearisation (that is, a linearisation which is very ample as a line
bundle). Letting $V=H^0(X,L)^*$, the natural graded ring map $\symdot H^0(X,L) \to
\kk[X,L]$ defines an embedding $\phi:X \hookrightarrow \PP(V)$. Dualising the action 
of $H$ on $H^0(X,L)$
of \eqref{eq:TrRe2} defines a canonical linearisation $H \act \OO_{\PP(V)}(1) \to
\PP(V)$ as in Example \ref{ex:TrRe2.2}, with respect to which $\phi$ is $H$-
equivariant and
$L=\phi^*\OO_{\PP(V)}(1)$ as linearisations. If $L \to X$ is sufficiently
positive then the restriction map $\phi^*:\kk[\PP(V),\OO(1)] \to \kk[X,L]$ is
surjective by Serre Vanishing \cite[Chapter 3, Proposition 5.3]{har77}, so that
$\kk[X,L]^H \cong (\kk[\PP(V),\OO(1)]/\ker(\phi^*))^H$. Note 
that in general the induced restriction map on invariants $\phi^*:\kk[\PP(V),\OO(1)]^H 
\to
\kk[X,L]^H$ is \emph{not} surjective; that is, not every invariant section over
$X$ extends to one over $\PP(V)$.   
\end{eg}
\begin{rmk} \label{rmk:TrRe3} When $X$ is a \emph{normal} quasi-projective variety
equipped with an action of a connected linear algebraic group $H$, then one can always
find an equivariant embedding of $X$ into some projective space $\PP^m$, with
the $H$-action on $\PP^m$ defined by some representation $H \to \GL(m+1,\kk)$ 
(cf. Example \ref{ex:TrRe2.2} \cite[Corollary 1.6]{mfk94} when $X$ is complete, and for the more general case \cite{su74,su75}). Hence any normal quasi-projective variety 
equipped with an action of a connected linear algebraic group admits a very ample
linearisation. For work on linearisations of actions on more general
varieties see \cite{bri15-2}.   
\end{rmk}

We saw earlier that, given a linear algebraic group $H$ acting on a variety $X$, the operations of
tensor product and dualising may be applied to $H$-linearisations. These give
an abelian group structure to the set $\Pic^H(X)$ of isomorphism classes of
$H$-linearised line bundles, such that the natural forgetful map $\Pic^H(X) \to
\Pic(X)$ to the usual Picard group of $X$ is a homomorphism. We shall see that
many constructions in geometric invariant theory are 
independent of taking positive tensor powers of a linearisation, therefore it is
convenient to consider the following notion.
\begin{defn} \label{def:TrRe3.1}
Given a linear algebraic group $H$ acting on a variety $X$, define a \emph{rational linearisation} to be an element of $\Pic^H(X)
\ten_{\mb{Z}} \mb{Q}$.   
\end{defn}
\begin{rmk} \label{rmk:TrRe3.2}
  Given an element $\mc{L} \in \Pic^H(X) \ten_{\mb{Z}} \mb{Q}$, we may write
  $\mc{L}=\frac{1}{n}L$ for some integer $n>0$ and $H$-linearisation $L \in
  \Pic^H(X)$, and if $\tilde{n} \in \mb{Z}_{>0}$ and $\tilde{L} \in \Pic^H(X)$
  are another such integer and linearisation then we have $L^{\ten \tilde{n}} =
  \tilde{L}^{\ten n}$ within $\Pic^H(X)$. This observation allows us to define various geometric invariant theoretic notions for rational linearisations.
\end{rmk} 

We conclude this section with a very useful observation regarding linearisations $L \to X$: the induced
actions on $H^0(X,L)$ are
\emph{locally finite} (also called \emph{rational} in     \index{locally finite action}
\cite{new78}). In other words, we have the following result (see 
\cite[Chapter 1, \S 1, Lemma]{mfk94}, or \cite[Proposition 1.9]{bor91} in the
case $L=\OO_X$).
\begin{lem} \label{lem:TrRe4} Let $H$ be a linear algebraic group acting on a
  variety $X$ and $L \to X$ a linearisation. Given a finite-dimensional linear
  subspace $W \subseteq H^0(X,L)$, there is a finite-dimensional
  rational $H$-module $V \subseteq H^0(X,L)$ containing $W$.  
\end{lem}


\subsubsection{Unipotent Groups and Reductive Groups}
\label{subsec:BgGrUnipotentAndReductive}

It turns out that the problem of constructing quotients and finding invariants
for a given 
linearisation depends very much on the sort of linear algebraic group one is
considering. In regards to this, it is helpful to focus one's attention on two
particular sub-classes of group: 
\emph{unipotent} groups and \emph{reductive} groups. These are defined as
follows. 

\begin{defn} \label{def:TrRe3} Let $H$ be a linear algebraic group. 
\begin{enumerate}
  \item \label{itm:TrRe3-1} \cite[Chapter 4]{bor91} We say $H$ is  \index{unipotent linear algebraic group}
    \emph{unipotent} if there is a closed 
    embedding $\rho:H \hookrightarrow \GL(n,\kk)$, for some $n\geq 0$, such that
    $\rho(h)-\rho(e)$ is nilpotent in $\GL(n,\kk)$ for each $h \in H$, i.e. $(\rho(h)-
\rho(e))^m=0$
    for some $m\geq 0$ (depending on $h$).\footnote{If this is the case, then in
      fact for \emph{any} closed embedding $\rho:H \hookrightarrow \GL(\tilde{n},\kk)$ 
into any
 general linear group $\GL(\tilde{n},\kk)$ one has $\rho(h)-\rho(e)$ nilpotent for 
each $h \in
 H$; cf. \cite[Theorem 4.4]{bor91}.}
\item \label{itm:TrRe3-2} \cite[\S 11.21]{bor91} The \emph{unipotent radical
    $H_u$}  \index{unipotent radical} \index{$H_u$ unipotent radical of $H$} of $H$ is the maximal connected\footnote{In fact, over characteristic
    zero all unipotent groups are necessarily connected; see \cite[Chapter II,
    \S 6.3]{dg70}.} normal unipotent
  subgroup of $H$. 
\item \label{itm:TrRe3-3} \cite[\S 11.21]{bor91} We say $H$ is \emph{reductive}  \index{reductive group}
  if $H_u=\{e\}$.     
  \end{enumerate}
\end{defn}

The following are well-known examples of reductive and unipotent groups.
\begin{eg} \label{ex:TrRe2} (Reductive groups; see \cite{spr94}.)  The classical groups $\GL(n,\kk)$,
  $\SL(n,\kk)$, $\Sp(2n,\kk)$ and $\SO(n,\kk)$ are reductive. Products of
  reductive groups are again reductive; in particular, groups isomorphic to
  products of $\GG_m$ (which are called \emph{tori}) are reductive. All finite groups and all
  semisimple linear algebraic groups are reductive. 
\end{eg}

\begin{eg} \label{ex:TrRe3} (Unipotent groups.) The group $\GG_a$ is
  unipotent. The group 
\[
\mb{U}_n:=\{(a_{ij}) \in \GL(n,\kk) \mid a_{ii}=1 \text{ for each } i=1,\dots,n
\text{ and } a_{ij}=0 \text{ whenever } j<i\} 
\]
 of strictly upper triangular inside $\GL(n,\kk)$ is unipotent, for each $n \geq
 1$. In fact, a group $H$ is unipotent if, and only if, it is
  isomorphic to a closed subgroup of $\mb{U}_n$ for some $n \geq 1$
  \cite[Theorem 4.8]{bor91}. Products of
  unipotent groups are unipotent, and all subgroups of unipotent groups are
  unipotent. 
\end{eg}
Given a linear algebraic group $H$, its quotient by the unipotent radical,
\[
H_r:=H/H_u
\]
is a reductive group. Moreover, given a variety $X$ and a normal subgroup $N$ of
$H$, it is easy to show that if $X$ has a geometric $N$-quotient $X/N$ and $X/N$
has a geometric $H/N$-quotient $(X/N)/(H/N)$, then the geometric quotient for
the $H$-action on $X$ exists, with $X/H=(X/N)/(H/N)$. This suggests that a
natural way to construct geometric quotients by general 
linear algebraic groups is to try and understand the construction of quotients
for unipotent groups and reductive groups. 

Reductive groups and unipotent groups behave rather differently from the point
of view  
of invariant theory. On the one hand, reductive groups have a well behaved
representation theory. Any representation of a reductive group can be
written as a direct sum of irreducible representations \cite[\S 4.6.6]{spr94},
which has a number of important consequences. Firstly, given a finitely generated
$\kk$-algebra $A$ and a reductive group $G$ acting on $A$ in a locally finite
fashion, the ring of invariants $A^G$ is 
also finitely generated over $\kk$ by a theorem of Nagata \cite{nag64}. Thus
$\spec(A^G)$ is an affine \emph{variety}. Secondly, one has the following result (see 
\cite[Lemma
5.1.A]{nag64}): if $I \subseteq A$ is a
$G$-invariant ideal, then any invariant element of $A/I$ lifts to an
invariant in $A$; geometrically stated, any invariant regular
function on a $G$-invariant closed subset $Z$ of $\spec A$ extends to
$G$-invariant regular function on the whole of $\spec A$. Thirdly, given two
distinct ideals $I,J$ of 
$A$ invariant under the $G$-action and such that $I+J=A$, one can always find an
element of $A^G$ 
contained in one but not the other (this follows from a result of Haboush
\cite{hab75}, see also \cite[Lemma 3.3]{new78}). Geometrically this 
says that any two disjoint closed invariant subsets of $\spec A$ can be
separated by an invariant function---this implies property \ref{itm:TrRe1-2c}
of Definition \ref{def:TrRe1} of a good quotient. The upshot
is that reductive group actions on affine varieties are amenable to constructing good
quotients; this will be the content of Theorem \ref{thm:BgRe1} in the next section.    

All three of the above properties fail for non-reductive group actions. The issue of
whether the ring of invariants is finitely generated, which is closely related
to the fourteenth problem of Hilbert\footnote{A good survey of counterexamples to
  Hilbert's fourteenth problem from an invariant theoretic perspective is 
\cite{fre01}.}, has arguably received the most attention
historically. The
following celebrated example of Nagata \cite{nag59} demonstrates that the ring of
invariants for a non-reductive group need not be finitely generated. (We follow
the exposition of \cite[\S 4.3]{dol03}.)
\begin{eg} \label{ex:TrRe2.4} Given $n>0$, let $X=\kk^{2}
  \oplus \dots \oplus \kk^2$ ($n$ times)
  and consider the action of $H^{\prime}_1 \times \dots \times H^{\prime}_n$ on $X$,
  where $H^{\prime}_i=\{\left(\begin{smallmatrix} c_i & a_i \\ 0 &
      c_i \end{smallmatrix} \right) \mid a_i,c_i \in \kk, \ c_i \neq 0\}$ acts
  on the $i$-th factor of $\kk^2$ in $X$ by usual matrix multiplication. Let
  $H \subseteq H^{\prime}_1 \times \dots \times H^{\prime}_n$ be the subgroup
  obtained by demanding that $c_1\cdots c_n=1$ and $(a_1,\dots,a_n)$ satisfy
  three suitable 
  linear equations $\sum_j x_{i,j}a_j = 0$, ($i=1,2,3$). Then for $n=16$, the
  ring of invariants $\OO(X)^H$ is not finitely generated over
  $\kk$. It follows that $\OO(X)^{H_u}$ is not
  finitely generated over $\kk$, 
  where $H_u$ is the unipotent radical of $H$ defined by $c_i=1$ for each
  $i=1,\dots,n$ (see \cite{nag60}). 
\end{eg}

Despite suffering representation-theoretic deficiencies from not being
reductive, unipotent group actions 
nevertheless have their own distinctive invariant theoretic flavour; indeed,
topologically they can be better behaved that reductive groups. For example,
every action of a unipotent group on an affine variety is closed. Somewhat more
strikingly, every unipotent group is \emph{special}: any principal bundle of a    \index{special group}
unipotent group is \emph{Zariski}-locally trivial, not just locally trivial in
the isotrivial topology \cite[Proposition 14]{ser58}. Finally, affine varieties
that admit affine locally trivial quotients are easily recognisable thanks to
the next result.  
\begin{propn} \label{prop:TrRe3.1} Suppose $X$ is an affine 
variety acted
  upon by a unipotent group $U$ and a locally trivial quotient $X \to X/U$
  exists. Then $X/U$ is affine if and only if $X \to X/U$ is a trivial
  $U$-bundle.   
\end{propn}

\begin{proof} This is an immediate consequence of  \cite[Theorem 3.14]{ad07}, which involves
 cohomological techniques from \cite{gp93}. However these techniques are not needed for this result: it is easy to prove by induction on the dimension of $U$ that every principal $U$-bundle over an affine base is trivial. Conversely if $X \to X/U$ is a trivial $U$-bundle, so that $X \cong X/U \, \times \, U$, then the unit morphism $\spec \kk \to U$ gives us a closed immersion $X/U \to X$ and so $X/U$ is affine.
\end{proof}

\subsection{Mumford's Geometric Invariant Theory for Reductive Groups}
\label{sec:BgReductive}

In the first edition of \cite{mfk94} Mumford introduced his geometric invariant
theory (GIT) to give both a theoretical framework and 
computational tools for finding invariant open sets of points inside a $G$-variety
$X$ that admit geometric quotients, when $G$ is a reductive linear algebraic
group. On the theoretical side, a basic strategy for constructing such open
subsets is to patch together quotients of open affines arising from the data
of a linearisation. Because the invariant theory of reductive groups is well
behaved, such quotients are easy to describe.          

\begin{thm} \label{thm:BgRe1} \cite[Chapter 1, \S 2]{mfk94}\footnote{While Theorem 
\ref{thm:BgRe1},
    \ref{itm:BgRe1-2} is not stated 
    explicitly in \cite[Chapter 1, \S 2]{mfk94}, it follows easily from the
    material there, together with the Closed Orbit Lemma \cite[Proposition
    1.8]{bor91} and the lower semi-continuity of the 
    function $x \mapsto \dim(H \cdot x)$. See \cite[Proposition 3.8]{new78} for
    a proof.} Let $X=\spec A$ be an affine variety upon which a
reductive group $G$ acts. Then
\begin{enumerate}
\item \label{itm:BgRe1-1} the natural map $\phi:X \to \spec (A^G)$ induced by
  the inclusion 
$A^G \hookrightarrow A$ is a good categorical quotient; and
\item \label{itm:BgRe1-2} the set $U:=\{x \in X \mid G \cdot x \text{ \upshape{is
      closed and} } \stab_G(x) \text{ \upshape{is finite}}\}$ is an open subset of $X
$, and the
restriction of $\phi$ to $U$ gives a geometric quotient $U \to
\phi(U)$ for the $G$-action on $U$, with $\phi(U)$ open in $\spec(A^G)$. 
\end{enumerate}
\end{thm} 
To deal with the more general case where $G$ acts on any 
variety $X$, Mumford used the extra data of a linearisation $L \to X$ to define
$G$-invariant open subsets which are obtained by patching affine open subsets of
the form $X_f$, for $f$ an invariant section of a positive tensor power of $L
\to X$.   

\begin{defn} \label{def:TrRe5} \cite[Chapter 1, \S 4]{mfk94} Let $X$ be a
  $G$-variety and $L \to X$ a linearisation. A point $x \in X$ is called
\begin{enumerate} 
\item \label{itm:TrRe5-1} \emph{semistable} if there is an invariant  \index{semistable}
$f \in 
H^0(X,L^{\otimes r})^G$, with $r>0$, such that $f(x) \neq 0$ and $X_f$ is affine; and
\item \label{itm:TrRe5-2} \emph{stable} if there is  \index{stable}
  an invariant $f\in H^0(X,L^{\otimes r})^G$, with $r>0$, such that $X_f$ is
  affine, the $G$-action on $X_f$ is closed and $\stab_G(y)$ is finite for all
  $y \in X_f$.
\end{enumerate} 
\end{defn}
We denote the subset of semistable (respectively, stable) points 
by $X^{\rmss(L)}$ (respectively $X^{\rms(L)}$), dropping the mention of $L$ if there 
is
no risk of confusion. Note that we have followed \cite[Chapter 3, \S 5]{new78} in 
requiring
finite stabilisers for Definition
\ref{def:TrRe5}, \ref{itm:TrRe5-2} of `stable'; this 
corresponds to Mumford's definition of `properly stable' in \cite[Definition 1.8]
{mfk94}.  

\begin{rmk} \label{rmk:BgRe3} The sets $X^{\rmss}$ and $X^{\rms}$ are
  $G$-invariant open subsets of
$X$ that may be defined for rational linearisations    \index{rational linearisation}  \index{fractional linearisation}
$\mc{L}$, in the following way: if $n>0$ is an integer such that $L=n\mc{L}$ is in $\Pic^G(X)$, then
define $X^{\rmss(\mc{L})}=X^{\rmss(L)}$ and $X^{\rms(\mc{L})}=X^{\rms(L)}$. Since 
$X_{f}=X_{f^m}$ for any global section $f$ of a line bundle and any integer $m>0$, this is
well-defined by Remark \ref{rmk:TrRe3.2}. In fact, for ample $L \to X$ the sets $X^{\rms(L)}$ and $X^{\rmss(L)}$ depend only on the \emph{fractional linearisation} class of $L$, in the sense of Thaddeus \cite{tha96}. 
\end{rmk}

The definitions of semistability and stability, respectively, are so specified as to 
allow one
to glue the good or geometric quotients $X_f \to \spec(\OO(X_f)^G)$,
respectively, of the affine $X_f$. The central result of Mumford's geometric
invariant theory says that these quotients can be glued into \emph{quasi-projective
varieties}.  

\begin{thm} \label{thm:BgRe2} \cite[Theorem 1.10]{mfk94} Let $G$ be a reductive
  group acting on a variety $X$ and $L \to X$ a linearisation for the action. Then
\begin{enumerate}
\item \label{itm:BgRe2-1} the semistable locus $X^{\rmss}$ has a good
  categorical quotient 
  $\phi:X^{\rmss} \to X \dblslash G$ onto a quasi-projective variety $X \dblslash G$, 
and there
  is an ample line bundle $M \to X \dblslash G$ pulling back to a positive
  tensor power of $L|_{X^{\rmss}}$ under $\phi$; and 
\item \label{itm:BgRe2-2} the image of $X^{\rms}$ under $\phi$ is an open subset of $X 
\dblslash G$, and the
restriction of $\phi$ to $X^{\rms}$ gives a geometric quotient $\phi:X^{\rms} \to
\phi(X^{\rms})$ for the action of $G$ on $X^{\rms}$.
\end{enumerate}
\end{thm}
The variety $X \dblslash G$ is called the \emph{GIT quotient}  \index{GIT quotient} for the
linearisation $G \act L \to X$. By a result of Seshadri \cite[Proposition
9]{ses77} the GIT quotient $X \dblslash G$ can be regarded topologically as the
quotient $X^{\rmss}/ \! \! \sim$ of $X^{\rmss}$ under the `S-equivalence' relation $
\sim$, where $x_1
\sim x_2$ if, and only if, the closures of $G \cdot x_1$ and $G \cdot x_2$ in
$X^{\rmss}$ intersect nontrivially.

The affine case of Theorem \ref{thm:BgRe1} can be recovered from Theorem
\ref{thm:BgRe2} by considering the linearisation of $L=\OO_X \to X$ defined by the
trivial character from Examples \ref{ex:TrRe2.1} and \ref{ex:TrRe3.1}: in
this case the constant function $1 \in H^0(X,L)$ is an invariant, so
$X^{\rmss(\OO_X)}=X$, and it follows immediately that $U$ from Theorem
\ref{thm:BgRe1} is equal to the stable locus $X^{\rms(\OO_X)}$. 

Another important special case of Theorem \ref{thm:BgRe2} is when $X$ is
\emph{projective} and $L \to X$ is an  
\emph{ample} linearised line bundle. In this case the ring of invariant sections
$S^G$ is a finitely generated $\kk$-algebra by Nagata's theorem \cite{nag64},
and $X \dblslash G = \proj(S^G)$ is the projective variety associated to the
graded ring $S^G$ \cite[Page
40]{mfk94}. The good quotient $\phi:X^{\rmss} \to X \dblslash G$ is
a representative of the rational map $X \dashrightarrow \proj(S^G)$ induced by
the inclusion $S^G \hookrightarrow S$, thus we have a commutative diagram, with
all inclusions open:
\begin{equation} \label{eq:BgRe1}
\begin{diagram}
X^{\rms} & \rEmpty~\subseteq & X^{\rmss} & \rEmpty~\subseteq & X  \\
\dTo^{geo}_{\phi} & & \dOnto^{good}_{\phi} & & \dDashto\\
X^{\rms}/G & \rEmpty~\subseteq & X \dblslash G & \rEmpty~\subseteq &
\proj(S^G)\\
\end{diagram}
\end{equation}
Note in the case that the GIT quotient $X
\dblslash G$ may therefore be regarded as a canonical compactification of the 
geometric
quotient $X^{\rms}/G$ of the stable locus.  

Another appealing feature of the case where $X$ is projective and $L \to X$ is ample
is that there is an effective way to compute the semistable and stable loci, via
the Hilbert-Mumford criteria. For completeness we present here two versions of
this result. Firstly, if $G$ is any reductive group, 
then a \emph{one-parameter subgroup} of $G$ (or \emph{1-PS} for short) is simply  \index{one-parameter subgroup}  \index{1-PS}
a non-trivial homomorphism $\lambda:\GG_m \to G$. Given a point $x \in X$ and a
one-parameter subgroup $\lambda$, the limit\footnote{By `$\lim_{t \to 0}
  \lambda(t) \cdot x$' we mean the value 
at $0 \in \kk$ of $\phi:\kk \to X$, where $\phi$ is the unique extension of the
morphism of varieties $t \mapsto \lambda(t) \cdot x$ to a morphism on $\kk$.}
$x_0:=\lim_{t \to 0} \lambda(t) \cdot x$  
is a well-defined point in $X$, because $X$ is proper. The subgroup $\lambda$
fixes $x_0$, hence $\GG_m$ acts on the fibre $L|_{x_0}$ over $x_0$ via $\lambda$; 
define
\[
\mu(x,\lambda):= \text{weight for the $\GG_m$-action on $L|_{x_0}$}.
\]    
Then the first form of the Hilbert-Mumford criteria can be stated as follows.
\begin{thm} \label{thm:BgRe3} \cite[Theorem 2.1]{mfk94}
Let $G$ be a reductive group, $X$ a projective $G$-variety and $L \to X$ an
ample linearisation. Then for any point $x \in X$,
\begin{align*}
x \in X^{\rmss} &\iff \mu(x,\lambda) \geq 0 \text{ for all 1-PS
  $\lambda:\GG_m \to G$;} \\
x \in X^{\rms} &\iff \mu(x,\lambda) > 0 \text{ for all 1-PS
  $\lambda:\GG_m \to G$.} 
\end{align*}   
\end{thm}
The second form of the Hilbert-Mumford criteria makes use of an
embedding in a projective space. Suppose still that $X$ is a projective
$G$-variety, with $G$ reductive, but now assume $L \to X$ is a very ample
linearisation. As in Example \ref{ex:TrRe3.2}, $X$ embeds into the projective
space $\PP(V)$ equivariantly, where $V=H^0(X,L)^*$. Fix a maximal torus $T
\subseteq G$ and let $\Hom(T,\GG_m)$ be the abelian group of characters of
$T$. The action of $T$ on $V$ is diagonalisable 
\cite[Proposition 8.4]{bor91}, so we may decompose $V$ into $T$-weight spaces:
\[
V=\bigoplus_{\chi \in \Hom(T,\GG_m)} V_{\chi}, \quad V_{\chi}:=\{v \in V \mid t
\cdot v = \chi(t)v \text{ for all } t \in T\}. 
\]   
Given $x \in X$, write $x = [v] \in \PP(V)$ with $v \in V \setminus \{0\}$ and let
$v=\sum_{\chi} v_{\chi}$ with $v_{\chi} \in V_{\chi}$. Define the \emph{weight 
polytope} of $x$ to be  \index{weight polytope}
\[
\Delta_x:=\ol{\text{convex hull of $\{\chi \mid v_{\chi} \neq 0\}$}}
\subseteq \Hom(T,\GG_m) \ten_{\mb{Z}} \mb{R}, 
\]
where the closure is taken with respect to the usual Euclidean topology on the
vector space $\Hom(T,\GG_m) \ten_{\mb{Z}} \mb{R}$. Denote the interior of
$\Delta_x$ inside $\Hom(T,\GG_m) \ten_{\mb{Z}} \mb{R}$ by
$\Delta_x^{\circ}$. Then the Hilbert-Mumford criteria can be stated in the
following way (see \cite[Theorem 9.2]{dol03} and \cite[Theorem 9.3]{dol03}).    
\begin{thm} \label{thm:BgRe4} 
Retain the preceding notation.
\begin{enumerate}
\item \label{itm:BgRe4-1} A point $x \in X$ is semistable (respectively,
stable) for $G \act L \to X$ if, and only if, for each $g \in G$ the point $gx$
is semistable (respectively, stable) for the restricted linearisation $T \act L
\to X$.
\item \label{itm:BgRe4-2} For any point $x \in X$ we have
\begin{align*}
\text{x is semistable for $T \act L \to X$} &\iff 0 \in \Delta_x; \\
\text{x is stable for $T \act L \to X$} &\iff 0 \in \Delta_x^{\circ}.
\end{align*}  
\end{enumerate}
\end{thm}

Thus we see that in the case where $X$ is projective and $L \to X$ is very ample,
semistability and stability can be computed in terms of weights for torus actions.


\subsection{Geometric Invariant Theory for Unipotent Groups}
\label{sec:BgUnipotent}

Given the effectiveness of Mumford's GIT for studying
quotients of reductive groups, there has been interest in developing a similar
geometric invariant theoretic approach to studying unipotent group actions. Such a
programme is taken up in the paper \emph{Towards non-reductive geometric
  invariant theory} \cite{dk07}, building on previous work such as
\cite{fau83,fau85,gp93,gp98,win03}. Given an irreducible projective variety $X$
with an action 
of a unipotent group $U$ and an ample linearsation $L \to X$ for the
action, the paper considers various notions of `stability' (intrinsic to the
data of the linearisation $L \to X$) that admit geometric
quotients, formulates an analogue of Mumford's reductive GIT
quotient in this context, and relates these notions to those of reductive GIT
for the purposes of computation. In this section we summarise the main
definitions and results presented there, which will form 
the backbone of our development of a geometric invariant theory for more general
linear algebraic 
groups in upcoming sections. We also take the opportunity to point out some
errors in \cite{dk07}, but leave details of how to correct them to Section 
\ref{sec:gitNonReductive}.

We assume for the rest of this section that $U$ is a unipotent group acting on an
irreducible projective variety $X$ with ample linearisation $L \to X$. 


\subsubsection{Intrinsic Notions of Semistability and Stability}
\label{subsec:BgUnIntrinsic}

As in Section \ref{subsec:BgGrLinearisations}, let $S=\kk[X,L]$ be the section ring. 
The inclusion $S^U \hookrightarrow S$ defines a rational
map
\[
q:X \dashrightarrow \proj(S^U)
\]
which is $U$-invariant on its maximal domain of definition. 
\begin{defn} \label{def:BgUn1} \cite[Definition 4.1.1]{dk07} Let $U$ be a
  unipotent linear algebraic group acting on an irreducible 
projective variety $X$ and $L \to X$ an ample linearisation. The \emph{naively
  semistable locus} is the open subset   \index{naively semistable}
\[
X^{\nss}:=\bigcup_{f \in I^{\nss}} X_f
\]
of $X$, where $I^{\nss}:=\bigcup_{r>0}H^0(X,L^{\ten r})^U$ is the set of invariant
sections of positive tensor powers of $L$.
\end{defn}
The rational map $q$ restricts to define a $U$-invariant morphism $q:X^{\nss}
\to \proj(S^U)$. As Nagata showed \cite{nag59}, the ring of
invariants $S^U$ need not be finitely generated over $\kk$, so $\proj(S^U)$ is in 
general a
non-noetherian scheme. It can also happen that this map is not surjective, with
the image only a dense constructible subset of
$\proj(S^U)$ in general, even if $\proj(S^U)$ is of finite type (an example of
this phenomenon---which features later in Example \ref{ex:BgUn16}---is given in 
\cite[\S
6]{dk07}). To address the first of these issues, it is natural to consider the
following subset of $X^{\nss}$.   
\begin{defn} \label{def:BgUn2} \cite[Definition 4.2.6]{dk07} Let $U$ be a
  unipotent linear algebraic group acting on an irreducible 
projective variety $X$ and $L \to X$ an ample linearisation. The
\emph{finitely generated semistable locus} is the open subset  \index{finitely generated semistable}
\[
X^{\ssfg}:= \bigcup_{f \in I^{\ssfg}} X_f
\]
of $X^{\nss}$, where 
\[
I^{\ssfg}:=\{f \in \textstyle{\bigcup_{r>0}} H^0(X,L^{\ten
r})^U \mid \OO(X_f)^U \text{ is a finitely generated $\kk$-algebra} \}.
\]  
\end{defn}
The image of $X^{\ssfg}$ under the map $q:X^{\nss} \to \proj(S^U)$ is contained
in the open subscheme of $\proj(S^U)$ obtained by patching together the affine
open subsets $\spec(\OO(X_f)^U)$ for which $\OO(X_f)^U$ is a finitely generated
$\kk$-algebra. 
\begin{defn} \label{def:BgUn3} \cite[Definition 4.2.7]{dk07} Let $U$ be a
  unipotent linear algebraic group acting on an irreducible 
projective variety $X$ and $L \to X$ an ample linearisation. The
\emph{enveloping quotient} is the open subscheme  \index{enveloping quotient}
\[ 
X \dblslash U:= \bigcup_{f \in I^{\ssfg}} \spec(\OO(X_f)^U)
\subseteq \proj(S^U) 
\]
of $\proj(S^U)$, together with the canonical map $q:X^{\ssfg} \to X \dblslash U$. The 
image
$q(X^{\ssfg})$ of this map is called the \emph{enveloped quotient}.  \index{enveloped quotient}
\end{defn} 

The enveloping quotient $X \dblslash U$ is
  canonical to the data of the linearisation and, as we will shortly see, in some
  sense plays the role of Mumford's reductive GIT quotient \cite{mfk94}. But
  there are two significant differences from the 
reductive case to be aware of (compare with the discussion after Theorem
\ref{thm:BgRe2}). Firstly, $X \dblslash U$ is not a projective variety in
general (however if
$S^U$ is a finitely generated $\kk$-algebra then $X^{\ssfg}=X^{\nss}$ and $X
\dblslash U=\proj(S^U)$ is a projective variety). Secondly, the map $q:X^{\ssfg}
\to X \dblslash U$ is not necessarily surjective and the image $q(X^{\ssfg})$
is not necessarily a variety (even when
$S^U$ is finitely generated). In particular, neither $X \dblslash U$ nor
$q(X^{\ssfg})$ are a categorical quotient of $X^{\ssfg}$ in general. 

\begin{rmk} \label{rmk:BgUn4}
 The enveloping quotient $X \dblslash U$ is a scheme \emph{locally} of finite
 type. In \cite[Proposition 4.2.9]{dk07} it is erroneously claimed that $X
 \dblslash U$ is a quasi-projective variety. The basic problem is that it is not
 necessarily quasi-compact: the ideal  
 in $S^U$ generated by $I^{\ssfg}$ may not satisfy the ascending chain
 condition and we cannot guarantee that finitely many of the affine open subsets
 $\spec(\OO(X_f)^U)$, for $f \in I^{\ssfg}$, cover $X \dblslash U$. (In the
 proof of \cite[Proposition 4.2.9]{dk07} it is implicitly assumed
 such a finite cover of $X \dblslash U$ exists in order to
 construct an embedding of $X \dblslash U$ into a projective space.) 
Of course $X^{\ssfg}$ and $X^{\nss}$ are quasi-compact because they are quasi-projective.
Geometrically
 speaking, the problem with the enveloping quotient $X \dblslash U$ is that even though finitely many of the open sets $X_f$,
 with $f \in I^{\ssfg}$, 
 cover $X^{\ssfg}$, the enveloping quotient map $q:X^{\ssfg} \to X \dblslash U$ is
 not surjective in general. However, if either $S^U$ is a finitely
generated $\kk$-algebra or the enveloping quotient map $q:X^{\ssfg} \to X
\dblslash U$ is surjective then the proof of \cite[Proposition 4.2.9]{dk07} goes
through to show that $X \dblslash U$ is a quasi-projective variety.         
\end{rmk}

The finitely generated semistable locus $X^{\ssfg}$ is analogous to Mumford's
notion of semistability in reductive GIT (cf. Definition \ref{def:TrRe5},
\ref{itm:TrRe5-1}), and 
indeed in \cite[Definition 5.3.7]{dk07} the set $X^{\ssfg}$ is dubbed the
`semistable' locus for the 
linearisation $U \act L \to X$. (In this article we will refrain from referring to
$X^{\ssfg}$ as the `semistable set' to preserve continuity with the more general
non-reductive setting, to be discussed in Section \ref{sec:gitNonReductive}.)

Various kinds of `stable' set are also considered in \cite{dk07}, each of which are 
subsets of
$X^{\ssfg}$ whose images under the 
enveloping quotient map $q:X^{\ssfg} \to X \dblslash U$ define geometric
quotients. One of the conclusions of that paper is that the following `locally
trivial' version of stability is well suited to studying linearised actions of
unipotent groups.   

\begin{defn} \label{def:BgUn5} \cite[Definition 4.2.6]{dk07} Let $U$ be a
  unipotent linear algebraic group acting on an irreducible 
projective variety $X$ and $L \to X$ an ample linearisation. The set of
\emph{locally trivial stable} points (later called the set of \emph{stable}  \index{locally trivial stable} \index{stable}
points in \cite[Definition 5.3.7]{dk07}) is the set
\[ 
X^{\rms}= X^{\textrm{lts}}=\bigcup_{f\in I^{\textrm{lts}}}X_f,
\] 
where
\[ 
I^{\textrm{lts}}:= \left\{ \begin{array}{c|c} \multirow{2}{*}{$f\in
\bigcup_{r>0}H^0(X,L^{\ten m})^U$} & \OO(X_f)^U \text{ is a finitely generated
$\kk$-algebra and } \\ 
& \text{$q:X_f \to \spec (\OO(X_f)^U)$ is a trivial $U$-bundle} 
\end{array} \right \}.
\]
\end{defn} 

\begin{propn} \label{prop:BgUn6} \cite[\S 4]{dk07} Let $U$ be a
  unipotent linear algebraic group acting on an irreducible 
projective variety $X$ and $L \to X$ an ample linearisation. The image
$q(X^{\rms})$ of $X^{\rms}$ under the enveloping quotient map $q:X^{\ssfg} \to X 
\dblslash U$ is
an open subscheme of $X \dblslash U $ that is a quasi-projective variety, and 
$q:X^{\rms} \to
q(X^{\rms})$ is a geometric quotient:
\[
\begin{diagram}
X^{\rms} & \rEmpty~\subseteq & X^{\ssfg} & \rEmpty~\subseteq & X^{\nss}  \\
\dTo^{geo} & & \dTo & & \dTo_q\\
q(X^{\rms}) & \rEmpty~\subseteq & X \dblslash U & \rEmpty~\subseteq &
\proj(S^U)\\
\end{diagram}
\] 
\end{propn}
It is helpful to compare this to the case where $G$ is reductive and $L \to X$
is an ample linearisation over a projective $G$-variety. The diagram in
Proposition \ref{prop:BgUn6} is similar to \eqref{eq:BgRe1}, but the unipotence
of $U$ leads to the two main differences mentioned earlier: the enveloping quotient
$X \dblslash U$ need not be projective and $q:X^{\ssfg} \to X \dblslash U$ is
not in general a good categorical quotient.  

\begin{rmk} \label{rmk:BgUn7}
It is clear from Remark \ref{rmk:TrRe3.2} and the definitions that $X^{\nss}$, $X^{\ssfg}$, $X^{\rms}$ and $X \dblslash U$ may be defined for rational linearisations.  
\end{rmk}

\subsubsection{Extending to Reductive Linearisations}
\label{subsec:BgUnExtending}

For the rest of this section, \emph{we assume that the
  linearisation $U \act L \to X$ is such that $X \dblslash U$ is
  quasi-projective} (see Remark \ref{rmk:BgUn4}). 

A rather helpful approach to studying the $U$-linearisation $L \to X$, and
the spaces $X^{\rms}$, $X^{\ssfg}$ and $X
\dblslash U$ thus arising, is to construct an associated linearisation of a
reductive group $G$ which contains $U$ as a closed subgroup, by making use of
the fibre space associated to the homogeneous space $G/U$. We take a moment to recall 
the general construction of such
fibre spaces.   

Let $H_1$ and $H_2$ be linear algebraic groups and suppose $H_1 \hookrightarrow
H_2$ is a closed embedding. For the moment suppose also that $X$ is any
$H_1$-variety. Then we may consider the diagonal action of $H_1$ on
the product $H_2 \times X$: 
\[
H_1 \act H_2 \times X, \quad h_1 \cdot (h_2,x):=(h_2h_1^{-1},h_1x), 
\]
where $h_1 \in H_1$, $h_2 \in H_2$ and $x\in X$. If $H_1$ is unipotent, or if
$X$ satisfies some mild assumptions---for example, if $X$ is quasi-projective,
or more generally if any finite 
set of points in $X$ is contained in an affine open subset---the geometric quotient
$H_2 \times^{H_1} X$ for this action exists as a variety \footnote{Indeed,
$H_2 \times X \to H_2 \times^{H_1} X$ is in fact a
principal $H_1$-bundle \cite[Proposition 4]{ser58}.} (see \cite[Proposition
23]{eg98} in case $H_1$ is unipotent,\footnote{When $H_1$ is unipotent the $H_1$-bundle $H_2 \to H_2/H_1$ is Zariski locally trivial and thus the geometric quotient $H_2 \times^{H_1} X$ can be shown to exist by working locally over $H_2/H_1$.
} or \cite[Theorem 4.19]{vp94}
otherwise). This quotient is the
associated fibre space of the principal $H_1$-bundle $H_2 \to H_2/H_1$ with
fibre $X$; see \cite[\S 3.2]{ser58}. We shall write points in $H_2
\times^{H_1} X$ as equivalence classes $[h_2,x]$ of points $(h_2,x) \in H_2
\times X$. The action of $H_2$ on 
$H_2 \times^{H_1} X$ induced by left multiplication of $H_2$ on
itself makes $H_2 \times^{H_1} X$ into an $H_2$-variety. Note there is a natural 
closed immersion
\[
\alpha:X  \hookrightarrow H_2 \times^{H_1} X, \quad 
x  \mapsto [e,x],   
\]
which is $H_1$-equivariant with respect to $H_1$ acting on $H_2 \times^{H_1} X$
through the action of $H_2$.

Suppose $L \to X$ is a linearisation for the $H_1$-action on $X$. Then this
extends to a natural
$H_2$-linearisation $H_2 \times^{H_1} L \to H_2 \times^{H_1} X$ which pulls back to $L 
\to X$ under
$\alpha$, using the same constructions as above. 
For brevity, we will usually abuse
notation and write $L \to H_2 \times^{H_1} X$ for this linearisation instead of
$H_2 \times^{H_1} L$, unless confusion is
likely to arise. Observe that, because the projection $H_2 \times X \to H_2
\times^{H_1} X$ is a categorical quotient, pullback along $\alpha$ induces an
isomorphism of graded rings
\[
\alpha^*:\kk[H_2 \times^{H_1} X,L]^{H_2} \overset{\cong}{\longrightarrow}
\kk[X,L]^{H_1}.
\] 

Let us now return to the setting where $U$ is a unipotent group acting on an
irreducible projective variety $X$ with ample linearisation $L \to X$. Following
\cite[\S 5.1]{dk07}, given a
closed embedding of $U$ into some reductive group $G$ (e.g. $G=\GL(n,\kk)$ for
suitable $n$), consider the
$G$-linearisation $G \act L=G \times^U L \to G \times^U X$. This is a 
linearisation over the quasi-projective variety $G \times^U X$, so it makes
sense to ask for semistability and stability, in the sense of Mumford's
reductive GIT (Definition \ref{def:TrRe5}).   
\begin{defn} \label{def:BgUn8} \cite[Definition 5.1.6]{dk07} Let $U$ be a
  unipotent group contained in a 
  reductive group $G$ as a closed subgroup and let $L \to X$ be an ample
  $U$-linearisation over a projective $U$-variety $X$. 
Define the set of \emph{Mumford stable} points to be  \index{Mumford stable}
\[
X^{\textrm{ms}}:=\alpha^{-1}((G \times^U X)^{\rms})
\]
and the set of \emph{Mumford semistable} points to be  \index{Mumford semistable}
\[
X^{\textrm{mss}}:=\alpha^{-1}((G \times^U X)^{\rmss})
\] 
where (semi)stability of $G \times^U X$ is defined as in Definition
\ref{def:TrRe5} with respect to the $G$-linearisation $G \times^U L \to G
\times^U X$, and $\alpha:X \hookrightarrow G \times^U X$ is the natural closed
immersion. 
\end{defn}
These sets would appear to depend on the choice of $G$ and embedding $U
\hookrightarrow G$, but in fact this is not the case by virtue of the following
result.   
\begin{propn} \label{prop:BgUn9} \cite[Lemma 5.1.7 and Proposition 5.1.10]{dk07}
Given a unipotent group $U$, a reductive group $G$ containing $U$ as a closed
subgroup and an ample $U$-linearisation $L \to X$ of a projective $U$-variety
$X$, we have equalities 
\[
X^{\mathrm{mss}}=X^{\mathrm{ms}}=X^{\mathrm{lts}}.
\]
\end{propn}
There are two main facts used in the proof that
$X^{\mathrm{mss}}=X^{\mathrm{ms}}$. The first is that any stabiliser
of a point with a closed $G$-orbit in $(G \times^U
X)^{\rmss}$ must have a \emph{reductive} stabiliser (by Matsushima's
criterion \cite[Theorem 4.17]{vp94}) that is also a subgroup of $U$, hence is
trivial. The second is that any $U$-orbit of a $U$-stable affine subvariety is
necessarily closed. Both of these rely on the unipotency of $U$ in an
essential way. The equality $X^{\mathrm{ms}}=X^{\mathrm{lts}}$ is established by
using descent to relate the notions of $U$-local triviality of suitable affine
open subsets of $X$ with $G$-local triviality of the corresponding subsets in $G
\times^U X$.   

As we observed above, $G \times^U X$ is only a \emph{quasi}-projective variety, so
computing (semi)stability for the linearisation $G \times^U L \to G \times^U X$
is difficult in general; on the other hand, reductive GIT is very effective at
dealing with
ample linearisations over \emph{projective} varieties. Therefore it is
reasonable to study $G$-equivariant projective completions $\ol{G  
  \times^U X}$ of $G \times^U X$, 
together with extensions $L^{\prime} \to \ol{G \times^U X}$ of the linearisation
$L \to G \times^U X$, in a bid to 
compute the stable locus $X^{\rms}=X^{\textrm{lts}}$ for $U \act L \to X$ and
study completions of the enveloping quotient $X \dblslash U$. More
precisely, the strategy adopted in \cite{dk07} is to look for $G \act L^{\prime}
\to \ol{G \times^U X}$ such that 
\begin{enumerate}
\item[(1)] \label{itm:BgUn10-1} the pre-image of the stable locus of $L^{\prime} \to
  \ol{G \times^U X}$ under $X \hookrightarrow \ol{G \times^U X}$ is contained in
  $X^{\rms}=X^{\textrm{lts}}$; and
\item[(2)] \label{itm:BgUn10-2} there is a naturally induced open embedding of $X
  \dblslash U$ into the GIT quotient $\ol{G \times^U X} \dblslash_{L^{\prime}}
  G$.    
\end{enumerate}
The following definition,
which can be regarded as an enhanced version of a collection of `separating
invariants' in \cite[Definition 2.3.8]{dk02}, facilitates this.\footnote{The
  definition we give is stated in a way 
that corrects a couple of small errors in the original \cite[Definition 5.2.2]{dk07}.}      

\begin{defn} \label{def:BgUn11} Let $U$ be a unipotent group acting on
a projective variety $X$, with ample linearisation $L \to X$, and let $G$ be a
reductive group containing $U$ as a closed 
subgroup. A finite collection $\mc{A} \subseteq \bigcup_{r>0}H^0(X,L^{\ten r})^U$ is
called a \emph{finite fully separating set of invariants} if  \index{finite fully separating set of invariants}
\begin{enumerate}
\item \label{itm:BgUn11-1} $X^{\nss}=\bigcup_{f \in \mc{A}} X_f$ and the set $\mc{A}$ 
is  \index{separating set of invariants}
  \emph{separating}: whenever $x,y
  \in X^{\nss}$ are distinct points and there exist 
$U$-invariant sections $g_0,g_1 \in H^0(X,L^{\ten r})^U$ (for some $r>0$) such that
$g_0(x) \neq 0$, $g_1(y) \neq 0$ and $[g_0(x):g_1(x)] \neq [g_0(y):g_1(y)]$ (as
points in $\PP^{1}$), then there are sections $f_0,f_1 \in \mc{A}$ of some common
tensor power of $L$ such that $f_0(x) \neq 0$, $f_1(y) \neq 0$ and
$[f_0(x):f_1(x)] \neq [f_0(y):f_1(y)]$;
\item \label{itm:BgUn11-2} for every $x \in X^{\rms}$ there is $f\in \mc{A}$
  with associated $G$-invariant $F$ such that $x\in (G \times_U X)_F$ and $(G
  \times_U X)_F$ is affine; and
\item \label{itm:BgUn11-3} we have $X \dblslash U \subseteq \bigcup_{f \in \mc{A}}
  \spec(\OO(X_f)^U) \subseteq \proj(S^U)$, and for every $x\in X^{\ssfg}$ there
  is $f\in \mc{A}$ such that $x\in X_f$ and 
$\OO(X_f)^U \cong \kk[\mc{A}]_{(f)}$ (where $\kk[\mc{A}]$ is the graded subalgebra of
$S^U=\kk[X,L]^U$ generated by $\mc{A}$). 
\end{enumerate}
\end{defn}
The existence of a finite fully separating set of invariants follows by a
suitable application of Hilbert's Basis Theorem inside $\kk[X,L]$ and the
quasi-compactness of $X^{\nss}$, $X^{\ssfg}$ and $X \dblslash U$ (recall from the
assumption made at the start of this section that $X \dblslash U$ is quasi-projective, as are  $X^{\nss}$ and $X^{\ssfg}$). The salient conditions in Definition \ref{def:BgUn11} relevant to points
(1) and (2) 
 above are the conditions
\ref{itm:BgUn11-2} and 
\ref{itm:BgUn11-3}, respectively. The idea now is to consider $L^{\prime} \to \ol{G
  \times^U X}$ such that some finite fully separating set of invariants $\mc{A}
\subseteq \kk[X,L]^U$ extends
to a collection of $G$-invariant sections over $\ol{G \times^U X}$, with various
further restrictions to increase their effectiveness for studying $U \act L \to X$.   
\begin{defn} \label{def:BgUn12} \cite[Definitions 5.2.4--5.2.7]{dk07} Let $U
  \act L \to X$ be an ample linearisation 
of a unipotent group over an irreducible projective variety, $G$ a reductive
group containing $U$ as a closed subgroup and $\mc{A}$ a finite fully
separating set of invariants. Suppose $\beta:G \times^U X \hookrightarrow \ol{G
  \times^U X}$ is a dominant $G$-equivariant open immersion into a projective
$G$-variety $\ol{G
  \times^U X}$ and $L^{\prime} \to \ol{G \times^U X}$ a
$G$-linearisation that restricts to $U \act L \to X$
under $\beta \circ \alpha$. If every $f\in \mc{A}$ extends to a $G$-invariant section 
of some
positive tensor power of $L^{\prime}$ over $\ol{G \times^U X}$, the pair
$(\ol{G\times^U X},L^{\prime})$ is called a \emph{reductive envelope} for $U \act L 
\to X$ (with       \index{reductive envelope}
respect to $\mc{A}$). Furthermore,
\begin{enumerate}
\item if each $f \in \mc{A}$ extends to a $G$-invariant $F$ over $\ol{G \times^U X}$
  such that $\ol{G \times^U X}_F$ is affine then $(\ol{G \times^U X},L^{\prime})$ is
  called a \emph{fine reductive envelope};  \index{fine reductive envelope}
\item if $L^{\prime}$ is an ample line bundle then 
$(\ol{G \times^U X},L^{\prime})$ is called an \emph{ample reductive envelope}; and  \index{ample reductive envelope}
\item if each $f \in \mc{A}$ extends to a $G$-invariant section $F$ over $\ol{G
    \times^U X}$ which vanishes on the codimension 1 part of the boundary $\ol{G
    \times^U X} \setminus (G \times^U X)$ then $(\ol{G \times^U X},L^{\prime})$
  is called a \emph{strong reductive envelope}.   \index{strong reductive envelope}
\end{enumerate}
\end{defn}
Clearly any ample reductive envelope is a fine reductive envelope.
In \cite[Proposition 5.2.8]{dk07} it is shown that for any ample linearisation $U
\act L \to X$, there is some positive tensor power $L^{\ten r}$ of $L$ which
possesses an ample reductive envelope. 

Associated to any reductive envelope
$L^{\prime} \to \ol{G \times^U X}$ are the \emph{completely semistable locus}  \index{completely semistable locus}
\[
X^{\ol{\rmss}}=(\beta \circ \alpha)^{-1}(\ol{G \times^U X}^{\rmss(L^{\prime})})
\]
and the \emph{completely stable locus}  \index{completely stable locus}
\[
X^{\ol{\rms}}=(\beta \circ \alpha)^{-1}(\ol{G \times^U X}^{\rms(L^{\prime})}).
\]  
The main theorem concerning reductive
envelopes, stated below, says that in the case where $L^{\prime} \to \ol{G
  \times^U X}$ is fine, the sets 
$X^{\ol{\rmss}}$ and $X^{\ol{\rms}}$ `bookend' the inclusion $X^{\mathrm{lts}} 
\subseteq
X^{\ssfg}$ associated to the $U$-linearisation $L \to X$, and the GIT quotient
$\ol{G \times^U X} \dblslash_{L^{\prime}} G$ contains the
enveloping quotient $X \dblslash U$.   
  
\begin{thm} \label{thm:BgUn13} \cite[Theorem 5.3.1]{dk07}\footnote{\cite[Theorem
    5.3.1]{dk07} actually says more than presented here and needs a
    normality assumption on $X$ to include this extra material. An examination
    of the proof shows that the version we give here does not need $X$ to be
    normal.} Let $X$ be an irreducible
projective variety with an ample linearisation $L \to X$ of a
unipotent group $U$, and let $(\ol{G \times^U X},L^{\prime})$ be a fine reductive
envelope, with open embedding $\beta:G \times^U X \hookrightarrow \ol{G \times^U
  X}$. Let $\pi:\ol{G \times^U X}^{\rmss(L^{\prime})} \to \ol{G \times^U X}
\dblslash_{L'} G$ be the GIT
quotient map and suppose $X \dblslash U$ is a quasi-projective variety. Then
there is a commutative diagram:  
\[
\begin{diagram}
X^{\overline{\rms}} & \rEmpty~\subseteq & X^{\mathrm{lts}}=X^{\mathrm{ms}}=X^{\mathrm{mss}} & \rEmpty~\subseteq & X^{\ssfg} & \rEmpty~\subseteq & X^{\overline{\rmss}}=X^{\nss} \\
\dTo^{q} & & \dTo^q & & \dTo^q & & \dTo^{\pi \circ \beta \circ \alpha}\\
q(X^{\overline{\rms}}) & \rEmpty~\subseteq & q(X^{\mathrm{lts}}) & \rEmpty~\subseteq & X \dblslash U & \rEmpty~\subseteq & \olGUX \dblslash_{L'} G\\
\end{diagram}
\]
with all inclusions open. 
\end{thm}
\begin{rmk}
Note that Theorem \ref{thm:BgUn13} holds for ample reductive envelopes in
particular, and in this case $\ol{G \times^U X} \dblslash_{L^{\prime}} G$ gives
a projective completion of $X \dblslash U$. If
furthermore $\kk[X,L]^U$ is a finitely generated $\kk$-algebra, then
$X^{\ssfg}=X^{\nss}=X^{\overline{\rmss}}$ and $X \dblslash U \cong  \ol{G \times^U X}
\dblslash_{L^{\prime}} G$.  
\end{rmk}
In the case where the reductive envelope is fine and strong with a completion
$\ol{G \times^U X}$ that is normal, the sets $X^{\rms}=X^{\mathrm{lts}}$ and
$X^{\ssfg}$ can be computed via 
the stable and semistable loci of the
reductive envelope:
\begin{thm} \label{thm:BgUn14} \cite[Theorem 5.3.5]{dk07}
Retain the notation of Theorem \ref{thm:BgUn13}. If furthermore $\ol{G \times^U
  X}$ is normal, and $(\ol{G \times^U X},L^{\prime})$ defines a fine strong reductive
envelope, then $X^{\overline{\rms}}=X^{\mathrm{lts}}$ and
$X^{\overline{\rmss}}=X^{\ssfg}$.  
\end{thm}

Given their use for computing the stable locus and the finitely generated semistable
locus of a linearisation of a unipotent group, the question of how to construct
ample strong reductive envelopes which are normal has special importance. Note
that ampleness is desirable, because it means the associated $X^{\ol{\rmss}}$ and 
$X^{\ol{\rms}}$
can be computed via the Hilbert-Mumford criteria applied to $L^{\prime} \to
\ol{G \times^U X}$. For sufficiently nice completions $\ol{G \times^U
  X}$ one can turn any $G$-linearisation $L^{\prime} \to \ol{G 
  \times^U X}$ into a strong reductive envelope by using the boundary divisors
of $\ol{G \times^U X}$
to modify the line bundle $L^{\prime}$ appropriately. 
\begin{defn} \label{def:Co1St3.1} \cite[Definition 5.3.8]{dk07} Let $X$ be a
  quasi-projective variety and 
  $\beta:X \hookrightarrow \ol{X}$ a projective completion of $X$. The
  completion is said to be \emph{gentle} if $\ol{X}$ is normal and every    \index{gentle completion}
  codimension 1 component of the boundary of $X$ in $\ol{X}$ is a
  $\mb{Q}$-Cartier divisor.   
\end{defn}   

\begin{rmk} \label{rem:gentleunipotent} As was observed in \cite[Remark 5.3.11]{dk07}, in general there does not exist a projective completion $\ol{G \times^U
  X}$ of ${G \times^U
  X}$ and a $G$-linearisation on a line bundle $L' \to \ol{G \times^U
  X}$ extending the induced linearisation on ${G \times^U
  X}$ with the following three properties all holding simultaneously, although we can ensure that any two of them hold together:

(i) $L' \to \ol{G \times^U
  X}$ is a reductive envelope for $U \act L \to X$ with respect to some finite fully separating set of invariants;

(ii) the completion  $\ol{G \times^U
  X}$ of  ${G \times^U
  X}$ is gentle;

(iii) $L'$ is ample.
\end{rmk}

Suppose $\ol{G \times^U X}$ is a gentle completion of $G \times^U X$ and $L^{\prime}
\to \ol{G \times^U X}$ is \emph{any}  
$G$-linearisation extending $L \to
G \times^U X$. Let $D_1,\dots,D_m \subseteq \ol{G \times^U X}$ be the
codimension 1 irreducible components of the 
complement of $G \times^U X$ in $\ol{G \times^U X}$ and define the
$\mb{Q}$-Cartier divisor  
\[
D:=\sum_{i=1}^m D_i.
\]
Then for any sufficiently divisible integer
$N>0$ the divisor $ND$ is Cartier and defines a line bundle
$\OO(ND)$ on $\ol{G \times^U X}$ which restricts to the trivial bundle on
$G \times^U X$. Define 
\begin{equation*} 
L^{\prime}_N:=L^{\prime} \ten \OO(ND) \to \ol{G \times^U X}.
\end{equation*}
If $G$ is connected, then the $G$-linearisation on $L \to G \times^U X$ extends
uniquely to a $G$-linearisation on $L^{\prime}_N$. The next proposition provides
a useful way for turning $L^{\prime} \to \ol{G \times^U X}$ into a strong
reductive envelope.
\begin{propn} \label{prop:BgUn14.1} \cite[Proposition 5.3.10]{dk07}
  Suppose $G$ is a connected reductive group and, as above, suppose $\ol{G
    \times^U X}$ is a gentle 
  completion of $G \times^U X$ and $L^{\prime} \to \ol{G \times^U X}$ is an
  extension of the $G$-linearisation $L \to G \times^U X$. Given a finite fully
  separating set of invariants $\mc{A}$ 
  on $X$, then $(\ol{G \times^U X},L^{\prime}_N)$ is a strong reductive envelope
  with respect to $\mc{A}$, for sufficiently divisible integers $N>0$. If in fact
  $(\ol{G \times^U X},L^{\prime})$ defines a fine 
  reductive envelope with respect to $\mc{A}$, then $(\ol{G \times^U
    X},L^{\prime}_N)$ defines a fine strong reductive
  envelope.    
\end{propn}
 
The above construction is
especially simple to describe explicitly when $X$ is normal and a reductive
group $G$ can be found that contains $U$ as a closed subgroup in such a way that
\begin{itemize}
\item the homogeneous space $G/U$ can be embedded in a normal affine variety
  $\ol{G/U}^{\aff}$ with codimension 2 complement; and
\item the $U$-linearisation $U \act L \to X$ extends to a $G$-linearisation $G
  \act L \to X$.
\end{itemize}
(Note that the first of these conditions is equivalent to $U$ being a
\emph{Grosshans subgroup} of $G$, about which more will be discussed in Section
\ref{subsec:Co1StImportant}.) The
extension of the linearisation 
leads to an isomorphism of $G$-linearisations
\[
  \begin{diagram} 
G\times^{U} L & \rTo^{\cong} & (G/U) \times L  & \rEmpty & [g,l] \mapsto (gU,gl) \\
\dTo &  & \dTo & & \\
G\times^U X & \rTo^{\cong} & (G/U) \times X &  & [g,x] \mapsto (gU,gx)\\ 
\end{diagram}
\]
with the corresponding $G$-linearisation on the right hand side being the
product of the linearisation $G \act L \to X$ and left multiplication on
$G/U$. Note that because $\ol{G/U}^{\aff} \times X$ is normal and $G$
reductive, the ring of invariants
\[
\kk[X,L]^U \cong (\OO(G/U) \ten \kk[X,L])^G =
(\OO(\ol{G/U}^{\aff}) \ten \kk[X,L])^G
\]
is a finitely generated $\kk$-algebra. In particular, this implies $X \dblslash U =
\proj(S^U)$ is a projective variety. One can choose a normal projective
$G$-equivariant completion 
$\ol{G/U}$ of $\ol{G/U}^{\aff}$ whose boundary consists of an effective Cartier divisor $D_{\infty}$ (not necessarily prime), and there is a very ample $G$-linearisation on the
associated line bundle $\OO(D_{\infty}) \to \ol{G/U}$ extending the canonical
linearisation on $\OO_{G/U} \to G/U$. As above, for any $N>0$, let
\[
L^{\prime}_N=\OO(ND_{\infty}) \boxtimes L \to \ol{G/U} \times X,
\]
equipped with its natural $G$-linearisation. Then using Proposition
\ref{prop:BgUn14.1} and Theorem \ref{thm:BgUn13}, one deduces 
\begin{propn} \label{prop:BgUn15} \cite[Lemma 5.3.14]{dk07}
In the above situation, the pair $(\ol{G/U} \times
X,L^{\prime}_N)$ defines an ample strong reductive envelope for sufficiently large
$N>0$, and $X \dblslash U = (\ol{G/U} \times X) \dblslash_{L^{\prime}_N} G$.  
\end{propn}

\begin{eg} \label{ex:BgUn16} \cite[\S 6]{dk07} 
Let $U=(\CC,+)$, embedded in $\GL(2,\CC)$ as the subgroup of upper triangular
matrices, act on $V=\sym^n \CC^2$ via the standard representation of
$\GL(2,\CC)$ on $V$, and consider the canonical $U$-linearisation on $L:=\OO(1) \to
X:=\PP(V)$. (Note that $X$ may be regarded as the space of degree $n$ divisors
on $\PP^1$, and the action of $U$ on $X$ 
corresponds to moving points on $\PP^1$ by 
the usual translation M\"{o}bius transformation.) This linearisation extends to one of 
$G=\SL(2,\CC)$ in the obvious
way. The homogeneous space $G/U$ is isomorphic to $\CC^2 \setminus \{0\}$ via the 
usual
transitive action of $G$ on $\CC^2 \setminus \{0\}$, and it has a normal
$G$-equivariant affine completion $\CC^2$. Embedding $\CC^2$ into $\PP^2$ by
adding a hyperplane at infinity, we arrive in the setting of Proposition
\ref{prop:BgUn15}, so $\OO_{\PP^2}(N) \boxtimes L \to \PP^2 \times X$ defines a
strong ample reductive envelope for $U \act L \to X$, for sufficiently large
$N>0$. Using the Hilbert-Mumford 
criterion on $\PP^2 \times X$ and Theorem \ref{thm:BgUn14}, one sees that
\begin{align*}
X^{\rms}&=\{\text{divisors $\textstyle{\sum_{i=1}^n p_i}$ where $<n/2$ of the $p_i$
  coincide}\}, \\
X^{\ssfg}&=\{\text{divisors $\textstyle{\sum_{i=1}^n p_i}$ where $\leq n/2$ of the 
$p_i$
  coincide}\}.
\end{align*}

In the case where $n$ is odd then $X^{\rms}=X^{\ssfg}$ and $X^{\rms}/U$ is an
open subset of $X \dblslash U=(\PP^2 \times X) \dblslash G$ with complement
given by the reductive GIT quotient $(\{0\} \times X) \dblslash G = X \dblslash
G$ for the classical action of $G=\SL(2,\CC)$ on $X$ (linearised with respect to
$\OO(1) \to X$). In particular, the enveloping quotient map $X^{\ssfg} \to X
\dblslash U$ is not surjective.

On the other hand, when $n$ is even then $X^{\rms}$ is a proper subset of
$X^{\ssfg}$ and the image of $X^{\ssfg} \to X \dblslash U$ is not a variety: it is
equal to the union of $X^{\rms}/U$ together with the point $\pt=(X \dblslash G)
\setminus (X^{\rms(G)}/G)$ given by the quotient of the strictly semistable
set for the $G$-linearisation on $X$.   
\end{eg}


\section{Geometric Invariant Theory for Non-Reductive Groups}
\label{sec:gitNonReductive}

In this section we extend the constructions of \cite{dk07}
described in Section \ref{subsec:BgUnIntrinsic} to the case where $H$ is a linear
algebraic group acting on a variety, with $H$ not necessarily unipotent. Let $H_u$ be the unipotent radical of $H$, so that $H_u$ is a unipotent normal subgroup of $H$ and $H_r = H/H_u$ is reductive. A number     \index{$H_r = H/H_u$}
of our definitions and results are simple generalisations of those found in
\cite[\S4 and \S 5.1]{dk07} 
to this more general context. Having said this, we also address
some errors that occur in \cite[\S 4]{dk07} and thus our work can be seen
as giving some new perspectives on the unipotent picture. A further guiding goal
is to develop a theory which reduces to Mumford's GIT \cite{mfk94} in the case
where $H=G$ is a reductive group acting on a projective variety
equipped with an ample linearisation. Throughout this section $X$ will be a
variety with an action of a linear 
algebraic group $H$ and $L \to X$ a linearisation of the action. We don't
necessarily assume $X$ is projective or irreducible, or $L$ is ample, unless
explicitly stated. 

We begin in Section \ref{sec:GiFinitely} by extending the finitely generated
semistable locus $X^{\ssfg}$ and the notions of 
enveloping quotients and enveloped quotients from Section \ref{subsec:BgUnIntrinsic} 
to the more general non-reductive case (Definitions
\ref{def:GiFi1} and \ref{def:GiFi3}). As a way to address the
observation that the enveloping quotient $X \env H$ need not be a variety (see Remark
\ref{rmk:BgUn4}) we introduce the concept of an \emph{inner enveloping
  quotient} in Definition
\ref{def:GiFi3.1}. These are subvarieties of the enveloping quotient that, in
some sense, play the role of 
the GIT quotient from Mumford's theory for reductive groups; indeed, in the
case where $H=G$ is reductive, $X$ is projective and $L \to X$ is ample, there
is only one inner enveloping quotient---namely,
the GIT quotient $X \dblslash G$. An inner enveloping quotient is in general
not intrinsic solely to the data of the linearisation $L \to X$, but instead
corresponds to a choice of a certain kind of linear system,
called an \emph{enveloping system}, introduced in Definition \ref{def:GiFi4}. We
explore ways in which the collection of all inner enveloping quotients gives a certain
`universality' with respect to $H$-invariant morphisms from $X^{\ssfg}$. In
Section \ref{sec:GiNaturality} we examine how the enveloping quotient
behaves under naturally induced group actions. In particular, we note some of
the difficulties that can arise when trying to take enveloping quotients `in
stages': first by a
normal subgroup $N$ of $H$ and then by the quotient group $H/N$. In Section
\ref{sec:GiStability} we introduce the stable locus $X^{\rms}$ for a general non-
reductive
linearisation over an irreducible variety $X$ (Definition
\ref{def:GiSt1.1}). This is an $H$-invariant open subset of $X$ that is intrinsic to 
the
linearisation $L \to X$ and admits a geometric quotient under the $H$-action. Our
notion of stability also reduces to 
Definition \ref{def:TrRe5} in the reductive case and to 
Definition \ref{def:BgUn5} in the unipotent case. Following the ideas of
\cite[\S 5]{dk07}, we relate our definition of stability for $H \act L \to X$ to 
stability for a
certain reductive linearisation obtained by extension to a reductive structure
group, which will be important for the work on reductive envelopes in Section
\ref{sec:compactifications1}. Finally, in Section
\ref{sec:GiSummary} we draw together all our definitions and key results into Theorem
\ref{thm:GiSu2}, which provides a summary of our geometric invariant theoretic picture 
for
non-reductive groups.

\subsection{Finitely Generated Semistability and Enveloping Quotients}
\label{sec:GiFinitely}

Let $H$ be a linear algebraic group acting on a variety $X$ equipped with a
linearisation $H \act L \to X$. As in Section \ref{sec:background}, we let
$S=\kk[X,L]=\bigoplus_{r \geq 0} H^0(X,L^{\ten r})$ 
be the graded $\kk$-algebra of global sections of positive tensor powers of
$L$ and $S^H$ be the subring of invariant sections under the action
\eqref{eq:TrRe2} of Section \ref{subsec:BgGrLinearisations}. The inclusion $S^H
\hookrightarrow S$ defines an $H$-invariant rational 
map of schemes
\begin{equation} \label{eq:GiFi1}
q:X \dashrightarrow \proj(S^H),
\end{equation}
whose maximal domain of definition contains the open subset of points where some
invariant section of a positive tensor power of $L$ does not vanish. As we have
already seen, the basic 
technique of geometric invariant theory is, roughly speaking, to use the
non-vanishing loci $X_f$ of 
invariant sections $f$ to construct $H$-invariant open subsets
of $X$ which admit geometric quotients in the category of
varieties. Since any such geometric quotient must be a scheme of finite type, it
makes sense to restrict which opens $X_f$ to include in the following manner.     

\begin{defn} \label{def:GiFi1} 
Let $H$ be a linear algebraic group acting on a variety $X$ and $L \to X$ a
linearisation of the action. The \emph{naively semistable locus} is the open subset   \index{naively semistable}
\[
X^{\nss}:=\bigcup_{f \in I^{\nss}} X_f
\]
of $X$, where $I^{\nss}:=\bigcup_{r>0}H^0(X,L^{\ten r})^H$ is the set of invariant
sections of positive tensor powers of $L$. The
\emph{finitely generated semistable locus} is the open subset  \index{finitely generated semistable}
\[
X^{\ssfg}:= \bigcup_{f \in I^{\ssfg}} X_f
\]
of $X^{\nss}$, where 
\[
I^{\ssfg}:=\left\{f \in \textstyle{\bigcup_{r>0}} H^0(X,L^{\ten
r})^H \mid (S^H)_{(f)} \text{ is a finitely generated $\kk$-algebra}\right\}.
\]  
\end{defn}
These definitions generalise Definitions \ref{def:BgUn1} and \ref{def:BgUn2} of the
naively semistable and finitely generated semistable loci, respectively, from
\cite{dk07}. The finitely generated semistable locus is also closely related to the `algebraic locus' of an affine scheme introduced in \cite[Remark 2.3]{dk15} (see the upcoming Example \ref{ex:GiFi1.3}). They depend on the choice of 
the linearisation $L$;
when necessary, we shall indicate this by writing $X^{\nss(L)}$ and
$X^{\ssfg(L)}$.   

\begin{rmk} \label{rmk:GiFi1.2}
It is clear from the definition that for any $r>0$ the subset
$X^{\nss}$ is unaffected by replacing the 
linearisation $L \to X$ with $L^{\ten r} \to X$, and there is a canonical isomorphism
$\proj(\kk[X,L]^H)\cong\proj(\kk[X,L^{\ten r}]^H)$. The subset $X^{\ssfg}$ is
also unaffected by this replacement. Indeed, it is easy to see that
$X^{\ssfg(L^{\ten r})} \subseteq X^{\ssfg(L)}$. For the reverse containment,
note that for any $f \in H^0(X,L^{\ten m})^H$ ($m>0$) with $(\kk[X,L]^H)_{(f)}$
a finitely generated $\kk$-algebra, we have $f^r \in \kk[X,L^{\ten r}]^H$ with
\[
(\kk[X,L^{\ten r}]^H)_{(f^r)}=(\kk[X,L]^H)_{(f^r)}=(\kk[X,L]^H)_{(f)}
\]
a finitely generated $\kk$-algebra, so that $X_f=X_{f^r} \subseteq
X^{\ssfg(L^{\ten r})}$. It thus makes sense to define $X^{\nss}$, $X^{\ssfg}$
and the scheme $\proj(\kk[X,L]^H)$ for rational linearisations using Remark \ref{rmk:TrRe3.2}.      
\end{rmk}

As we noted in Section \ref{subsec:BgGrLinearisations}, the most common 
linearisations one comes across are either when $X$ is
affine and $L=\OO_X$ is the trivial bundle, or else when $X$ is a projective variety 
and
$L$ is an ample line bundle. We take a moment to consider the rational map
\eqref{eq:GiFi1} and Definition \ref{def:GiFi1} in each of these cases.
\begin{eg} \label{ex:GiFi1.3} 
In the case where $X=\spec A$ affine and $L=\OO_X$, recall from Example
\ref{ex:TrRe3.1} that the
linearisation is defined by a character $\chi:H \to
  \mb{G}_m$ and that $S^H$ is the graded subring of semi-invariants, 
\[
\bigoplus_{r \geq 0} A^H_{\chi^r}, \quad A^H_{\chi^r}:=\{f \in A \mid
\text{$f(hx)=\chi(h)^rf(x)$ for all $x \in X$, $h \in H$}\}.
\]
The rational map $q:X
\dashrightarrow \proj(\bigoplus_{r \geq 0} A^H_{\chi^r})$ corresponds to the
natural map $\bigoplus_{r \geq 0} A^H_{\chi^r} \to A$ induced by the inclusions
$A^H_{\chi^r} \hookrightarrow A$, and
$X^{\nss}$ is in this case the maximal domain of definition of $q$, consisting of
points $x \in X$ where $f(x) \neq 0$ for some $f \in A^H_{\chi^r}$ with
$r>0$. 

In the special case where $\chi=1$ is the trivial character, then the ring of
semi-invariants is just $\bigoplus_{r\geq 0} A^H$, so that
$\proj(\kk[X,L]^H)=\spec(A^H)$. Furthermore, $X^{\nss}=X$ because the constant
function $1 \in H^0(X,L)^H$, and $X^{\ssfg}$ is the union of $X_f$ with $f
\in A^H$ such that $(A^H)_f$ is a finitely generated $\kk$-algebra. (In fact, $X^{\ssfg}$ is the preimage of Dufresne and Kraft's algebraic locus of $\spec(A^H)$ under $q:X \to \spec(A^H)$ in this case; see \cite[Remark 2.3]{dk15}.)    
\end{eg} 

\begin{eg} \label{ex:GiFi1.4} 
If now $X$ is projective and $L$ is ample, then each of the open subsets $X_f$
arising in Definition \ref{def:GiFi1} is affine, so the restriction of the
rational map $q:X \dashrightarrow \proj(S^H)$ to $X^{\nss}$ and $X^{\ssfg}$ defines an
\emph{affine} morphism. Moreover, by taking a
  sufficiently positive tensor power $L^{\ten r}$ of $L$ we may embed $X$
  equivariantly into the projective
  space $\PP(V^*)$ using the complete linear system $V=H^0(X,L^{\ten
    r})$, and the linearisation $L^{\ten r}$ extends to $\OO_{\PP(V^*)}(1) \to
  \PP(V^*)$; see Example \ref{ex:TrRe3.2}. If $L$ is very ample (so that we may
  take $r=1$), $I_X$ is the kernel of the
  restriction map $\kk[\PP(V^*),\OO(1)] \to \kk[X,L]$ and
  $R_X=\kk[\PP(V^*),\OO(1)]/I_X$, then by 
\cite[Chapter 2, Lemma 5.14 and Remark 5.14.1]{har77}  for some $m>0$ the $m$-th Veronese subring
  $(R_X)^{(m)} \subseteq R_X$ is isomorphic to $\kk[X,L^{\ten m}]$. Then $
\proj(\kk[X,L]^H)
  \cong \proj((R_X)^H)$, and in light of Remark \ref{rmk:GiFi1.2} computing
  $X^{\nss}$ and $X^{\ssfg}$ for
  $H \act L \to X$ is essentially equivalent to
  studying the action of $H$ on $R_X$. Thus when $L$ is ample and $X$ is
  projective we can always reduce to the case where $H$
  acts on a projective space $\PP^n$ via a
  representation $H \to \GL(n+1,\kk)$ and $X \subseteq \PP^n$ is a closed
  subvariety invariant under the action.     
\end{eg} 

\begin{rmk} \label{rmk:GiFi1.1}
In general the finitely generated semistable locus $X^{\ssfg}$ is strictly
contained in $X^{\nss}$, due to the fact that the subring of invariant sections
can be non-noetherian (even if $S$ is a finitely generated
$\kk$-algebra). Indeed, when $X=\spec A$ is an affine variety and $L=\OO_X$
is equipped with the canonical $H$-linearisation (i.e. defined by the trivial
character $1:H \to \mb{G}_m$), then as seen in Example \ref{ex:GiFi1.3}
$X^{\nss}=X$ and $X^{\ssfg}$ is the union of all $X_f$ where $(A^H)_f$ is
finitely generated over $\kk$. In \cite[Proposition
2.10]{dk08} Derksen and Kemper show that the set 
\[
I^{\ssfg} \cup \{0\}=\{f \in A^H \mid \text{$(A^H)_f$ is
  finitely generated}\} \cup \{0\}
\]
is in fact a radical ideal of $A^H$. In \cite{gro76} Grosshans shows that if $A$
is an integral domain then there is
a nonzero $f \in A^H$ such that $(A^H)_f$ is finitely generated, so if $A^H$ is not 
finitely
generated then $I^{\ssfg}$ is a
proper nonzero ideal. It follows that any irreducible affine example in which
the ring of invariant 
global functions is not finitely generated will result in $\emptyset \neq
X^{\ssfg} \neq X^{\nss}$; for example, the Nagata counterexample in Example
\ref{ex:TrRe2.4}.           
\end{rmk} 

The rational map of \eqref{eq:GiFi1} restricts to
define a morphism on $X^{\ssfg}$ whose image is contained in the following open
subscheme of $\proj(S^H)$.

\begin{defn}
\label{def:GiFi3} Let $H$ be a linear algebraic group and $H \act L \to X$ a
linearisation of an 
$H$-variety $X$. The \emph{enveloping quotient} is the scheme   \index{enveloping quotient}
\[ \label{deg-0-localisation}
X \env H:= \bigcup_{f \in I^{\ssfg}} \spec \left((S^H)_{(f)}\right)
\subseteq \proj(S^H) 
\]
together with the canonical map $q:X^{\ssfg} \to X \env H$. We call the image
$q(X^{\ssfg})$ of this map the \emph{enveloped quotient}.   \index{enveloped quotient}
\end{defn}
When it is necessary to do so, we will include the data of the linearisation in
an enveloping quotient by
writing $X \env_{\! L} H$. Definition \ref{def:GiFi3} is simply an extension of
the definition of enveloping quotient and enveloped quotient in
\cite{dk07} to the case of linearisations for any linear algebraic group. Observe
that we do not use the ``$\dblslash$'' notation, since one can define the
enveloping quotient for a linearisation of a reductive group and in general this
is \emph{not} equal to Mumford's reductive GIT 
quotient from \cite{mfk94} (more will be said about this in Section
\ref{subsec:GiFiComparison}). 
Observe that the enveloping quotient is a canonically defined reduced, separated 
scheme
locally of finite type over $\kk$.      

\begin{rmk} \label{rmk:GiFi3.1}
  As we noted earlier in Remark \ref{rmk:BgUn4}, the enveloping quotient is only
  a scheme \emph{locally} 
  of finite type in general. However, when $S^H$ is a finitely
generated $\kk$-algebra, or when the enveloping quotient map $q:X^{\ssfg} \to X
\env H$ is surjective, then $X \env H$ is noetherian and hence a variety. In
particular, if $X$ is projective, $L \to X$ is ample and
$S^H$ finitely generated over $\kk$ then $X \env H = \proj(S^H)$ is a projective
variety.    
\end{rmk}

\begin{rmk} \label{rmk:GiFi3.2}
When $X=\spec A$ is affine and $L=\OO_X$ has linearisation defined by the trivial character (see Example \ref{ex:TrRe2.1}) the enveloping quotient is precisely the algebraic locus of $\spec(A^H)$, in the sense of \cite[Remark 2.3]{dk15}.
\end{rmk}

In the next lemma we make some initial observations about the rational map $q:X
\dashrightarrow \proj(S^H)$ of \eqref{eq:GiFi1} associated to the linearisation $H
\act L \to X$. As well as using standard facts about the $\proj$
construction, we need the
following commutative algebra result from \cite[Proposition 2.9]{dk08}: if $A$
is an integral domain over $\kk$ and $a,b \in A \setminus \{0\}$ are
such that $A_a$ and $A_b$ are both finitely generated $\kk$-algebras and the
ideal generated by $a,b$ is equal to $A$, then $A$ is 
also a finitely generated $\kk$-algebra (and in fact $A=A_a \cap A_b$).   
 
\begin{lem} \label{lem:GiFi3.1} 
Suppose $S^H \neq \kk$, let $\mc{Y}=\proj(S^H)$ and let
$\underline{L}$ \label{sheaf-of-line-bundle} denote the sheaf of sections of $L$
on $X$. Then for each $r \geq 0$, pulling back along $q$ defines
  inclusions of sheaves $q^*:\OO_{\mc{Y}}(r) \hookrightarrow
  (q_*(\underline{L}^{\ten r}|_{X^{\nss}}))^H \subseteq q_*(\underline{L}^{\ten
    r}|_{X^{\nss}})$. If $S$ is furthermore assumed to be an integral domain, then 
\begin{enumerate}
\item \label{itm:GiFi3.1-1} for each $r \geq 0$ the twisting sheaf $\OO_{\mc{Y}}(r)$ 
is
  identified with $(q_*(\underline{L}^{\ten r}|_{X^{\nss}}))^H$ via $q^*$; and
\item \label{itm:GiFi3.1-2} if $S$ is finitely generated over $\kk$ then the
  ideal $\mf{a} \subseteq S^H$ generated by
  $I^{\ssfg}$ is a non-zero graded radical ideal of $S^H$ satisfying $\mf{a} \cap
  S^H_r=I^{\ssfg} \cap S^H_r$ for each $r \geq 0$. In particular, $X^{\ssfg}
  \neq \emptyset$.   
\end{enumerate} 
\end{lem}

\begin{proof}
Fix $r \geq 0$ and 
let $f \in S_m^H=H^0(X,L^{\ten m})^H$ for $m>0$. Let $S(r)$ be the graded $S$-module
with degree $d$ piece equal to $S_{d+r}$ for each $d \in \mb{Z}$ and let 
\[ \label{twisted-module}
M=\bigoplus_{n\geq 0}S(r)_{mn}=\bigoplus_{n\geq 0}H^0(X,L^{\ten r} \ten L^{\ten mn})
\]
with its $S^{(m)}=\kk[X,L^{\ten m}]$-module structure. Since $X$ is
quasi-compact and separated, by \cite[Chapter 2, Lemma 5.14 and Remark
5.14.1]{har77} there is a 
canonical identification of $\OO(X_f)$-modules $H^0(X_f,L^{\ten r}) =
M_{(f)}$, 
\label{deg-0-localised-module} where the module structure on the right hand side
comes from the identification $\OO(X_f)=(S^{(m)})_{(f)}$. By definition we have
$H^0(\spec((S^H)_{(f)}),\OO_{\mc{Y}}(r)) = (M^H)_{(f)}$ with its
$((S^H)^{(m)})_{(f)}$-module structure, and the pullback map 
$q^*:H^0(\spec((S^H)_{(f)}),\OO_{\mc{Y}}(r)) \to
H^0(X_f,L^{\ten r})$ corresponds to the inclusion $(M^H)_{(f)}
\hookrightarrow (M_{(f)})^H \subseteq M_{(f)}$ under these identifications. Hence 
\[ 
q^*:\OO_{\mc{Y}}(r) \hookrightarrow (q_*(\underline{L}^{\ten r}|_{X^{\nss}}))^H
\subseteq q_*(\underline{L}^{\ten r}|_{X^{\nss}})  
\]
is an inclusion of sheaves. If furthermore $S$ is an integral domain, then in
fact $(M^H)_{(f)}=(M_{(f)})^H$ (with notation as above): for if $g \in S(r)_{mn}$ is 
such that $g/f^{n}
\in (M_{(f)})^H$ and $h \in H$, then $g/f^{n} = h\cdot(g/f^{n})=(h \cdot
g)/f^n$ and so $h \cdot g=g$. Statement \ref{itm:GiFi3.1-1} of the lemma
follows. 
  
Now we prove \ref{itm:GiFi3.1-2}, assuming $S$ is an integral domain and
finitely generated over $\kk$. We first show that $I^{\ssfg} \neq
\emptyset$. Since $S^H \neq \kk$ we can 
find a nonzero homogeneous $f \in S^H$ of positive degree. Then $A:=S_{(f)}$ is a 
finitely
generated integral domain over $\kk$, so applying Grosshans' localisation result
\cite{gro76} (see Remark \ref{rmk:GiFi1.1}) to 
$A=S_{(f)}$ we conclude that there exists $a \in A^H \setminus \{0\}$ such that $
(A^H)_a$ is
finitely generated over $\kk$. Because $S$ is an integral
domain we have $A^H=(S^H)_{(f)}$, so $a=g/f^m$ for some integer $m
\geq 0$ and $g \in S^H$ homogeneous of degree equal to $m \deg f$, and
$(A^H)_a=(S^H)_{(fg)}$. Hence $fg \in I^{\ssfg}$. Note this implies
$X^{\ssfg} \neq \emptyset$.   

It is immediate that $\mf{a}$ is a graded ideal of $S^H$. The fact that it
is radical follows from the equality $(S^H)_{(f)}=(S^H)_{(f^m)}$ for each $f \in
S^H$ homogeneous and each $m>0$. It
remains to show
$\mf{a} \cap S^H_r = I^{\ssfg} \cap S^H_r$ for all $r \geq 0$. The inclusions
$\mf{a} \cap S^H_r \supseteq I^{\ssfg} 
\cap S^H_r$ are obvious. For the reverse containments, it suffices to show that
for any $g_1,g_2 \in I^{\ssfg}$ and any $f \in H^0(X,L^{\ten r^{\prime}})^H$
($r^{\prime}>0$), we have $fg_i \in I^{\ssfg}$ and $\tilde{g}:=g_1+g_2 \in I^{\ssfg}$. 
To
this end, note that $fg_i \in I^{\ssfg}$ because
\[
(S^H)_{(fg_i)}=((S^H)_{(g_i)})_{\frac{f^r}{g_i^{r^{\prime}}}}
\] 
is the localisation of a finitely generated algebra. On the other hand, setting
$a_i:=g_i/\tilde{g} \in (S^H)_{(\tilde{g})}$, we see that each
$((S^H)_{(\tilde{g})})_{a_i}=((S^H)_{(\tilde{g}g_i)})$ is a finitely generated 
integral domain over $\kk$. 
Since $(a_1,a_2)$ is the unit ideal in $(S^H)_{(\tilde{g})}$, the ring
$(S^H)_{(\tilde{g})}$ is therefore finitely generated by the result
\cite[Proposition 2.9]{dk08} quoted before the statement of the lemma, so
$\tilde{g}=g_1+g_2 \in I^{\ssfg}$. Thus, $\mf{a} \cap H^0(X,L^{\ten
  r})^H \subseteq I^{\ssfg} \cap H^0(X,L^{\ten r})^H$ for each $r \geq 0$.               
\end{proof}

\begin{rmk} \label{rmk:GiFi2.5}
When $X$ is irreducible the ring of sections $S$ is an integral domain, so for any
section $f \in I^{\nss}=\bigcup_{r>0}H^0(X,L^{\ten r})^H$ we have a natural
identification $\OO(X_f)^H=(S_{(f)})^H=(S^H)_{(f)}$, by Lemma \ref{lem:GiFi3.1},
\ref{itm:GiFi3.1-1}. Under this identification, the morphism $q:X_f
\to \spec((S^H)_{(f)})$ defined by \eqref{eq:GiFi1} corresponds to the natural map 
$X_f
\to \spec(\OO(X_f)^H)$ induced by $\OO(X_f)^H \hookrightarrow \OO(X_f)$.      
\end{rmk}

From the lemma we see that for general $X$ the natural map of
sheaves $q^{\#}: \OO_{\mc{Y}} \to q_*\OO_{X^{\nss}}$ is injective with image contained 
in
$(q_*\OO_{X^{\nss}})^H$ and $q:X^{\nss} \to \mc{Y}$ is a dominant
morphism. Similarly, if $\mc{U} \subseteq \mc{Y}$ is any nonempty open
subscheme (such as the enveloping quotient $X \env H$) then the sheaves 
$\OO_{\mc{Y}}(r)|_{\mc{U}}$ ($r \geq 0$)
are quasi-coherent sheaves whose sections are included
in the $H$-invariant sections of $L^{\ten r}|_{q^{-1}(\mc{U})}$ under pullback by
$q$. Notice also that, as a corollary of \ref{itm:GiFi3.1-2} of Lemma
\ref{lem:GiFi3.1}, any situation where a linear algebraic group $H$ acts on
a finitely generated integral $\kk$-algebra $S$ such that the ring of invariants
$S^H$ is not finitely 
generated over $\kk$ will result in an example of a projective variety $X=\proj
S$ with ample linearisation $L \to X$ such that $X^{\nss} \neq X^{\ssfg}$ and $X
\env H \neq \proj(S^H)$. This holds, in particular, for the projectivised
version of the Nagata example (cf. Example \ref{ex:TrRe2.4}).       


\subsubsection{`Universality' of the Enveloping Quotient}
\label{subsec:GiFiUniversality}

Given a linear algebraic group $H$ acting on a variety $X$ with linearisation $L
\to X$, it is not necessarily the case that $q:X^{\ssfg} \to X \env H$ is
  surjective, thus $X \env H$ is not in general a categorical quotient of
  $X^{\ssfg}$. 
\begin{eg} \label{eg:GiFi3} Let $X=\SL(2,\kk)$ and let $H \subseteq X$ be the
  subgroup of strictly upper triangular matrices, acting on $X$ via matrix
  multiplication. There is a unique linearisation of the trivial bundle $\OO_X
  \to X$. By \cite[Theorem 6.8]{bor91} the geometric quotient $X / H$ exists; in
  fact, $H$ is precisely the stabiliser of the standard action of $\SL(2,\kk)$
  on $\kk^2 \setminus \{0\}$, therefore $X/H \cong \kk^2 \setminus \{0\}$ and
  $\kk[X,\OO_X]^H \cong \kk[z_0,z_1]$ is finitely generated. So in this case
  $X \env H=\spec(\kk[X,\OO_X]^H)\cong \kk^2$, and the image of $X^{\ssfg}=X$ under
  the enveloping quotient map is identified with $\kk^2 \setminus \{0\}$.
\end{eg}  

An example of the failure of surjectivity of $q:X^{\ssfg} \to X \env H$ in the
projective case was
given in Example \ref{ex:BgUn16}. There we also saw
  examples where the enveloped quotient $q(X^{\ssfg})$ is not a variety, so in
  general a categorical quotient of $X^{\ssfg}$ need not exist at all. This
  raises the question of whether there is any sort of way in which to view the
  enveloping quotient as `universal' for $H$-invariant morphisms. Here we give one
  possible way to answer this. (We should say that our use of the word
  `universal' here is informal---while we do prove a sort of uniqueness and existence
  result regarding morphisms induced by certain $H$-invariant morphisms from
  $X^{\ssfg}$, we don't formulate this in terms of a universal property
  within some category, though it is surely possible to do so. This is simply because
  we won't have need for such a formal usage in what follows.)    
       
The key observation is that, even though $X \env H$ may not be quasi-compact,
the enveloped quotient---being the image $q(X^{\ssfg})$ of
$X^{\ssfg}$---is quasi-compact as a subset of $X \env H$. So it is natural to
look at \emph{quasi-compact} open
subschemes $\mc{U}$ of $X \env H$ that contain
$q(X^{\ssfg})$. Observe that it is easy to give examples of such subsets (at least non-constructively): because $X^{\ssfg}$ is quasi-compact there is a finite collection of sections $f_i \in I^{\ssfg}$ such that $X^{\ssfg}$ is covered by the basic opens $X_{f_i}$ and $\bigcup_i \spec((S^H)_{(f_i)})$ is a quasi-compact open subscheme of $X \env H$ containing $q(X^{\ssfg})$. Furthermore, it is easy to see that we have an
equality of sets
\[
q(X^{\ssfg})= \bigcap \{\mc{U} \mid \text{$\mc{U} \subseteq X \env H$ is open, quasi-
compact and contains
  $q(X^{\ssfg})$}\}.   
\]
In fact, with a bit more work we can see that $q(X^{\ssfg})$ is a constructible subset of $X \env H$. Indeed, any quasi-compact open subscheme $\mc{U}$ of $\proj(S^H)$ is of finite type and separated, since $\proj(S^H)$ is separated. Choosing such a $\mc{U}$ and restricting attention to the induced morphism between varieties $q:X^{\ssfg} \to \mc{U}$, we may apply Chevalley's Theorem \cite[Tag 05H4]{stacks-project} to conclude that $q(X^{\ssfg})$ is a constructible subset of $\mc{U}$, and since $\mc{U}$ is a quasi-compact open subset of $X \env H$ it follows that $q(X^{\ssfg})$ is constructible inside $X \env H$ too.

This suggests it is natural to study diagrams of the form
\[
\begin{diagram} 
  &       & X^{\ssfg} & \rEmpty~\subseteq & (G/U) \times L  & \rEmpty~\subseteq & X \\
  & \ldTo & \dTo &  & \dTo^q & & \\
\mc{U} & \rEmpty~\subseteq  & X \env H & \rEmpty~\subseteq  & \proj(S^H) &   & \\ 
\end{diagram}
\]
where the inclusions are open and the $\mc{U}$ are quasi-compact.

More generally, because $q:X^{\ssfg} \to X \env H$ is dominant \emph{any} nonempty 
open
set $\mc{U} \subseteq X \env H$ intersects $q(X^{\ssfg})$ and is
covered by basic affine open subsets of the form $\spec((S^H)_{(f)})$ with $f$ such 
that
$(S^H)_{(f)}$ is finitely generated. Thus the pre-image of $\mc{U}$ under the
enveloping quotient map $q$ is a nonempty union of the associated open subsets
$X_f$. So given any open subset $U \subseteq X$ that is a union of $X_f$ with $f
\in I^{\ssfg}$ we may also consider its image $q(U)$ as an intersection of those
quasi-compact open $\mc{U} \subseteq X \env H$ containing it. This motivates the
following definition.  

\begin{defn} \label{def:GiFi3.1}
Let $H$ be a linear algebraic group acting on a variety $X$ with linearisation
$L \to X$ and let
$U \subseteq X^{\ssfg}$ be a nonempty $H$-invariant open subset. An \emph{inner
  enveloping quotient of $U$} is a quasi-compact open subscheme 
of $X \env H$ that contains the image $q(U)$ of $U$ under the enveloping
quotient map $q:X^{\ssfg} \to X \env H$. An inner enveloping quotient of   \index{inner enveloping quotient}
$U=X^{\ssfg}$ is simply called an \emph{inner enveloping quotient}.
\end{defn}

\begin{eg} \label{ex:GiFi3.2}
In the case where $X$ is an irreducible projective $H$-variety and $L \to X$ an ample
linearisation we can intrinsically define a collection of inner
enveloping quotients, as follows. The section ring $S=\kk[X,L]$ is an integral
domain finitely generated over 
$\kk$, so by Lemma \ref{lem:GiFi3.1}, \ref{itm:GiFi3.1-2} the set
$I^{\ssfg} \cap H^0(X,L^{\ten r})^H$ is a finite dimensional vector space over
$\kk$, for each $r 
>0$. Taking $r>0$ such that $X^{\ssfg} = \bigcup \{X_f \mid f
\in I^{\ssfg} \cap H^0(X,L^{\ten r})^H\}$, the associated open subscheme 
\[
\mc{U}^{(r)}:= \bigcup \{\spec((S^H)_{(f)}) \mid f \in I^{\ssfg} \cap
H^0(X,L^{\ten r})^H\} \subseteq X \env H
\]
is an inner enveloping quotient: choosing any basis $\{f_i\}$ of
$I^{\ssfg} \cap H^0(X,L^{\ten r})^H$ yields a finite open cover
$\{\spec((S^H)_{(f_i)})\}$ of $\mc{U}^{(r)}$ by quasi-compact open subsets.     
\end{eg}

From the discussion above we see it is natural to regard the image of an
$H$-invariant open subset $U$ of $X^{\ssfg}$ under the
enveloping quotient $q:X^{\ssfg} \to X \env H$ as sitting inside a `germ' of
inner enveloping
quotients of $U$. The following proposition makes this idea more precise in the
case where $U$ is a union of open subsets of the form $X_f$ with $f \in f \in
I^{\ssfg}$. 

\begin{propn} \label{prop:GiFi3.2}
Let $H$ be a linear algebraic group acting on an irreducible variety $X$ with
linearisation 
$L \to X$ and let $U = \bigcup_{f \in \mc{S}} X_f$, where $\mc{S}$ is a nonempty
subset of $I^{\ssfg}$. Suppose we are given the data of a
quasi-projective variety $Z$ together with a very ample line bundle $M
\to Z$ and an $H$-invariant morphism
$\phi:U \to Z$ with $\phi^*M \cong L^{\ten r}|_{U}$ for some $r>0$. Then 
\begin{enumerate}
\item \label{itm:GiFi3.2-1} there is an inner enveloping
  quotient $\mc{U}$ of $U$ and a morphism
  $\ol{\phi}:\mc{U} \to Z$ such that $\phi = \ol{\phi} \circ q|_{U}$ and
  $\ol{\phi}^*M \cong \OO_{\mc{U}}(r)$; and
\item \label{itm:GiFi3.2-2} if $\mc{U},\mc{U}^{\prime} \subseteq X \env H$ are two
  inner enveloping quotients of $U$ and $\psi:\mc{U} \to Z$ and
  $\psi^{\prime}:\mc{U}^{\prime} \to Z$ two morphisms such that
  $\psi \circ q|_{U} = \psi^{\prime} \circ q|_{U}$, then $\psi$
  and $\psi^{\prime}$ agree on $\mc{U}^{\prime} \cap \mc{U}$.     
\end{enumerate}
\end{propn}

\begin{proof}
(Proof of \ref{itm:GiFi3.2-1}.) Let $\iota:Z \hookrightarrow \PP^n$ be a
locally closed immersion defined by sections $\sigma_0,\dots,\sigma_n \in
H^0(Z,M)$, so that the composition $\iota \circ
\phi:U \to \PP^n$ is defined by the $H$-invariant sections
$f_0=\phi^*\sigma_0,\dots,f_n=\phi^*\sigma_n \in
H^0(U,L^{\ten r})^H$. Let $\mc{U}_0=\bigcup_{f \in \mc{S}} \spec((S^H)_{(f)})
\subseteq X \env H$. Then $U=q^{-1}(\mc{U}_0)$, so appealing to Lemma
\ref{lem:GiFi3.1}, \ref{itm:GiFi3.1-1} 
there are $g_0,\dots,g_n \in
  H^0(\mc{U}_0,\OO(r))$ such that $q^*g_i=f_i$ for each $i$. The sections
$g_0,\dots,g_n$ define a morphism $\Phi:\mc{U} \to \PP^n$ on some nonempty
quasi-compact open subscheme $\mc{U} \subseteq \mc{U}_0$ that contains
$q(U)$, since the collection of $q^*g_i=f_i$ is basepoint free on
$U$, and by construction $\Phi \circ q=\iota \circ \phi$. Because $q^*$ is
injective, any section of a power of $\OO_{\PP^n}(1)$ that vanishes on
$Z$ will pull back under $\ol{\phi}$ to a zero section over $\mc{U}$, so the image
  of $\mc{U}$ under $\Phi$ is contained in the closure $\ol{\iota(Z)}$ of $\iota(Z)$ 
in
  $\PP^n$; by shrinking $\mc{U}$ if 
necessary we may assume that $\Phi(\mc{U}) \subseteq \iota(Z) \cong Z$. Then
$\ol{\phi}:=\Phi|_{\mc{U}}:\mc{U} \to Z$ is a morphism such that $\ol{\phi} \circ
q|_U=\phi$ and $\ol{\phi}^*M \cong \OO_{\mc{U}}(r)$.   

(Proof of \ref{itm:GiFi3.2-2}.) Suppose we have $\psi:\mc{U} \to Z$ and
$\psi^{\prime}:\mc{U}^{\prime} \to Z$ with $q(U) \subseteq \mc{U} \cap \mc{U}^{\prime}
$ and
$\psi \circ q|_U=\psi^{\prime} \circ q|_U$. Then $q(U)$ is a dense
constructible subset of the noetherian space $\mc{U} \cap \mc{U}^{\prime}$, so
the interior $q(U)^{\circ}$ is a nonempty dense open 
subscheme of $\mc{U} \cap \mc{U}^{\prime}$ on which $\psi$ and $\psi^{\prime}$
agree. Since $X \env H$ is separated, so too is $\mc{U} \cap \mc{U}^{\prime}$ and thus 
we
have $\psi=\psi^{\prime}$ on $\mc{U} \cap \mc{U}^{\prime}$.  
\end{proof}

\begin{rmk} \label{rmk:GiFi3.25}
If the sections $\sigma_0,\dots,\sigma_n \in H^0(Z,M)$ defining an embedding $Z
\hookrightarrow \PP^n$ in the statement
of Proposition \ref{prop:GiFi3.2} are such that each $\phi^*\sigma_i$ extends to a 
global
section of $L^{\ten r} \to X$, then in fact one can prove 
\ref{itm:GiFi3.2-1} and \ref{itm:GiFi3.2-2} for reducible $X$ and \emph{any}
$H$-invariant open subset $U \subseteq X$.   
\end{rmk}

As a corollary of Proposition \ref{prop:GiFi3.2}, we obtain a sort of universal
property for the enveloping  \index{enveloping quotient ! universal property}
quotient $q:X^{\ssfg} \to X \env H$ when $X$ is irreducible (for reducible $X$ an
appropriate statement can be formulated from Remark \ref{rmk:GiFi3.25}).
Given a quasi-projective variety $Z$ embedded in some
projective space and an $H$-invariant morphism
$\phi:X^{\ssfg} \to Z$ defined by sections of some positive
power of $L|_{X^{\ssfg}}$, 
there is an inner enveloping quotient $\mc{U} \subseteq X \env H$ of $X^{\ssfg}$ 
and a morphism $\ol{\phi}:\mc{U} \to Z$ such that the diagram          
\[
\begin{diagram}
& & X^{\ssfg} & & \\
&  &  \dTo^q & \rdTo^\phi & \\
X \env H & \rEmpty~\supseteq & \mc{U} & \rTo{\ol{\phi}} & Z    
\end{diagram}
\]        
commutes, and any other inner enveloping quotient $\mc{U}^{\prime}$ and morphism
$\ol{\phi}^{\prime}$ with $\ol{\phi}^{\prime}\circ q=\phi$ defines the same
rational map $X \env H \dashrightarrow Z$ as $(\mc{U},\ol{\phi})$. 

\begin{rmk} \label{rmk:GiFi3.3}
The inner enveloping quotient $\mc{U}$ and the map $\ol{\phi}:\mc{U} \to Z$
constructed above depend on the choice of embedding of $Z$ into a projective
space $\PP^n$, and the whole construction furthermore relies on the requirement
that the morphism
$\phi:X^{\ssfg} \to Z \subseteq \PP^n$ is defined by sections of some positive
tensor power of $L \to X^{\ssfg}$ (or for reducible $X$, of $L \to X$). Contrast
this to Mumford's GIT quotient arising from a reductive group $G$ acting on a variety 
$X$
with linearisation $L$: then the GIT quotient $X^{\rmss} \to X\dblslash G$ is a 
categorical
quotient of the semistable locus $X^{\rmss}$ \emph{in the category of
  varieties}, so that a $G$-invariant morphism $X^{\rmss} \to Z$ factors
uniquely through $X^{\rmss} \to X \dblslash G$ without any further assumptions
on $X^{\rmss} \to Z$. So we see that the universality of the collection of
inner enveloping quotients for a general linear algebraic
group $H$ is a considerably weaker notion than the universal property of a
reductive GIT quotient. The reason for this can
be traced in large part to the fact that the enveloping quotient $q:X^{\ssfg}
\to X \env H$ is not surjective. In \ref{subsubsec:ExVaCase2} of Section 
\ref{sec:example} we will find
examples of enveloping quotients where $q:X^{\ssfg} \to X \env H$ is surjective
and indeed $X \env H$ is a geometric---and hence categorical---quotient of
$X^{\ssfg}$.  
\end{rmk}

Any inner enveloping quotient $\mc{U} \subseteq X \env H$ is quasi-compact, so for
sufficiently large integers $r>0$ the twisting sheaf $\OO_{\mc{U}}(r)$ defines a
line bundle on $\mc{U}$. We
shall soon see that, for $r$ large enough, 
$\OO_{\mc{U}}(r)$ is in fact very ample. In order to prove this---as well as a similar
statement for inner enveloping quotients of more general open subsets of $X^{\ssfg}
$---it is
convenient to make the following definition.  
\begin{defn} 
\label{def:GiFi4} Let $H$ be a linear algebraic group acting on a variety $X$ and $L 
\to X$ a
linearisation. For $r>0$ and $\mc{S} \subseteq
H^0(X,L^{\ten r})^H$ a finite subset of invariant sections, we
say a linear subspace
$V\subseteq H^0(X,L^{\ten r})$ is an \emph{enveloping system adapted to $\mc{S}$} if   \index{enveloping system adapted to $\mc{S}$}
\begin{enumerate}
\item \label{itm:GiFi4-1} it is finite dimensional, contains $\mc{S}$ and is
  stable under the $H$-action; and 
\item \label{itm:GiFi4-2} for each $f\in \mc{S}$ the $\kk$-algebra
  $(S^H)_{(f)}$ is finitely generated with
generating set $\{\tilde{f}/f \mid \tilde{f} \in V^H\}$.
\end{enumerate}
We call $V$ simply an \emph{enveloping system} if it is an enveloping system \index{enveloping system}
adapted to a subset $\mc{S}$ such that $X^{\ssfg} = \bigcup_{f \in \mc{S}} X_f$. 
\end{defn} 

\begin{eg} \label{ex:GiFi4.2}
Suppose $S^H$ is a finitely generated
$\kk$-algebra. Then there exists $r>0$ such that the $r$-th Veronese subring
$(S^H)^{(r)}$ is generated by its degree 1 piece $S^H_r=H^0(X,L^{\ten
  r})^H$. Therefore $X^{\ssfg}=X^{\nss}$ is covered by the basic open subsets
$X_f$ with $f \in H^0(X,L^{\ten r})^H$, and for each such $f$ we have
$(S^H)_{(f)}=((S^H)^{(r)})_{(f)}$ generated by $\tilde{f}/f$ with $\tilde{f} \in 
H^0(X,L^{\ten
  r})^H$. So $H^0(X,L^{\ten r})^H$ is an enveloping system. 
\end{eg}

The following basic result asserts that finding enveloping systems adapted to
finite subsets is
essentially equivalent to finding quasi-compact open subschemes of the
enveloping quotient $X \env H$ and giving ways to embed them into projective
spaces. 

\begin{propn}
  \label{prop:GiFi4.1} Suppose $H$ is a linear algebraic group and $L \to X$ a
  linearisation of an $H$-variety $X$.
  \begin{enumerate}
  \item \label{itm:GiFi4.1-1} For any quasi-compact open subscheme $\mc{U} \subseteq
    X \env H$, there is an enveloping system $V \subseteq H^0(X,L^{\ten r})^H$
    adapted to a finite 
    subset $\mc{S} \subseteq H^0(X,L^{\ten r})^H$ with $\mc{U}= \bigcup_{f \in \mc{S}}
    \spec((S^H)_{(f)})$, for some $r>0$ such that $\OO_{\mc{U}}(r)$ is a very ample
    line bundle. Moreover, the natural
    map $V \to H^0(\mc{U},\OO_{\mc{U}}(r))$
    defines a locally closed embedding $\mc{U} \hookrightarrow \PP(V^*)$.
\item \label{itm:GiFi4.1-2} Conversely, suppose $H^0(X,L^{\ten r})$ contains an
enveloping system $V$ adapted to a finite subset $\mc{S} \subseteq H^0(X,L^{\ten
  r})^H$,
let $\mc{U}= \bigcup_{f \in \mc{S}} \spec((S^H)_{(f)}) \subseteq X \env H$ and let
$\phi:U := \bigcup_{f \in \mc{S}} X_f \to \PP((V^H)^*)$ be the
$H$-invariant map defined by the inclusion $V^H \subseteq H^0(X,L^{\ten r})$.  
Then there is a locally closed embedding $\ol{\phi}:\mc{U} \hookrightarrow
\PP((V^H)^*)$ such that $\phi=\ol{\phi} \circ q$ on $U$ and
$\ol{\phi}^*\OO_{\PP((V^H)^*)}(1) = \OO_{\mc{U}}(r)$.  
\item \label{itm:GiFi4.1-3} If $V \subseteq H^0(X,L^{\ten r})$ is any enveloping
  system adapted to $\mc{S}$, then the
  image of the natural multiplication map $V^{\ten n} \to H^0(X,L^{\ten rn})$
defines an enveloping system adapted to the set $\{f^n \mid f \in \mc{S}\}$, for each 
$n>0$.        
  \end{enumerate}
\end{propn}

\begin{proof}
(Proof of \ref{itm:GiFi4.1-1}.) The argument we use can essentially be found in
\cite[Proposition 4.2.2]{dk07} and is based on a slight
modification of the argument used to prove quasi-projectivity of the GIT quotient in
reductive GIT (cf. \cite[Theorem 1.10]{mfk94}). 
For completeness, it runs as follows. Let $\mc{Y}=\proj(S^H)$. Since $\mc{U}$ is
quasi-compact, we may find
finitely many invariants $f_1,\dots,f_m \in I^{\ssfg}$ such that the basic open
subsets $\spec((S^H)_{(f_i)})$ cover $\mc{U}$. Using the 
reducedness of $S^H$ we can take powers of the $f_i$ and assume, without loss of
generality, that there is $r_0>0$ such that $f_i \in
S_{r_0}^H$ for each $i$, so that $\OO_{\mc{U}}(r_0)$ is the trivial line bundle over
$\spec((S^H)_{(f_i)})$. The $\kk$-algebras $(S^H)_{(f_i)}$
have finite generating sets, which we can write as
$\{g_{i1}/(f_i^{r_1}),\dots,g_{in_i}/(f_i^{r_1})\}$ for $g_{ij} \in
S^H_{r_0r_1}$ and some $n_i>0$, with one common $r_1>0$ working for each
$i=1,\dots,m$. Resetting $f_i=f_i^{r_1}$ for each $i$ and letting
$\mc{S}:=\{f_1,\dots,f_m\} $, we can assume that we have found $r>0$ and a set 
\[
A:=\mc{S} \cup \{g_{i,j} \mid i=1,\dots,m, \ j=1,\dots,n_i\}
\]
of invariant sections such that $\mc{U} = \bigcup_{f \in \mc{S}}
\spec((S^H)_{(f)})$, the sheaf $\OO_{\mc{U}}(r)$ is locally free and 
$(S^H)_{(f_i)} = \kk[g_{i,1}/f_i,\dots,g_{i,n_i}/f_i]$ for each 
$i$. Taking $V \subseteq S^H_r=H^0(X,L^{\ten r})^H$ to be the $\kk$-span of the
elements of $A$, we see that $V$ is an enveloping system adapted to
$\mc{S}$. The image of the natural
map $V \to H^0(\mc{U},\OO_{\mc{U}}(r))$ induced by the structure map $S^H_r \to
H^0(\mc{Y},\OO_{\mc{Y}}(r))$ is basepoint-free on $\mc{U}$, so $V \to
H^0(\mc{U},\OO_{\mc{U}}(r))$ defines a morphism 
\[
\psi:\mc{U} \to \PP(V^*)
\]
such that $\psi^*\OO_{\PP(V^*)}(1) = \OO_{\mc{U}}(r)$. Now $H^0(\mc{U}_{f_i},
\OO_{\mc{U}}(r)) \cong
(S^H)_{(f_i)}$ and the restriction of $\psi$ to $\mc{U}_{f_i}$ maps into
the affine open subset $\PP(V^*)_{f_i}$ of points of $\PP(V^*)$ where $f_i \in
H^0(\PP(V^*),\OO(1))$ doesn't vanish. So $\psi:\mc{U}_{f_i} \to \PP(V^*)_{f_i}$
corresponds to the natural ring homomorphism  
\[
(\symdot V)_{(f_i)} \to (S^H)_{(f_i)} 
\]
given by multiplying sections, which is surjective because the generators
$g_{i,1}/f_i,$ $\dots, g_{n_i,1}/f_i$ of
$(S^H)_{(f_i)}$ are contained in the image. Thus $\psi:\mc{U}_{f_i} \to \PP(V^*)_{f_i}
$ is a closed
immersion. Since $\mc{U}$ is covered by the $\mc{U}_{f_i}$, the map $\psi:\mc{U} \to 
\PP(V^*)$
is a locally closed immersion and $\OO_{\mc{U}}(r)$ is very ample.         

(Proof of \ref{itm:GiFi4.1-2}.) Suppose $V \subseteq H^0(X,L^{\ten r})$ is an
enveloping system adapted to $\mc{S} \subseteq H^0(X,L^{\ten r})^H$ and let $\mc{U}=
\bigcup_{f \in \mc{S}}
\spec((S^H)_{(f)}) \subseteq \mc{Y}=\proj(S^H)$. As above, the structure map
$S_r^H \to H^0(\mc{Y},\OO_{\mc{Y}}(r))$ defines a linear map 
\[
\alpha:H^0(\PP((V^H)^*),\OO(1))=V^H \to H^0(\mc{U},\OO_{\mc{U}}(r))
\]
such that the composition $q^* \circ \alpha$ is equal to
$\phi^*:H^0(\PP((V^H)^*),\OO(1)) \to H^0(U,L^{\ten r})$. Now $\mc{S} \subseteq
V^H$, so $\OO_{\mc{U}}(r)$ is globally generated by the sections in the image of
$\alpha$ and thus $\alpha$ defines a morphism 
\[
\ol{\phi}:\mc{U} \to \PP((V^H)^*)
\]     
such that $\ol{\phi}^*\OO_{\PP((V^H)^*)}(1)=\OO_{\mc{U}}(r)$ and $\phi=\ol{\phi}
\circ q$. By \ref{itm:GiFi4-2} of
Definition \ref{def:GiFi4}, for each $f \in \mc{S}$ the algebra $(S^H)_{(f)}$ is
generated by $\tilde{f}/f$, where $\tilde{f} \in V^H$, and now the argument used
in the proof of \ref{itm:GiFi4.1-1} above shows that $\ol{\phi}$ is a locally
closed immersion. 

(Proof of \ref{itm:GiFi4.1-3}.) Given an enveloping system $V \subseteq
H^0(X,L^{\ten r})$ adapted to $\mc{S}$ and $n>0$, the
image $V^{\prime}$ of the natural multiplication map $V^{\ten n} \to
H^0(X,L^{\ten nr})$ is a finite dimensional $H$-stable subspace of $H^0(X,L^{\ten nr})
$ that contains
the set of $n$-fold products of invariant sections $A^{\prime}:=\{f_1\cdots
f_{n} \mid \text{each } f_k \in V^H\}$. For any $f \in \mc{S}$ the algebra
$(S^H)_{(f^n)}=(S^H)_{(f)}$ is generated by 
$A^{\prime}$, since we have
$\tilde{f}/f=(\tilde{f}f^{n-1})/f^n$ in $(S^H)_{(f)}$ for all $\tilde{f}
\in V^H$. Hence $V^{\prime}$ is an enveloping system adapted to $\{f^n \mid f
\in \mc{S}\}$.  
\end{proof} 

\begin{rmk} \label{rmk:GiFi4.1} Given an enveloping system $V$ consisting of invariant
  sections, it follows from Proposition \ref{prop:GiFi4.1},
  \ref{itm:GiFi4.1-1} that any basis of $V$ will give a set of invariants of some 
positive
  tensor power of $L \to X$ that separates points \emph{in $X^{\ssfg}$} (compare
  with Definition \ref{def:BgUn11}, \ref{itm:BgUn11-1}).  
\end{rmk}

We have already seen that when $X$ is projective, $L \to X$ is ample and the ring of 
invariants $S^H$ is finitely
generated then the enveloping quotient $X \env H=\proj(S^H)$ is a projective
variety. As a first
application of enveloping systems, we can prove a sort of converse to this fact
for irreducible $X$.
\begin{cor} \label{cor:GiFi5} Suppose $H$ is a linear algebraic group, $X$ an
  irreducible $H$-variety and $L \to X$ a linearisation. If the
  enveloping quotient $X \env H$ is 
 complete, then $X \env
  H=\proj(S^H)$. Furthermore, for suitably
  divisible integers $r>0$ the sheaf $\OO_{X \env H}(r)$ is an ample line bundle
  on $X \env H$ and the natural structure map 
\[
\kk[X,L^{\ten r}]^H = (S^H)^{(r)} \to \kk[X \env H,\OO_{X \env H}(r)]
\]
is an isomorphism. (In particular, $\kk[X,L^{\ten r}]^H$  
is a finitely generated $\kk$-algebra for such $r$ and we have
$X^{\nss}=X^{\ssfg}$.)   
\end{cor}

\begin{proof}
Recall that $q:X^{\nss} \to \proj(S^H)$ is a dominant morphism, as a result of
Lemma \ref{lem:GiFi3.1}. Because $X$ is irreducible, by \ref{itm:GiFi3.1-2} of
the same lemma $X^{\ssfg}$ is a dense open subset of $X^{\nss}$, so the
enveloped quotient $q(X^{\ssfg})$ 
is a dense subset of $\proj(S^H)$ and hence the enveloping quotient $X
\env H$ is a dense open subscheme of 
$\proj(S^H)$. Because $X \env H$ is complete, and hence quasi-compact, it is universally
closed over $\spec \kk$, and since $\proj(S^H)$ is separated over
$\spec \kk$ the open immersion ${X \env H} \hookrightarrow \proj(S^H)$ is a
closed morphism \cite[Tag 01W0]{stacks-project}. Thus $X \env
H=\proj(S^H)$. Using Proposition \ref{prop:GiFi4.1}, \ref{itm:GiFi4.1-1} find
$r^{\prime}>0$ and an enveloping system $V \subseteq H^0(X,L^{\ten r^{\prime}})^H$ so 
that the 
natural map $V \to H^0(X \env H,\OO(r^{\prime})) =
H^0(\proj(S^H),\OO(r^{\prime}))$ defines a closed immersion 
$\proj(S^H) \hookrightarrow \PP(V^*)$ (the fact the immersion is closed is implied
from the completeness of $\proj(S^H)=X \env H$). The line bundle $\OO(r^{\prime})$ on
$\proj(S^H)$ is (very) ample, so by Serre vanishing
\cite[Chapter 3, Proposition 5.3]{har77} there is $m_0>0$ such that for all $m
\geq m_0$ the restriction map
\[
H^0(\PP(V^*),\OO(m))=\sym^mV \to H^0(\proj(S^H),\OO(mr^{\prime})) 
\]
is surjective. Letting $r$ be any positive multiple of $m_0r^{\prime}$ and
$m=r/r^{\prime}$, we see that the restriction map $\kk[\PP(V^*),\OO(m)] \to 
\kk[\proj(S^H),\OO(r)]$
is surjective and therefore $\kk[\proj(S^H),\OO(r)]$ is a finitely generated
$\kk$-algebra. The map $\kk[\PP(V^*),\OO(m)] \to \kk[\proj(S^H),\OO(r)]$
factors through the canonical structure map 
\[
\kk[X,L^{\ten r}]^H=(S^H)^{(r)} \to \kk[\proj(S^H),\OO(r)],
\]
thus this too is a surjective map onto a finitely generated $\kk$-algebra. On the 
other hand,
the composition of this map with pull-back 
along the natural map $q$ from \eqref{eq:GiFi1} agrees with restriction of
sections $\kk[X,L^{\ten r}]^H \to \kk[X^{\nss},L^{\ten r}]^H$, 
which is injective because $X$ is irreducible. It follows that
\[
\kk[X,L^{\ten r}]^H=(S^H)^{(r)} \cong \kk[\proj(S^H),\OO(r)].
\]
In particular, $\kk[X,L^{\ten r}]^H$ is a finitely generated $\kk$-algebra. The
equality $X^{\nss}=X^{\ssfg}$ now follows from the Definitions \ref{def:GiFi1}
of these sets and Remark \ref{rmk:GiFi1.2}.    
\end{proof}
  
\subsubsection{Comparison with Mumford's Reductive GIT}
\label{subsec:GiFiComparison}

The definitions of the naively semistable locus, finitely generated semistable
locus and enveloping quotient are direct generalisations of the corresponding
notions in \cite{dk07} for unipotent groups to the context of general varieties with 
actions of any
linear algebraic group. As such they apply to the situation where $H=G$ is a
reductive group, so we take a moment to compare these notions to those arising in
Mumford's GIT \cite{mfk94} for reductive groups.

Firstly, \emph{in the case where $X$ is affine with a linearisation of the trivial
  bundle $L=\OO_X \to X$, or $X$ is projective with ample linearisation $L \to X$, 
then
  $X^{\nss}=X^{\ssfg}$ is equal to Mumford's semistable locus $X^{\rmss}$ for $G
  \act L \to X$ \cite[Definition 1.7]{mfk94}, and the enveloping quotient is
  precisely the GIT quotient $X \dblslash G=\proj(S^G)$ of \cite[Theorem
  1.10]{mfk94}}. Indeed, we have $I^{\ssfg}=\bigcup_{r>0}H^0(X,L^{\ten r})^G$ in this 
case: for any invariant
section $f$ of a positive tensor power of $L$, the algebra $(S^G)_f$ is the
localisation of a finitely generated algebra $S^G$ by Nagata's theorem \cite{nag64}, and so by Nagata's theorem again
$(S^G)_{(f)}$ is also a finitely generated algebra, being the subalgebra of invariants 
for the
$\GG_m$-action defining the grading on $S^G$. Thus $X \env G =
\proj(S^G) $. Because $L$ is ample $X_f$ is affine for each $f \in I^{\ssfg}$,
from which it follows that $X^{\rmss}=X^{\nss}=X^{\ssfg}$. 

However, the similarities with Mumford's GIT when $G$ is reductive do not extend
beyond these cases. For a general variety $X$ with possibly
non-ample linearisation $L \to X$ of $G$, there may be invariant sections $f$
whose non-vanishing loci $X_f$ are not affine. In Mumford's theory only those
$X_f$ that are affine are included in the definition of the semistable locus
$X^{\rmss}$; 
see Definition \ref{def:TrRe5}, \ref{itm:TrRe5-1}. So
\emph{for a general linearisation $G \act L \to X$ with $G$ reductive,
  Mumford's semistable locus $X^{\rmss}$ is contained in $X^{\ssfg}$ as a
  (possibly empty) open subset.} Given any inner enveloping quotient
$q:X^{\ssfg} \to \mc{U} \subseteq X \env G$, the restriction to Mumford's semistable
locus $X^{\rmss}$ coincides with the GIT quotient map, thus $X \dblslash 
G=q(X^{\rmss})$ is an
open subvariety of $\mc{U}$. Hence \emph{the GIT quotient $X \dblslash G$ is a
  (possibly empty) quasi-compact open subscheme of each inner enveloping quotient
  inside $X \env G$.} Finally, as discussed in Remark \ref{rmk:GiFi3.3}, \emph{the
  enveloping quotient $q:X^{\ssfg} \to X \env G$ is not in general a categorical
  quotient in the category of varieties for the $G$-action on $X^{\ssfg}$,
  whereas Mumford's GIT quotient $X^{\rmss} \to X \dblslash G$ is a categorical
  quotient for the $G$-action on $X^{\rmss}$ \cite[Theorem 1.10]{mfk94}}.


\subsection{Natural Properties of Enveloping Quotients with Respect to Induced
  Group Actions}
\label{sec:GiNaturality}

In this section we will study various natural properties of the enveloping quotient 
and inner enveloping
quotients (adapted to some finite subset) with respect to various natural
operations on groups.

\subsubsection{Restriction and Extension of the Structure Group}
\label{subsec:GiNaEnveloping}

We first look at the case of restricting a linearisation under a surjective
homomorphism 
$\rho:H_1\to H_2$ of linear algebraic groups. Suppose $X$ is an
$H_2$-variety and $L \to X$ line bundle with an $H_2$-linearisation. For precision, 
let us denote
this linearisation as $\mc{L}_2 \to X$. The homomorphism $\rho$ induces an
$H_1$-linearisation on the line bundle $L \to X$, which we denote $\mc{L}_1 \to
X$. Clearly $\kk[X,\mc{L}_1]^{H_1}=\kk[X,\mc{L}_2]^{H_2}$, from which it
follows that there are canonical identifications 
\[
X^{\nss(\mc{L}_1)}=X^{\nss(\mc{L}_2)}, \quad X^{\ssfg(\mc{L}_1)} =
X^{\ssfg(\mc{L}_2)}, \quad X \env_{\! \! \mc{L}_1} H_1 = X \env_{\! \! \mc{L}_2} H_2 
\]
and the natural maps $q_1:X^{\nss(\mc{L}_1)} \to X \env_{\! \! \mc{L}_1} H_1$
and $q_2:X^{\nss(\mc{L}_2)} \to X \env_{\! \! \mc{L}_2} H_2$ of \eqref{eq:GiFi1}
coincide under these identifications.   
 
Now let us consider extensions of the structure group. Suppose we have an
inclusion $H_1 \hookrightarrow H_2$ and $L \to X$ is an
$H_1$-linearisation. Recall from Section
\ref{subsec:BgUnExtending} that we may consider the fibre space $H_2
\times^{H_1} X$ associated to the principal $H_1$-bundle $H_2 \to H_2/H_1$,
which is a variety if $H_1$ is unipotent or if any finite subset of points in $X$ is 
contained in
an affine open subset (e.g. if $X$ is quasi-projective). Then $H_2 \times^{H_1} L 
\to H_2 \times^{H_1} X$ is a line bundle and there is a natural
$H_2$-linearisation on $H_2 \times^{H_1} L \to H_2 \times^{H_1} X$ induced by left
multiplication. This extends the $H_1$-linearisation $L \to X$ under the closed
immersion  
\[
\alpha:X \hookrightarrow H_2 \times^{H_1} X, \quad x \mapsto [e,x].  
\]
As before, we will usually abuse
notation and write $L \to H_2 \times^{H_1} X$ for this linearisation instead of
$H_2 \times^{H_1} L$, unless confusion is
likely to arise. Recall also that pullback along $\alpha$ induces an
isomorphism of graded rings
\[
\alpha^*:\kk[H_2 \times^{H_1} X,L]^{H_2} \overset{\cong}{\longrightarrow}
\kk[X,L]^{H_1}.
\]
Applying $\proj$ gives an isomorphism of schemes
\[
\ol{\alpha}:\proj(\kk[X,L]^{H_1}) \overset{\cong}{\longrightarrow} \proj(\kk[H_2
\times^{H_1} X,L]^{H_2})
\]
such that $\ol{\alpha}^*$ identifies the corresponding twisting sheaves. Let
$q_{H_1}:X^{\ssfg} \to X \env H_1$ be the enveloping quotient map for the
linearisation $H_1 \act L \to X$ and let $q_{H_2}:(H_2 \times^{H_1} X)^{\ssfg} \to 
(H_2 \times^{H_1}
X) \env \,  H_2$ be the enveloping quotient map
for the linearisation $H_2 \act H_2 \times^{H_1} L \to H_2 \times^{H_1}
X$. Clearly pulling back along $\alpha$ establishes a bijection $I^{\ssfg(H_2
  \times^{H_1} L)} \longleftrightarrow I^{\ssfg(L)}$. 
This implies that $\alpha$ restricts to give a closed immersion of $X^{\ssfg(H_1)}$
into $(H_2 \times^{H_1} X)^{\ssfg(H_2)}$ and
$\ol{\alpha}$ restricts to an isomorphism of the enveloping
quotients, fitting into the following commutative diagram:    
\[
\begin{diagram} 
X^{\ssfg(H_1)} & \rInto^\alpha &  (H_2 \times^{H_1}
X)^{\ssfg(H_2)} \\
 \dTo^{q_{H_1}} & & \dTo^{q_{H_2}}\\
X \env H_1 & \rTo^{\cong}_{\ol{\alpha}} & (H_2 \times^{H_1} X) \env H_2\\
\end{diagram}
\]  

\subsubsection{Induced Actions of Quotient Groups on Enveloping Quotients}
\label{subsec:GiNaInduced}

Let us return to the situation where a linear algebraic group $H$ acts on a
variety $X$ and is equipped with a linearisation $L \to X$, but now also
suppose $H$ has a normal subgroup $N$. (We will in particular be interested in the case when $N=H_u$ is the unipotent radical of $H$ and $H_r = H/N$ is reductive). Then we may consider the restricted
linearisation $N \act L \to X$ and form its naively semistable locus $X^{\nss(N)}$, 
semistable
finitely generated locus $X^{\ssfg(N)}$ and enveloping quotient
$q_N:X^{\ssfg(N)} \to X \env N$. Because $N$ is normal in $H$, the action of $H$
on $S=\kk[X,L]$ induces a natural $H/N$-action on the ring $S^N$ of
$N$-invariant sections. For any $h \in H$, the action on $X$ induces an isomorphism 
\[
X_f \overset{\cong}{\longrightarrow} X_{h\cdot f}, \quad x \mapsto hx
\]
(with inverse given by acting by $h^{-1}$), so the action of $H$ on $X$ restricts
to an action on $X^{\nss(N)}$. Moreover, for any $f \in H^0(X,L^{\ten r})^N$
(with $r>0$) and
$\ol{h}=hN \in H/N$ the application of $\ol{h}$ induces an
isomorphism
\[
\ol{h}\cdot(-):(S^N)_{(f)} \overset{\cong}{\longrightarrow} (S^N)_{(h
  \cdot f)},
\]
from which it follows that the action of $H/N$ on $S^N$ preserves
$I^{\ssfg(N)}$. Thus $X^{\ssfg(N)}$ is also stable under the $H$-action.    

\begin{propn} \label{prop:GiNa7} 
Retain the notation of the preceding discussion. Then
the action of $H/N$ on $S^N$ defines a canonical action of $H/N$ on
$\mc{Y}=\proj(S^N)$ such that 
the map $q_N:X^{\nss(N)} \to \mc{Y}$ of \eqref{eq:GiFi1} is equivariant with respect 
to the
quotient $H \to H/N$. In addition, if $\mc{S} \subseteq I^{\nss(N)}$ is any
subset that is stable under the canonical $H/N$-action on $S^N$, then the open
subscheme $\mc{U}=\bigcup_{f \in \mc{S}} \spec((S^N)_{(f)})$ of $\mc{Y}$ is
preserved under this action. (In particular, $X \env N$ is preserved under the
action and $q_N:X^{\ssfg(N)} \to X \env N$ is equivariant.)
\end{propn}  


Before proving Proposition \ref{prop:GiNa7} we need to introduce some notation
and prove a lemma. Let 
\[
\Sigma:H \times X \to X
\]
be the action morphism. A linearisation of $H$ on $L$ is equivalent to a choice
of line bundle isomorphism 
\[ 
\Theta:\Sigma^*L \overset{\cong}{\longrightarrow} \OO_H \boxtimes L=H \times L
\]
over $H \times X$, satisfying an appropriate cocycle condition (see
\cite[Chapter 1, \S 3]{mfk94}). This naturally extends
to isomorphisms of tensor powers of the bundles, giving
isomorphisms 
\[
 \theta:H^0(H \times X,\Sigma^*(L^{\ten r})) \overset{\cong}{\longrightarrow}
 H^0(H \times X,H \times L^{\ten r}) =\OO(H) \ten
H^0(X,L^{\ten r}), \quad r \geq 0,
\]
where the last equality comes from the K\"{u}nneth formula \cite[Tag
02KE]{stacks-project}.  (Note we abuse notation and suppress mention of $r$ in
the map $\theta$.) Composition of 
$\theta$ with $\Sigma$ thus gives us the co-action (or dual action, cf. \cite[Definition 1.2]{mfk94})  
\begin{equation} \label{eq:GiNa5.1}
\Sigma^*_{\theta}:H^0(X,L^{\ten r}) \overset{\theta \circ
  \Sigma^*}{\longrightarrow} \OO(H) \ten
H^0(X,L^{\ten r}), \quad r \geq 0. 
\end{equation}
For any $h \in H$ and any $r \geq 0$, the linearisation $L^{\ten r} \to X$
yields a linear automorphism of
$H^0(X,L^{\ten r})$ given by the composition
\[
H^0(X,L^{\ten r}) \overset{\Sigma^*_{\theta}}{\longrightarrow} \OO(H) \ten 
H^0(X,L^{\ten
  r}) \overset{\ev_h \ten \id_X^*}{\longrightarrow} H^0(X,L^{\ten r}),    
\]
which satisfies 
\begin{equation} \label{eq:GiNa5.2}
(\ev_h \ten \id_X^*)(\Sigma^*_{\theta}(f))=h^{-1} \cdot f
\end{equation}
for all $f \in H^0(X,L^{\ten r})$. 

\begin{lem} \label{lem:GiNa6}
Let $r \geq 0$ and suppose $V \subseteq H^0(X,L^{\ten r})^N$ is an $H$-stable
subspace of sections. Then the image of $V$ under
$\Sigma^*_{\theta}$ lies in $\OO(H)^N \ten V$, where $\OO(H)^N$ is the ring of
functions invariant under the right multiplication action of $N$ on $H$. (In 
particular, this holds for
$V=H^0(X,L^{\ten r})^N$.)   
\end{lem}

\begin{proof}
This follows from the $H$-equivariance of the co-action (\ref{eq:GiNa5.1}) when $H$ acts on the right-hand side of (\ref{eq:GiNa5.1}) via right multiplication on $H$. In more detail, suppose $f \in V$ is non-zero. Then we may
write $\Sigma^*_{\theta}f= \sum_{j=1}^m a_j \ten f_j$, with $m >0$, $q_j \in
\OO(H)$ and $f_j \in  H^0(X,L^{\ten r})$ such that the $a_j$ and the $f_j$ are
linearly independent over $\kk$. For any 
$h \in H$ we have 
\[
h^{-1} \cdot f=(\ev_h \ten \id_X^*)(\Sigma^*_{\theta}(f))=\sum_j a_j(h)f_j.
\]
We can find $h_1,\dots,h_m \in H$ such that the matrix 
\[
(a_j(h_i))_{i,j}
\]
is invertible:\footnote{A result like this is used in the proof of \cite[Lemma
  3.1]{new78}.} indeed, the morphism $H \to \kk^m$ defined by the $a_j$ has image
not contained in any proper linear subspace of $\kk^m$, so there are
$h_1,\dots,h_m$ such that the $(a_1(h_i),\dots,a_m(h_i)) \in \kk^m$
span $\kk^m$. For such $h_i$, the system of linear equations
\[
h_i^{-1} \cdot f=\sum_{j=1}^m a_j(h_i)f_j, \quad i=1,\dots,m
\]
tells us that each $f_j$ is in the span of $\{h_1^{-1} \cdot f,\dots,h_m^{-1}
\cdot f\} \subseteq V$. So $f_j \in V$ for each $j$. Because $V \subseteq
H^0(X,L^{\ten r})^N$, by the associativity
property of an action and the fact that $N$ is normal in $H$ we have 
\[
\sum_j a_j(hn)f_j=(n^{-1}h^{-1})\cdot f = h^{-1} \cdot f = \sum_j a_j(h)f_j.
\]
for any $n \in N$ and $h \in H$. Since the $f_j$ are linearly independent
$a_j(hn)=a_j(h)$ for all $n \in N$, $h \in H$, so $a_j \in \OO(H)^N$ for each
$j$. Hence $\Sigma^*_{\theta}f \in \OO(H)^N \ten V$.
\end{proof}

\begin{proof}[Proof of Proposition \ref{prop:GiNa7}]

The proof is divided into two steps. We begin by constructing the morphism
$\ol{\Sigma}:(H/N) \times \mc{Y} \to \mc{Y}$ which defines the desired action and show 
that it
maps $(H/N) \times \mc{U}$ to $\mc{U}$, for $\mc{U}$ as in the statement of the
proposition. We then show $\ol{\Sigma}$
satisfies the axioms for a group action and prove the equivariance of
$q_N$. 

(\textbf{Step 1:} Definition of $\ol{\Sigma}$ and restriction to $\mc{U}$.)
Recall that 
\[
(H/N) \times \mc{Y} = \proj(\OO(H/N) \ten S^N),
\]
where the grading in $\OO(H/N) \ten S^N$ is induced by $S^N$, with
$\OO(H/N)$ having degree $0$. The corresponding twisting sheaves $\OO(r)$ are
given by the exterior tensor product $\OO_{H/N} \boxtimes \OO_{\mc{Y}}(r)$ for
each $r \geq 0$ \cite[Tag 01MX]{stacks-project}. Pullback along the quotient map
$H \to H/N$ identifies 
$\OO(H/N)$ with $\OO(H)^N$ and by virtue of Lemma \ref{lem:GiNa6} the diagram 
\begin{equation} \label{eq:GiNa6-1}
\begin{diagram}
S^N & \rTo^{\Sigma^*_{\theta}|_{S^N}} & \OO(H)^N \ten S^N \\
\dInto^{q_N^*} & & \dInto \\
S  & \rTo^{\Sigma^*_{\theta}} & \OO(H) \ten S\\
\end{diagram}
\end{equation} 
of graded rings is well defined and commutes, where $\Sigma_{\theta}^*$ is as in
\eqref{eq:GiNa5.1}. Applying the $\proj$ functor to the top horizontal map
defines a rational map, which we claim is in fact a morphism 
\[
\ol{\Sigma}:=\proj(\Sigma_{\theta}^*|_{S^n}):(H/N) \times \mc{Y} \to \mc{Y}.
\]
To see this, we need to verify that if $f \in S^N$ is a homogenous element of
positive degree, then there is a homogeneous prime ideal of $\OO(H)^N \ten S^N$,
different to the 
irrelevant ideal and not containing $\Sigma^*_{\theta}(f)$. But since $S^N$ is reduced 
there is a
homogeneous prime $\mf{p} \in \mc{Y}=\proj(S^N)$ not containing $f$, and it follows 
that 
\[
(\ev_e \ten \id_S)^{-1}(\mf{p}) \in  \proj(\OO(H)^N \ten S^N)
\]
is a homogeneous prime which does not contain $\Sigma^*_{\theta}(f)$ and is
different to the irrelevant ideal.
    
Now let $\mc{S} \subseteq I^{\nss(N)}$ be a subset that is stable under the
$H$-action on $S$. Notice that this includes the case $\mc{S}=I^{\ssfg(N)}$
by virtue of the discussion before the statement of Proposition
\ref{prop:GiNa7}. Let $\mc{U}=\bigcup_{f \in \mc{S}} \spec((S^N)_{(f)}) \subseteq
\mc{Y}$ and consider the restriction of $\ol{\Sigma}$ to $(H/N) \times
\mc{U}$. Given $y \in \mc{U} \subseteq \mc{Y}$ and $f \in \mc{S}$ such that $f(y) \neq
0$, for any $h \in H$ the section $h \cdot f$ is contained in $\mc{S}$ and maps to
$f$ under the composition $(\ev_h \ten \id_X^*) \circ \Sigma_{\theta}^*$ by
\eqref{eq:GiNa5.2}. Applying $\proj$, this says that  
\[
(h \cdot f)(\ol{\Sigma}(\ol{h},y)) \neq 0 \iff f(y) \neq 0,
\]   
where $\ol{h}=hN$ and we think of $f$ as a section of some
power of $\OO_{\mc{Y}}(1)$. If follows that 
$\ol{\Sigma}$ maps $(\ol{h},y)$ into $\spec((S^N)_{(h \cdot f)}) \subseteq
\mc{U}$, hence $\ol{\Sigma}$ restricts to a map $(H/N) \times \mc{U} \to \mc{U}$. In 
the case
where $\mc{S}=I^{\ssfg(N)}$, we conclude that $\ol{\Sigma}$ restricts to a
morphism $(H/N) \times X \env N \to X \env N$.

(\textbf{Step 2:} $\ol{\Sigma}$ is an action of $(H/N)$ on $\mc{Y}$ and $q_N$ is 
equivariant.) Let $\mu$
(respectively, $\ol{\mu}$) be the morphism defining group multiplication on $H$
(respectively, on $(H/N)$). By using the $\proj$ functor, the commutative
diagrams that $\ol{\Sigma}:(H/N) \times \mc{Y} \to \mc{Y}$ needs to satisfy in order 
to be
an action follow immediately from verifying that the following diagrams of graded
rings commute:
\begin{equation} \label{eq:GiNa6-3}
\begin{diagram}
S^N & &  \\
\dTo^{\Sigma^*_{\theta}|_{S^N}} & \rdTo^{id_{S^N}}  & \\
\OO(H)^N \ten S^N & \rTo_{\ev_{e} \ten \id_{S^N}} & S^N\\
\end{diagram} \tag{Identity}
\end{equation}
and 
\begin{equation} \label{eq:GiNa6-4}
\begin{diagram}
S^N & \rTo^{\Sigma^*_{\theta}|_{S^N}}  & \OO(H)^N \ten
S^N\\
\dTo_{\Sigma^*_{\theta}|_{S^N}} & & \dTo_{\ol{\mu}^* \ten \id_{S^N}} \\
\OO(H)^N \ten S^N & \rTo^{\id_{H/N}^* \ten (\Sigma^*_{\theta}|_{S^N})} &
\OO(H)^N \ten \OO(H)^N \ten S^N\\
\end{diagram}\tag{Associativity}
\end{equation}
Note that  \eqref{eq:GiNa6-4} is well defined by Lemma \ref{lem:GiNa6}. The
diagram \eqref{eq:GiNa6-3} is simply \eqref{eq:GiNa5.2} applied to 
$h=e \in H$. To verify commutativity of diagram \eqref{eq:GiNa6-4},
note that
\[
\ol{\mu}^*:\OO(H)^N \to \OO(H)^N \ten \OO(H)^N
\]
is just the restriction of $\mu^*:\OO(H) \to \OO(H) \ten \OO(H)$ to the subring
$\OO(H)^N$, so \eqref{eq:GiNa6-4} is obtained by restricting the diagram
\[
\begin{diagram}
S & \rTo^{\Sigma^*_{\theta}}  & \OO(H) \ten S \\
\dTo_{\Sigma^*_{\theta}} & & \dTo_{\mu^* \ten
  \id_S} \\
\OO(H) \ten S & \rTo^{\id_H^* \ten \Sigma^*_{\theta}} & \OO(H) \ten \OO(H) \ten S\\
\end{diagram}
\]
to subalgebras of $N$-invariants. But this diagram commutes because $\Sigma:H \times
X \to X$ defines an action.  

Finally, let $\pi:H \to H/N$ be the canonical quotient map. Applying $\proj$ to
the commuting diagram \eqref{eq:GiNa6-1}, we see that $\ol{\Sigma}$ makes the diagram 
\[
\begin{diagram}
H \times X^{\nss(N)} & \rTo^{\Sigma} & X^{\nss(N)}\\
\dTo_{\pi \times q_N}& & \dTo_{q_N} \\
(H/N) \times \mc{Y} & \rTo^{\ol{\Sigma}} & \mc{Y}\\
\end{diagram}
\]
commute, which is to say that $q_N$ is equivariant with respect to the
projection $\pi:H \to H/N$. 
\end{proof}

A consequence of Proposition \ref{prop:GiNa7} is that there is a canonical
action of $H/N$ on the enveloping quotient $X \env N$. The next result says
that, if $\mc{U} \subseteq X \env N$ is a quasi-compact $H/N$-stable open
subscheme, then any sufficiently divisible positive tensor power of the twisting
sheaf $\OO_{\mc{U}}(1)$ has a uniquely defined natural $H/N$-linearisation. 
 
\begin{propn} \label{prop:GiNa8}
Retain the notation preceding Proposition \ref{prop:GiNa7}. Let $\mc{S}
\subseteq I^{\ssfg(N)}$ be a
finite subset such that $\mc{U}=\bigcup_{f \in \mc{S}} \spec((S^N)_{(f)})$ is
stable under the $H/N$-action on $X \env N$ of Proposition
\ref{prop:GiNa7}. 
\begin{enumerate}
\item \label{itm:GiNa8-1} If $r>0$ and $V \subseteq H^0(X,L^{\ten r})^N$ is an
  $H$-stable enveloping system adapted to $\mc{S}$ for the restricted linearisation $N 
\act L \to
  X$, then the immersion $\ol{\phi}:\mc{U}
  \hookrightarrow \PP(V^*)$ of Proposition \ref{prop:GiFi4.1},
  \ref{itm:GiFi4.1-2} is equivariant, and pullback of the canonical linearisation $H/N 
\act
  \OO_{\PP(V^*)}(1) \to \PP(V^*)$ along $\ol{\phi}$ defines a linearisation $(H/N) 
\act
\OO_{\mc{U}}(r) \to \mc{U}$ such that the
natural morphism $L^{\ten r}|_{q_N^{-1}(\mc{U})} \to \OO_{\mc{U}}(r)$ is equivariant 
with respect
to the projection $H \to H/N$. 
\item \label{itm:GiNa8-2} Given
  $r>0$ such that $\OO_{\mc{U}}(r) \to \mc{U}$ is very ample, there is at most one
  $H/N$-linearisation on $\OO_{\mc{U}}(r) \to \mc{U}$ making the natural map $L^{\ten
    r}|_{q_N^{-1}(\mc{U})} \to \OO_{\mc{U}}(r)$ equivariant with respect
to the projection $H \to H/N$. 
\end{enumerate}         
\end{propn}

\begin{proof}
(Proof of \ref{itm:GiNa8-1}.) The action of $H$ on $V$ descends to an action of $H/N$ 
on
$V$, which defines a linearisation $(H/N) \act \OO(1) \to
\PP(V^*)$. We show that $\ol{\phi}$ is equivariant with respect to this
action on $\PP(V^*)$. Let $\Sigma:H \times X \to X$ denote the action
morphism and $\ol{\Sigma}:(H/N) \times \mc{Y} \to \mc{Y}$ the action morphism on
$\mc{Y}=\proj(S^N)$ constructed in Proposition \ref{prop:GiNa7}; note that
$\ol{\Sigma}$ restricts to a morphism $(H/N) \times \mc{U} \to \mc{U}$ by assumption. 
By
Lemma \ref{lem:GiNa6} the linear map $\Sigma_{\theta}^*$ of \eqref{eq:GiNa5.1}
restricts to define a map 
$\Sigma^*_{\theta}|_V:V \to \OO(H)^N \ten V$. Applying the $\symdot$ functor, we
get a homomorphism of graded rings
\[
\symdot(\Sigma^*_{\theta}|_V):\symdot V \to \OO(H)^N \ten \symdot V,
\]
where $\OO(H)^N$ is in degree zero in the latter ring, and applying
$\proj$ to this homomorphism recovers the action of $H/N$ on $\PP(V^*)$ just
described. Furthermore, the following diagram of graded rings commutes (recall
$(S^N)^{(r)}$ is the $r$-th Veronese subring of $S^N$):
\[
\begin{diagram}
 \symdot V & \rTo^{\symdot(\Sigma^*_{\theta}|_V)}  & \OO(H)^N
 \ten \symdot V \\
 \dTo_{\text{mult}} & & \dTo_{\id_H^* \ten \text{mult}} \\
(S^N)^{(r)} & \rTo^{\Sigma^*_{\theta}|_{S^N}} & \OO(H)^N \ten (S^N)^{(r)}\\
\dTo & & \dTo \\
\kk[\mc{U},\OO_{\mc{U}}(r)] & \rTo^{\ol{\Sigma}^*} & \OO(H)^N \ten
\kk[\mc{U},\OO_{\mc{U}}(r)] \\
\end{diagram}
\]     
Under the identification $\symdot V = \kk[\PP(V^*),\OO(1)]$, the
composition of the left-hand vertical arrows corresponds to pull-back along 
$\ol{\phi}$ and, by the K\"{u}nneth isomorphism, the composition of the
right-hand vertical arrows corresponds to pulling back along $\id_{H/N} \times
\ol{\phi}$. Applying $\proj$ to this diagram, it follows that $\ol{\phi}:\mc{U}
\hookrightarrow \PP(V^*)$ is $H/N$-equivariant. 

Define $(H/N) \act \OO_{\mc{U}}(r) \to \mc{U}$ 
to be the linearisation obtained by pulling back $H/N \act \OO(1) \to \PP(V^*)$
under $\ol{\phi}$. We have
$q_N^*\OO_{\mc{U}}(r)=L^{\ten r}|_{q_N^{-1}(\mc{U})}$ as line bundles; let $
\psi:L^{\ten
  r}|_{q_N^{-1}(\mc{U})} \to \OO_{\mc{U}}(r)$ be the naturally induced bundle map. To
show $\psi$ is equivariant with respect to $H
\to H/N$, argue as follows. The image of $\ol{\phi}^*:H^0(\PP(V^*),\OO(1)) \to
H^0(\mc{U},\OO(r))$ is an $H/N$-stable subspace of $H^0(\mc{U},\OO(r))$ that
pulls back  under $q_N$ to the linear system $V \subseteq H^0(X,L^{\ten r})^N$,
which is basepoint free on $q_N^{-1}(\mc{U})$. Let $x \in q_N^{-1}(\mc{U})$, let
$f \in V$ such that $f(x) \neq 0$ and let $F \in H^0(\mc{U},\OO(r))$ with
$q_N^*F=f$. Then because $q_N^*$ is equivariant with respect to the natural
$H/N$-actions on $V$ and $H^0(\mc{U},\OO(r))$, for any $h \in H$ we have
\[
h f(x)=(h\cdot f)(hx)=(h\cdot (q_N^*F))(hx)=(q_N^*(\ol{h} \cdot F))(hx),
\]   
(where $\ol{h}=hN \in H/N$), whence
\[
\psi(hf(x))=(\ol{h} \cdot F)(q_N(hx)) = \ol{h} F(q_N(x))=\ol{h} \psi(f(x)).
\]
It follows by linearity that $\psi(hl)=\ol{h}\psi(l)$ for any $l \in L^{\ten
  r}|_x$. Hence $\psi$ is equivariant with respect to $H \to H/N$.   

(Proof of \ref{itm:GiNa8-2}.) Suppose now $r>0$ is such that $\OO_{\mc{U}}(r)
\to \mc{U}$ is equipped with two $H/N$-linearisations $\mc{L}_1,\mc{L}_2$ such
that the natural maps $L^{\ten r}|_{q_N^{-1}(\mc{U})} \to \mc{L}_1$ and $L^{\ten 
  r}|_{q_N^{-1}(\mc{U})} \to \mc{L}_2$ are both equivariant with respect to the 
projection
$H \to H/N$. Then the inclusions 
\begin{align*}
q_N^*:H^0(\mc{U},\mc{L}_1) &\hookrightarrow H^0(q_N^{-1}(\mc{U}),L^{\ten r})^N, \\
q_N^*:H^0(\mc{U},\mc{L}_2) &\hookrightarrow H^0(q_N^{-1}(\mc{U}),L^{\ten r})^N
\end{align*}
are both $H/N$-equivariant linear maps, therefore the $H/N$-actions on
$H^0(\mc{U},\mc{L}_1)$ and $H^0(\mc{U},\mc{L}_2)$ agree. Because
$\OO_{\mc{U}}(r) \to \mc{U}$ is very ample, by
Lemma \ref{lem:TrRe4} we can find a finite dimensional rational $H/N$-module
$W \subseteq H^0(\mc{U},\mc{L}_1)=H^0(\mc{U},\mc{L}_2)$, that is also a complete
linear system, with which to equivariantly embed $\mc{U}$ into $\PP(W^*)$. Then
the restriction of $H/N 
\act \OO_{\PP(W^*)}(1) \to \PP(W^*)$ to $\mc{U}$ is equal to both the
linearisations 
$\mc{L}_1$ and $\mc{L}_2$, so that $\mc{L}_1=\mc{L}_2$.    
\end{proof}

Given an inner enveloping quotient $q_H:X^{\ssfg} \to \mc{U}$ for the
$H$-linearisation $L \to X$, it is natural and desirable to want to factorise it
through an inner enveloping quotient for the $N$-linearisation $N \act L \to X$
obtained by restriction. Unfortunately there is a theoretical obstruction to doing
this, in that there may be $H$-invariant sections $f$ over $X$ where
$(S^H)_{(f)}$ is a finitely generated algebra, but $(S^N)_{(f)}$ is not. 

\begin{eg} \label{eg:GiNa9}
Consider any example where $N$ is a linear
algebraic group acting linearly on a finitely generated graded $\kk$-algebra
$A=\bigoplus_{d \geq 0} A_d$, with $A_0=\kk$, such that $A^N$ is not 
finitely generated over $\kk$ (e.g. Nagata's Example
\ref{ex:TrRe2.4}). Let $X=\spec A$ and $L=\OO_X$ with the canonical
$N$-linearisation. Because $N$ respects the grading on $A$, there is a
linearisation of $H=N \times \GG_m$ on $L \to X$, where $\GG_m \act L \to X$ is the
canonical linearisation defined by the grading on $A$. Now consider $f=1 \in
H^0(X,L)^{H}=A^H$. Then $(S^H)_{(1)}=(\kk[X,L]^H)_{(1)}=A^H=\kk$ is finitely
generated over $\kk$, because the only
$\GG_m$-invariants in $A$ are the constant functions. But $(S^N)_{(1)} = A^N$ is
not finitely generated over $\kk$.    
\end{eg}

One can remedy this issue by considering a definition of
`semistability' that forces finite generation of both the algebras
$N$-invariants and $H$-invariants, as follows. 
The inclusion $S^H \hookrightarrow S^N$ induces a rational map
\begin{equation} \label{eq:GiNa10}
q_{H/N}:\proj(S^N) \dashrightarrow \proj(S^H)
\end{equation}
that is invariant with respect to the canonical $H/N$-action on $\proj(S^N)$ of
Proposition \ref{prop:GiNa7}. For the purpose of this
discussion, let $X^{\rmss,N \dashfg}$ denote the union of all $X_f$ such that
$f \in \bigcup_{r>0}H^0(X,L^{\ten r})^H$ is an $H$-invariant section with
both of the $\kk$-algebras  
$(S^N)_{(f)}$ and $(S^H)_{(f)}$ finitely generated. For each such $f$ we
have $\spec((S^N)_{(f)}) \subseteq X \env N$ and the
restriction of
$q_{H/N}$ to $\spec((S^N)_{(f)})$ maps into $\spec((S^H)_{(f)})  \subseteq X \env H$. 
Letting $\mc{U}$ be the union of the
$\spec((S^N)_{(f)})$ defined by such $f$, we see that the rational map $q_{H/N}$
restricts to give a well-defined morphism $q_{H/N}:\mc{U} \to X \env
H$. Furthermore, if $q_H:X^{\ssfg(H)} \to X \env H$ and $q_N:X^{\ssfg(N)} \to X
\env N$ are the enveloping quotients for the $H$- and $N$-linearisations on $L
\to X$, respectively, then the
diagram
\[
\begin{diagram}
X^{\ssfg(H)} & \rEmpty~\supseteq & X^{\rmss,N \dashfg} & \rEmpty~\subseteq &
X^{\ssfg(N)}\\
& \rdTo(2,4)_{q_H} &  \dTo_{q_N} & & \dTo_{q_N} \\
&  & \mc{U} & \rEmpty~\subseteq &  X \env N \\
& & \dTo_{q_{H/N}}& & \\
& & X \env H & & \\  
\end{diagram}
\]
commutes, with all inclusions open. We will take up this theme in the next
section when talking about stability.  

    








    
\subsection{Stability for Non-Reductive Linearisations}
\label{sec:GiStability}

We now turn to the question of defining an open subset of `stable' points of $X$
for a given linearisation $L \to X$ of a linear algebraic group $H$, which
admits a geometric $H$-quotient. A basic requirement we demand of such a definition is
that it should extend the definitions of
stability in the cases where $H$ is reductive (Definition \ref{def:TrRe5},
\ref{itm:TrRe5-2}) or unipotent (Definition \ref{def:BgUn5}). We will do this in
the case where $X$ is \emph{irreducible}, since it will be helpful to make 
use of Remark \ref{rmk:GiFi2.5}.

\subsubsection{An Intrinsic Definition of Stability}
\label{sec:GiStIntrinsic}
 
Recall from Definition \ref{def:TrRe3}, \ref{itm:TrRe3-2} that for any linear
algebraic group $H$ there is a canonical  
normal unipotent subgroup $H_u$ of $H$, called the unipotent radical of $H$,
with the property that the quotient $H_r=H/H_u$ is a reductive group. According to 
Proposition
\ref{prop:GiNa7}, the enveloping quotient
$X \env H_u$ for the restricted linearisation $H_u
\act L \to X$ has a canonical $H_r$-action which makes the enveloping quotient
map $q_{H_u}:X^{\ssfg(H_u)} \to X \env H_u$ equivariant
with respect to $H \to H_r$. This action comes from an
action of $H_r$ on $\proj(S^{H_u})$ which has the property that the
rational map 
\begin{equation} \label{eq:GiSt1}
q_{H_r}:\proj(S^{H_u}) \dashrightarrow \proj(S^H)
\end{equation}
defined via \eqref{eq:GiNa10} is $H_r$-invariant. 

Given an $H$-invariant section $f$ of some
positive tensor power of $L \to X$ such that $(S^{H_u})_{(f)}$ is a
finitely generated $\kk$-algebra, the basic open set
$\spec((S^{H_u})_{(f)})$ is stable under the
$H_r$-action on $X \env H_u$. The composition of $q_{H_u}:X_f \to
\spec((S^{H_u})_{(f)})$ with $q_{H_r}$ coincides with the restriction of the 
enveloping
quotient map $q_H:X_f \to \spec((S^{H})_{(f)})$. Because $X$ is irreducible, we
have
\[
(S^{H})_{(f)} = \left((S^{H_u})^{H_r}\right)_{(f)}= \left((S^{H_u})_{(f)}\right)^{H_r}
\]
(see Remark \ref{rmk:GiFi2.5}) and since $H_r$ is reductive we thus have
$(S^H)_{(f)}$ finitely generated over $\kk$ and
$\spec((S^{H})_{(f)}) \subseteq X \env H$. Following the ideas of unipotent GIT
in Section \ref{subsec:BgUnIntrinsic}, we define a notion of stability for the
linearisation $H \act L \to X$ by requiring the restriction of the enveloping
quotient map $q_H:X^{\ssfg(H)} \to X 
\env H$ for $H \act L \to X$ to give a geometric quotient $X_f \to
\spec((S^H)_{(f)})$ for the $H$-action on $X_f$, for suitably chosen sections $f$. 
There
are a number of ways one could go about doing this. For example, it is easy to see
that if each of $q_{H_u}$ and $q_{H_r}$ define geometric quotients for the $H_u$-
and $H_r$-actions on $X_f$ and
$\spec((S^{H_u})_{(f)})$, respectively, then the composition $q_H$ is a
geometric quotient for $H \act X_f$. But we also want to build on Definition
\ref{def:BgUn5} of stability from
\cite{dk07}, where one takes the $X_f$ which are affine and admit a locally
trivial geometric quotient. So it makes sense to further require that
$X_f$ is affine and $q_{H_u}:X_f \to \spec((S^{H_u})_{(f)})$ is a
principal bundle for the action of $H_u$ on $X_f$. By
\cite[Amplification 1.3]{mfk94}, the induced action of the reductive group $H_r$ on $
\spec((S^{H_u})_{(f)})$ has a geometric quotient if, and only if, all the
orbits are closed in $\spec((S^{H_u})_{(f)})$, and following the ideas of
stability in reductive GIT it also natural to demand that the stabilisers for
this action are finite. Because the action of
$H_u$ on $X_f$ is free and $H_u$ is normal in $H$, these last conditions can be
lifted to the action of $H$ on $X_f$ using the following lemma.

\begin{lem} \label{lem:GiSt1}
Suppose $H$ is a linear algebraic group, $N$ is a normal subgroup of $H$ and
$X$ is an $H$-variety (not necessarily assumed 
irreducible). Suppose all the stabilisers for the restricted
action $N \act X$ are finite and this action has a geometric quotient
$\pi:X \to X/N$. Note that $H/N$ acts canonically on $X/N$. Then 
\begin{enumerate}
\item \label{itm:GiSt1-1} for all the $H/N$-orbits in $X/N$ to be closed, it is
  necessary and sufficient that all the $H$-orbits in $X$ are closed;
\item \label{itm:GiSt1-2} given $y \in X/N$, the stabiliser $\stab_{H/N}(y)$ is finite 
if, and
only if, $\stab_H(x)$ is finite for some (and hence all) $x \in
\pi^{-1}(y)$; and
\item \label{itm:GiSt1-3} if $H/N$ is reductive and $X/N$ is affine, then $X/N$
  has a geometric $H/N$-quotient if, and only if, all $H$-orbits in $X$ are closed.  
\end{enumerate}        
\end{lem} 

\begin{proof}
(Proof of \ref{itm:GiSt1-1}.) Let $x \in X$ and $y=\pi(x)$. We first show that
$H \cdot x = \pi^{-1}((H/N) \cdot y)$. Clearly $H \cdot x \subseteq
\pi^{-1}((H/N) \cdot y)$, because $\pi$ is equivariant with respect to the
projection $H \to H/N$. On the other hand, 
if $x^{\prime} \in \pi^{-1}((H/N) \cdot y)$, then there is $h \in H$
such that $y=\ol{h}\pi(x^{\prime})=\pi(hx^{\prime})$. Since $\pi^{-1}(y)=N \cdot
x$, there is therefore $n \in N$ such that $x^{\prime} = h^{-1}nx \in H \cdot
x$. Hence $H \cdot x=\pi^{-1}((H/N)\cdot y)$. Because $\pi$ is a submersion, $H
\cdot x$ is closed if, and only if, $(H/N)\cdot y$ is closed. Since $\pi$ is
surjective, this suffices to prove \ref{itm:GiSt1-1}.   

(Proof of \ref{itm:GiSt1-2}.) Suppose $y \in X/N$ has finite stabiliser in $H/N$
and again let $x \in \pi^{-1}(y)$. Then  
\[
\stab_{H/N}(y) = \{ g_1N, \dots, g_m N\}
\]
for some finite collection of representatives $g_1,\dots,g_m \in H$, which we
fix once and for all, such that the cosets $g_iN$
are pairwise disjoint. If $h \in \stab_H(x)$ then $\ol{h}=hN \in
\stab_{H/N}(y)$, so $h$ is contained in a unique coset $g_{i(h)}N$,
where $i(h) \in \{1,\dots,m\}$. In this way we define a function
\[
\stab_H(x) \to \{g_1,\dots,g_m\}, \quad h \mapsto g_{i(h)}.
\]
We claim the fibres of this function are finite. Indeed, let $h \in
\stab_H(x)$ and suppose $\tilde{h} \in \stab_H(x)$ is such that  
$g_{i_0}:=g_{i(h)}=g_{i(\tilde{h})}$, with $i_0 \in \{1,\dots,m\}$. Then we may
find $n,\tilde{n} \in N$ such that $h=ng_{i_0}$ and
$\tilde{h}=\tilde{n}g_{i_0}$, so $ng_{i_0}x=\tilde{n}g_{i_0}x$. It follows that,
for some $p \in \stab_N(g_{i_0}x)$, we have
$\tilde{n}=pn$ and $\tilde{h}=ph$. Since all stabilisers for the $N$-action
on $X$ are finite, there are finitely many choices for $\tilde{h}$ and hence the
fibre containing $h$ is finite, as claimed. We 
conclude that $\stab_H(x)$ is finite for any $x \in \pi^{-1}(y)$.  

Conversely, suppose $x \in X$ has finite stabiliser in $H$ and let $y=\pi(x) \in
X/N$. Let $h \in H$ such that $\ol{h}=hN \in \stab_{H/N}(y)$. Then
$\pi(x)=\ol{h}\pi(x) = \pi(hx)$ and, because $\pi$ is a geometric $N$-quotient
and $N$ is normal in $H$,  
there is $n \in N$ such that $hnx=x$. Hence
$hn \in \stab_H(x)$ and $\ol{h}$ is in the image 
of $\stab_H(x)$ under the quotient map $H \to H/N$. Thus $\stab_{H/N}(y)$ is
finite. 

(Proof of \ref{itm:GiSt1-3}.) If a geometric quotient $X/N \to (X/N)/(H/N)$
exists then the composition $X 
\overset{\pi}{\longrightarrow} X/N \longrightarrow (X/N)/(H/N)$ is a geometric
quotient for the $H$-action on $X$, which implies that all the $H$-orbits in $X$
are closed. Now suppose all the $H$-orbits in $X$ are closed. Because $X/N$ is
affine and $H/N$ is
reductive the categorical quotient of $X/N$ by $H/N$ exists by Theorem
\ref{thm:BgRe1}. Every $H/N$-orbit in $X/N$ is closed by \ref{itm:GiSt1-1}, so
the categorical quotient  
of $X/N$ by $H/N$ is a geometric quotient \cite[Amplification 1.3]{mfk94}.      
\end{proof}

In light of the above lemma and the preceding discussion, we make the following 
definition.
\begin{defn} \label{def:GiSt1.1}
Let $H$ be a linear algebraic group (with as usual unipotent radical $H_u$ and $H_r = H/H_u$ reductive) acting on an irreducible variety $X$ and let $L \to X$ be
a
linearisation for the action. The \emph{stable locus} is the open subset  \index{stable}
\[
X^{\rms}:= \bigcup_{f \in I^{\rms}} X_f
\]
of $X^{\nss}$, where $I^{\rms} \subseteq \bigcup_{r>0} H^0(X,L^{\ten r})^H$ is the 
subset
of $H$-invariant sections satisfying the following conditions:
\begin{enumerate}
\item \label{itm:GiSt1.1-1} the open set $X_f$ is affine;
\item \label{itm:GiSt1.1-2} the action of $H$ on $X_f$ is closed with all
  stabilisers finite groups; and
\item \label{itm:GiSt1.1-3} the restriction of the $H_u$-enveloping quotient map
 \[
q_{H_u}:X_f \to \spec((S^{H_u})_{(f)})
\]
is a principal $H_u$-bundle for the action of $H_u$ on $X_f$. 
\end{enumerate}
\end{defn}

\begin{rmk} \label{rmk:GiSt2}
It is clear that this definition of stability extends the definition of
stability in \cite{dk07} for unipotent groups (see Definition \ref{def:BgUn5}). In the 
case
where $H$ is reductive, then $H_u$ is trivial and our definition reduces to Mumford's 
notion of 
properly stable points \cite{mfk94} (see Definition \ref{def:TrRe5}, 
\ref{itm:TrRe5-2}).    
\end{rmk}

\begin{rmk} \label{rmk:GiSt2.1}
Observe that requiring \ref{itm:GiSt1.1-1} and
\ref{itm:GiSt1.1-3} in Definition \ref{def:GiSt1.1} is equivalent to demanding
that $X_f$ be an affine open subset of $X$ that is a \emph{trivial}
$H_u$-bundle by Proposition \ref{prop:TrRe3.1}. 
\end{rmk}

The significance of assuming that $X$ is irreducible in Definition
\ref{def:GiSt1.1} is that it ensures  
\[
\spec((S^{H})_{(f)})=\spec(((S^{H_u})_{(f)})^{H_r}),
\]
so that by reductive GIT for affine varieties $q_{H_r}:\spec((S^{H_u})_{(f)}) \to 
\spec((S^{H})_{(f)})$
is at least a good categorical $H_r$-quotient when $f$ is an $H$-invariant
(see Theorem \ref{thm:BgRe1}) and $\spec((S^H)_{(f)}) \subseteq X \env H$. Then
if $f\in I^{\rms}$, conditions
\ref{itm:GiSt1.1-1}--\ref{itm:GiSt1.1-3} in Definition \ref{def:GiSt1.1},
combined with Lemma \ref{lem:GiSt1}, \ref{itm:GiSt1-3}, tell us that
\[
q_H:X_f \overset{q_{H_u}}{\longrightarrow} \spec((S^{H_u})_{(f)})
\overset{q_{H_r}}{\longrightarrow} \spec((S^{H})_{(f)})
\]
is a composition of geometric quotients, hence a geometric quotient for $H \act X_f$. 
Because
the property of being a geometric quotient is local on the base, it 
follows that the enveloping quotient 
$q_H:X^{\ssfg(H)} \to X \env H$ restricts to define a geometric quotient 
\[
q_H:X^{\rms} \to X^{\rms}/H = q_H(X^{\rms}). 
\]
This factorises through the restriction of the enveloping quotient for $H_u$
in a natural way, and we have the following commutative diagram, with
all inclusions open:
\[
\begin{diagram}
X^{\ssfg(H)} & \rEmpty~\supseteq & X^{\rms} & \rEmpty~\subseteq & X^{\ssfg(H_u)} \\
& & \dTo_{q_{H_u}}^{\text{geo}} & & \dTo_{q_{H_u}} \\
\dTo_{q_H} & & X^{\rms}/H_u & \rEmpty~\subseteq & X \env H_u \\
 & & \dTo_{q_{H_r}}^{\text{geo}} & &  \\
X \env H & \rEmpty~\supseteq & X^{\rms}/H & & \\ 
\end{diagram}
\] 
\begin{rmk} \label{rmk:GiSt2.2}
If $X$ is irreducible and $q_H:X^{\ssfg(H)} \to \mc{U}$ is any inner enveloping 
quotient for the
linearisation $H \act L \to X$, then the geometric quotient $X^{\rms}/H$ of
$X^{\rms}$ is naturally an open subvariety of $\mc{U}$.    
\end{rmk}

One of the features of the stable locus defined in Definition \ref{def:GiSt1.1}
is that it behaves well under affine locally closed immersions, as the next
lemma shows. 

\begin{lem} \label{lem:GiSt1.2}
Let $H$ be a linear algebraic group acting on an irreducible variety $X$ with
linearisation $L \to X$. Suppose $Y$ is another irreducible variety and $\phi:Y
\hookrightarrow X$ is an $H$-equivariant locally closed immersion that is an
affine morphism. Then $\phi^{-1}(X^{\rms})$ is an open subset of $Y^{\rms(\phi^*L)}$,
the image of $\phi^{-1}(X^{\rms})$ under the enveloping quotient
$q^{\prime}:Y^{\ssfg(\phi^*L)} \to Y \env_{\! \! \phi^*L} \, H$ is a geometric
quotient for the $H$-action on $\phi^{-1}(X^{\rms})$, and there is a locally
closed immersion $\bar{\phi}:\phi^{-1}(X^{\rms})/H \hookrightarrow
X^{\rms}/H$ such that the following diagram commutes (with unmarked inclusions
open)
\[
\begin{diagram}
Y^{\rms(\phi^*L)} &  \rEmpty~\supseteq &  \phi^{-1}(X^{\rms}) & \rInto^\phi & X^{\rms}\\
\dTo_{q^{\prime}} & & \dTo_{q^{\prime}}^{\mathrm{geo}} & &  \dTo_{q} \\
Y^{\rms(\phi^*L)}/H  & \rEmpty~\supseteq &  \phi^{-1}(X^{\rms})/H & \rInto^{\ol{\phi}} & X^{\rms}/H\\    
\end{diagram}
\]
\end{lem}

\begin{proof}
Let $R=\kk[Y,\phi^*L]$ and $S=\kk[X,L]$. The set $\phi^{-1}(X^{\rms})$
is covered by open subsets of the form $Y_{\phi^*f}$, where $f$ is a section
in $I^{\rms} \subseteq \bigcup_{r>0} H^0(X,L^{\ten r})^H$ of
Definition \ref{def:GiSt1.1}, and by Remark \ref{rmk:GiSt2.1} the map $q:X_f \to
\spec((S^H)_{(f)})$ is a trivial $H_u$-bundle for such an $f$. For each such $f$ the 
open subset
$Y_{\phi^*f}=\phi^{-1}(X_f)$ is
affine because $\phi$ is affine and $X_f$ is affine. It is clear that the action of $H
$
on $Y_{\phi^*f}$ is closed with all stabilisers finite. By restriction
$Y_{\phi^*f}$ also has the structure of a trivial $H_u$-bundle, thus 
$Y_{\phi^*f}/H_u$ is affine and isomorphic to $\spec((R^{H_u})_{(\phi^*f)})$
(because $Y$ is irreducible). Thus $Y_{\phi^*f} \subseteq
Y^{\rms(\phi^*L)}$ and the restriction of the enveloping quotient map,
\[
q^{\prime}:Y_{\phi^*f} \to \spec((R^H)_{(\phi^*f)}),
\]
is a geometric quotient for the $H$-action on $Y_{\phi^*f}$. On the other hand, by the 
submersion
property of a geometric quotient the image of $q \circ 
\phi:Y_{\phi^*f} \to X_f/H$ is a locally closed subset of
$X_f/H$ that is also a geometric quotient for
the $H$-action on $Y_{\phi^*f}$, hence
there is a unique locally closed immersion 
\[
\spec((R^H)_{(\phi^*f)}) \hookrightarrow X_f/H
\]
factoring $(q \circ \phi)|_{Y_{\phi^f}}$ through $q^{\prime}|_{Y_{\phi^f}}$. By
varying over suitable $f$, we see that $\phi^{-1}(X^{\rms}) \subseteq 
Y^{\rms(\phi^*L)}$ and the above morphisms $\spec((R^H)_{(\phi^*f)}) \hookrightarrow 
X_f/H$ glue
to give the locally closed immersion 
\[
\ol{\phi}:q^{\prime}(\phi^{-1}(X^{\rms}))=\phi^{-1}(X^{\rms})/H \hookrightarrow
X^{\rms}/H
\]
making the required diagram commute.
\end{proof} 


\subsubsection{Relation to Stability of Reductive Extensions}
\label{subsec:GiStRelation}

We next consider how the notion of stability for a linear action of $H$ proposed in Definition \ref{def:GiSt1.1}
relates to stability for a reductive group acting on the fibre space $\GUX$
associated to certain embeddings $H_u \hookrightarrow G$, with $G$
reductive, where $H_u$ is the unipotent radical of $H$. The resulting Proposition \ref{prop:GiSt3} provides an extension of
the fact that the Mumford stable locus is equal to the locally trivial stable
locus in the unipotent setting (see Proposition \ref{prop:BgUn9}). This will
be important for the theory of reductive envelopes in the upcoming
Section \ref{sec:compactifications1}, and shows more generally how Mumford's
reductive GIT can be used to study non-reductive group actions.  

It will be convenient to make the following definition. 
\begin{defn} \label{def:GiSt1.3} 
Let $H$ be a linear algebraic group and $G$ a reductive group. A homomorphism
$H \to G$ is called \emph{$H_u$-faithful} if its restriction to the  \index{$H_u$-faithful homomorphism} 
unipotent radical $H_u$ of $H$ defines a closed embedding $H_u \hookrightarrow G$.    
\end{defn}

Fix a reductive group $G$ and an $H_u$-faithful homomorphism $\rho:H \to G$. Then we
can consider the fibre
space $G \times^{H_u} X$ associated to $X$ and the homomorphism $\rho|_{H_u}:H_u 
\hookrightarrow
G$, together with its natural closed immersion $\alpha:X \hookrightarrow \GUX$.
This space exists as a variety because $H_u$ is unipotent \cite[Proposition
23]{eg98}, hence the natural projection 
\[
G \times X \to G\times^{H_u}X
\]
is a geometric quotient for the diagonal action of $H_u$ in the category of
varieties. Because $H_u$ is normal in $H$ the diagonal action of $H$ on $G \times X$,
induced by the action of $H$ on $X$ and right multiplication on $G$ through
$\rho$, descends 
through this projection to give an action of $H_r=H/H_u$ on
$\GUX$. Explicitly, this given by
\[
\ol{h} \cdot [g,x] = [g\rho(h)^{-1},hx] 
\]
for each $g\in G$, $x\in X$, and $\ol{h}=hH_u \in H_r$. This action of $H_r$
commutes with the $G$-action on $\GUX$, so we can view 
$\GUX$ as a $G\times H_r$-variety in a natural way. Notice that the inclusion
$\alpha:X \hookrightarrow \GUX$ is equivariant with respect to
the diagonal embedding $H \hookrightarrow G \times H_r$ induced by $\rho$ and the
quotient $H \to H_r$. 
 
As noted in Section \ref{subsec:BgUnExtending},
there is a natural
$G$-linearisation over $\GUX$ which extends the $H_u$-linearisation on $L$ under
the inclusion $\alpha$. By abuse of notation, we denote this linearisation $L=G
\times^{H_u} L \to \GUX$. The diagonal  
$H_r$-action on $\GUX$ canonically lifts to 
the line bundle $L$ to define an $H_r$-linearisation on $L \to \GUX$ which
commutes with the $G$-linearisation, hence there is a natural linearisation 
\[
G \times H_r \act L \to \GUX.
\]
This provides an extension of the $H$-linearisation $H \act L \to X$, when we
let $H$ act on $L=G \times^{H_u}L$ via the diagonal homomorphism 
$H \to G \times H_r$. As such, pulling back sections
along $\alpha$ induces an isomorphism
\[
\alpha^*:\kk[\GUX,L]^{G\times H_r} \overset{\cong}{\longrightarrow} \kk[X,L]^{H}.
\]
We may now state
\begin{propn} \label{prop:GiSt3}
Let $H$ be a linear algebraic group (with unipotent radical $H_u$ and $H_r = H/H_u$ reductive) acting on an irreducible variety $X$ with
linearisation $H \act L \to X$, and let
$G$ be a reductive 
group with an $H_u$-faithful homomorphism $\rho:H \to G$. Let $\alpha$ be the natural 
closed
immersion of $X$ into $\GUX$ and let 
$(\GUX)^{\rms(L)}$ be the stable locus for the $G
\times H_r$-linearisation $L \to \GUX$ in the sense of Definition
\ref{def:TrRe5}, \ref{itm:TrRe5-2}. Then
\[
X^{\rms} = \alpha^{-1}((\GUX)^{\rms(L)}).
\]
\end{propn}
 
\begin{proof} 
(Proof of $\alpha^{-1}((\GUX)^{\rms(L)})  \subseteq X^{\rms}$.) Suppose $x \in
X$ and $\alpha(x) \in (\GUX)_{F}$, where $F$ is a $G\times H_r$-invariant
section of some positive tensor power of $L$ over
$\GUX$ such that
$(\GUX)_{F}$ is affine and the $G \times H_r$-action on $(\GUX)_{F}$ is closed
with 
finite stabilisers. Let $f = \alpha^*F$ be the corresponding
$H$-invariant over $X$, so that $x \in X_f$. Then $\alpha$ restricts to an 
$H$-equivariant closed immersion $X_f \hookrightarrow (\GUX)_{F}$, thus $X_f$ is
affine. For any $y \in X_f$ the orbit $H \cdot y =
\alpha^{-1}((G \times H_r) \cdot \alpha(y))$ is closed in $X_f$ and
$\stab_H(y) \subseteq \stab_{G\times H_r}(\alpha(y))$, so all $H$-orbits in
$X_f$ are closed and all stabilisers for the
$H$-action on $X_f$ are finite. It remains to show that the restriction of the
$H_u$-enveloping quotient 
\[
q_{H_u}:X_f \to \spec((S^{H_u})_{(f)})
\]
gives $X_f$ the structure of a principal $H_u$-bundle. The action of
$G$ on $(\GUX)_{F}$ is set-theoretically free, because all its stabilisers are
conjugate to subgroups of the unipotent group $H_u$ and, since they are finite,
are thus 
trivial. Furthermore, all
$G$-orbits in $(\GUX)_{F}$ are closed because the action of $G \times
H_r$ on $(\GUX)_{F}$ is proper \cite[Corollary 2.5]{mfk94}. Hence
$(\GUX)_{F}$ is in the stable locus for the restricted linearisation $G \act
L \to \GUX$. The action of $G$ on $(\GUX)_{F}$ is therefore proper
(\cite[Corollary 2.5]{mfk94} again) and so the action of $G$ on $(\GUX)_{F}$
is free by Lemma \ref{lem:TrRe1}. The subset $(\GUX)_{F}$ has an affine geometric
quotient $(\GUX)_{F}/G \cong X_f/H_u$ by Theorem \ref{thm:BgRe1},
\ref{itm:BgRe1-2}, which by \cite[Proposition 0.9]{mfk94} is actually a locally
trivial quotient. By descent \cite[Proposition 10]{ser58} this means $X_f$ has
an affine locally trivial geometric quotient, isomorphic to
$\spec(\OO(X_f)^{H_u})=\spec((S^{H_u})_{(f)})$. So $q_{H_u}:X_f \to
\spec((S^{H_u})_{(f)})$ is a locally trivial $H_u$-quotient.      

(Proof of $X^{\rms} \subseteq \alpha^{-1}((\GUX)^{\rms(L)})$.) Let $x \in X_f$,
where $f$ is an $H$-invariant section of some positive tensor power of $L \to X$
such that $X_f$ is affine, has closed 
$H$-orbits with all stabilisers finite and $q_{H_u}:X_f \to
\spec((S^{H_u})_{(f)})$ is a principal $H_u$-bundle. Let $F$ be the
$G \times H_r$-invariant over $\GUX$ pulling back to $f$ under
$\alpha$. By \cite[Proposition 5]{ser58}, the natural
morphism $(\GUX)_{F}=G \times^{H_u}(X_f) \to X_f/H_u$ is a principal $G$-bundle
with affine base. By \cite[Proposition 0.7]{mfk94} this means $(\GUX)_{F} \to
X_f/H_u$ is an affine morphism, hence $(\GUX)_{F}$ is affine. Now, any $G
\times H_r$-orbit in $(\GUX)_{F}$ is the image $G \times^{H_u} O$ of a 
subset of the form $G \times O \subseteq G \times X_f$ under the geometric
quotient $G \times X_f \to (\GUX)_{F}$, where $O \subseteq X_f$ is an
$H$-orbit. Since $O$ is 
closed in $X_f$, so too is the $G \times H_r$-orbit $G \times^{H_u} O$ inside
$(\GUX)_{F}$. Hence all $G \times H_r$-orbits in $(\GUX)_{F}$ are
closed. Moreover, because any point in $(\GUX)_{F}$ is in the
$G$-sweep of a point in $X_f$ via $\alpha$, any stabiliser for the $G
\times H_r$-action on $(\GUX)_{F}$ is conjugate to an
$H$-stabiliser for a point in $X_f$ 
under the inclusion $H \hookrightarrow G \times H_r$ induced by $\rho$ and $H \to
H_r$. Hence all stabilisers for the $G \times H_r$-action on $(\GUX)_{F}$ are
finite. It follows that $(\GUX)_{F} \subseteq
(\GUX)^{\rms(L)}$ and $x \in \alpha^{-1}((\GUX)^{\rms(L)})$.    
\end{proof}

\begin{rmk} \label{rmk:GiSt3}
For future reference we note the following fact, which was shown during the
proof of 
Proposition \ref{prop:GiSt3}: given an $H$-linearisation $L \to X$, an
$H_u$-faithful homomorphism $H \to G$ and an $H$-invariant
section $f$ of some positive tensor power of $L \to X$ with associated $G \times
H_r$-invariant section $F$ over $\GUX$, 
then $(\GUX)_F=G \times^{H_u} (X_f)$ is affine if, and only if, $X_f$ is affine
and $X_f \to 
\spec((S^{H_u})_{(f)})$ a principal $H_u$-bundle.      
\end{rmk}

It immediately follows from Proposition \ref{prop:GiSt3} that we have an equality
\[
(\GUX)^{\rms(L)} = G \times^{H_u} (X^{\rms}).
\]
Because $G$ is a closed reductive subgroup of $G \times H_r$ it follows that
$(\GUX)^{\rms(L)}$ is contained in the stable locus for the restricted
linearisation $G \act L \to \GUX$ and hence has a geometric
quotient for the $G$-action. The inclusion $\alpha_{H_u}:X^{\rms} \hookrightarrow
G \times^{H_u} (X^{\rms})$ thus induces an $H_r$-equivariant isomorphism
\[
X^{\rms}/H_u \cong (\GUX)^{\rms(L)}/G,
\]
and since $(\GUX)^{\rms(L)}/(G\times H_r)=((\GUX)^{\rms(L)}/G)/H_r$, we conclude
that $\alpha_{H_u}$ descends further to an isomorphism 
\[
X^{\rms}/H \cong (\GUX)^{\rms(L)}/(G \times H_r).
\]
So the significance of Proposition \ref{prop:GiSt3} is that it allows us to describe
stability for the linearisation $H \act L \to X$ and its geometric quotient in terms 
of
stability and the associated quotient for the \emph{reductive}
linearisation $G \times H_r \act L \to \GUX$.       


\begin{rmk} \label{rmk:GiSt4}
Even if $H \act L \to X$ is a linearisation of an ample line bundle over a
projective variety, the induced fibre space $\GUX$  will only be
quasi-projective with an ample linearisation, so care needs to be taken when
computing stability. Similarly, if $X$ is affine then $\GUX$ is not
necessarily affine.      
\end{rmk}

\begin{rmk} \label{rmk:GiSt4.1}
One can also consider the fibre space $G \times^H X$ associated to the
$H_u$-faithful homomorphism $\rho:H \to G$, together with its natural
$G$-linearisation $\tilde{L}:=G \times^H L \to G \times^H X$ and inclusion $\alpha_H:H
\hookrightarrow G \times^H X$ (assuming these spaces exists
as varieties). If $\ker \rho$ is finite, then is can be shown that
$X^{\rms}=\alpha_{H}^{-1}((\GHX)^{\rms(\tilde{L})})$ and the induced embedding
$X^{\rms} \hookrightarrow (\GHX)^{\rms(\tilde{L})}$ descends to an isomorphism
$X^{\rms}/H \cong ((\GHX)^{\rms(\tilde{L})})/G$. Since we will not use this in the
sequel, we omit the details of the proofs.     
\end{rmk}

\subsection{Summary of the Intrinsic Picture}
\label{sec:GiSummary}

We shall shortly draw together the work done so far to give a result that we believe
provides a good theoretical basis for doing geometric invariant
theory, in the case where $H$ is any linear algebraic group acting on an
irreducible variety $X$ with linearisation $L \to X$. Before doing so, we make one 
final
observation about the relationship between the notion of stability in Definition
\ref{def:GiSt1.1} and the various notions of semistability considered before. 

As already observed, we have 
\[
X^{\rms} \subseteq X^{\ssfg} \subseteq X^{\nss}.
\]
This can be further refined using the ideas at the end of Section
\ref{subsec:GiNaInduced}. The stable locus is
patched together with affine open subsets $X_f$, for certain $H$-invariant
sections $f$ of a positive tensor power of $L \to X$ which, among other things,
have the property that $(S^{H_u})_{(f)}$ is a finitely generated
$\kk$-algebra (cf. Definition \ref{def:GiSt1.1}, \ref{itm:GiSt1.1-3}). Because
$H_r$ is reductive and $X$ is irreducible then the full invariant
algebra $(S^{H})_{(f)}$ is finitely generated over $\kk$. This idea suggests it
is useful to consider another
notion of `semistability' that sits inside the finitely generated semistable locus.  

\begin{defn} \label{def:GiSt4}
Let $H$ be a linear algebraic group with unipotent radical $H_u$, let $H_r = H/H_u$ and let $H \act L \to X$ be a linearisation of an
irreducible $H$-variety $X$. We define the \emph{$H_u$-finitely generated
  semistable locus} to be the open subset \index{$H_u$-finitely generated
  semistable locus}
\[
X^{\rmss,H_u \dashfg}:=\bigcup_{f \in I^{\rmss,H_u \dashfg}} X_f,
\]
of $X^{\nss}$, where
\[
I^{\rmss,H_u \dashfg}=\left\{
f \in \cup_{r>0} H^0(X,L^{\ten r})^H \mid (S^{H_u})_{(f)}  
 \text{ is a finitely generated $\kk$-algebra}
\right\}.
\]  
\end{defn}   
It follows from the discussion above that $X^{\rms} \subseteq X^{\rmss,H_u
  \dashfg} \subseteq X^{\ssfg}$. The image of $X^{\rmss,H_u \dashfg}$ under the
enveloping quotient map 
$q_{H_u}:X^{\ssfg(H_u)} \to X \env H_u$ for the restricted linearisation $H_u
\act L \to X$ is contained in the $H_r$-invariant open subscheme
\[
\bigcup_{f \in I^{\rmss,H_u \dashfg}} \spec((S^{H_u})_{(f)}) \subseteq X \env H_u.
\]
This subscheme is not necessarily quasi-compact, but we can always find a finite
subset $\mc{S} \subseteq I^{\rmss,H_u \dashfg}$ of invariant sections such that
the image $q_{H_u}(X^{\rmss,H_u \dashfg})$ is contained in the quasi-compact
open subscheme
\[
\mc{U}:=\bigcup_{f\in \mc{S}} \spec((S^{H_u})_{(f)}) \subseteq X \env H_u,
\] 
which is an inner enveloping quotient of $X^{\rmss,H_u \dashfg}$ under the 
linearisation $H_u \act L \to X$. The rational map
$q_{H_r}:\proj(S^{H_u}) \dashrightarrow \proj(S^{H})$ of \eqref{eq:GiSt1} is
defined on $\mc{U}$ and gives a morphism 
\[
q_{H_r}:\mc{U} \to X \env H.
\]
In fact, the image under $q_{H_r}$ is precisely the reductive GIT
quotient $\mc{U} \dblslash H_r$ for the natural $H_r$-linearisation
$\OO_{\mc{U}}(r) \to \mc{U}$ (for $r>0$ sufficiently large) defined in
Proposition \ref{prop:GiNa8}, noting that the semistable set for this
linearisation is the whole of $\mc{U}$. Indeed, we have
\[
\mc{U} \dblslash H_r = \bigcup_{f \in \mc{S}} \spec((S^H)_{(f)}) \subseteq X
\env H.
\]
Thus $q_{H_r}:\mc{U} \to \mc{U} \dblslash H_r$
is a good categorical quotient by $H_r$, and $q_H:X^{\rmss,H_u \dashfg} \to
\mc{U} \dblslash H_r \subseteq X \env H$ is an inner enveloping quotient of
$X^{\rmss,H_u \dashfg}$  for the full
linearisation $H \act L \to X$. It is also clear that the geometric quotient
$X^{\rms}/H$ of the  
stable locus $X^{\rms}$ by $H$ is naturally an open subvariety of $\mc{U}
\dblslash H_r$. 

The following theorem summarises the main points
of our work so far.

\begin{thm} \label{thm:GiSu2}
Let $H$ be a linear algebraic group (with unipotent radical $H_u$) acting on an irreducible variety $X$ and $L
\to X$ a 
linearisation for the action. Let $H_r = H/H_u$, let $S=\kk[X,L]$ and let 
\begin{align*}
q_H:\proj S &\dashrightarrow \proj(S^H) \\
q_{H_u}:\proj S &\dashrightarrow \proj(S^{H_u}) \\ 
q_{H_r}:\proj(S^{H_u}) &\dashrightarrow \proj(S^H) 
\end{align*}
be the rational maps defined by the obvious inclusions of graded algebras. Also let
\[
\mc{U}=\bigcup_{f \in \mc{S}} \spec((S^{H_u})_{(f)}),
\]
where $\mc{S}$ is a
finite subset of $I^{\rmss, H_u \dashfg}$ such that $X^{\rmss,H_u \dashfg} =
\bigcup_{f \in \mc{S}} X_f$.  
\begin{enumerate}
\item \label{itm:GiSu2-1} There is a commutative diagram
\[
\begin{diagram}
X^{\rms} & \rEmpty~\subseteq & X^{\rmss,H_u\dashfg}
&  \rEmpty~\subseteq & X^{\ssfg} & \rEmpty~\subseteq &
X^{\nss} & \rEmpty~\subseteq & X \\
\dTo_{q_{H_r}}^{\text{\upshape{geo}}} & & \dTo_{q_{H_u}} && &&&& \dDashto_{q_{H_u}}\\
X^{\rms}/H_u  & \rEmpty~\subseteq & \mc{U} & & \dTo_{q_H} && \dTo_{q_H} && \proj(S^{H_u})\\
\dTo_{q_{H_r}}^{\text{\upshape{geo}}}  &  & \dTo_{q_{H_r}}^{\text{\upshape{good}}} & & & && & \dDashto_{q_{H_r}}\\
X^{\rms}/H & \rEmpty~\subseteq & \mc{U} \dblslash H_r & \rEmpty~\subseteq &  X \env H & \rEmpty~\subseteq &
\proj(S^{H}) & \rEmpty~= & \proj(S^{H}) \\
\end{diagram}
\]
with good or geometric quotients as indicated and all inclusions open. The
induced morphism $q_H:X^{\rmss,H_u \dashfg} \to
\mc{U} \dblslash H_r$ is an inner enveloping quotient of $X^{\rmss,H_u \dashfg}$.
\item \label{itm:GiSu2-2} Given any reductive group $G$ and $H_u$-faithful 
homomorphism $H
\to G$, we have 
\[
X^{\rms} = \alpha^{-1}((\GUX)^{\rms(L)}),
\]
inducing a natural isomorphism
\[
X^{\rms}/H \cong (\GUX)^{\rms(L)}/(G \times H_r)
\]
where $\alpha:X \hookrightarrow \GUX$ is the natural
inclusion and $L=G \times^{H_u} L$ is the natural $G \times H_r$-linearisation
defined in Section \ref{subsec:GiStRelation}.   
\end{enumerate}     
\end{thm}

\begin{rmk} \label{rmk:GiSu3} The spaces involved in the statement of
  Theorem \ref{thm:GiSu2} are unchanged when we replace the linearisation $L \to
  X$ by any positive tensor power $L^{\ten r} \to X$. (In the case of
  $X^{\nss}$ and $X^{\ssfg}$ this was observed in Remark \ref{rmk:GiFi1.2}.) It
  thus makes sense to talk about the notions of stability, finitely generated
  semistability, enveloping quotients etc. for rational linearisations (see Remark \ref{rmk:TrRe3.2}).
\end{rmk}

Theorem \ref{thm:GiSu2} is a culmination of all the intrinsic notions we have
discussed so far and provides what we believe is a good basis for doing
geometric invariant theory for non-reductive groups. One reason for this is
because, in the case where $L \to X$  is an ample linearisation 
over a projective variety $X$, it extends the main geometric
invariant theoretic theorems in both
the reductive and unipotent settings. However, given a general non-reductive
linearisation the question remains of how one
can study the stable locus
$X^{\rms}$, finitely generated semistable locus $X^{\ssfg}$, inner enveloping
quotients of $X^{\ssfg}$ and the geometric quotient $X^{\rms}/H$, in a more
explicit manner. In the next section we consider some
techniques for doing this.


\section{Projective Completions of Enveloped Quotients and Reductive Envelopes} 
\label{sec:compactifications1}

\setcounter{defn}{0}

In Section \ref{sec:gitNonReductive} we developed a theoretical framework for
identifying open subsets of varieties that admit geometric quotients under a
given group action, which seeks to provide an analogue of Mumford's work
\cite{mfk94} to the non-reductive setting. Due to the fact that a non-reductive
group does not have  
such a well behaved invariant theory (notably the possibility of non-finite
generation of rings of invariants) there are significant differences with
Mumford's theory for reductive groups. 

Throughout this section, as before, let $H$ be a linear algebraic group with unipotent radical $H_u$, and let $H_r = H/H_u$.
Recall that when $H=H_r = G$ is a reductive group acting on a projective
variety $X$ with an ample linearisation $L \to X$, the quotient of the stable
locus (in the sense of Definition \ref{def:TrRe5}, \ref{itm:TrRe5-2}) admits a
canonical projective completion $X \dblslash G$, which is a good
categorical quotient of the semistable locus. We have $X \dblslash G =
\proj(\kk[X,L]^G)$, where $\kk[X,L]^G$ is a finitely generated $\kk$-algebra by
Nagata's theorem \cite{nag64}, and topologically $X \dblslash G$ can be described
as the quotient of $X^{\rmss}$ under the S-equivalence relation; see Section
\ref{sec:BgReductive}. When $H$ is not
reductive, this picture does not carry over into our non-reductive GIT, due to
the fact that the ring of invariants 
$\kk[X,L]^H$ is not necessarily finitely generated and the image of the
enveloping quotient map $q:X^{\ssfg} \to X \env H$ (within which the
quotient $X^{\rms}/H$ of the stable locus $X^{\rms}$ is contained) is not
necessarily a variety. To address the first issue, we
introduced the notion of an inner enveloping quotient (Definition
\ref{def:GiFi3.1}). Recall this is a choice of quasi-compact open subscheme $\mc{U}
\subseteq X \env H$ which contains the image of the enveloping quotient map
$q:X^{\ssfg} \to X \env H$ as a dense subset. Every inner enveloping quotient is
quasi-projective, so we may talk about their projective completions. 

A reasonable way to try and recover a picture similar to Mumford's for
reductive groups is to therefore consider how to construct projective completions of
inner enveloping quotients containing the enveloped quotient
as a dense subset. We make
\begin{defn} \label{def:Co10}
Let $H$ be a linear algebraic group acting on a variety $X$ with linearisation   \index{enveloped quotient!projective completion of}
$L \to X$. We call a projective variety $Z$ a \emph{projective completion of the
  enveloped quotient} if there is an inner enveloping quotient $\mc{U} \subseteq
X \env H$ and an open immersion $\mc{U} \hookrightarrow Z$.  
\end{defn}

The purpose of this section is to describe a method for constructing projective
completions of
the enveloped quotient, based on extending the work of \cite[\S 5]{dk07} described in 
Section
\ref{subsec:BgUnExtending}. In Section \ref{sec:Co1Reductive} we
extend the notion of a reductive envelope in Definition \ref{def:BgUn12} to the
more general case 
where $H$ is not necessarily unipotent, nor the linearisation $L \to X$ ample
over a projective variety (see Definition \ref{def:Co1Re5}). The idea here is to
consider equivariant projective completions 
$\beta:\GUX \hookrightarrow \olGUX$ (where $G$ is a reductive group $G$ and
$\GUX$ is formed with respect to an $H_u$-faithful homomorphism $H \to G$),
together with an extension of the linearisation $L \to X$ to a $G \times
H_r$-linearisation $L^{\prime} \to \olGUX$, with conditions imposed in order to
yield inclusions 
\begin{equation} \label{eq:Co10}
\begin{array}{c} 
X \cap (\olGUX^{\rms(L^{\prime})}) \subseteq X^{\rms} \subseteq X^{\ssfg} \subseteq
X \cap (\olGUX^{\rmss(L^{\prime})}) \\ \\
\text{and } \quad q(X^{\ssfg}) \subseteq \mc{U} \subseteq \olGUX
\dblslash_{L^{\prime}} (G \times H_r), 
\end{array}   
\end{equation}
where $\mc{U}$ is some inner enveloping quotient of $X^{\ssfg}$. The main result
concerning reductive envelopes is Theorem \ref{thm:Co1Re11}. In Section
\ref{sec:Co1Strong} we consider 
certain kinds of reductive envelope, called \emph{strong reductive envelopes},
which yield equalities $X \cap \olGUX^{\rms(L^{\prime})}=X^{\rms}$ and
$X^{\ssfg} = X \cap \olGUX^{\rmss(L^{\prime})}$. These are therefore interesting from
the point of view of computing the stable and finitely generated semistable loci
for $H \act L \to X$. Here we also give an explicit way to construct strong reductive 
envelopes when some extra simplifying conditions on the linearisation $H \act L \to X$ 
and
group $G$ are satisfied.

\subsection{Reductive Envelopes: General Theory and Main Theorem}
\label{sec:Co1Reductive}

Let $H$ be a linear algebraic group acting on an irreducible variety $X$ with
linearisation $L \to X$ and suppose we have an $H_u$-faithful homomorphism $H
\to G$ into a reductive group $G$. Consider the fibre space $\GUX$ associated to
this homomorphism, as in Section \ref{subsec:GiStRelation}. As there, we
abuse notation and write $L$ for the natural $G \times H_r$ linearisation $G
\times^{H_u} L \to \GUX$, unless confusion is likely to  
occur, and also let $\alpha:X \hookrightarrow \GUX$ be the natural closed
immersion. We wish to extend the theory of reductive
envelopes from \cite{dk07} to the more general setting 
where $H$ is not necessarily a unipotent group. It is natural to
preserve as much of the intrinsic non-reductive GIT picture as possible:
we would like the reductive quotients of $L^{\prime} \to
\olGUX$ by $G$ and $G \times H_r$ to reflect the `quotienting in stages' aspect
of the diagram in Theorem
\ref{thm:GiSu2}, \ref{itm:GiSu2-1}. This requires identifying collections of
invariant sections in $\kk[X,L]$ that are large 
enough to detect the subsets $X^{\nss}$,
$X^{\ssfg}$, $X^{\rmss, H_u \dashfg}$ and $X^{\rms}$, and ensuring
these extend to sections over $\olGUX$. 

\begin{defn} \label{def:Co1Re3} Let $H$ be a linear algebraic group acting on an
  irreducible variety $X$ and $L \to X$ a linearisation. Fix a reductive
  group $G$ and an $H_u$-faithful homomorphism, where $H_u$ is the unipotent radical of $H$ and $H_r = H/H_u$. We say an enveloping system $V$  \index{fully separating enveloping system} \index{enveloping system ! fully separating}
  is \emph{fully separating} if it is adapted to a finite
  subset $\mc{S} \subseteq V$ such that the following properties are satisfied: 
  \begin{enumerate}
  \item \label{itm:Co1Re3-1} $X^{\nss} = \bigcup_{f \in V^H} X_f$;
  \item \label{itm:Co1Re3-2} there is a subset $\mc{S}^{\ssUfg} \subseteq \mc{S}$
    such that $X^{\ssUfg}=\bigcup_{f \in \mc{S}^{\ssUfg}} X_f$ and $V$ defines
    an enveloping system adapted to 
    $\mc{S}^{\ssUfg}$ for the restricted linearisation $H_u \act L \to X$; and 
  \item \label{itm:Co1Re3-3} for every $x\in X^{\rms}$ there is $f
  \in \mc{S}$ with corresponding $G \times H_r$-invariant $F$ over $\GUX$ such
  that $(\GUX)_F$ is affine. (Equivalently, for every $x \in X^{\rms}$ there if
  $f \in \mc{S}$ such that $X_f$ is affine and $X_f \to \spec((S^{H_u})_{(f)})$
  is a principal $H_u$-bundle; see Remark \ref{rmk:GiSt3}.)  
  \end{enumerate}
\end{defn}

It is not difficult to modify the proof of Proposition \ref{prop:GiFi4.1},
\ref{itm:GiFi4.1-1} to prove the existence of fully separating enveloping
systems, for any given linearisation $L \to X$ with $X$ irreducible, and that
such an enveloping system is stable under taking products of sections. This is
done in the next lemma.


\begin{lem} \label{lem:Co1Re4} Given any irreducible $H$-variety $X$ with
  linearisation $L \to X$ and any $H_u$-faithful homomorphism $H \to G$ with $G$
  reductive, for some $r>0$ there exists a fully separating enveloping
  system $V \subseteq H^0(X,L^{\ten r})^{H_u}$. Furthermore, for any fully
  separating enveloping system $V \subseteq H^0(X,L^{\ten m})$ and any $n>0$ the
  image of the natural multiplication map 
\[
V^{\ten n} \to H^0(X,L^{\ten mn}), 
\]
is again a fully separating enveloping system.
\end{lem}

\begin{proof}
By Proposition \ref{prop:GiFi4.1}, \ref{itm:GiFi4.1-1} there is an
enveloping system $V^{\prime} \subseteq H^0(X,L^{\ten r})^{H}$, for some $r>0$, 
adapted to
a finite subset $\mc{S}$ with $X^{\ssfg}=\bigcup_{f \in \mc{S}} X_f$. We will
augment this enveloping system by taking a suitably large multiple of $r$ and
replacing $V^{\prime}$ and $\mc{S}$ by their images under the natural multiplication 
map of sections,
as in Proposition \ref{prop:GiFi4.1}, \ref{itm:GiFi4.1-3}. We repeatedly use the
fact that $X_f=X_{f^n}$ for any section $f$ over $X$ and integer $n>0$, as well
as the equalities $(S^H)_{(f)}=(S^H)_{(f^n)}$ and
$(S^{H_u})_{(f)}=(S^{H_u})_{(f^n)}$ for invariant sections $f$.    

Because $X$ is quasi-compact, by taking a large multiple of $r$ and replacing
$V^{\prime}$ and $\mc{S}$ appropriately 
we may assume there are subsets $\mc{S}^{\rms}$ and $
\mc{S}^{\ssUfg}$ of $\mc{S}$ such that $X^{\rms}$ and
$X^{\ssUfg}$ are covered by open subsets of
the form $X_f$ with $f \in \mc{S}^{\rms}$ and $f \in
\mc{S}^{\ssUfg}$, respectively. We may also assume that $r$ is chosen so that
$H^0(X,L^{\ten r})^H$ contains sections $f_1,\dots,f_n$ 
with $X^{\nss}=\bigcup_{i=1}^n X_{f_i}$, and furthermore we can choose
$\mc{S}^{\rms}$ so that each $f \in 
\mc{S}^{\rms}$ extends to $F$ over $\GUX$ such that 
$(\GUX)_F$ is affine. Following an argument 
similar to the construction of the subset 
$A$ in the proof of Proposition \ref{prop:GiFi4.1}, \ref{itm:GiFi4.1-1}, by
taking another suitably large multiple of $r$ and replacing $V^{\prime}$ and the sets
$\mc{S}$ and $\{f_1,\dots,f_n\}$ by their images under the appropriate multiplication 
map on
sections we may assume there is a subset $A^{\ssUfg}
\subseteq H^0(X,L^{\ten r})^{H_u}$,
containing $\mc{S}^{\ssUfg}$, such
that $(S^{H_u})_{(f)}$ is generated by $\{\tilde{f}/f \mid \tilde{f} \in
A^{\ssUfg}\}$ for each $f \in \mc{S}^{\ssUfg}$. By Lemma \ref{lem:TrRe4} there
is a finite dimensional rational $H$-module $V \subseteq H^0(X,L^{\ten
  r})^{H_u}$ for the
natural $H$-action that contains $V^{\prime} \cup A^{\ssUfg} \cup
\{f_1,\dots,f_n\}$. Then $V$ 
is an enveloping system adapted to $\mc{S}$ such that properties
\ref{itm:Co1Re3-1}--\ref{itm:Co1Re3-3} of Definition \ref{def:Co1Re3} are satisfied.   

The statement about images of fully separating enveloping systems under
multiplication maps follows immediately from
Proposition \ref{prop:GiFi4.1}, \ref{itm:GiFi4.1-3} and the equalities
$X_{f^n}=X_f$, $(S^H)_{(f)}=(S^H)_{(f^n)}$ and
$(S^{H_u})_{(f)}=(S^{H_u})_{(f^n)}$ for any $H$-invariant $f$ and $n>0$.    
\end{proof}      

\begin{eg} \label{ex:Co1Re1}
If $X$ is an irreducible projective $H$-variety and $L \to X$ an ample
linearisation, then each space of sections $H^0(X,L^{\ten r})$, where $r>0$, is a
finite dimensional vector space. Then given an $H_u$-faithful homomorphism $H
\to G$, an easy consequence of Lemma
\ref{lem:Co1Re4} is that the space
$V=H^0(X,L^{\ten r})$ defines a fully separating enveloping system for
sufficiently divisible $r>0$.      
\end{eg}

We now turn to the definition of a reductive envelope. Given an
$H_u$-faithful homomorphism $H \to G$ with $G$ reductive, the idea is to
extend the linearisation $G \times H_r$ on $G
\times^{H_u} L \to \GUX$ over a suitable equivariant projective completion
$\olGUX$ of $\GUX$. A key condition for obtaining the diagram \eqref{eq:Co10} is to
ensure enough invariants over $X$ (or equivalently over $\GUX$) extend to
invariants over $\olGUX$. 

\begin{defn} \label{def:Co1Re5} Let $H \act L \to X$ be a linearisation of a
  non-reductive group $H$ and suppose $H \to G$ is an $H_u$-faithful homomorphism
  into a reductive group $G$, where $H_u$ is the unipotent radical of $H$ and $H_r = H/H_u$. Let $\olGUX$ be a
  projective $G \times 
  H_r$-variety with $G \times H_r$-equivariant dominant open immersion $\beta:\GUX
  \hookrightarrow \olGUX$ and 
$L^{\prime} \to \olGUX$ a $G \times H_r$-linearisation that restricts to some
positive tensor power of $L
\to \GUX$ under $\beta$. We call $(\olGUX,\beta,L^{\prime})$ a \emph{reductive
  envelope} for the linearisation $H \act L \to X$ if there is fully separating  \index{reductive envelope}
enveloping system $V$ for $H \act L \to X$ such that   
\begin{enumerate}
\item \label{itm:Co1Re1-1} each section in $V^{H_u}$ extends under $\beta \circ
  \alpha$ to a $G$-invariant section of some tensor power of $L^{\prime}$ over $\olGUX
$;
\item \label{itm:Co1Re1-2} each section in $V^H$ extends under $\beta \circ
  \alpha$ to a $G \times H_r$-invariant section of some tensor power of
  $L^{\prime}$ over $\olGUX$; and
\item \label{itm:Co1Re1-3} for $f \in V^{H_u}$ 
  with extension to $F$ over $\olGUX$, the open subset $(\olGUX)_{F}$ is affine.  
\end{enumerate} 
If the line bundle $L^{\prime}$ is ample, then we call
$(\olGUX,\beta,L^{\prime})$ an \emph{ample reductive envelope}. \index{ample reductive envelope}
\end{defn}
\begin{rmk} \label{rmk:Co1Re6} In the case where $H$ is unipotent, our notion of
  reductive envelope corresponds to 
  that of a \emph{fine} reductive envelope in \cite{dk07}; cf. Definition   \index{fine reductive envelope}
  \ref{def:BgUn12}. 
\end{rmk}
\begin{rmk} \label{rmk:Co1Re6.1}
The case where $L^{\prime}$ is ample is of most interest to us, because
it ensures the GIT quotient $\olGUX \dblslash_{L^{\prime}} (G \times H_r)$ is
a projective variety. It also means that condition \ref{itm:Co1Re1-3} of
Definition \ref{def:Co1Re5} is automatically 
satisfied, so verifying that the data $(\olGUX,\beta,L^{\prime})$ defines a
reductive envelope 
reduces to checking that invariant sections from a fully separating
enveloping system extend to sections over
$\olGUX$.
\end{rmk}

The next proposition asserts the existence of an ample reductive envelope for
any given \emph{ample} $H$-linearisation $L \to X$, under the 
standing assumption that stabilisers of general points in $X$ are finite (see 
Remark \ref{rmk:TrRe0}). 
  
\begin{propn} \label{prop:Co1Re6} Let $H$ be a linear algebraic group acting on
  an irreducible quasi-projective variety $X$ such that stabilisers of general
  points in $X$ are finite and $L \to X$ an ample
  linearisation for the action. Then 
  $H \act L \to X$ possesses an ample reductive envelope for some
  reductive group $G$ containing $H$ as a closed subgroup.
\end{propn}

\begin{proof} 
Begin by using Lemma \ref{lem:Co1Re4} to find $r>0$ with a fully separating
enveloping system $V \subseteq H^0(X,L^{\ten r})$. (Note that by Remark
\ref{rmk:GiSt3} it makes sense to talk about fully separating enveloping systems
without reference to any reductive group $G$ and $H_u$-faithful homomorphism $H
\to G$.) The line bundle $L$ is ample, so by taking a 
sufficiently large multiple of $r$, replacing $V$
by its image under the natural multiplication map of sections
and enlarging the resulting $V$ using Lemma \ref{lem:TrRe4} if
necessary, we may assume that $V$ is a rational $H$-module and defines an
$H$-equivariant locally 
closed immersion $X \hookrightarrow 
\PP(V^*)$. The action of $H$ on
$V$ defines a homomorphism $\rho: H \to G:=\GL(V)$ and there is a canonical
$G$-linearisation on 
$\OO_{\PP(V^*)}(1) \to \PP(V^*)$ extending the $H$-linearisation $L^{\ten r} \to
X$ via $\rho$. Note that any element of the kernel of $\rho$ must fix every
point in $X$, so $\ker \rho$ is a finite group and cannot contain any
unipotent elements. Thus $\rho:H \to G$ is $H_u$-faithful. Consider
the fibre bundle $G \times^{H_u} \PP(V^*)$, together with its $G 
\times H_r$-linearisation $G \times^{H_u}\OO_{\PP(V^*)}(1)$. Then there is an
isomorphism of $G \times H_r$-varieties
\begin{align*}
G \times^{H_u} \PP(V^*) &\cong (G/H_u) \times \PP(V^*), \\
[g,y] &\mapsto (gH_u,gy),
\end{align*}              
and the corresponding $G \times H_r$-linearisation over $(G/H_u) \times
\PP(V^*)$ has underlying line bundle $\OO_{G/H_u} \boxtimes
\OO_{\PP(V^*)}(1)$.   

Now because $H_u$ is unipotent the homogeneous space $G/H_u$ is quasi-affine
\cite[Corollary 2.8 and Theorem 2.1]{gro97}, so there is a finite dimensional
vector subspace $W \subseteq \OO(G/H_u)$ defining a locally closed immersion of
$G/H_u$ into the affine space
$\A=\spec(\symdot W^*)$. Right multiplication by $H$
on $G$ via $\rho$ descends to define an $H_r$-action on $G/H_u$, and $G$ acts by
left 
multiplication on $G/H_u$. By \cite[Proposition 1.9]{bor91} we may assume that
$W$ is invariant under  
the corresponding actions of $H_r$ and $G$ on $\OO(G/H_u)$. This induces
an action of $G \times H_r$ on $\A$ together with a linearisation on
the trivial line bundle $\OO_{\A} \to \A$ which restricts to the canonical $G
\times H_r$-linearisation $\OO_{G/H_u} \to G/H_u$ under $G/H_u \hookrightarrow
\A$. Let $\kk$ be a copy of the ground field
equipped with the trivial $G \times H_r$-representation, set $\PP:=\PP(W^*
\oplus \kk)$, and let $\beta_1:G/H_u \hookrightarrow \ol{G/H_u}$ be the
projective completion of $G/H_u$ resulting from the embedding $G/H_u
\hookrightarrow \A$ and the standard open immersion $\A
\hookrightarrow \PP$. Then the restriction
$\OO_{\ol{G/H_u}}(1)=\OO_{\PP}(1)|_{\ol{G/H_u}}$ of the canonical $G \times H_r$-
linearisation $\OO_{\PP}(1) \to
\PP$ to $\ol{G/H_u}$ pulls back to
the $G \times H_r$-linearisation $\OO_{G/H_u} \to G/H_u$ under
$\beta_1$. Consider the linearisation
\[
G \times H_r \act \OO_{\ol{G/H_u}}(1) \boxtimes \OO_{\PP(V^*)}(1) \to \ol{G/H_u}
\times \PP(V^*)
\]
obtained by taking the product of $\OO_{\ol{G/H_u}}(1)$ with the $G \times
H_r$-linearisation on $\OO_{\PP(V^*)}(1)$ defined by $G$ and the trivial
$H_r$-action on $\PP(V^*)$. Let $\beta:\GUX \hookrightarrow \olGUX \subseteq \ol{G/
H_u} \times \PP(V^*)$
be the projective completion
of $\GUX$ obtained by the composition of the embedding $\GUX \hookrightarrow G
\times^{H_u} \PP(V^*) \cong (G/H_u) \times \PP(V^*)$ with the open immersion
$\beta_1 \times \id_{\PP(V^*)}$, and let $L^{\prime}$ be the restriction of 
$\OO_{\ol{G/H_u}}(1) \boxtimes \OO_{\PP(V^*)}(1)$ to $\olGUX$. Then
$\beta^*L^{\prime}=L^{\ten r} \to \GUX$ as $G \times H_r$-linearisations.   

To conclude we observe that the required extension properties
\ref{itm:Co1Re1-1}--\ref{itm:Co1Re1-3} of Definition \ref{def:Co1Re5} hold for
$(\olGUX,\beta,L^{\prime})$, as 
follows. Any invariant section $f \in V \subseteq H^0(X,L^{\ten
  r})$ extends to an invariant (which we also call $f$) of $\OO_{\PP(V^*)}(1) \to 
\PP(V^*)$ by
construction, and if $f$ is $H_u$-invariant (respectively, $H$-invariant) over
$\PP(V^*)$ then it extends to the $G$-invariant (respectively, $G \times
H_r$-invariant) section 
\[
1 \ten f \in \OO(G/H_u) \ten H^0(\PP(V^*),\OO(1)) = H^0((G/H_u) \times
\PP(V^*),\OO_{G/H_u} \boxtimes \OO_{\PP(V^*)}(1)).
\] 
(Here we have used the K\"{u}nneth formula \cite[Tag 02KE]{stacks-project}.) But now 
if $\epsilon \in H^0(\PP,\OO_{\PP}(1))$ is the
homogeneous coordinate of $\PP=\PP(W^* \oplus \kk)$ corresponding to the trivial
$G \times H_r$-summand $\kk$ then $1 \ten f$ extends to $\epsilon \ten f$ under
$\beta$, which is $G$ or $G \times H_r$-invariant if $1 \ten
f$ is $G$ or $G \times H_r$-invariant, respectively. Thus $f \in V^{H_u}$
(respectively, $f \in V^H$) extends to the $G$-invariant (respectfully, $G \times
H_r$-invariant) $(\epsilon \ten f)|_{\olGUX}$ of $L^{\prime} \to \olGUX$. This
shows properties \ref{itm:Co1Re1-1} and \ref{itm:Co1Re1-2}. Finally, property 
\ref{itm:Co1Re1-3} holds because $L^{\prime}$ is (very) ample.      
\end{proof} 

\begin{rmk} \label{rmk:Co1Re6.3} In practice, the group $G=\GL(V)$ containing
  $H$ constructed in the proof of Proposition \ref{prop:Co1Re6} is too large to
  be computationally useful. In Section \ref{subsec:Co1StImportant} we will look
  at a class of reductive envelopes where the reductive group $G$ 
  contains $H_u$ as a \emph{Grosshans subgroup}, in which case the geometry of
  the 
  homogeneous space $G/H_u$ lends itself to more explicit calculations.       
\end{rmk}

Associated to any reductive envelope $(\olGUX,\beta,L^{\prime})$ are open
subsets of $X$ obtained by pulling back the stable and semistable loci for the
$G \times H_r$-linearisation $L^{\prime} \to \olGUX$. As we will see shortly,
one of the key properties 
of these sets is that they `bookend' the intrinsically defined
notions of stability and semistability considered in Section
\ref{sec:gitNonReductive}. In analogy to
\cite[Definition 5.2.11]{dk07} (see the statement of Theorem \ref{thm:BgUn13}),
we make the following definition. 

\begin{defn} \label{def:Co1Re7}
Let $H \act L \to X$ be a linearisation of a linear algebraic group $H$ with unipotent radical $H_u$, let $H_r = H/H_u$ and let $H
\to G$ be an $H_u$-faithful homomorphism into a reductive group $G$. Suppose 
$(\olGUX,\beta,L^{\prime})$ is a reductive envelope.  The \emph{completely
  semistable locus} is the set \index{completely semistable locus} 
\[ 
X^{\ol{\rmss}}:=(\beta \circ \alpha)^{-1}((\olGUX)^{\rmss(L^{\prime})})  
\]
and the \emph{completely stable locus} is the set \index{completely stable locus}
\[
X^{\ol{\rms}}:=(\beta\circ \alpha)^{-1}((\olGUX)^{\rms(L^{\prime})}),
\]  
where $\olGUX^{\rmss(L^{\prime})}$ and $\olGUX^{\rms(L^{\prime})}$ are the
semistable and stable loci, respectively, for the reductive linearisation $G
\times H_r \act L^{\prime} \to \olGUX$.
\end{defn}

\begin{propn} \label{prop:Co1Re8} Let $H$ be a linear algebraic group with an
  $H_u$-faithful morphism $H \to G$, with $G$ reductive, and let $X$ be an
  irreducible quasi-projective $H$-variety with an ample linearisation $L \to
  X$. If $(\olGUX,\beta,L^{\prime})$ is a
  reductive envelope for the linearisation, then
  $X^{\ol{\rmss}}=X^{\nss}$ and $X^{\ol{\rms}} \subseteq X^{\rms}$.
\end{propn}


 
\begin{proof}
Let $V$ be a fully separating enveloping system adapted to a finite subset
$\mc{S} \subseteq V$ satisfying properties
\ref{itm:Co1Re3-1}--\ref{itm:Co1Re3-3} of Definition \ref{def:Co1Re3}, and
suppose $V$ satisfies the extension properties
\ref{itm:Co1Re1-1}--\ref{itm:Co1Re1-3} of Definition \ref{def:Co1Re5} for the
reductive envelope 
$(\olGUX,\beta,L^{\prime})$. Then there is a basis $f_1,\dots,f_n$ of $V^H$ such
that $X^{\nss} = \bigcup_{i=1}^n X_{f_i}$ and each $f_i$
extends to a $G \times H_r$-invariant $F_i$ of some positive tensor power of
$L^{\prime} \to \olGUX$ such that $(\olGUX)_{F_i}$ is affine, so
$X^{\nss} \subseteq X^{\ol{\rmss}}$. On the other hand, any $G \times H_r$-invariant
of a tensor power of $L^{\prime} \to \olGUX$ restricts to a $G \times
H_r$-invariant over $\GUX$ under $\beta$, which in turn corresponds to an
$H$-invariant over $X$ via $\alpha$. Hence $X^{\ol{\rmss}} \subseteq
X^{\nss}$ also. 


Now suppose $x\in X^{\ol{s}}$. Then there is a $G\times H_r$-invariant $F$ of some
positive tensor power of $L^{\prime} \to \olGUX$ such that $(\olGUX)_F$ is
an affine open subset containing $(\beta \circ \alpha)(x)$, and the $G\times
H_r$-action on 
$(\olGUX)_F$ is closed with all stabilisers finite. By abuse of notation, write
$F$ for the section $\beta^*F$ over $\GUX$. Invoking
Proposition \ref{prop:GiSt3}, to prove $x \in X^{\rms}$ it suffices to show that
$(\GUX)_F \subseteq (\GUX)^{\rms(L)}$, where stability is with respect to the
canonical $G \times H_r$-linearisation $L \to \GUX$. Note that $(\GUX)_F$ is a
$G \times H_r$-invariant open subset of $(\olGUX)_F$, and $(\olGUX)_F$ has a geometric
$G \times H_r$-quotient $\pi:(\olGUX)_F \to
(\olGUX)_F/(G \times H_r)$ with affine base (Theorem
\ref{thm:BgRe1}). The image $\pi((\GUX)_F)$ of $(\GUX)_F$ is an open
subset of $(\olGUX)_F/(G \times H_r) $, thus we may cover $\pi((\GUX)_F)$ with
basic affine 
open subsets of $\pi((\olGUX)_F)$. Each of these takes the form
$\pi((\olGUX)_{F\tilde{F}})=(\olGUX)_{F\tilde{F}}/(G 
\times H_r)$, for a $G \times H_r$-invariant section $\tilde{F}$ over $\olGUX$,
by virtue of the canonical isomorphism   
\[
\OO(\pi((\GUX)_F)) \overset{\pi^{\#}}{\longrightarrow} (\OO((\GUX)_F))^{G \times
  H_r}=(\kk[\olGUX,L^{\prime}]^{G \times H_r})_{(F)}.
\]
(In the final equality we have used the fact that $\olGUX$ is 
irreducible, which is necessarily the case because $\GUX$ is irreducible and
$\beta$ is dominant.) Thus, for suitable $G \times H_r$-invariant sections $F_i$
over $\olGUX$, we have 
\[
(\GUX)_F = \bigcup_{i}
\pi^{-1}(\pi((\olGUX)_{FF_i}))=\bigcup_{i} (\GUX)_{FF_i}
\]
and
by affineness of $\pi$ each $\pi^{-1}(\pi((\olGUX)_{FF_i}))
=(\GUX)_{FF_i}$ is an affine open subset of
$\GUX$. By restriction the $G \times H_r$-action on each
$(\GUX)_{FF_i}$ is closed with finite stabilisers, hence
$(\GUX)_{FF_i} \subseteq (\GUX)^{\rms(L)}$ for each $i$. Therefore $(\GUX)_F
\subseteq (\GUX)^{\rms(L)}$, as desired.         
\end{proof}

\begin{cor} \label{cor:Co1Re8.1} For a linearisation $H \act L \to X$ with $X$ 
projective and
  $L$ ample, the restriction of the enveloping quotient map to the completely
  stable locus $X^{\ol{\rms}}$ for 
  a reductive envelope $(\olGUX,\beta,L^{\prime})$ defines a geometric
  quotient $q_H:X^{\ol{\rms}} \to X^{\ol{\rms}}/H$, and the composition 
\[
\begin{diagram}
X^{\ol{\rms}} & \rInto^{\beta \circ \alpha} &
\olGUX^{\rms(L^{\prime})} & \rTo & \olGUX^{\rms(L^{\prime})}/(G \times H_r) 
\end{diagram}
\]
induces a natural open immersion $X^{\ol{\rms}}/H \hookrightarrow
\olGUX^{\rms(L^{\prime})}/(G \times H_r)$.   
\end{cor}

\begin{proof}
Since $X^{\ol{\rms}}$ is an $H$-invariant open subset of $X^s$ the map $q_H:X^s
\to X^s/H$ restricts to define a geometric quotient $q_H:X^{\ol{\rms}} \to
q_H(X^{\ol{\rms}}) \subseteq X^s / H$. By definition $G
\times^{H_u}(X^{\ol{\rms}})$ is an open subset of
$\olGUX^{\rms(L^{\prime})}$ via $\beta$ and hence 
\[
X^{\ol{\rms}}/H= G \times^{H_u}(X^{\ol{\rms}})/(G \times H_r) \subseteq 
\olGUX^{\rms(L^{\prime})}/(G \times H_r)
\]  
with the inclusion open.
\end{proof}

Suppose we have a reductive envelope $(\olGUX,\beta,L^{\prime})$ for the
linearisation $H \act L \to X$. Then we may
consider the reductive GIT quotients
\begin{align*} 
\pi_G&:\olGUX^{\rmss(G)} \to \olGUX \dblslash G, \\
\pi_{G \times H_r}&:\olGUX^{\rmss(G\times H_r)} \to \olGUX \dblslash (G\times
H_r) 
\end{align*}
for the $G$ and $G \times H_r$-linearisations on $L^{\prime}$,
respectively. According to Proposition \ref{prop:Ap1} of the Appendix, there is
an induced ample $H_r$-linearisation $M^{\prime} \to \olGUX \dblslash G$ such
that $\pi_G$ maps $\olGUX^{\rmss(G\times H_r)}$ into $(\olGUX \dblslash
G)^{\rmss (M^{\prime})}$, and if 
\[
\ol{\pi}_{H_r}:(\olGUX \dblslash G)^{\rmss (M^{\prime})} \to (\olGUX \dblslash G)
  \dblslash_{M^{\prime}} H_r
\]
is the reductive GIT quotient for the linearisation $H_r \act M^{\prime} \to
\olGUX \dblslash G$, then there is a canonical open immersion 
\[
\psi: \olGUX \dblslash(G \times H_r) \hookrightarrow (\olGUX \dblslash G)
\dblslash_{M^{\prime}} H_r  
\]
such that the diagram 
\begin{equation}  \label{eq:Co1Re1}
\begin{diagram}
 & & \olGUX^{\rmss(G \times H_r)} \\
 & \ldTo(2,4)^{\pi_{G \times H_r}} & \dTo_{\pi_G} \\
 & & (\olGUX \dblslash G)^{\rmss(M^{\prime})} \\
 & &  \dTo_{\ol{\pi}_{H_r}} \\
\olGUX \dblslash (G \times H_r) & \rInto^{\psi} & (\olGUX \dblslash G) \dblslash_{M^{\prime}} H_r  \\
\end{diagram}
\end{equation}
commutes.
\begin{propn} \label{prop:Co1Re10} Let $H$ be a linear algebraic group acting on
  an irreducible variety $X$ with linearisation $L \to X$, let $H \to G$ be an
  $H_u$-faithful homomorphism with $G$ reductive and suppose
  $(\olGUX,\beta,L^{\prime})$ is a 
reductive envelope for $H \act L \to X$. Retain the notation above.
\begin{enumerate}
\item \label{itm:Co1Re10-1} There is an inner enveloping quotient $\mc{V} \subseteq X 
\env H$,
together with an open immersion 
\[
\theta_H: \mc{V} \hookrightarrow \olGUX \dblslash(G
\times H_r) 
\]
such that $\theta_H^*\psi^*N^{\prime}=\OO_{\mc{V}}(n)$ for some $n>0$, where
$N^{\prime} \to (\olGUX \dblslash G) \dblslash_{M^{\prime}} H_r$ is a very
ample line bundle pulling back to a 
positive tensor power of the line bundle $M^{\prime} \to (\olGUX \dblslash 
G)^{\rmss(M^{\prime})}$.
\item \label{itm:Co1Re10-2} There is an inner enveloping quotient $\mc{U}
  \subseteq X \env H_u$ of $X^{\ssUfg}$ 
that is stable under the canonical $H_r$-action on $X \env H_u$ of Proposition
\ref{prop:GiNa7} and an 
$H_r$-equivariant open immersion  
\[
\theta_{H_u}:\mc{U} \hookrightarrow (\olGUX \dblslash G)^{\rmss(M^{\prime})}
\]
such that $\theta_{H_u}^*M^{\prime}$ defines the same linearised polarisation
over $\mc{U}$ as the natural one on $\OO_{\mc{U}}(n)$, for $n$ as in
\ref{itm:Co1Re10-1}. Furthermore, $\mc{U}$ is such that the natural rational map
$q_{H_r}:\proj(S^{H_u}) 
\dashrightarrow \proj(S^H)$ of \eqref{eq:GiSt1} restricts to define a good categorical 
quotient
$q_{H_r}:\mc{U} \to \mc{U} \dblslash H_r$ for the $H_r$-action on $\mc{U}$, with
$\mc{U} \dblslash H_r$ contained in $\mc{V}$ as an open subscheme.  
\item \label{itm:Co1Re10-3} The following diagram commutes (where all unmarked 
inclusions are
natural open immersions): 
\[
\begin{diagram}
                            &               &  X^{\ssfg}       &            &  \rInto                 &                   &   X^{\ol{\rmss}}   \\
                            &   \ldInto  &  \vLine           &            &                            &  \ruEqual    &                            \\
X^{\ssUfg}           &  \rInto  &  \HonV            &              &  X^{\ol{\rmss}}    &                   &                             \\
\dTo^{q_{H_u}}   &               &  \vLine^{q_H} &            &  \dTo^{\pi_{G} \circ \beta \circ \alpha} & & \dTo^{\pi_{G \times H_r}\circ \beta \circ \alpha}\\
\mc{U}                 &  \rInto^{\theta_{H_u}} &  \HonV   &                  & (\olGUX \dblslash G)^{\rmss(M^{\prime})} & & \\
                           &               &  \dTo            &               &  \dTo^{\bar{\pi}_{H_r}}  &           &  \\  
\dTo^{q_{H_r}}   &               & \mc{V}           &  \rHook^{\theta_H} &  \VonH & \rTo &   \olGUX \dblslash (G
\times H_r) \\
                           &  \ruInto &                      &             &            &  \ldInto^{\psi} &   \\
\mc{U} \dblslash H_r &        &  \rInto           &             &  (\olGUX \dblslash G)\dblslash_{M^{\prime}} H_r & &\\   
  \end{diagram}
\]
\end{enumerate}
\end{propn}

\begin{proof} 
We begin by fixing some notation. Suppose $L^{\prime} \to \olGUX$ pulls back to
$L^{\ten r} \to X$ under $\beta \circ \alpha$, with $r>0$. Let $V$ be a
fully separating enveloping system associated to $(\olGUX,\beta,L^{\prime})$ with $V
\subseteq H^0(X,L^{\ten r_1})$ for some positive integral multiple $r_1$ of
$r$. Let $\mc{S} \subseteq V$ be a finite subset to which $V$ is adapted and
such that properties
\ref{itm:Co1Re3-1}--\ref{itm:Co1Re3-3} of Definition \ref{def:Co1Re3} are
satisfied. By Theorem \ref{thm:BgRe2}, \ref{itm:BgRe2-1} there is an ample line bundle
\[
N^{\prime} \to (\olGUX \dblslash G) \dblslash_{M^{\prime}} H_r
\]
that pulls back to $(M^{\prime})^{\ten m} \to (\olGUX \dblslash
G)^{\rmss(M^{\prime})}$ under $\ol{\pi}_{H_r}$, for some $m>0$. Similarly,
$M^{\prime}$ pulls back to $(L^{\prime})^{\ten l} \to \olGUX^{\rmss(G)}$ under $\pi_G
$, for
some $l>0$. By replacing
$N^{\prime}$ by a sufficiently positive tensor power of itself, we may assume
the following: $N^{\prime} \to (\olGUX \dblslash G) \dblslash_{M^{\prime}} H_r$
and $(M^{\prime})^{\ten m} \to \olGUX \dblslash G$ are very ample; and there is
$r_2>0$ such that $r_1r_2=n:=lmr$. Using this second assumption and Lemma
\ref{lem:Co1Re4}, we may use the multiplication map $V^{\ten r_2} \to
H^0(X,L^{\ten n})$ to further assume that $\mc{S} \subseteq V \subseteq
H^0(X,L^{\ten n})$.   

(Proof of \ref{itm:Co1Re10-1}.) We now construct the inner enveloping quotient $\mc{V}
$ and open immersion
$\theta_H$. Let 
\[
\mc{V}=\bigcup_{f \in \mc{S}} \spec((S^H)_{(f)}) \subseteq X \env H.
\]
Recall from Definition \ref{def:GiFi4} of an enveloping system that $\mc{S}$
satisfies $X^{\ssfg} = \bigcup_{f \in \mc{S}} X_f$ and $(S^H)_{(f)}$ has
generating set $\{ \tilde{f}/f \mid \tilde{f} \in V^H\}$ for each $f \in
\mc{S}$. Given $f \in \mc{S}$, by Definition \ref{def:Co1Re5}, \ref{itm:Co1Re1-3} of a
reductive  
envelope, there is an extension
$F \in H^0(\olGUX,(L^{\prime})^{\ten lm})^{G \times H_r}$ of $f$ under $\beta
\circ \alpha$ such that $(\olGUX)_F$ is affine. Pulling back along $\pi_{G
  \times H_r}$ identifies the ring of
regular functions on the affine open subset $\pi_{G \times H_r}((\olGUX)_F)
\subseteq \olGUX \dblslash (G \times H_r)$ with $\OO((\olGUX)_F)^{G \times
  H_r}$, and because $X$ is irreducible
$q_H^{\#}:(S^H)_{(f)} \hookrightarrow \OO(X_f)^H$ is an isomorphism by Lemma
\ref{lem:GiFi3.1}, \ref{itm:GiFi3.1-1}. Therefore there is a unique ring
homomorphism $\Theta_f:\OO(\pi_{G \times H_r}((\olGUX)_F)) \to (S^H)_{(f)}$
making the diagram      
\[
\begin{diagram}
\OO(\pi_{G \times H_r}((\olGUX)_F)) & \rTo^{(\pi_{G \times
    H_r})^{\#}}_{\cong} &  \OO((\olGUX)_F)^{G \times
  H_r}  \\ 
\dTo_{\Theta_f}  & & \dTo_{(\beta \circ \alpha)^{\#}} \\
(S^H)_{(f)} &  \rTo^{q_H^{\#}}_{\cong} & \OO(X_f)^H\\
\end{diagram}
\]
commute. In fact, $\Theta_f$ is an isomorphism. Indeed, by Definition
\ref{def:GiFi4}, \ref{itm:GiFi4-2} of 
an enveloping system $\OO(X_f)^H=(S^H)_{(f)}$ is generated by the regular functions
$q_H^{\#}(\tilde{f}/f)$, where $\tilde{f} \in V^H$. Each such $\tilde{f}$
extends to some 
$\tilde{F} \in H^0(\olGUX,(L^{\prime})^{\ten lm})^{G \times H_r}$ under $\beta
\circ \alpha$ by Definition \ref{def:Co1Re5}, \ref{itm:Co1Re1-2} of a
reductive envelope, and the regular function in $\OO((\olGUX)_F)^{G \times H_r}$
defined by $\tilde{F}/F$ pulls back to $q_H^{\#}(\tilde{f}/f)$ under $\beta \circ
\alpha$. It follows that $\Theta_f$
is surjective. On the other hand, because $\beta:(\GUX)_F \hookrightarrow
(\olGUX)_F$ is a dominant morphism and
$\alpha^{\#}$ identifies $\OO((\GUX)_F)^{G \times H_r}$ with $\OO(X_f)^H$
the map $(\beta \circ \alpha)^{\#}$ is injective, hence $\Theta_f$ is injective also.

It follows that $\Theta_f$ defines an isomorphism of affine varieties 
\[
(\theta_H)_f:\spec((S^H)_{(f)}) \overset{\cong}{\longrightarrow} \pi_{G \times
  H_r}((\olGUX)_F)
\]
with $(\theta_H)_f \circ q_H|_{X_f} = \pi_{G \times H_r} \circ \beta \circ
\alpha|_{X_f}$. Because $\Theta_f$ is defined in terms of compositions of
(inverses of) sheaf homomorphisms and taking invariants is natural with respect
to equivariant inclusions, it can easily be shown that the maps $(\theta_H)_f$
glue over the $\spec((S^H)_{(f)})$ with $f \in \mc{S}$ to define an open immersion
\[
\theta_H:\mc{V} \hookrightarrow \olGUX \dblslash(G \times H_r)
\]     
such that $\theta_H \circ q_H=\pi_{G \times H_r} \circ \beta \circ \alpha$
on $X^{\ssfg}$. We can see $\theta_H^*\psi^*N^{\prime}=\OO_{\mc{V}}(n)$ as
follows. Because of the extension property \ref{itm:Co1Re1-2} of Definition
\ref{def:Co1Re5} of a reductive envelope, the space $V^H$ extends isomorphically under 
$\beta
\circ \alpha$ to a subspace $W$ of $H^0(\olGUX,(L^{\prime})^{\ten
  lm})^{G \times H_r}$. Each section in $W$ descends through the GIT quotient map $
\pi_{G
  \times H_r}$ to a section in $H^0(\olGUX \dblslash(G
\times H_r),\psi^*N^{\prime})$, thus defining a rational map 
\[
\gamma_H:\olGUX \dblslash (G \times H_r) \dashrightarrow \PP(W^*)
\]
that defines a morphism on the image of $\mc{V}$ under $\theta_H$. There is a
natural isomorphism $\PP((V^H)^*) \cong
\PP(W^*)$ induced by $(\beta \circ \alpha)^*$ and, by inspection, one sees
that the composition $\gamma_H \circ \theta_H:\mc{V} \to
\PP(W^*)$ corresponds to the natural locally closed immersion $\mc{V} \hookrightarrow
\PP((V^H)^*)$ defined by $V^H$ in Proposition \ref{prop:GiFi4.1}. Hence
$\theta_H^*\psi^*N^{\prime}=\OO_{\mc{V}}(n)$.

(Proof of \ref{itm:Co1Re10-2}.) The map $\theta_{H_u}$ is constructed in a
similar way to $\theta_H$. By Definition
\ref{def:Co1Re3}, \ref{itm:Co1Re3-2} there is $\mc{S}^{\ssUfg} \subseteq \mc{S}$
such that $V$ defines an enveloping system adapted to $\mc{S}^{\ssUfg}$ for the
restricted linearisation $H_u \act L \to X$. Letting 
\[
\mc{U}=\bigcup_{f \in \mc{S}^{\ssUfg}} \spec((S^{H_u})_{(f)}) \subseteq X \env H_u,
\]
it follows from property \ref{itm:Co1Re1-1} of Definition \ref{def:Co1Re5} that,
for $f \in \mc{S}^{\ssUfg}$ with extension $F$ over $\olGUX$, 
there are natural isomorphisms $\spec((S^{H_u})_{(f)}) \cong
\OO(\pi_G((\olGUX)_F))$ which patch to define an open immersion
\[
\theta_{H_u}:\mc{U} \hookrightarrow (\olGUX \dblslash G)^{\rmss(M^{\prime})}
\] 
such that $\theta_{H_u} \circ q_{H_u} = \pi_G \circ \beta \circ \alpha$ on
$X^{\ssUfg}$, and $\theta_{H_u}^*(M^{\prime})^{\ten m} = \OO_{\mc{U}}(n)$ as
line bundles. The arguments are analogous to those for the construction of $\theta_H:
\mc{V}
\hookrightarrow \olGUX \dblslash (G \times H_r)$. Notice that each section in
$\mc{S}^{\ssUfg}$ is fixed by the 
$H$-action on $H^0(X,L^{\ten n})^H$, so by Proposition \ref{prop:GiNa7} $\mc{U}$
is stable under the $H_r$-action on $X \env H_u$. Furthermore, the
equality $\theta_{H_u} \circ q_{H_u} = \pi_G \circ \beta \circ \alpha$
implies that $\theta_{H_u}$ is $H_r$-equivariant on the image
$q_{H_u}(X^{\ssUfg})$. The interior $q_{H_u}(X^{\ssUfg})^{\circ}$ inside
$\mc{U}$ is therefore a dense open subset of $\mc{U}$ on which $\theta_{H_u}$ is
equivariant, so it follows from the separatedness of $(\olGUX \dblslash
G)^{\rmss(M^{\prime})}$ that $\theta_{H_u}$ is equivariant on the whole of
$\mc{U}$. Because $\theta_{H_u} \circ q_{H_u} = \pi_G \circ \beta \circ
\alpha$, there is a naturally induced map $L^{\ten n}|_{X^{\rmss,H_u \dashfg}} \to
\theta_{H_u}^*(M^{\prime})^{\ten m}$ which is equivariant with respect to $H \to H_r$,
and since $\theta_{H_u}^*(M^{\prime})^{\ten m} \cong \OO_{\mc{U}}(n)$ as line
bundles it follows from Proposition \ref{prop:GiNa8}, \ref{itm:GiNa8-2} that the
$H_r$-linearisation on 
$\theta_{H_u}^*(M^{\prime})^{\ten m}$ defines the same linearisation 
as the natural one on $\OO_{\mc{U}}(n) \to \mc{U}$. 

Because $\mc{S}^{\ssUfg}
\subseteq \mc{S}$, the rational map $q_{H_r}:\proj(S^{H_u})
\dashrightarrow \proj(S^H)$ defines an $H_r$-invariant morphism $q_{H_r}:\mc{U} \to
\mc{V}$, whose restriction to $\spec((S^{H_u})_{(f)})$ for $f \in
\mc{S}^{\ssUfg}$ is the map 
\[
\spec((S^{H_u})_{(f)}) \to \spec((S^{H})_{(f)}) = \spec(((S^{H_u})_{(f)})^{H_r})
\]
induced by the inclusion $((S^{H_u})_{(f)})^{H_r} \hookrightarrow
(S^{H_u})_{(f)}$. By reductive GIT for affine varieties (Theorem
\ref{thm:BgRe1}), each of these restrictions is a good categorical quotient for the
$H_r$-action on $\spec((S^{H_u})_{(f)})$, and since good categorical quotients
are local on the base it follows that $q_{H_r}:\mc{U} \to \mc{U}
\dblslash H_r=q_{H_r}(\mc{U})$ 
is a good categorical quotient for the $H_r$-action on $\mc{U}$.              


(Proof of \ref{itm:Co1Re10-3}.) It remains to prove the commutativity of the
diagram in \ref{itm:Co1Re10-3}. Most of this follows from the construction of $
\theta_H$ and
$\theta_{H_u}$---all that is left is to show is the equality 
\[
\ol{\pi}_{H_r} \circ \theta_{H_u}=\psi \circ \theta_H \circ q_{H_r}:\mc{U} \to (\olGUX
  \dblslash G) \dblslash_{M^{\prime}} H_r.  
\]
Note first that both of these morphisms are indeed well defined on $\mc{U}$. By
construction of $\theta_H$ and diagram \eqref{eq:Co1Re1} we have 
\[
\psi \circ \theta_H \circ q_H = \psi \circ \pi_{G \times H_r} \circ \beta \circ \alpha
= \ol{\pi}_{H_r} \circ \pi_G \circ \beta \circ \alpha
\]
on $X^{\ssfg}$. Since $\pi_G \circ \beta \circ \alpha=\theta_{H_u} \circ
q_{H_u}$ and $q_H=q_{H_r} \circ q_{H_u}$ on $X^{\ssUfg}$, it follows that
\[
\psi \circ \theta_H \circ q_{H_r} \circ q_{H_u}= \ol{\pi}_{H_r} \circ
\theta_{H_u} \circ q_{H_u}:X^{\ssUfg} \to (\olGUX
  \dblslash G) \dblslash_{M^{\prime}} H_r.   
\]
Applying Proposition \ref{prop:GiFi3.2}, \ref{itm:GiFi3.2-2} to this morphism, we 
conclude the
desired equality $\ol{\pi}_{H_r} \circ \theta_{H_u}=\psi \circ \theta_H \circ
q_{H_r}:\mc{U} \to (\olGUX \dblslash G) \dblslash_{M^{\prime}} H_r$.   
\end{proof}




By appealing to Proposition \ref{prop:Ap1}, \ref{itm:Ap1-3} in the Appendix, we
obtain a corollary which is particularly relevant for the aims of
constructing projective completions of the enveloped quotient.
  
\begin{cor} \label{cor:Co1Re10.1}
If $H \act L \to X$ is an ample linearisation and $(\olGUX,\beta,L^{\prime})$
is an ample reductive envelope, then $\olGUX \dblslash (G \times H_r)=(\olGUX
\dblslash G) \dblslash_{M^{\prime}} H_r$ is projective and $\theta_H:\mc{V}
\hookrightarrow \olGUX \dblslash(G \times H_r)$ defines a projective completion of
the enveloped quotient $q(X^{\ssfg})$, as in Definition \ref{def:Co10}. Moreover,
$\theta_{H_u}:\mc{U} \hookrightarrow \olGUX \dblslash G$ defines an
$H_r$-equivariant projective completion of the inner enveloping quotient
$q_{H_u}:X^{\ssUfg} \to \mc{U}$ of $X^{\ssUfg}$.    
\end{cor}

We now come to the main theorem of this section, which is the \emph{raison
  d'\^{e}tre} of reductive envelopes within non-reductive GIT. 

\begin{thm} \label{thm:Co1Re11}
Let $H$ be a linear algebraic group with unipotent radical $H_u$ and $H_r = H/H_u$, let $L
\to X$ be an ample linearisation of a quasi-projective 
$H$-variety $X$ and suppose $(\olGUX,\beta,L^{\prime})$ is a reductive envelope
for the linearisation, formed with respect to an $H_u$-faithful homomorphism $H
\to G$ with $G$ reductive. Let 
\begin{align*}
\pi: \olGUX^{\rmss(L^{\prime})} \to \olGUX \dblslash (G
\times H_r) \\
\pi: \olGUX^{\rms(L^{\prime})} \to \olGUX^{\rms(L^{\prime})} /(G
\times H_r)
\end{align*}
be the GIT quotient and geometric quotient of the semistable and stable locus,
respectively, for the $G \times H_r$-linearisation $L^{\prime} \to\olGUX$. 
\begin{enumerate}
\item \label{itm:Co1Re11-1} There is an 
inner enveloping quotient $q_{H_u}:X^{\ssUfg} \to \mc{U}$ of $X^{\ssUfg}$, with
$\mc{U}$ an $H_r$-invariant open subset of $X \env H_u$ with a good categorical
$H_r$-quotient $q_{H_r}:\mc{U} \to \mc{U} \dblslash H_r$, and an inner
enveloping quotient $q_H:X^{\ssfg} \to \mc{V}$ making the diagram  
\[
\begin{diagram}
X^{\ol{\rms}}  & \rEmpty~\subseteq & X^{\rms} & \rEmpty~\subseteq & X^{\ssUfg} 
& \rEmpty~\subseteq & X^{\ssfg} &  \rEmpty~\subseteq & X^{\ol{\rmss}}=X^{\nss}  \\
\dTo_{q_H}^{\text{\upshape{geo}}} & &  \dTo_{q_H}^{\text{\upshape{geo}}}   &&\dTo_{q_H} & & \dTo_{q_H} && \dTo_{\pi \circ \beta \circ \alpha} \\
X^{\ol{s}}/H & \rEmpty~\subseteq & X^{\rms}/H & \rEmpty~\subseteq & \mc{U} \dblslash H_r &
\rEmpty~\subseteq & \mc{V} & \rEmpty~\subseteq & \olGUX 
\dblslash (G \times H_r)    
\end{diagram}
\]
commute, where all the inclusions are natural open immersions and the two
left-most vertical arrows are geometric 
quotients.
\item \label{itm:Co1Re11-2} The inclusion $\beta \circ \alpha:X^{\ol{\rms}}
\hookrightarrow \olGUX^{\rms(L^{\prime})}$ induces a natural open immersion 
\[
 X^{\ol{\rms}}/H \hookrightarrow \olGUX^{\rms(L^{\prime})}/(G \times H_r).
\]
\item \label{itm:Co1Re11-3} If moreover the reductive envelope
  $(\olGUX,\beta,L^{\prime})$ is ample, then
$\olGUX \dblslash (G \times H_r)$ is projective and thus gives a
projective completion of the enveloped 
quotient, as in Definition \ref{def:Co10}.
\end{enumerate}
\end{thm}

\begin{proof}
This follows immediately by combining Theorem \ref{thm:GiSu2} with Proposition
\ref{prop:Co1Re8}, Corollary \ref{cor:Co1Re8.1}, Proposition 
\ref{prop:Co1Re10} and Corollary \ref{cor:Co1Re10.1}.
\end{proof}

We make a couple of remarks regarding this result.

\begin{rmk} \label{rmk:Co1Re11.1}
Notice from Proposition \ref{prop:Co1Re10} that the inner enveloping quotient
$\mc{U}$ of $X^{\rmss,H_u \dashfg}$ embeds naturally into $(\olGUX \dblslash
G)^{\rmss(M^{\prime})}$, where $M^{\prime} \to \olGUX \dblslash G$ is the
  naturally induced ample linearisation on the GIT quotient of the linearisation
  $G \act L^{\prime} \to \olGUX$. When
  $(\olGUX,\beta,L^{\prime})$ is an ample reductive envelope then $\olGUX
  \dblslash G$ provides an $H_r$-equivariant projective completion of $\mc{U}$, and
  moreover the composition 
\[
X^{\rmss,H_u \dashfg} \overset{q_{H_u}}{\longrightarrow} \mc{U}
\overset{q_{H_r}}{\longrightarrow} \mc{U} \dblslash H_r
\]
can be studied by by doing reductive GIT on $\olGUX$ in stages---first by $G$,
then by $H_r$ (see Corollary \ref{cor:Co1Re10.1}.        
\end{rmk}

\begin{rmk} \label{rmk:Co1Re12}
If one happens to know that semistability and stability for a
reductive envelope $(\olGUX,\beta,L^{\prime})$ coincide, then $X^{\ol{\rms}} =
X^{\ol{\rmss}}$ and we have a string of equalities
\[
X^{\ol{\rms}} = X^{\rms} = X^{\ssUfg} = X^{\rmss} = X^{\nss}=X^{\ol{\rmss}}.
\]   
\end{rmk}

From the point of view of constructing projective completions of
enveloped quotients, the most important application of Theorem
\ref{thm:Co1Re11} is to the case where $(\olGUX,\beta,L^{\prime})$ is an ample
reductive envelope, as assumed in statement \ref{itm:Co1Re11-3}. In this
case, the 
associated completely semistable and stable loci can be computed using the
Hilbert-Mumford criteria and the GIT quotient $\olGUX \dblslash (G
\times H_r)$ can be described set-theoretically as the quotient space of
$\olGUX^{\rmss(L^{\prime})}$ modulo the S-equivalence relation (see Section
\ref{sec:BgReductive}).

\subsection{Strong Reductive Envelopes}
\label{sec:Co1Strong}

We saw in Proposition \ref{prop:Co1Re8} that, given a linear algebraic group $H$ with unipotent radical $H_u$
acting on an irreducible variety $X$ with 
ample linearisation $L \to X$, the
completely stable and completely
semistable loci associated to a reductive
envelope $(\olGUX,\beta,L^{\prime})$ provide approximations of the intrinsically
defined stable locus $X^{\rms}$ and the finitely generated semistable locus
$X^{\ssfg}$: one has $X^{\ol{\rms}} \subseteq X^{\rms}$ and $X^{\ssfg} \subseteq
X^{\nss}=X^{\ol{\rmss}}$. In this section we
discuss particular kinds of reductive envelope $(\olGUX,\beta,L^{\prime})$, which
give equalities $X^{\ol{\rms}}=X^{\rms}$ and $X^{\ol{\rmss}}=X^{\ssfg}$,
thus providing potential ways to compute the finitely generated stable set and
stable set for the original linearisation $H \act L \to X$ using methods from
reductive GIT. In light of Theorem \ref{thm:Co1Re11} and Proposition
\ref{prop:GiSt3}, to obtain these equalities we need to make sure that  
\begin{itemize}
\item no points inside $(\GUX)^{\rms(L)}$ suddenly become
unstable with respect to $L^{\prime}$ as points in $\olGUX$; and
\item any point in $X \subseteq \olGUX$ that is semistable for $L^{\prime}$ must
  lie in $X_f$ for some invariant $f$ over $X$ with $(S^H)_{(f)}$ finitely
  generated over $\kk$.
\end{itemize}
Following the ideas used in \cite[\S 5.2]{dk07}, we adopt the strategy of
effectively forcing out any complications arising from the codimension 1
boundary components of $\olGUX \setminus (\GUX)$, by demanding that extensions
of appropriate invariants over $\GUX$ vanish on these components. This, together with 
a normality
assumption on $\olGUX$, turns out to be enough to get the desired equalities
$X^{\ol{\rms}} = X^{\rms}$ and $X^{\ssfg} = X^{\ol{\rmss}}$. 

\begin{defn} \label{def:Co1St1} Let $H$ be a linear algebraic group acting on an
  irreducible variety $X$ with linearisation $L \to X$. Let $H \to G$ be an
  $H_u$-faithful homomorphism with $G$ a reductive
  group and let $(\olGUX,\beta,L^{\prime})$ be a reductive
  envelope. We call
$(\olGUX,\beta,L^{\prime})$ a \emph{strong reductive envelope} if there is a   \index{strong reductive envelope}
fully separating enveloping system $V$ for $H \act L \to X$ satisfying the
extension properties
\ref{itm:Co1Re1-1}--\ref{itm:Co1Re1-3} of Definition \ref{def:Co1Re5} and the
further property that every $f\in V^H$ extends to
a $G\times H_r$-invariant over $\olGUX$ that vanishes on each codimension 1
component of the boundary of $\GUX$ inside $\olGUX$.
\end{defn}

\begin{propn} \label{prop:Co1St2} Let $H$ be a linear algebraic group acting on
  an irreducible variety $X$ with ample linearisation $L \to X$ and $H \to G$ an
  $H_u$-faithful homomorphism into a reductive group $G$. Suppose
  $(\olGUX,\beta,L^{\prime})$ is a strong
  reductive envelope with $\olGUX$ a normal variety. Then
$X^{\rms}=X^{\overline{\rms}}$ and $X^{\ssUfg}=X^{\ol{\rmss}}$.
\end{propn}

\begin{proof} 
By Proposition \ref{prop:Co1Re8} it suffices to
show $X^{\rms}\subseteq X^{\ol{\rms}}$ and $X^{\ol{\rmss}} \subseteq
X^{\ssUfg}$. Throughout the proof we let $V$ be a fully separating enveloping
system associated to $(\olGUX,\beta,L^{\prime})$ satisfying the conditions in
Definition \ref{def:Co1St1} and let $D_1,\dots,D_m$ be the codimension 1
components of $\olGUX \setminus \beta(\GUX)$.  

(Proof of $X^{\rms} \subseteq X^{\ol{\rms}}$.) Let $x\in X^{\rms}$. By
Definition \ref{def:Co1Re3}, \ref{itm:Co1Re3-3} of a 
fully separating enveloping system there is $f \in V^H$ with extension to a $G
\times H_r$-invariant section $F$
over $\GUX$ such that $f(x) \neq
0$ and $G \times^{H_u}(X_f)=(\GUX)_F$ is affine. By \ref{itm:Co1Re1-3} of Definition
\ref{def:Co1Re5} and the definition 
of a strong reductive envelope, there is a section of some positive tensor
power of $L^{\prime} \to \olGUX$, which we also call $F$, such that $(\beta
\circ \alpha)^*F=f$, with the open set $(\olGUX)_F$ affine and $F$ vanishing on
$\bigcup_{i}D_i\subseteq 
\olGUX$. Thus the complement of $G \times^{H_u}(X_f) = \beta^{-1}((\olGUX)_F)$
inside $(\olGUX)_F$ has codimension at least 2. Because $(\olGUX)_F$ is normal,
pullback along $\beta$ yields an isomorphism $\OO((\olGUX)_F) \cong
\OO(G\times^{H_u}(X_f))$ and, since both $(\olGUX)_F$ and $G
\times^{H_u}(X_f)$ are affine, the open inclusion $\beta:G
\times^{H_u}(X_f) \hookrightarrow (\olGUX)_F$ is therefore an isomorphism. But $G 
\times^{H_u}(X_f)=(\olGUX)_F$ is contained in the stable locus for the $G \times
H_r$-linearisation $L \to \GUX$ by Proposition \ref{prop:GiSt3}, so the $G \times H_r
$-action on
$(\olGUX)_F$ is closed with all stabilisers finite. It follows that $(\olGUX)_F
\subseteq \olGUX^{\rms(L^{\prime})}$ and thus $x \in X^{\ol{\rms}}$.  

(Proof $X^{\ol{\rmss}} \subseteq X^{\rmss,H_u \dashfg}$.) Now suppose $x \in
X^{\ol{\rmss}}$. By Definition \ref{def:Co1Re3}, 
\ref{itm:Co1Re3-1} of a fully separating enveloping system we have
$X^{\ol{\rmss}} = X^{\nss}=\bigcup_{f \in V^H}X_f$, 
so by Definition \ref{def:Co1St1} there is $f \in V^H$ with extension to a $G \times 
H_r$-invariant
section $F$ of some positive tensor power of $L^{\prime} \to \olGUX$, with
$(\olGUX)_F$ affine, such that
$x \in X_f$ and $F$ vanishes on $\bigcup_iD_i$. As
above, the complement of $G \times^{H_u}(X_f)$ is therefore of codimension at least
2 in $(\olGUX)_F$, so from the normality of $(\olGUX)_F$ it follows that the
pullback map 
$\beta^{\#}:\OO((\olGUX)_F) \to \OO(G \times^{H_u}(X_f))$ is a $G \times
H_r$-equivariant isomorphism. Thus $(\beta \circ \alpha)^{\#}$ yields isomorphisms
\begin{align*} 
\OO((\olGUX)_F)^{G} &\overset{\cong}{\longrightarrow}
\OO(X_f)^{H_u}=(S^{H_u})_{(f)}, \\
\OO((\olGUX)_F)^{G \times H_r} &\overset{\cong}{\longrightarrow}
\OO(X_f)^{H}=(S^{H})_{(f)}.
\end{align*} 
Since $(\olGUX)_F$ is affine and $G$ and $G \times H_r$ are reductive, the
$\kk$-algebras $(S^{H_u})_{(f)}$ and $(S^H)_{(f)}$ are therefore finitely
generated and thus
$x \in X_f \subseteq X^{\ssUfg}$. 
\end{proof}

\begin{rmk} \label{rmk:Co1St3} Observe that as a corollary to Proposition
  \ref{prop:Co1St2}, for any given linearisation $H \act L \to X$ a necessary
  condition for the existence of a strong 
  reductive envelope with normal $\olGUX$ is that
  $X^{\ssUfg}=X^{\ssfg}=X^{\nss}$ (cf.  Remark
\ref{rem:gentleunipotent}). 
\end{rmk}
 
For the remainder of this section we will consider ways to try to construct strong
reductive envelopes. Recall that this means (1) choosing an
equivariant completion $\olGUX$ of $\GUX$, together with (2) an extension
$L^{\prime} \to \olGUX$ of some positive tensor power of the linearisation $L
\to \GUX$ such that (3) (roughly stated) enough invariant sections over $\GUX$
extend to sections over $\olGUX$ vanishing on the boundary divisors. In
general (2) depends heavily on the singularities of the completion in (1).
We don't wish to explore in depth here; instead it suffices for us to note that if
$\olGUX$ 
is $\mb{Q}$-factorial then some positive tensor power of $L \to \GUX$ extends
over the boundary, and if $\olGUX$ is even factorial (for example, smooth) then
$L$ itself extends. Moreover, if $X$ is irreducible and $H$ and $G$ are
\emph{connected} linear algebraic groups such that the action of $G \times H_r$
on $\GUX$ extends to one on $\olGUX$, then any line bundle $L^{\prime} \to
\olGUX$ extending $L^{\ten r} \to \GUX$ (with $r>0$)
  has a unique linearisation extending $L^{\ten r} \to \GUX$ (see the proof of
  \cite[Converse 1.13]{mfk94}). 

On the other hand, given any extension $L^{\prime} \to \olGUX$ of (a power of)
the $G \times H_r$-linearisation $L \to \GUX$, then as in \cite{dk07} the
question of (3) can be 
approached by making a fairly mild assumption on the nature of the completion
$\olGUX$: namely, that it is a \emph{gentle} completion of $\GUX$, in the sense
of Definition \ref{def:Co1St3.1}: recall this means that $\olGUX$ is a normal
projective variety such that every
  codimension 1 component of the boundary of $\GUX$ in $\olGUX$ is a
  $\mb{Q}$-Cartier divisor. What follows is a
  generalisation of the constructions in \cite[\S\S 5.3.1--5.3.2]{dk07}
  (described in Section \ref{subsec:BgUnExtending}) that
  applies to our current setting. 


So let us now assume that $H$ and $G$ are \emph{connected} linear algebraic groups,
with $G$ reductive, and $\rho:H \to G$ is a fixed $H_u$-faithful homomorphism. 
Let $\olGUX$ be a
gentle completion of $\GUX$ and $L^{\prime} \to \olGUX$ a 
$G \times H_r$-linearisation extending $L^{\ten r} \to
\GUX$, for some $r>0$. Also let $D_1,\dots,D_m \subseteq \olGUX$ be the
codimension 1 irreducible components of the 
complement of $\GUX$ in $\olGUX$ and define the $\mb{Q}$-Cartier divisor 
\[
D:=\sum_{i=1}^m D_i.
\]
Then for any sufficiently divisible integer
$N>0$ the divisor $ND$ is Cartier and defines a line bundle
$\OO(ND)$ on $\olGUX$ which restricts to the trivial bundle on
$\GUX$. As in Section \ref{subsec:BgUnExtending}, given any line bundle $M \to
\olGUX$, define 
\begin{equation} \label{eq:Co1St0.1}
M_N:=M \ten \OO(ND) \to \olGUX.
\end{equation}
In the case $M=L^{\prime}$ then, because $G \times H_r$ is connected, the $G \times
H_r$-linearisation on $L^{\ten r} \to \GUX$ extends uniquely to a $G\times
H_r$-linearisation on $L^{\prime}_N$.  

The next lemma consists of an expanded version of the argument found in
\cite[Proposition 5.3.10]{dk07}.

\begin{lem} \label{lem:Co1St4}
Let $\beta:\GUX \hookrightarrow \olGUX$ be a gentle completion and $L^{\prime}
\to \olGUX$ a line bundle extending $L$. Retain the preceding notation.  
\begin{enumerate}
\item \label{itm:Co1St3-1} Let $f$ be a section of $L^{\ten rr^{\prime}} \to
  \GUX$ for some 
  $r^{\prime}>0$. Then for $N>0$ a sufficiently divisible integer, $f$ extends
  to a section of $(L^{\prime}_N)^{\ten r^{\prime}} \to \olGUX$ under $\beta$.
\item \label{itm:Co1St3-2} For any $N>0$ such that $ND$ is Cartier and for any
  integers $r^{\prime},m>0$, there is a natural 
  inclusion $H^0(\olGUX,(L^{\prime}_N)^{\ten r^{\prime}})
\hookrightarrow H^0(\olGUX,(L^{\prime}_{mN})^{\ten r^{\prime}})$
whose image consists of sections of $(L^{\prime}_N)^{\ten r^{\prime}}$
that vanish on each of $D_1,\dots,D_m$.
\end{enumerate}
\end{lem}
\begin{proof} 
Throughout the proof we denote the field of rational functions on an irreducible
variety $Y$ by $\kk(Y)$.  

(Proof of \ref{itm:Co1St3-1}.) 
It suffices to show that for a suitable $N>0$ the section $f$ extends to a
section of $L'_N$ over codimension 1 boundary components---it will canonically
extend over components of codimension at least 2 because $\olGUX$ is normal.
We use the fact that $L^{\prime}$
corresponds to a Cartier divisor \cite[Proposition 6.15]{har77}; that is, there is a 
finite
collection of pairs $(U_{j},t_{j})$, with $U_j \subseteq \olGUX$ open and $t_j
\in \kk(U_j)=\kk(\olGUX)$, such that $L^{\prime}$ is represented by a Weil
divisor whose restriction to $U_j$ is the principal divisor defined by
$t_j$. Note that $(L^{\prime})^{\ten r^{\prime}}$ is then 
represented by the collection
$(U_{j},t_{j}^{r^{\prime}})$. There is a positive integer $N$ such that
the order of the vanishing of $t_{j}$ along each $D_i$ is less than $N$ for each
$j$. Thinking of sections of $L^{\ten rr^{\prime}}$ as sections of the constant
sheaf of rational functions on $\GUX$ in the standard way \cite[Chapter 2, \S
6]{har77}, we can write
$f|_{U_{j}\cap(\GUX)}=b_{j}/t_{j}^{r^{\prime}}$ with $b_{j}$ a
regular function on $U_{j}\cap(\GUX)$. We can also
assume without loss of generality that $ND$ is a Cartier divisor of $\olGUX$, 
since the completion $\olGUX$ is gentle. Let $(V_{k},s_{k})$
represent $ND$, with $V_k \subseteq \olGUX$ open and $s_k \in
\kk(V_k)=\kk(\olGUX)$, again with the index set for
$k$ being finite, and set $a_{jk}=b_{j}s_{k}^{r^{\prime}} \in
\kk(\olGUX)$. Now $b_j$ may have poles along $U_{j} \cap
V_{k} \cap (\bigcup_i D_i)$; on the other hand each $s_{k}^{r^{\prime}}$
vanishes on $U_{j} \cap
V_{k} \cap (\bigcup_i D_i)$ with order $r^{\prime}N$, so by further increasing
$N$ if necessary we
may assume each $a_{jk}$ defines a regular function on $U_{j} \cap
V_{k}$. Thus each $a_{jk}/(t_j s_k)^{r^{\prime}}$ is a section of $(L^{\prime}
_N)^{\ten
r^{\prime}}$ over $U_j \cap V_k$, whose restriction to $(\GUX) \cap
U_{j} \cap V_{k}$ is defined by $b_{j}/t_{j}^{r^{\prime}}$. One can check that
all the $a_{jk}/(t_js_k)^{r^{\prime}}$ (for all $j$ and $k$) agree on overlaps,
so they patch together to give a global section of $(L^{\prime}_N)^{\ten
  r^{\prime}}$ which extends $f$.


(Proof of \ref{itm:Co1St3-2}.) Assume $N$ is large enough so that $ND$ is Cartier, and
let $r^{\prime}>0$. Continuing with the notation used above to prove
\ref{itm:Co1St3-1}, the sheaf $\OO(ND)$ is represented by the
collection $(V_k,s_k)$, with $s_k \in \OO(V_k)$ such that
$\OO(ND)|_{V_k}=\OO_{V_k} \langle 1/s_k \rangle$ as sheaves of
$\OO_{V_k}$-modules. Then for each $m>0$ there is a
well-defined inclusion $\OO(r^{\prime}ND) 
\hookrightarrow \OO(mr^{\prime}ND)$, whose restriction to $V_k$ corresponds to
the multiplication-by-$s_k^{r^{\prime}(m-1)}$ map $\OO_{V_k} \to
\OO_{V_k}$. Note that sections in the image of this map vanish of each of the
$D_i$. Because
$(L^{\prime})^{\ten r^{\prime}}$ is locally free the natural map of sheaves
\[
(\underline{L^{\prime}})^{\ten r^{\prime}} \ten \OO(r^{\prime}ND) \to
(\underline{L^{\prime}})^{\ten r^{\prime}} \ten \OO(mr^{\prime}ND)
\]
is again injective \cite[Chapter 3, Proposition 9.2]{har77}, and since taking global
sections is left-exact we see that 
this yields an injection
\[
H^0(\olGUX,(L^{\prime}_N)^{\ten r^{\prime}})
\hookrightarrow H^0(\olGUX,(L^{\prime}_{mN})^{\ten r^{\prime}}).
\]         
It is immediate that any section in the image of this map vanishes on each of
the $D_i$. 
\end{proof}

Lemma \ref{lem:Co1St4} says that given any finite collection of sections of some power
of $L^{\ten r} \to \GUX$ we can always modify the extension $L^{\prime}$
so that these sections extend to sections of the resulting linearisation which
vanish on the boundary of $\olGUX$. With additional
assumptions on $L^{\prime}$ we can use this fact to produce a strong reductive
envelope.           

\begin{propn} \label{prop:Co1St5} Let $H$ be a connected linear algebraic group
  acting on an irreducible variety $X$ with linearisation $L \to X$ and let $H
  \to G$ be an $H_u$-faithful homomorphism into a connected reductive group $G$ where $H_u$ is the unipotent radical of $H$.
  Suppose $\beta:\GUX \hookrightarrow \olGUX$ is a gentle $G \times
  H_r$-equivariant 
  projective completion and let
  $L^{\prime} \to \olGUX$ be any $G\times H_r$-linearisation extending
  some positive tensor power of the
  $H$-linearisation $L \to X$. If either 
\begin{itemize} 
\item $(\olGUX,\beta,L^{\prime})$ defines a reductive
  envelope for $H \act L \to X$; or 
\item the line bundle $L^{\prime}_n$ of \eqref{eq:Co1St0.1} is ample, for
  sufficiently divisible integers $n>0$;
\end{itemize} 
then for sufficiently divisible integers $N>0$ the triple
$(\olGUX,\beta,L^{\prime}_{N})$ defines a strong reductive
envelope.
\end{propn}

\begin{proof} 
Suppose $L^{\prime} \to \olGUX$ pulls back to $L^{\ten r} \to X$ under $\beta
\circ \alpha$ and let $D_1,\dots,D_m$ be the irreducible codimension 1
components of the boundary of the gentle completion $\beta:\GUX \hookrightarrow
\olGUX$. We first show the following: given $r^{\prime}>0$ and an 
enveloping system $V \subseteq H^0(X,L^{\ten rr^{\prime}})$, for sufficiently
divisible integers $N>0$ (depending on $V$) each section $f \in V^{H_u}$
(respectively, $f\in V^H$) extends to a $G$-invariant (respectively, $G \times
H_r$-invariant) section $F$ of $(L^{\prime}_N)^{\ten r^{\prime}} \to \olGUX$ under $
\beta
\circ \alpha$ which vanishes on each $D_i$. To this end, let $f\in V^{H_u}$ 
(respectively,
$f \in V^{H}$). By Lemma 
\ref{lem:Co1St4} there is an integer $N_f>0$ such that $f$ extends to a
section $F$ of $(L^{\prime}_{N_f})^{\ten r^{\prime}}$ over $\olGUX$
which vanishes on the codimension 1 boundary components of $\GUX$ inside
$\olGUX$. Note that $F$ must be $G$-invariant (respectively,
$G \times H_r$-invariant): by the normality of $\olGUX$
the section $f$ extends canonically to an invariant over the boundary components of
codimension at least 2 in $\olGUX$, and since $F$ vanishes on the
remaining boundary components it too must be invariant. Now take bases
$\mc{B}_{H_u}$ of $V^{H_u}$ and $\mc{B}_H$ of $V^H$ and let $N>0$ 
be any positive integer which is properly divisible by 
all the $N_{f}$ for $f \in \mc{B}_{H_u} \cup \mc{B}_{H}$. Then by Lemma
\ref{lem:Co1St4}, \ref{itm:Co1St3-2} any $f$ in $V^{H_u}$ or $V^H$
extends to a section $F$ of $(L^{\prime}_N)^{\ten r^{\prime}}$, invariant in the
appropriate sense, which vanishes on the codimension 1 boundary components of
$\GUX$ in $\olGUX$, which was to be shown.   

Now suppose $(\olGUX,\beta,L^{\prime})$ defines
a reductive envelope for $H \act L \to X$ and let $V$ be an associated fully 
separating
enveloping system satisfying \ref{itm:Co1Re1-1}--\ref{itm:Co1Re1-3} of
Definition \ref{def:Co1Re5} of a reductive envelope. As above, for sufficiently
divisible $N>0$ each $f$ in $V^{H_u}$ or $V^H$ extends to an invariant $F$ (in the
appropriate sense) of some positive tensor power of $L^{\prime}_N \to
\olGUX$ which vanishes on the codimension 1 complement $\bigcup_iD_i$ of $\GUX$
in $\olGUX$. To show that $(\olGUX,\beta,L^{\prime}_N)$ is a strong reductive
envelope, we are left to show that each such $(\olGUX)_F$ is affine (see
Definition \ref{def:Co1Re5}, \ref{itm:Co1Re1-3}). But by Definition
\ref{def:Co1Re5} of a reductive envelope applied to $L^{\prime}$, the section $f$ does
extend to an
invariant $F^{\prime}$ of some positive tensor power of $L^{\prime} \to
\olGUX$ with $(\olGUX)_{F^{\prime}}$ affine. Observe that the restrictions of
$L^{\prime}_N$ and $L^{\prime}$ 
to $\olGUX \setminus (\bigcup_i D_i)$ are equal, and $\GUX$ and $\olGUX \setminus
(\bigcup_i D_i)$ differ only in codimension at least 2, so by normality of $\olGUX$
the sections $F$ and $F^{\prime}$ are equal over $\olGUX \setminus
(\bigcup_iD_i)$. It follows that
\[
(\olGUX)_F=(\olGUX)_{F^{\prime}} \setminus (\textstyle{\bigcup_i} D_i).
\]          
Because $(\olGUX)_{F^{\prime}}$
is affine and the complement of the support of a Cartier divisor on an affine
variety is again affine \cite[Tag 01WQ]{stacks-project}, $(\olGUX)_F$ is
therefore also affine. Hence, 
$(\olGUX,\beta,L^{\prime}_N)$ is a strong reductive envelope, for sufficiently
divisible $N>0$.
   
Finally, consider the case where $L^{\prime}_n$ is ample for sufficiently divisible
$n>0$. Appealing to Lemma \ref{lem:Co1Re4} there is an integer $r^{\prime}>0$ 
such that $H^0(X,L^{\ten rr^{\prime}})$ contains a fully
separating enveloping system $V$, and for sufficiently
divisible $N>0$ each $f$ in $V^{H_u}$ or $V^H$ extends to an invariant $F$ (in the
appropriate sense) of some positive tensor power of $L^{\prime}_N \to
\olGUX$ which vanishes on each $D_i$. We can choose $N>0$ sufficiently divisible
so that that $L^{\prime}_N$ is ample and thus each
$(\olGUX)_F$ affine. Then
$(\olGUX,\beta,L^{\prime}_N)$ is a strong ample reductive envelope.  
\end{proof}

\begin{cor} \label{cor:Co1St5.1} 
In the setting of Proposition \ref{prop:Co1St5}, suppose we are in the situation
that $X$ is projective and $L^{\prime}_n$ is ample for sufficiently divisible
integers $n>0$. If 
$S^H=\kk[X,L]^H$ is a finitely generated $\kk$-algebra, then for sufficiently
divisible integers $N>0$ the inclusion $\beta \circ \alpha:X^{\ssfg} \hookrightarrow
(\olGUX)^{\rmss(L^{\prime}_N)}$ induces a natural isomorphism
\[
X \env H = \proj(S^H) \cong \olGUX \dblslash_{L^{\prime}_N} (G \times H_r).
\]    
\end{cor}

\begin{proof}
Because $X$ is projective and $L \to X$ is necessarily ample, the invariant ring
$S^H$ is finitely generated with degree 0 piece equal to the ground field $\kk$,
hence $X \env H=\proj(S^H)$ is a projective variety. Thus $X \env H$ is an inner
enveloping quotient of $X^{\ssfg} = X^{\nss}$. As
in the proof of Proposition \ref{prop:Co1St5}, by appealing to Lemma 
\ref{lem:Co1Re4} we may find $r^{\prime}>0$ such that $H^0(X,L^{\ten rr^{\prime}})$ 
contains a fully
separating enveloping system $V$. By taking the image of $V$
under a suitable multiplication map (see Lemma \ref{lem:Co1Re4}) and enlarging
using Lemma \ref{lem:TrRe4} if
necessary, we can assume that
$V$ is a fully separating system that also contains a (finite) collection of
generators $f_1,\dots,f_m$ of the ring $\kk[X,L^{\ten rr^{\prime}}]^H$. Then 
\[
X \env H \cong \proj(\kk[X,L^{\ten rr^{\prime}}]^H)= \bigcup_{j=1}^m
\spec((S^H)_{(f_j)}) 
\] 
and each $(S^H)_{(f_j)}$ is generated by $\{ \tilde{f}/f_j \mid \tilde{f} \in
V^H\}$. Without loss of generality we may therefore view $V$ as a fully
separating enveloping system which is adapted to to a subset $\mc{S}$ 
containing $\{f_1,\dots,f_m\}$. For sufficiently
divisible $N>0$ we have $L^{\prime}_N$ ample and, as
in the proof of Proposition \ref{prop:Co1St5}, each $f \in V^H$ extends to a
$G \times H_r$-invariant $F$ of some positive tensor power of $L^{\prime}_N \to 
\olGUX$. The inclusion $\beta \circ \alpha:X^{\ssfg} \hookrightarrow
(\olGUX)^{\rmss(L^{\prime}_N)}$ is therefore well-defined, and defines a
dominant open immersion  
\[
\theta_H:X \env H \hookrightarrow \olGUX \dblslash_{L^{\prime}_N} (G \times H_r),
\]
as in the proof of Proposition \ref{prop:Co1Re10}, \ref{itm:Co1Re10-1}.
Since $X \env H$ is proper and $\olGUX \dblslash_{L^{\prime}_N} (G \times H_r)$
is separated, the image of $\theta_H$ is also closed in $\olGUX
\dblslash_{L^{\prime}_N} (G \times H_r)$ \cite[Tag 01W0]{stacks-project} and
thus is the whole of $\olGUX 
\dblslash_{L^{\prime}_N} (G \times H_r)$. Therefore $\theta_H$ defines an
isomorphism $X \env H \cong \olGUX \dblslash_{L^{\prime}_N} (G \times H_r)$.        
\end{proof}

The situation from Proposition \ref{prop:Co1St5} where
the linearisation $L^{\prime}_N \to \olGUX$ is ample for sufficiently divisible
$N$ is potentially the most useful in applications, because it does not require
any verification that $(\olGUX,\beta,L^{\prime})$ forms a reductive envelope
(in particular, that enough invariants extend to sections $F$ with $(\olGUX)_F$
affine). For the proposition to apply however, one needs to know that $\olGUX$
provides a gentle completion of $\GUX$ and this imposes restrictions on the
linearisation $H \act L \to X$ (cf. Remarks \ref{rem:gentleunipotent} and \ref{rmk:Co1St3}). For example, if $X$ is affine and $L=\OO_X \to
X$ with the canonical linearisation, then without loss of generality we may
assume that $N>0$ is such that $L'_N \to \olGUX$ forms an ample, strong
reductive envelope with respect to an enveloping system inside $H^0(X,L)=\OO(X)$
that includes a nonzero 
constant function $f$; let $F$ denote the extension of $f$ to 
$L'_N \to \olGUX$. Then $(\olGUX)_F$ is an affine variety containing $(\GUX)_F=X_f=X$
as a codimension 2 complement, so if $\olGUX$ is gentle then
$\OO(X)^H=\kk[(\olGUX)_F,L'_N]^G$ is a finitely generated $\kk$-algebra. 
    
In the next section we consider a special case which allows us to explicitly construct
strong ample reductive envelopes $(\olGUX,\beta,L^{\prime})$ with gentle
completions of $\GUX$.  
 
\subsubsection{Special Case: Extension to a $G$-Linearisation and Grosshans
  Subgroups} 
\label{subsec:Co1StImportant}

For this section we suppose that $X$ is a
\emph{projective} 
irreducible $H$-variety. If $\rho:H \to G$ is an $H_u$-faithful homomorphism into a
reductive group $G$ and the linearisation $H \act L \to X$ can be partially
extended to a $G$-linearisation, one can reduce the task of constructing ample strong 
reductive envelopes to understanding the geometry of the homogeneous space
$G/H_u$ by `untwisting' the $G$-action on $\GUX$. More precisely, suppose there is a
$G$-linearisation on the line bundle $L \to X$ satisfying the following condition:
\begin{description}\nonumber
\item[(C1)] \label{itm:Co1St0-1} The linearisation $H_u \act L \to X$ arising from
 restricting the $H$-linearisation extends to the $G$-linearisation
 through $\rho|_{H_u}$.   
\end{description}

The extension condition (C1) yields an isomorphism of
$G \times H_r$-linearisations:    

\begin{equation} \label{eq:Co1St1} 
  \begin{diagram}
G\times^{H_u} L & \rTo^{\cong} & (G/H_u) \times L & \rEmpty &
\text{$[g,l] \mapsto (gH_u,gl)$} \\
\dTo & & \dTo & & \\
\GUX & \rTo^{\cong} & (G/H_u) \times X & \rEmpty & \text{$[g,x] \mapsto (gH_u,gx)$} \\
  \end{diagram}
\end{equation} 
The corresponding $G$-linearisation on the right hand side of this diagram is the one
given by taking the product of the linearisation on $L \to X$ and left multiplication 
on
$G/H_u$. The $H_r$-linearisation is more complicated and in general cannot be
expressed as the product of linearisations over $G/H_u$ and $X$. We make
an additional assumption to demand this:
\begin{description}\nonumber
\item[(C2)] \label{itm:Co1St0-2} There is a linearisation $G \times H_r \act \tilde{L}
  \to X$, with $\tilde{L}=L$ as line bundles, such that the $G \times
  H_r$-linearisation $(G/H_u) \times L
  \to (G/H_u) \times X$ arising from \eqref{eq:Co1St1} coincides with the
  product of $\tilde{L} \to X$ and the $G\times H_r$-action on $G/H_u$ given by
  left 
  multiplication by $G$ and right multiplication by $H_r$\footnote{That is,
    $(g,\ol{h}) \cdot g_0H_u=gg_0\rho(h)^{-1}H_u$ for all $\ol{h}=hH_u \in
    H_r=H/H_u$, $g\in G$ and all $g_0H_u \in G/H_u$.}.   
\end{description}

\begin{eg} \label{ex:Co1St0}
Suppose the $H_u$-faithful homomorphism $\rho$ is just a closed embedding $H
\hookrightarrow G$. Then the $H_r$-linearisation on $(G/H_u) \times L \to
(G/H_u) \times X$ under the isomorphism \eqref{eq:Co1St1} is simply the product
of the right multiplication action on $G/H_u$ together with $\tilde{L} \to X$,
where $\tilde{L} \to X$ is the line bundle $L \to X$ equipped with the trivial
$H_r$-linearisation.    
\end{eg}

Assuming $X$ is projective and conditions (C1) and
(C2) are satisfied, then a
natural way to complete $\GUX \cong (G/H_u) \times X$ is to study $G \times
H_r$-equivariant projective
completions $\ol{G/H_u}$ of $G/H_u$. For example, if $L$ is ample and it
is known that $\ol{G/H_u}$ can be chosen to be normal with $\mb{Q}$-Cartier
prime boundary divisors $E_1,\dots,E_m$ such that $\OO(N\sum_i E_i) \to \ol{G/H_u}$ is
ample for sufficiently divisible $N>0$, then the codimension 1 components of the
boundary $\ol{G/H_u}
\times X$ are precisely the $E_i \times X$. These are $\mb{Q}$-Cartier divisors,
so if $X$ is
further assumed to be normal then $\beta:\GUX \hookrightarrow \olGUX$ is a gentle
completion and Proposition \ref{prop:Co1St5} applies to
$L^{\prime}=\OO(N\sum_i E_i) \boxtimes \tilde{L}$ (with $\tilde{L} \to X$ the $G
\times H_r$-linearisation of (C2)).        

This works out particularly well in the case where $H_u$ is a \emph{Grosshans
  subgroup} of the reductive group $G$. Recall from \cite[\S 4]{gro97} that this means 
the pair $H_u
\subseteq G$ satisfies the following equivalent conditions:
\begin{itemize}
\item $\OO(G/H_u)$ is a finitely generated $\kk$-algebra; 
\item there is a finite dimensional $G$-module $W$ and a vector $w \in W$ such
  that $H_u=\stab_G(w)$ and $G/H_u$ embeds into the closure $\ol{G \cdot w}
  \subseteq W$ with complement having codimension at least 2, via the natural map
  $G/H_u \to G \cdot w$.   
\end{itemize}
In this case, since
$\OO(G/H_u)=\OO(G)^{H_u}$ is a normal ring
$\spec(\OO(G/H_u))=\spec(\OO(G)^{H_u})$ is a normal affine variety upon which $G
\times H_r$ naturally acts, and there is a canonical open immersion 
\[
G/H_u \hookrightarrow \spec(\OO(G/H_u)) 
\]
that is $G \times H_r$-equivariant. It follows from \cite[Theorem
4.3]{gro97} that the boundary of $G/H_u$ in fact sits inside $\spec(\OO(G/H_u))$
with a complement of codimension at least 2.    

\begin{rmk}\label{rmk:Co1St6}
In the situation where (C1) holds we have
\[
\kk[X,L]^{H_u}=\kk[\GUX,L]^G=\kk[G/H_u \times X,L]^G \cong (\OO(G/H_u) \ten
\kk[X,L])^G,
\]
where the last isomorphism follows from the K\"{u}nneth formula \cite[Tag
02KE]{stacks-project}. If
$H_u$ is a Grosshans subgroup of $G$ and $H \act L \to X$ is an ample
linearisation of a projective variety, then by Nagata's theorem \cite{nag64} it
follows that $\kk[X,L]^{H_u}$ is
finitely generated. Then $\kk[X,L]^H=(\kk[X,L]^{H_u})^{H_r}$ is also finitely
generated, hence $X^{\ssUfg}=X^{\ssfg}=X^{\nss}$ by definition and  $X \env
H=\proj(\kk[X,L]^H)$ is a projective variety.      
\end{rmk}

Now, any normal affine completion $\ol{G/H_u}^{\aff}$ of $G/H_u$ admits an equivariant 
closed
immersion into some $G \times H_r$-module $W$. As in the proof of
Proposition \ref{prop:Co1Re6}, let $\kk$ be a copy of the ground field equipped
with the trivial $G \times H_r$-action and consider the closure $\ol{G/H_u}$ of
$\ol{G/H_u}^{\aff}$ under the natural open immersion
$W \hookrightarrow \PP:=\PP(W \oplus \kk)$. The complement
$D_{\infty}:=\ol{G/H_u} \setminus \ol{G/H_u}^{\aff}$ is an effective Cartier
divisor corresponding to the hyperplane line bundle
$\OO_{\ol{G/H_u}}(1)=\OO_{\PP}(1)|_{\ol{G/H_u}}$. Let $\nu:\widetilde{G/H_u}
\to \ol{G/H_u}$ be the normalisation of $\ol{G/H_u}$. Then $\widetilde{G/H_u}$
naturally contains $\ol{G/H_u}^{\aff}$ as an open subset with $\widetilde{G/H_u}
\setminus \ol{G/H_u}^{\aff} = \nu^{-1}(D_{\infty})$, and because $\nu$ is a
finite map $\nu^{-1}(D_{\infty})$ is a divisor corresponding to the ample line
bundle $\OO_{\widetilde{G/H_u}}(1)=\nu^*\OO_{\ol{G/H_u}}(1)$. The naturally induced $G 
\times
H_r$-linearisation on $\OO_{\ol{G/H_u}}(1) \to \ol{G/H_u}$ canonically defines a
linearisation on the normalisation $\OO_{\widetilde{G/H_u}}(1) \to
\widetilde{G/H_u}$ \cite[Chapter 1]{ses62}, which pulls back to the canonical $G
\times H_r$-linearisation on $\OO_{\ol{G/H_u}^{\aff}} \to \ol{G/H_u}^{\aff}$. 

\begin{eg} \label{ex:Co1St6.1}
As a simple example, consider the situation where $H$ is a linear algebraic
group with $H_u \cong \GG_a$ and $\rho:H \to G=\SL(2,\kk)$ is an $H_u$-faithful
homomorphism. Without loss of generality we may assume $\rho$ maps $H_u$ onto the 
subgroup of strictly upper-triangular matrices of $G$ and, as is well known,
$H_u$ is a Grosshans subgroup of
$G$. Indeed, consider the defining
representation of $G$ on 
$\kk^2$. The orbit of $\left(\begin{smallmatrix} 1 \\
    0 \end{smallmatrix}\right)$ is $\kk^2 
\setminus \{0\}$ and has stabiliser equal to $H_u$, so that $G/H_u \cong \kk^2
\setminus \{0\}$. This has a $G \times H_r$-equivariant normal affine
completion $\ol{G/H_u}^{\aff}:= \kk^2$, containing $G/H_u$ with codimension 2
complement. By adding a line $D_{\infty}$ at infinity, we obtain an
equivariant normal projective completion $\ol{G/H_u}=\PP^2$ and then
$\OO(D_{\infty})=\OO_{\PP^2}(1)$ has a natural $G \times H_r$-linearisation
extending the one on $\OO \to \kk^2 \setminus \{0\}$.      
\end{eg}

Pulling the previous two observations together, we conclude that when $H_u$ is a
Grosshans subgroup of $G$ we may find a gentle, $G \times H_r$-equivariant projective completion $\ol{G/H_u}$ of $G/H_u$ such that the codimension 1 part of the boundary is an effective Cartier divisor $D_{\infty}$ corresponding to an ample, $G
\times H_r$-linearised line bundle $\OO_{\ol{G/H_u}}(1) \to \ol{G/H_u}$ which
restricts to the canonical $G \times H_r$-linearisation on $\OO_{G/H_u} \to
G/H_u$. Given this, and assuming condition
(C2), let
\[
\beta:\GUX \cong (G/H_u) \times X \hookrightarrow \ol{G/H_u} \times X
\]
be the resulting open immersion and let $L^{\prime}=\OO_{\ol{G/H_u}}
\boxtimes \tilde{L} \to \ol{G/H_u} \times X$ be the $G \times H_r$-linearisation
required by (C2). Let $E_1,\dots,E_m$ be prime divisors such that $D_{\infty} = \sum_i E_i$. Then $\sum_i E_i \times X$ is the Cartier divisor corresponding to the codimension 1 part of the boundary of $\GUX$ inside $\ol{G/H_u} \times X$, so for integers $N>0$ 
\eqref{eq:Co1St0.1} yields the linearisation $L^{\prime}_N=\OO(ND_{\infty})
\boxtimes \tilde{L}$ and we obtain a triple
\begin{equation} \label{eq:Co1St3}
(\ol{G/H_u} \times X,\beta,L'_N=\OO(ND_{\infty}) \boxtimes \tilde{L})
\end{equation}     
such that $\ol{G/H_u} \times X$ is a gentle completion of $\GUX$ (assuming $X$ is 
normal) and the $G \times H_r$-linearisation
$L^{\prime}_N \to \ol{G/H_u} \times X$ extends the $H$-linearisation $L \to X$
under $\beta \circ \alpha$. In the case where $L \to X$ is ample, we
obtain the following corollary to Proposition \ref{prop:Co1St5} and Corollary
\ref{cor:Co1St5.1}.    

\begin{cor} \label{cor:Co1St7}
Let $H$ be a connected linear algebraic group acting on a normal, irreducible
projective variety $X$ with ample linearisation $L \to X$. Suppose there is a
connected reductive group $G$ with an $H_u$-faithful homomorphism $H \to G$ such
that the unipotent radical $H_u$ of $H$ embeds as a Grosshans subgroup of $G$ and conditions
(C1) and
(C2) hold. Then there is a gentle
$G \times H_r$-equivariant projective completion $\ol{G/H_u}$ of $G/H_u$ (where $H_r = H/H_u$) whose boundary divisor $D_{\infty}$ is ample, effective and Cartier. Furthermore, given any such completion $\ol{G/H_u}$ of $G/H_u$ and associated $D_{\infty}$, for
sufficiently divisible $N>0$ the triple
$(\ol{G/H_u} \times X,\beta,L^{\prime}_N=\OO(ND_{\infty}) \boxtimes \tilde{L})$ of
\eqref{eq:Co1St3} 
is an ample strong reductive envelope for $H \act L \to X$, and 
\[
X \env H \cong (\ol{G/H_u} \times X) \dblslash_{L^{\prime}_N} (G \times H_r).
\]       
\end{cor}

%
%

\begin{rmk} \label{rmk:Co1St8}
One
can use arguments analogous to 
those found in the proof of \cite[Corollary 5.3.19]{dk07} to also show
that the ring of invariants  $\kk[X,L]^H$ is a finitely generated $\kk$-algebra,
and $X \env H \cong (\ol{G/H_u} \times X) \dblslash_{L^{\prime}_N} (G \times
H_r)$ for sufficiently divisible $N>0$. (Note that \cite[Corollary 5.3.19]{dk07}
appears as a corollary to \cite[Theorem 5.3.18]{dk07}, whose proof contains an
error---see \cite[Remark 2.2]{bk15}. However, since \cite[Corollary
5.3.19]{dk07} includes additional hypotheses of ampleness, its validity is
unaffected by this error.)           
\end{rmk}

\begin{rmk} \label{rmk:Co1St9}
Because the various intrinsic notions of conventional reductive GIT and
non-reductive GIT may be defined for rational
linearisations, one can work with rational linearisations in
the setting of reductive envelopes. For example, in Corollary \ref{cor:Co1St7}
if one assumes $H \act L \to X$ and $\tilde{L} \to X$ are rational linearisations
satisfying the natural rational versions of
(C1) and
(C2), then for $N\gg 0$ some positive
integral multiple of the rational
linearisation $L^{\prime}_N=\OO(ND_{\infty}) \boxtimes \tilde{L} \to 
\ol{G/H_u} \times X$ will define a strong reductive envelope for the
corresponding multiple of $L \to X$. The stable locus, finitely generated 
semistable locus and enveloping quotient for $H \act L \to X$ can thus still be
computed using the rational linearisation
$L^{\prime}_N$, which is often more convenient to work with in
computations.
\end{rmk}

Corollary \ref{cor:Co1St7} is a useful result that has the potential
to be applied to the study of a number of interesting examples of non-reductive
group actions. We will make use of it during the extended example in the upcoming
Section \ref{sec:example}.


\section{An Example: $n$ Unordered Points on $\PP^1$}
\label{sec:example}

In this section we undertake a detailed study of an example to demonstrate the
use of strong reductive 
envelopes for computing the semistable and stable loci and constructing
projective completions of the enveloped quotient, in a non-reductive GIT
set-up. 
We will follow this in the final section with an outline of ongoing research applying non-reductive GIT to study moduli spaces occurring naturally in algebraic geometry.

For our detailed study we consider
the space of $n$ unordered points on $\PP^1$ up to compositions of translations
and dilations (that is, under the action of the standard Borel subgroup of
$\SL(2,\kk)$). This extends Example \ref{ex:BgUn16} from \cite[\S 6]{dk07},
which only looked at the translation actions. Including the dilations gives a somewhat 
richer picture, due to the possibility of variation of linearisations and the
associated birational transformations on the quotients (as we shall explore in
Section \ref{subsec:Co1ExVariation}).      

We first fix some notation. Let $G:=\SL(2,\kk)$ and consider its action on
$\PP^1$ via M\"{o}bius transformations:
\[
\begin{pmatrix} a & b \\ c &
    d \end{pmatrix}: \PP^1 \to \PP^1, \quad 
[z_0:z_1] \mapsto [az_0+bz_1:cz_0+dz_1], \quad 
\begin{pmatrix} a & b \\ c &
    d \end{pmatrix} \in G
\]
Fix an integer $n>0$. Then $G$ acts dually on the complete linear system
associated to $\OO_{\PP^1}(n)$, which we identify as $X:=\PP(V)$ with
$V:=\kk[x,y]_n$ the vector space of homogeneous polynomials of degree $n$ in two
indeterminates $x$ and $y$. We write points of $X$ as $[\sigma(x,y)]$ with
$\sigma(x,y) \in V$; note that $[\sigma(x,y)]$ defines an effective divisor of
zeros on 
$\PP^1$, which can be thought of as a collection of $n$ unordered points on
$\PP^1$ with multiplicities. 

We consider the action of the subgroup
\[
H:=\left\{
\begin{array}{c|c} 
\begin{pmatrix} t & a \\ 0 & t^{-1}
\end{pmatrix} \in G & t\in \kk^{\times}, \ a\in \kk 
\end{array}
\right\} 
\] 
on $X$. Geometrically $H$ corresponds to the M\"{o}bius
transformations on $\PP^1$ that are compositions of scalings and translations,
all of which fix $\infty=[1:0] \in \PP^1$. The unipotent
radical of $H$ is
\[
H_u=\left\{
\begin{array}{c|c} 
\begin{pmatrix} 1 & a \\ 0 & 1
\end{pmatrix} & a\in \kk 
\end{array}
\right\} \cong \kk^+
\]     
while the quotient $H_r=H/H_u$ is isomorphic to the torus
\[
T:=\left\{
\begin{array}{c|c} 
\begin{pmatrix} t & 0 \\ 0 & t^{-1}
\end{pmatrix} & t\in \kk^{\times} 
\end{array}
\right\} \cong \kk^{\times}
\] 
via the composition $T \hookrightarrow H \to H_r$. We will typically identify
$H_r$ with $T$ in this way throughout this section.

We next consider the possible linearisations $L$ of $H$ over $X$. As line
bundles, any linearisation is isomorphic to $\OO_{X}(m)$ for $m \in \mb{Z}$
\cite[Corollary 6.17]{har77}. The action of $H$ on $V$ defines, for each $m \in
\mb{Z}$, a canonical 
linearisation $\OO_X(m)^{\mathrm{can}} \to X$ on the line bundle
$\OO_X(m)$. Because $X$ is irreducible and proper over $\kk$, any other
linearisation on $\OO_X(m)$ is 
obtained by twisting the canonical linearisation by a character of $H$
\cite[Corollary 7.1]{dol03}. We take a moment to recall our conventions here: if
$\chi: H \to \kk^{\times}$ is a 
character of $H$, then let $\OO_X^{(\chi)}$ denote the linearisation on the
trivial bundle $\OO_X=X \times \kk$ defined by the character $\chi^{-1}$:
\[
H \times X \times \kk \to X \times \kk, \quad (h,x,z) \mapsto (hx,\chi^{-1}(h)z).
\]
Then any other linearisation on $\OO_X(m)$ is of the form
$\OO_X(m)^{(\chi)}:=\OO_X(m)^{\mathrm{can}} \ten \OO_X^{(\chi)}$. Because unipotent
groups have no nontrivial characters, the inclusion $T 
\hookrightarrow H$ induces an identification between the groups of characters of
$H$ and $T$, so that $\chi$ is of the form
\[
\chi:\begin{pmatrix} t & a \\ 0 & t^{-1} \end{pmatrix} \mapsto t^r,
\quad \begin{pmatrix} t & a \\ 0 & t^{-1} \end{pmatrix} \in H
\]
for some weight $r \in \mb{Z}$. For each $m,r \in \mb{Z}$ we therefore define
linearisations 
\[
L_{m,r}:=\OO_X(m)^{(\chi)}, \quad \text{where $\chi$ has weight $r \in \mb{Z}$}.
\]
of $H$ over $X$. It is these linearisations that we will study in this example.



When $m \leq 0$ the sets $X^{\rms(L_{m,r})}$,
$X^{\nss(L_{m,r})}$, $X^{\ssfg(L_{m,r})}$, $X^{\ssUfg(L_{m,r})}$ and the
enveloping quotient $X \env_{\! \! L_{m,r}} H$ for the
linearisations $L_{m,r}$ can be easily described by inspection. Indeed,
if $m<0$ then $\OO_X(m)$ has no nonzero
global sections as a line bundle, so that the stable and naively semistable
locus, and enveloping 
quotient, are all empty. On the other hand, if $m=0$, then the ring of invariants
$\kk[X,L_{0,r}]^H$ is isomorphic to $\symdot \kk$ if $r=0$,
and $\kk$ (in degree $0$) otherwise. Thus when $r=0$, we have
$X^{\rms(L_{0,0})}=\emptyset$ and
$X^{\ssUfg(L_{0,0})}=X^{\ssfg(L_{0,0})}=X^{\nss(L_{0,0})}=X$, while $X
\env_{\! \! L_{0,0}} H = \text{pt}$; on the other hand, if $r \neq 0$ then all
of $X^{\rms(L_{0,r})}$, 
$X^{\ssUfg(L_{0,r})}$, $X^{\ssfg(L_{0,r})}$, $X^{\nss(L_{0,r})}$ and $X
\env_{\! \! L_{0,r}} H$ are empty.      

In what follows we therefore consider the linearisations $L_{m,r}$ with
$m>0$. In the next 
section we shall use the methods of Corollary \ref{cor:Co1St7} of Section
\ref{subsec:Co1StImportant} to 
construct strong ample reductive envelopes for these linearisations,
which will allow us to compute the stable locus $X^{\rms(L_{m,r})}$, the
 various semistable loci $X^{\ssUfg(L_{m,r})}$, $X^{\ssfg(L_{m,r})}$,
 $X^{\nss(L_{m,r})}$ (which will all be equal) and a projective completion of the
 enveloped quotient; cf. Proposition \ref{prop:Co1St2} and Theorem
 \ref{thm:Co1Re11}.   

\subsection{The Strong Reductive Envelopes}
\label{subsec:Co1ExReductive}

Fix $m,r \in \mb{Z}$, with $m>0$, and let $\chi:H \to \kk^{\times}$ be the
character of $H$ of weight $r$. Consider the inclusion $H \hookrightarrow
G$, which is clearly an $H_u$-faithful homomorphism. By construction, the
restricted linearisation of 
the unipotent radical
$H_u \act L_{m,r} \to X$ extends to a 
linearisation of $G=\SL(2,\kk)$, and we are in the setting of Example
\ref{ex:Co1St0}. Furthermore, as we saw in Example \ref{ex:Co1St6.1}, $H_u$ is a
Grosshans subgroup of $G$, and $G/H_u \cong \kk^2
\setminus \{0\}$ via the defining representation of $G$ on $\kk^2$. This has a
$G \times H_r$-equivariant normal affine 
completion $\ol{G/H_u}^{\aff}:= \kk^2$, containing $G/H_u$ with codimension 2
complement. 
\begin{rmk} \label{rmk:Co1Ex0}
Because the restricted linearisation $H_u \act L_{m,r} \to X$ extends to one of
$G=\SL(2,\kk)$ 
and $H_u$ is a Grosshans subgroup of $G$, by Remark \ref{rmk:Co1St6} we know
that $\kk[X,L_{m,r}]^{H_u}$ and $\kk[X,L_{m,r}]^H$ are both finitely generated
$\kk$-algebras. Therefore $X \env_{\! \! L_{m,r}} H=\proj(\kk[X,L_{m,r}]^H)$ is a
projective variety and $X^{\ssUfg(L_{m,r})}=X^{\ssfg(L_{m,r})}=X^{\nss(L_{m,r})}$.   
\end{rmk}

In what follows we regard elements of $\kk^3$ as column vectors. As in Example
\ref{ex:Co1St6.1}, by adding a hyperplane at infinity
we obtain a normal (in fact,
smooth) $G \times H_r$-equivariant projective completion $\PP^2$ of
$\ol{G/H_u}^{\aff}$. Here we write $\PP^2=\{[v_0:v_1:v_2] \mid 0 \neq
(v_0,v_1,v_2)^{\mathrm{t}} \in 
\kk^3\}$ with the hyperplane at infinity defined by $v_0=0$. The
action of $G \times H_r=G \times T$ on $\PP^2=\PP(\kk^3)$ is the one induced by
the representation given in block form
\[
G \times T \to \GL(3,\kk) , \quad (g,\left(\begin{smallmatrix} t & 0 \\ 0 & t^{-1} 
\end{smallmatrix} \right))
\mapsto \left(\begin{array}{c|c} 1 & 0 \\ \hline 0 &
    g\left(\begin{smallmatrix}t^{-1} & 0 \\ 0 &
        t^{-1} \end{smallmatrix}\right)  \end{array}\right) 
\]
where $\GL(3,\kk)$ acts on $\kk^3$ by left multiplication. For each $N>0$ this 
representation
canonically defines a $G \times H_r$-linearisation $\OO_{\PP^2}(N) \to \PP^2$ which 
restricts
to the canonical linearisation on $\OO_{G/H_u} \to G/H_u$. 

Let $\beta:\GUX \cong (\kk^2 \setminus \{0\}) \times X \hookrightarrow \PP^2
\times X$ be the naturally induced open immersion and for integers $N>0$ let  
\[
L^{\prime}_{m,r,N}:= \OO_{\PP^2}(N) \boxtimes L_{m,r} \to \PP^2 \times X,
\]
equipped with its natural $G \times H_r$-linearisation. By Corollary
\ref{cor:Co1St7}, for $N>0$
sufficiently divisible depending on $m$ and $r$, the triple  
\[
(\PP^2 \times X, \beta,L^{\prime}_{m,r,N})
\]
defines a strong ample reductive envelope for $H \act L_{m,r} \to X$, and 
\[
X \env_{\! \! L_{m,r}} H = (\PP^2 \times X) \dblslash_{L^{\prime}_{m,r,N}} (G \times 
H_r). 
\]
The stable locus $X^{\rms(L_{m,r})}$ and finitely generated 
semistable locus $X^{\ssfg(L_{m,r})}$ for the linearisation $L_{m,r}$ may
therefore be computed as the
completely stable and completely semistable loci, respectively, associated to
the $G \times H_r$-linearisation $L^{\prime}_{m,r,N}$ by Proposition
\ref{prop:Co1St2}. We therefore next compute the semistable and stable loci for
$L^{\prime}_{m,r,N} \to \PP^2 \times X$.

\subsection{Semistability and Stability for the $G \times H_r$-Linearisations
  $L^{\prime}_{m,r,N}$}

In order to compute semistability and stability for the $G \times
H_r$-linearisations $L^{\prime}_{m,r,N} \to \PP^2 \times X$ we will use the
Hilbert-Mumford criteria as stated in Theorem \ref{thm:BgRe4}. To do this we
use the maximal torus $T_1 \times T_2 ,\subseteq G  
\times H_r$ where $T_1$ is the maximal torus $T$ of
$G=\SL(2,\kk)$ and $T_2:= H_r$, also identified with $T$. The group of characters of 
$T_1 \times
T_2$ is then identified with $\mb{Z} \times \mb{Z}$ in the natural way.

The set of fixed points for the action $T_1 \times T_2$-action on $\PP^2
\times X$ is
\[
\{([1:0:0],[x^{n-i}y^{i}]), \, ([0:1:0],[x^{n-i}y^{i}]), \, ([0:0:1],[x^{n-i}y^{i}]) 
\mid
i=0,\dots,n\}. 
\]
Table \ref{tab:Co1Ex1} gives the weights for each fixed point with respect to
the linearisation $L^{\prime}_{m,r,N}$ and a general plot of
these weights is given in Figure
\ref{fig:Co1Ex2}.



\begin{table}[h]
\begin{center}
\begin{tabular}{c|c} Fixed point ($i=0,\dots,n$) & Weight in
$\Hom(T_1 \times T_2,\kk^{\times})=\mb{Z} \times \mb{Z}$ \\ \hline 
$([1:0:0],[x^{n-i}y^{i}])$ & $(m(2i-n),r)$ \\
$([0:1:0],[x^{n-i}y^{i}])$ & $(N+m(2i-n),-N+r)$ \\ 
$([0:0:1],[x^{n-i}y^{i}])$ & $(-N+m(2i-n),-N+r)$
\end{tabular}
\end{center}
\caption{Weights of the fixed points of $T_1 \times T_2 \act \PP^2 \times X$
  with respect to the linearisation $L^{\prime}_{m,r,N}$.}
\label{tab:Co1Ex1}
\end{table}

\begin{center}
\begin{figure}[h]
\begin{tikzpicture}[scale=0.7] \draw[->] (0,-4) -- (0,4)
node[anchor=south]{$\mb{Z} \cong \Hom(T_2,\GGm)$}; \draw (-2pt,2.5) -- (2pt,2.5); 
\foreach \x
in {-0.9,-0.6,...,0.9}{ \fill (\x,2.5) circle (1.5pt); 
}; 
\node at (1.1,2.7) [anchor=south west]{$(m(2i-n),r)$}; 
\node at (8,2.7) [anchor=south west]{$(i=0,\dots,n)$};
\foreach \x in
{4.1,4.4,...,5.9}{ 
\fill (\x,-3.5) circle (1.5pt); 
}; 
\node at (5,-3.7) [anchor=north]{$(N+m(2i-n),-N+r)$};
\foreach \x in {-5.9,-5.6,...,-4.1}{ 
\fill (\x,-3.5) circle (1.5pt); 
}; 
\node at (-5,-3.7)
[anchor=north]{$(-N+m(2i-n),-N+r)$}; \draw (-2pt,-3.5) -- (2pt,-3.5); \draw[->]
(-7,0) -- (7,0) node[anchor=west]{$\mb{Z} \cong \Hom(T_1,\GGm)$};
\end{tikzpicture}
\caption{The weight diagram for $T_1 \times T_2 \act L^{\prime}_{m,r,N} \to \PP^2 
\times
X$.}
\label{fig:Co1Ex2}
\end{figure}
\end{center}

By the Hilbert-Mumford criteria Theorem \ref{thm:BgRe4}, \ref{itm:BgRe4-2}, a
point $p=([v_0:v_1:v_2],[\sigma(x,y)]) \in \PP^2 \times X$ is semistable 
(respectively, stable) for
the restricted linearisation $T_1 \times T_2 \act L^{\prime}_{m,r,N} \to \PP^2
\times X$ if, and only if, the origin of $\mb{R} \ten_{\mb{Z}}(\mb{Z} \times
\mb{Z}) = \mb{R} \times \mb{R}$ is contained in the weight polytope
$\Delta_p \subseteq \mb{R} \times \mb{R}$ (respectively, the interior
$\Delta_p^{\circ}$) associated to $p$. When 
$N$ is taken to be very large with respect to 
$m$ and $r$, we find that $p$ can only be $T_1 \times T_2$-(semi)stable by
satisfying one the following criteria, split into four cases: 
\begin{description}
\item[Case $v_0v_1v_2 \neq 0$:] 
\[ 
    \begin{array}{rl} 
    \multirow{2}{*}{$0 \in \Delta_p \iff$} & \text{$[1:0]$ and $[0:1]$ are both zeros
      of $\sigma(x,y)$} \\
    & \text{with multiplicity $\leq (n+\tfrac{r}{m})/2$} \\
    \\ 
    \multirow{2}{*}{$0 \in \Delta_p^{\circ} \iff$} & \text{$[1:0]$ and $[0:1]$ are 
both zeros
      of $\sigma(x,y)$} \\
    & \text{with multiplicity $< (n+\tfrac{r}{m})/2$}
    \end{array}
\]
\item[Case $v_0v_1 \neq 0$, $v_2=0$:]
\[ 
    \begin{array}{rl} 
    \multirow{2}{*}{$0 \in \Delta_p \iff$} & \text{$[1:0]$ is a zero of
      $\sigma(x,y)$ with multiplicity $\leq (n-\tfrac{r}{m})/2$ and} \\
    & \text{$[0:1]$ is a zero of
      $\sigma(x,y)$ with multiplicity $\leq (n+\tfrac{r}{m})/2$} \\
    \\ 
    \multirow{2}{*}{$0 \in \Delta_p^{\circ} \iff$} & \text{$[1:0]$ is a zero of
      $\sigma(x,y)$ with multiplicity $< (n-\tfrac{r}{m})/2$ and} \\
    & \text{$[0:1]$ is a zero of
      $\sigma(x,y)$ with multiplicity $< (n+\tfrac{r}{m})/2$}
    \end{array}
\]
\item[Case $v_0v_2 \neq 0$, $v_1=0$:]
\[ 
    \begin{array}{rl} 
    \multirow{2}{*}{$0 \in \Delta_p \iff$} & \text{$[1:0]$ is a zero of
      $\sigma(x,y)$ with multiplicity $\leq (n+\tfrac{r}{m})/2$ and} \\
    & \text{$[0:1]$ is a zero of
      $\sigma(x,y)$ with multiplicity $\leq (n-\tfrac{r}{m})/2$} \\
    \\ 
    \multirow{2}{*}{$0 \in \Delta_p^{\circ} \iff$} & \text{$[1:0]$ is a zero of
      $\sigma(x,y)$ with multiplicity $< (n+\tfrac{r}{m})/2$ and} \\
    & \text{$[0:1]$ is a zero of
      $\sigma(x,y)$ with multiplicity $< (n-\tfrac{r}{m})/2$}
    \end{array}
\]
\item[Case $v_0 \neq 0$, $v_1=v_2=0$ and $r=0$:]
\[ 
    \begin{array}{rl} 
    \multirow{2}{*}{$0 \in \Delta_p \iff$} & \text{$[1:0]$ and $[0:1]$ are both zeros
      of $\sigma(x,y)$} \\
    & \text{with multiplicity $\leq n/2$} \\
    \\ 
    \multirow{2}{*}{$0 \in \Delta_p^{\circ} \iff$} & \text{$[1:0]$ and $[0:1]$ are 
both zeros
      of $\sigma(x,y)$} \\
    & \text{with multiplicity $< n/2$}
    \end{array}
\]
\end{description}
By Theorem \ref{thm:BgRe4}, \ref{itm:BgRe4-1}, the point $p$ is
(semi)stable for the whole $G
\times H_r$-linearisation if, and only if, $(g,\ol{h}) \cdot p$ is $T_1 \times
T_2$-(semi)stable for each $(g,\ol{h}) \in G \times H_r$. Using this, one
deduces the following:

\begin{propn} \label{prop:Co1Ex1}
Let $m>0$ and $r$ be integers. Then for sufficiently large $N>0$ (depending on $m$
and $r$) the semistable and stable loci for the 
$G \times H_r$-linearisations $L^{\prime}_{m,r,N} \to \PP^2 \times X$
are as follows:
\begin{description}
\item[(Case $r<0$ or $\tfrac{r}{m}>n$:)] Then $(\PP^2 
\times
  X)^{\rms(L^{\prime}_{m,r,N})} = (\PP^2 \times X)^{\rmss(L^{\prime}_{m,r,N})} =
  \emptyset$.
\item[(Case $r=0$):] Then $(\PP^2 \times
  X)^{\rms(L^{\prime}_{m,r,N})} = \emptyset$ and 
\[
(\PP^2 \times X)^{\rmss(L^{\prime}_{m,r,N})} =
\left\{
\begin{array}{c|c}
\multirow{2}{*}{$([1:v_1:v_2],[\sigma(x,y)])$} & \text{$\sigma(x,y)$
  has no zeros} \\
 & \text{of multiplicity $>\frac{n}{2}$}
\end{array}
\right\}.
\]
\item[(Case $0 < \tfrac{r}{m} < n$):] Then 
{\footnotesize 
\begin{align*} 
\text{\normalsize $(\PP^2 \times X)^{\rms(L^{\prime}_{m,r,N})}$} &= 
\left\{
\begin{array}{c|c}
\multirow{4}{*}{\normalsize $([1:v_1:v_2],[\sigma(x,y)])$} &
(v_1,v_2) \neq (0,0), [v_1:v_2] \text{ is a zero of} \\ 
 & \text{$\sigma(x,y)$ of multiplicity $< (n-\tfrac{r}{m})/2$} \\
 & \text{and all zeros of $\sigma(x,y)$ have} \\
 & \text{multiplicity $<(n+\tfrac{r}{m})/2$}
\end{array}
\right\},
& \\ 
\\
\text{\normalsize $(\PP^2 \times X)^{\rmss(L^{\prime}_{m,r,N})}$} &= 
\left\{
\begin{array}{c|c}
\multirow{4}{*}{\normalsize$([1:v_1:v_2],[\sigma(x,y)])$} &
(v_1,v_2) \neq (0,0), [v_1:v_2] \text{ is a zero of} \\ 
 & \text{$\sigma(x,y)$ of multiplicity $\leq (n-\tfrac{r}{m})/2$} \\
 & \text{and all zeros of $\sigma(x,y)$ have} \\
 & \text{multiplicity $\leq (n+\tfrac{r}{m})/2$}
\end{array}
\right\}.
\end{align*}
} 
\item[(Case $\tfrac{r}{m}=n$):] Then $(\PP^2 \times
  X)^{\rms(L^{\prime}_{m,r,N})} = \emptyset$ and 
\[
(\PP^2 \times X)^{\rmss(L^{\prime}_{m,r,N})} =
\left\{
\begin{array}{c|c}
\multirow{2}{*}{$([1:v_1:v_2],[\sigma(x,y)])$} & \text{$[v_1:v_2]$ is not a} \\ 
 & \text{zero of $\sigma(x,y)$} 
\end{array}
\right\}.
\]
\end{description}
\end{propn}

For fixed $m,r,N$ the completely semistable and completely stable loci are by
Definition \ref{def:Co1Re7} the intersections of the semistable and stable loci
for $G \times H_r \act L^{\prime}_{m,r,N} \to \PP^2 \times X$ under the
inclusion
\[
\beta \circ \alpha:X \hookrightarrow \PP^2 \times X, \quad [\sigma(x,y)]
\mapsto ([1:1:0],[\sigma(x,y)]).
\]
Using Corollary \ref{cor:Co1St7}, we therefore deduce

\begin{cor} \label{cor:Co1Ex2}
For integers $m>0$ and $r$, the semistable and stable loci for the
linearisations $H \act L_{m,r} \to X$ are as follows:
\begin{description}
\item[(Case $r<0$ or $\tfrac{r}{m}>n$):] Then
  $X^{\rms(L_{m,r})} = X^{\ssfg(L_{m,r})} = \emptyset$.
\item[(Case $r=0$):] Then $X^{\rms(L_{m,r})} =
  \emptyset$ and 
\[
X^{\ssfg(L_{m,r})} =
\left\{
[\sigma(x,y)] \in X \mid \text{$\sigma(x,y)$ has no zeros of multiplicity
  $>n/2$} \right\}
\] 
\item[(Case $0 < \tfrac{r}{m} < n$):] Then 
\begin{align*} 
X^{\rms(L_{m,r})} &= 
{
\small
\left\{
\begin{array}{c|c}
\multirow{3}{*}{\normalsize $[\sigma(x,y)] \in X$} & \text{$[1:0]$ is a zero of
  $\sigma(x,y)$ with multiplicity} \\
 & \text{$< (n-\tfrac{r}{m})/2$ and all other zeros} \\
 & \text{of $\sigma(x,y)$ have multiplicity $< (n+\tfrac{r}{m})/2$}
\end{array}
\right\}
}, \\ 
& \\
X^{\ssfg(L_{m,r})} &= 
{
\small
\left\{
\begin{array}{c|c}
\multirow{3}{*}{\normalsize $[\sigma(x,y)] \in X$} & \text{$[1:0]$ is a zero of
  $\sigma(x,y)$ with multiplicity} \\
 & \text{$\leq (n-\tfrac{r}{m})/2$ and all other zeros} \\
 & \text{of $\sigma(x,y)$ have multiplicity $\leq (n+\tfrac{r}{m})/2$}
\end{array}
\right\}
}. 
\end{align*}
\item[(Case $\tfrac{r}{m}=n$):] Then $X^{\rms(L_{m,r})} = 
\emptyset$ and 
\[
X^{\ssfg(L_{m,r})} =\{[\sigma(x,y)]) \in X \mid \text{$[1:0]$ is not a zero of
  $\sigma(x,y)$}\}.  
\]
\end{description}    
\end{cor}

\begin{rmk} \label{UhatHilbertMumford}
Notice that in each case $x \in X^{\rms(L_{m,r})}$ (respectively $x \in X^{\ssfg(L_{m,r})}$) if and only if $x$ is stable (respectively semistable) for every one-parameter subgroup $\lambda: \mathbb{G}_m \to H$, or equivalently if and only if $hx$ is stable (respectively semistable) with respect to the action of the standard maximal torus of $\SL(2,\kk)$ for every $h \in H$. Thus the analogues of the Hilbert--Mumford criteria for reductive GIT hold in these examples for the action of the non-reductive group $H$; they fail, of course, for the action of its unipotent radical $H_u$, since there are no one-parameter subgroups $\lambda: \mathbb{G}_m \to H_u$. 

When $r \neq 0$ then, again just as for reductive GIT but  in contrast to the unipotent case,
the quotient morphism 
$q_{m,r} : X^{\ssfg(L_{m,r})} \to X \env_{\! \! L_{m,r}} H
$
from the semistable locus to the enveloping quotient is surjective, and the projective variety $X \env_{\! \! L_{m,r}} H$ is a categorical quotient of the semistable locus, with $q_{m,r}(x) = q_{m,r} (y)$ for $x,y \in  X^{\ssfg(L_{m,r})}$ if and only if the closures of the $H$-orbits of $x$ and $y$ meet in $X^{\ssfg(L_{m,r})}$.
\end{rmk}

\subsection{Variation of the Enveloping Quotients}
\label{subsec:Co1ExVariation} 

We conclude this example by studying how the enveloping quotients
$X^{\ssfg(L_{m,r})} \to X \env_{\! L_{m,r}} H$ and geometric quotients
$X^{\rms(L_{m,r})} \to X^{\rms(L_{m,r})}/H$ therein change as we range over the
different possible ample linearisations 
$L_{m,r} \to X$. This
can be done by examining the variation of the reductive GIT quotients (in
the sense of the VGIT of \cite{tha96,dh98,res00}) of
the linearisations $L^{\prime}_{m,r,N} \to \PP^2 \times X$. In order to keep the
exposition brief we suppress the details of the VGIT analysis on the reductive
envelopes and instead concentrate on the consequences for the $H$-linearisations,
referencing relevant VGIT results from \cite{tha96}.  
 
From now on we assume $N$ is sufficiently divisible with respect to $m$ and $r$ so
as to satisfy the conclusions of 
  Corollary \ref{cor:Co1St7} and Proposition \ref{prop:Co1Ex1}. Note that two
  linearsations $L_{m,r}$ and $L_{m',r'}$ will have the same stable locus,
  finitely generated semistable locus and enveloping quotient if
  $\frac{r}{m} = \frac{r'}{m'}$ (see Remarks \ref{rmk:TrRe3.2} and
  \ref{rmk:GiSu3}). By inspecting Corollary
\ref{cor:Co1Ex2} we see that the changes in stability and finitely generated
semistability occur when $\frac{r}{m} = 0$, or $n-\frac{r}{m} \in 2\mb{Z}$, or
$\frac{r}{m}=n$, and clearly we only need consider the cases where $\frac{r}{m} \in
\mb{Q} \cap [0,n]$. It makes 
sense therefore to consider four cases: (1) when $r=0$; (2) when $0 <
\frac{r}{m} < n$ and 
$n-\frac{r}{m} \in \mb{Q} \setminus 2\mb{Z}$; (3) when $0 < \frac{r}{m} < n$ and
$n-\frac{r}{m} \in 2\mb{Z}$; and (4) when $\frac{r}{m} = n$. 

\subsubsection{Case $r=0$}

Observe from Corollary \ref{cor:Co1Ex2} that 
  $X^{\ssfg(L_{m,0})}$ is precisely the semistable locus for the canonical
  $G=\SL(2,\kk)$-linearisation on $\OO_X(1) \to X=\PP(V)$ \cite[\S
  10.2]{dol03} (which is the
  classical reductive GIT problem of configurations 
  of $n$ unordered points on $\PP^1$ up to M\"{o}bius transformations). Notice
  also that there is a $G$-equivariant retraction  
\[
(\PP^2 \times X)^{\rmss(L^{\prime}_{m,0,N})} = \kk^2 \times X^{\ssfg(L_{m,0})} \to 
\{0\}
  \times X^{\ssfg(L_{m,0})} \cong  X^{\ssfg(L_{m,0})}
\]
(where $\kk^2 \subseteq \PP^2$ corresponds to the gentle affine completion
$\ol{G/H_u}^{\aff}$ of $G/H_u$), defined by taking a limit along the flow of $t
\in T_2 \cong \kk^{\times}$ as $t \to \infty$. Thus two points $[\sigma(x,y)],
[\tilde{\sigma}(x,y)]$ from $X^{\ssfg(L_{m,0})}$ get identified in $X \env_{\!
  \! L_{m,0}} H=(\PP^2 \times X) 
\dblslash_{L^{\prime}_{m,0,N}}(G \times H_r)$ if, and only if,
  $[\sigma(x,y)]$ and $[\tilde{\sigma}(x,y)]$ are S-equivalent for the
  standard action of $G$ on $X$ (see Section \ref{sec:BgReductive}). Writing $X
  \dblslash G$ for the GIT quotient of 
  the canonical linearisation $G \act \OO_X(1) \to X$, it follows that the
  inclusion $X^{\ssfg(L_{m,0})} \hookrightarrow 
  (\PP^2 \times X)^{\rmss(L^{\prime}_{m,0,N})}$ induces an isomorphism       
\[ 
X \env_{\! \! L_{m,0}} H \cong X \dblslash G.
\] 
In particular, we see that the dimension of $X \env_{\! L_{m,0}} H$
is one less than the anticipated dimension. 

\subsubsection{Case $0 < \frac{r}{m} < n$ and $n-\frac{r}{m} \in \mb{Q} \setminus 
2\mb{Z}$}
\label{subsubsec:ExVaCase2}

In this case we see from Proposition \ref{prop:Co1Ex1} and Corollary
\ref{cor:Co1Ex2} that  $(\PP^2
\times X)^{\rms(L^{\prime}_{m,r,N})}=(\PP^2 \times
X)^{\rmss(L^{\prime}_{m,r,N})}$ and
$X^{\ssfg(L_{m,r})}=X^{\rms(L_{m,r})}$. Moreover, by inspection
\[
(\PP^2 \times X)^{\rmss(L^{\prime}_{m,r,N})} \subseteq (\kk^2 \setminus\{0\})
\times X \cong \GUX,
\] 
thus $(\PP^2 \times X)^{\rmss(L^{\prime}_{m,r,N})} \cong G \times^{H_u}
(X^{\ssfg(L_{m,r})})$ and $(\PP^2 \times X)^{\rms(L^{\prime}_{m,r,N})} \cong G
\times^{H_u} (X^{\rms(L_{m,r})})$. From Corollary \ref{cor:Co1St7} we thus have
\[
X \env_{\! \! L_{m,r}} H=(\PP^2 \times X) \dblslash_{L^{\prime}_{m,r,N}} (G
\times H_r) = (\PP^2 \times X)^{\rms(L^{\prime}_{m,r,N})}/(G \times 
H_r)=X^{\rms(L_{m,r})}/H. 
\]
In particular,
the enveloping quotient map $X^{\ssfg(L_{m,r})} \to X \env_{\! \! L_{m,r}} H$ is a
geometric quotient of $X^{\ssfg(L_{m,r})}$ and the enveloped quotient is equal
to the enveloping quotient, 
which itself is the canonical choice of inner enveloping quotient (Definition
\ref{def:GiFi3.1}). Indeed, the quotient of $X^{\ssfg(L_{m,r})}$ by $H$ is even
projective. So in this case we obtain the best possible geometric picture we
could hope for.  

\subsubsection{Case $0 < \frac{r}{m} < n$ and $n-\frac{r}{m} \in 2\mb{Z}$}

Now $\frac{r}{m}$ lies on a `wall' (in the sense of Thaddeus \cite[Theorem
2.3]{tha96}) and $X^{\rms(L_{m,r})}$ is a proper
subset of $X^{\ssfg(L_{m,r})}$. As in the above case $n-\frac{r}{m}
\notin 2\mb{Z}$ we still have 
\[
(\PP^2 \times X)^{\rmss(L^{\prime}_{m,r,N})} \cong G \times^{H_u}
(X^{\ssfg(L_{m,r})}), \quad (\PP^2 \times X)^{\rms(L^{\prime}_{m,r,N})} \cong G
\times^{H_u} (X^{\rms(L_{m,r})})
\]
and $X \env_{\! \! L_{m,r}} H=(\PP^2 \times X) \dblslash_{L^{\prime}_{m,r,N}} (G
\times H_r)$. In particular, $X^{\rms(L_{m,r})}/H \cong (\PP^2 \times
X)^{\rms(L^{\prime}_{m,r,N})}/(G \times H_r)$ and the enveloping quotient map
$q:X^{\ssfg(L_{m,r})} 
\to X \env_{\! \! L_{m,r}} H$ is surjective by Theorem \ref{thm:Co1Re11} (which
means that again the notions of enveloped quotient, inner enveloping quotient and
enveloping quotient all coincide in this case). We next compute the complement of
$X^{\rms(L_{m,r})}/H$ inside $X \env_{\! L_{m,r}} H$. Let
$[\sigma(x,y)] \in X^{\ssfg(L_{m,r})} \setminus X^{\rms(L_{m,r})}$. We claim
that 
\[
[x^{(n+\frac{r}{m})/2}y^{(n-\frac{r}{m})/2}] \in \ol{H \cdot [\sigma(x,y)]} \cap
X^{\ssfg(L_{m,r})} \quad \text{(closure taken in $X$)}. 
\]
Indeed, by inspection of Corollary \ref{cor:Co1Ex2} either 
(1) $\sigma(x,y)$ has 
$[1:0]$ as a root of multiplicity $(n-\tfrac{r}{m})/2$; or (2) $\sigma(x,y)$ has
a root 
$[u_1:u_2] \neq [1:0]$ of multiplicity $(n+\tfrac{r}{m})/2$. In the first case,
the limit of any zero of $\sigma(x,y)$ different to $[1:0]$ under the flow of $
\left(\begin{smallmatrix} t_1 & 0 \\ 0 &
  t_1^{-1} \end{smallmatrix} \right) \in T_1 \subseteq G$ as $t_1 \to 0$ is
equal to $[0:1]$, so $[x^{(n+\frac{r}{m})/2}y^{(n-\frac{r}{m})/2}]
\in \ol{H \cdot [\sigma(x,y)]} \cap X^{\ssfg(L_{m,r})}$. In the
second case, there is $h \in H_u \cong \kk^+$ taking $[u_1:u_2]$ to
 $[0:1]$, and any other zero of $\sigma(x,y)$ is taken to a point of the form
$[v_1:v_2]$ with $v_1 \neq 0$. Any such $[v_1:v_2]$ flows to $[1:0]$ under
$\left(\begin{smallmatrix} t_1 & 0 \\ 0 & t_1^{-1} \end{smallmatrix} \right)$
as $t_1 \to \infty$, so that $\ol{H \cdot [\sigma(x,y)]} \cap
X^{\ssfg(L_{m,r})}$ also contains 
$[x^{(n+\frac{r}{m})/2}y^{(n-\frac{r}{m})/2}]$. This proves our claim. 

It follows that any
two points of $X^{\ssfg(L_{m,r})} \setminus X^{\rms(L_{m,r})}$ 
are S-equivalent inside $(\PP^2 \times X)^{\rmss(L^{\prime}_{m,r,N})}$ and so
\[ 
X \env_{\! \! L_{m,r}} H = (X^{\rms(L_{m,r})} / H) \amalg \pt,
\]
where $\pt$ is the image of 
$X^{\ssfg(L_{m,r})} \setminus X^{\rms(L_{m,r})}$ under the enveloping
quotient $q:X^{\ssfg(L_{m,r})} \to X \env_{\! \! L_{m,r}} H$. Note that multiple
orbits get collapsed to $\pt$, so the enveloping quotient fails to be a
geometric quotient.

\subsubsection{Case $\tfrac{r}{m}=n$}

In this case $X^{\rms(L_{m,r})} = \emptyset$, while for $X^{\ssfg(L_{m,r})}$ a
similar analysis to the above case $0 < \frac{r}{m} < n$ and $n-\frac{r}{m} \in
2\mb{Z}$ holds: again we have $(\PP^2 \times X)^{\rmss(L^{\prime}_{m,r,N})}
\cong G \times^{H_u}(X^{\ssfg(L_{m,r})})$, 
and any $[\sigma(x,y)]\in X^{\rmss(L_{m,r})}$ has limit point
equal to $[x^n]$ inside $X^{\rmss(L_{m,r})}$ under the action of $T_1$. So we
see that any two points in 
$X^{\ssfg(L_{m,r})}$ are S-equivalent inside $(\PP^2 \times
X)^{\rmss(L^{\prime}_{m,r,N})}$. Since $X \env_{\! \! L_{m,r}} H = (\PP^2 \times
X) \dblslash_{L^{\prime}_{m,r,N}} (G \times H_r)$, we deduce that 
\[
X \env_{\! \! L_{m,r}} H =\pt.
\]

\subsubsection{Birational Transformations of $X^{\rms(L_{m,r})}/H$}  

Finally, we examine how the geometric
quotients $X^{\rms(L_{m,r})}/H$ transform birationally as $\tfrac{r}{m}$ crosses the 
`walls' of integers
congruent to $n$ modulo 2 between $0$ and $n$, or else equal to $0$ or $n$. 


First consider the case where $\tfrac{r}{m} \in \mb{Z} \cap (0,n)$ and
$\frac{r}{m} \equiv n \mod 2$, with $n \geq 3$. Let
$0<\epsilon<1$ be a small rational 
number and let $L_{m,r}^+ \to X$ and $L_{m,r}^-
\to X$ be the perturbed $H$-linearisations, corresponding to the rational numbers
$\frac{r}{m}+ \epsilon$ and $\frac{r}{m}- \epsilon$, respectively. By inspecting
Corollary \ref{cor:Co1Ex2} we see there are inclusions 
\[
X^{\rms(L_{m,r}^-)} \subseteq X^{\ssfg(L_{m,r})} \supseteq
X^{\rms(L_{m,r}^+)}, \quad
X^{\rms(L_{m,r}^-)} \supseteq X^{\rms(L_{m,r})} \subseteq
X^{\rms(L_{m,r}^+)},
\]    
from which we obtain proper birational morphisms 
\begin{align*}
\psi_-:X \env_{\! \! L_{m,r}^-} H &\to X \env_{\! \! L_{m,r}} H, \\
\psi_+:X \env_{\! \! L_{m,r}^+} H &\to X \env_{\! \! L_{m,r}} H
\end{align*}
fitting into the following
commutative diagram (with all unmarked 
inclusions natural open immersions):
\[
\begin{diagram}
X^{\rms(L_{m,r}^-)} & \rEmpty~\subseteq & X^{\ssfg(L_{m,r})}
& \rEmpty~\supseteq &
X^{\rms(L_{m,r}^+)} \\
\dTo && \dTo_q && \dTo \\
X \env_{\! \! L_{m,r}^-} H & \rOnto^{\psi_-} & X \env_{\! \!
  L_{m,r}} H & \lOnto^{\psi_+} & 
X \env_{\! \! L_{m,r}^+} H \\ 
& \luInto & \uInto & \ruInto & \\
 & & X^{\rms(L_{m,r})}/H & & \\   
\end{diagram}
\]
(Cf. \cite[Theorem 3.3]{tha96}.) If $n=3$ then in fact 
\[
X \env_{\! \! L_{m,r}^-} H \overset{\psi_-}{\cong} X \env_{\! \! L_{m,r}} H
\overset{\psi_+}{\cong} X \env_{\! \! L_{m,r}^+} H
\]
are all isomorphic to $\PP^1$.\footnote{To see this, note that $X \env_{ \!
    L_{m,r}} H = (\PP^2 \times X) \dblslash_{L^{\prime}_{m,r,N}}(G \times H_r)$
  is a one dimensional reductive GIT quotient of a smooth variety, hence is isomorphic 
to $\PP^1$ \cite{kem80}; similarly for the linearisations $L_{m,r}^{\pm}$. Alternatively observe that 
$X \env_{ \!
    L_{m,r}} H$ is a normal projective unirational curve, and hence is isomorphic to $\PP^1$. } Otherwise 
$\psi_-$ and $\psi_+$ are
both small contractions, and the induced birational morphism 
\[
X \env_{\! \! L^-_{m,r}} H \dashrightarrow X \env_{\! \! L^+_{m,r}} H
\]
is a blow-down of
$E_-:=\psi_-^{-1}(\pt)$ followed by a blow-up of
$E_+:=\psi_+^{-1}(\pt)$, where here $\pt= (X\env_{\! \! L_{m,r}}H) \setminus
(X^{\rms(L_{m,r})}/H)$ \cite[Theorem 3.5]{tha96}.



In the case where $n \geq 4$ we claim that $E_+$ and $E_-$ are
isomorphic to the following weighted projective spaces\footnote{The fact that
  $E_+$ and $E_-$ 
  are weighted projective spaces also follows from \cite[Theorem 5.6]{tha96}.}:
$E_+ \cong \PP(1,2,\dots,s)$ and $E_- \cong
\PP(1,2,\dots,n-s)$, where $s=(n-\frac{r}{m})/2$. Indeed, recall that 
\[
E_-=(X^{\rms(L_{m,r}^-)} \setminus X^{\rms(L_{m,r})})/H, \quad
E_+=(X^{\rms(L_{m,r}^+)} \setminus X^{\rms(L_{m,r})})/H. 
\] 
In the case of $E_+$,
any $H_u \cong \kk^+$-orbit 
in $X^{\rms(L_{m,r}^+)} \setminus X^{\rms(L_{m,r})}$ contains a unique point
$[\sigma(x,y)]$ such that $[0:1]$ is a zero of 
$\sigma(x,y)$ of multiplicity $n-s$ and $[1:0]$ is a zero of
multiplicity $0 \leq l<s$. Thus the locally closed subset
\[
Z_+=\{[a_0x^n+a_1x^{n-1}y+\dots +a_sx^{n-s}y^s] \in X \mid \text{$a_s \neq 0$ and
  $a_i \neq 0$ for some $0 \leq i <s$}\} \\
\]    
of $X$ provides a slice to the $H_u$-action on $X^{\rms(L_{m,r}^+)} \setminus
X^{\rms(L_{m,r})}$ which is stable under the 
$T_1$-action. Now $Z_+ \cong \kk^s \setminus \{0\}$ in the obvious way
and one can check that the corresponding $T_1$-action on $\kk^s \setminus \{0\}$
has weights $-2s,-2s-2,\dots,-2$. It follows that $E_+=Z_+/T_1$ is isomorphic to
$\PP(1,2,\dots,s)$. The proof that $E_- \cong \PP(1,2,\dots,n-s)$ is similar.

In the case where $r=0$, we of course have $X^{\ssfg(L_{m,0}^-)}=X^{\rms(L_{m,0}^-)}=
\emptyset$ and
$X^{\rms(L_{m,0}^+)} \subseteq X^{\ssfg(L_{m,0})}=X^{\rmss(G)}$, where recall
$X^{\rmss(G)}$ is the semistable locus for the canonical linearisation $G \act
\OO_X(1) \to X$. In the case where $n$ is even, the boundary
of the stable locus $X^{\rms(G)}$ inside $X^{\rmss(G)}$ is precisely the
  $G$-orbit of the point $[x^{n/2}y^{n/2}] \in X$ \cite[\S 10.2]{dol03}, so by
  inspection 
  $X^{\rms(L_{m,0}^+)} \subseteq X^{\rms(G)}$. Thus there is a commuting diagram
\[
\begin{diagram}
X^{\rms(L_{m,0}^+)} & \rEmpty~\subseteq & X^{\rms(G)} &\rEmpty~\subseteq & X^{\ssfg(L_{m,0})} & \rEmpty~= &
X^{\rmss(G)} \\
\dTo_{q} && \dTo && \dTo_{q} && \dTo_{\text{S-equivalence}} \\
X^{\rms(L_{m,0}^+)}/H & \rTo^{\psi} & X^{\rms(G)}/G &
\rEmpty~\subseteq & X \env_{\! \! L_{m,0}} H & 
\rEmpty~= & X \dblslash G  \\
\end{diagram}
\]
where $\psi:X^{\rms(L_{m,0}^+)}/H \to X^{\rms(G)}/G$ is a fibration. (Indeed, it
is a geometric quotient for the unipotent subgroup $(H_u)^{\mathrm{opp}}$ of
strictly lower triangular matrices in $G$
opposite to $H_u$.) When $n$ is odd, a similar diagram holds, except now
$X^{\rms(G)}=X^{\rmss(G)}$ and $X^{\rms(G)}/G=X \dblslash G$.

Lastly, the case where $\tfrac{r}{m}=n$ is trivial: we now have
$X^{\ssfg(L_{m,r}^+)}=X^{\rms(L_{m,r}^+)}=\emptyset$ and
$X^{\rms(L_{m,r}^-)} \subseteq X^{\ssfg(L_{m,r})}$ induces the unique map
$X \env_{\! \! L_{m,r}^-} H \to 
X \env_{\! \! L_{m,r}} H=\pt$.


\section{Applications}
\label{sec:applications}
In the example studied in the previous section, where $H$ is the standard Borel subgroup of $\SL(2,\kk)$ and $X = \PP(V)$ with $V$ an irreducible representation of $\SL(2,\kk)$, the quotient morphism 
$$q_{m,r} : X^{\ssfg(L_{m,r})} \to X \env_{\! \! L_{m,r}} H
$$
from the semistable locus to the enveloping quotient is surjective whenever $r \neq 0$, even though the corresponding morphism when $H$ is replaced with its unipotent radical $H_u$ is not surjective (see Remark \ref{UhatHilbertMumford}). Furthermore   the analogues of the Hilbert--Mumford criteria for reductive GIT hold in these examples for the action of the non-reductive group $H$; that is, $x \in X^{\rms(L_{m,r})}$ (respectively $x \in X^{\ssfg(L_{m,r})}$) if and only if $x$ is stable (respectively semistable) for every one-parameter subgroup $\lambda: \mathbb{G}_m \to H$. In addition when $r \neq 0$ the enveloping quotient is a categorical quotient of the semistable locus, with $q_{m,r}(x) = q_{m,r} (y)$ for $x,y \in  X^{\ssfg(L_{m,r})}$ if and only if the closures of the $H$-orbits of $x$ and $y$ meet in $X^{\ssfg(L_{m,r})}$, just as for reductive GIT.  Indeed for generic choice of $r/m \in \mathbb{Q} \setminus\{0\}$ the enveloping quotient is a projective variety which is a geometric quotient of $X^{\ssfg(L_{m,r})} = X^{\rms(L_{m,r})}$. In this final section we will describe without proof some ongoing research which generalises these observations and has applications to moduli spaces occurring naturally in algebraic geometry.

\subsection{Graded unipotent group actions}
\label{subsec:graded unipotent group actions}

In \cite{bk15, bedhk16a, bedhk16b} the situation is studied when the unipotent radical $H_u$ of a linear algebraic group $H$ has  a semi-direct product $\hat{{H_u}} = {H_u} \rtimes \GG_m$ by the multiplicative group $\GG_m$ of $\kk$
such that the weights of the action of $\GG_m$ on the Lie algebra of $H_u$ are all  strictly positive;  such a unipotent group is called \emph{graded unipotent}. 
Given any action of ${\hat{H_u}}$ on a projective variety $X$ which is linear with respect to an ample line bundle $L$ on $X$, it is shown in \cite{bedhk16a, bedhk16b} that {provided} two conditions are satisfied:

(i) that we are willing to replace $L$ with a suitable tensor power and to twist the linearisation of the action of ${\hat{H_u}}$ by a suitable (rational) character of ${\hat{H_u}}$, and 

(ii)  roughly speaking, that \lq semistability coincides with stability' for the action of $\hat{H_u}$,

\noindent  then
 the ${\hat{H_u}}$-invariants form a finitely generated algebra. Moreover in this situation the natural quotient morphism $q$ from the semistable locus $X^{\ssfg,{\hat{H_u}}}$ to the enveloping quotient $X\env {\hat{H_u}}$ is surjective, and indeed expresses the projective variety $X\env {\hat{H_u}}$ as a geometric quotient of $X^{\ssfg,{\hat{H_u}}}$, and this locus $X^{\ssfg,{\hat{H_u}}} = X^{\rms,{\hat{H_u}}}$  can be described using Hilbert--Mumford-like criteria.

Suppose that $H$ acts linearly on a projective variety $X$, and that $H_r = H/H_u$ itself contains a central one-parameter subgroup whose conjugation action on the Lie algebra of $H_u$ has all weights strictly positive. Then  the corresponding semi-direct product $\hat{H_u}$ is a subgroup of $H$, and provided that the condition (ii) that  \lq semistability coincides with stability' for the action of $\hat{H_u}$ is satisfied, $X$ can be quotiented first by $\hat{H_u}$ for a suitably twisted linearisation as above,  and then by the induced action of the reductive group ${H}/\hat{H_u} \cong H_r/\GG_m$, to obtain   a projective variety $X \env {H}$ which is a categorical quotient by ${H}$ of $X^{\ssfg,{{H}}}$.
More generally suppose that  the linear action of $H$ on $X$  extends to a linear action of a semi-direct product $\hat{H}$ of $H$ by $\GG_m$ acting by conjugation on the Lie algebra of $H_u$ with all weights strictly positive, and whose induced conjugation action on $H_r=H/H_u$ is trivial. Then if the condition (ii)  is satisfied, we can  quotient first by $\hat{H_u}$ for a suitably twisted linearisation as above and then by the induced action of the reductive group $\hat{H}/\hat{H_u}$, to obtain   a projective variety $X \env \hat{H}$ which is a categorical quotient by $\hat{H}$ of the $\hat{H}$-invariant open subset $X^{\ssfg,{\hat{H}}}$ of $X$.

In order to describe the condition (ii),  that  \lq semistability coincides with stability' for the action of $\hat{H_u}$, more precisely,
let $L \to X$ be a very ample linearisation
of  the action of $\hat{H}$ on an irreducible projective variety $X$. 
Let $\chi: \hat{H} \to \GG_m$ be a character of $\hat{H}$ with kernel containing $H$;  such characters $\chi $ can be identified with integers so that the integer 1 corresponds to the character which fits into the exact sequence $H \to \hat{H} \to \GG_m$. Let $\omega_{\min}$ be the minimal weight for the $\GG_m$-action on
$V:=H^0(X,L)^*$ and let $V_{\min}$ be the weight space of weight $\omega_{\min}$ in
$V$. Suppose that $\omega_{\min} < \omega_{\min +1} < 
\cdots < \omega_{\max} $ are the weights with which the one-parameter subgroup $\GG_m \leq {\hat{H_u}} \leq \hat{H}$ acts on the fibres of the tautological line bundle $\mathcal{O}_{\PP((H^0(X,L)^*)}(-1)$ over points of the connected components of the fixed point set $\PP((H^0(X,L)^*)^{\GG_m}$ for the action of $\GG_m$ on $\PP((H^0(X,L)^*)$; since $L$ is very ample $X$ embeds in $\PP((H^0(X,L)^*)$ and the line bundle $L$ extends to the dual $\mathcal{O}_{\PP((H^0(X,L)^*)}(1)$ of the tautological line bundle $\mathcal{O}_{\PP((H^0(X,L)^*)}(-1)$.
Without loss of generality we may assume that there exist at least two distinct such weights, since otherwise the action of the unipotent radical $H_u$ of $H$ on $X$ is trivial, and so the action of $H$ is via an action of the reductive group $H_r=H/H_u$ and reductive GIT can be applied.
Let $c$ be a positive integer such that 
$$ 
\frac{\chi}{c} = \omega_{\min} + \epsilon$$
where $\epsilon >0$ is sufficiently small; we will call rational characters $\chi/c$  with this property {\emph well adapted } to the linear action of $\hat{H}$, and we will call the linearisation well adapted if $\omega_{\min} <0\leq \omega_{\min} + \epsilon$ for sufficiently small $\epsilon >0$.
 The linearisation of the action of $\hat{H}$ on $X$ with respect to the ample line bundle $L^{\otimes c}$ can be twisted by the character $\chi$ so that the weights $\omega_j$ are replaced with $\omega_jc-\chi$;
let $L_\chi^{\otimes c}$ denote this twisted linearisation. 
Let $X^{s,\GG_m}_{\min+}$ denote the stable subset of $X$ for the linear action of $\GG_m$ with respect to the linearisation $L_\chi^{\otimes c}$; by the theory of variation of (classical) GIT \cite{dh98,tha96}, if $L$ is very ample then  $X^{s,\GG_m}_{\min+}$ is the stable set for the action of $\GG_m$ with respect to any rational character $\chi/c$ such that 
$\omega_{\min} < \chi/c < \omega_{\min + 1}$.
Let $$X^{s,{\hat{H_u}}}_{\min+} = X \setminus {\hat{H_u}} (X \setminus X^{s,\GG_m}_{\min+}) = \bigcap_{u \in H_u} u X^{s,\GG_m}_{\min+}$$ be the complement of the ${\hat{H_u}}$-sweep (or equivalently the $H_u$-sweep) of the complement of $X^{s,\GG_m}_{\min+}$,
let
\[
Z_{\min}:=X \cap \PP(V_{\min})=\left\{
\begin{array}{c|c}
\multirow{2}{*}{$x \in X$} & \text{$x$ is a $\GG_m$-fixed point and} \\ 
 & \text{$\GG_m$ acts on $L^*|_x$ with weight $\omega_{\min}$} 
\end{array}
\right\}
\]
and
\[
X^0_{\min}:=\{x\in X \mid \lim_{t \to 0, \,\, t \in \GG_m} t \cdot x \in Z_{\min}\}.
\]   
Then $X^0_{\min}$ is $\hat{H_u}$-invariant and $X^{s,{\hat{H_u}}}_{\min+}  = X^0_{\min} \setminus H_u Z_{\min}$.

The condition that \lq semistability coincides with stability' for the linear action of ${\hat{H_u}}$ required in \cite{bedhk16a} is slightly stronger than that required in \cite{bedhk16b}, where the hypothesis needed for the
$\hat{H_u}$-linearisaton $L \to X$ is that 

\begin{equation}  
\text{$\stab_{H_u}(z)) = \{ e \} $ for every $z \in Z_{\min}$} \tag{$\mf{C}^*$}
\end{equation}
(note that this condition is satisfied in the examples studied in $\S$5) 
and the following result is proved.

\begin{thm} \label{mainthm}  {\rm \cite{bedhk16b}}
Let $H$ be a linear algebraic group over $\kk$ with unipotent radical $H_u$.
 Let $\hat{H} = H \rtimes \GG_m$ be a semidirect product of $H$ by $\GG_m$ with subgroup $\hat{H_u} = H_u \rtimes \GG_m$,
where the conjugation action of $\GG_m$ on $H_u$ is such that all the weights
of the induced $\GG_m$-action on the Lie algebra of $H_u$ are strictly positive, while the induced conjugation action of $\GG_m$ on $H_r=H/H_u$ is trivial.
Suppose that ${\hat{H}}$  acts linearly on an irreducible projective variety $X$ with respect to an  ample line bundle $L$, and that $c$ is a sufficiently divisible positive integer and $\chi: {\hat{H}} \to \GG_m$ is a character of $\hat{H}$ with kernel containing $H$  such that the rational character $\chi/c$ is well adapted for the linear action of ${\hat{H_u}}$.
Suppose also that the linear action of ${\hat{H_u}}$ on $X$ satisfies the condition  ($\mf{C}^*$) above. 
Then the algebras of invariants $\oplus_{m=0}^\infty H^0(X,L_{m\chi}^{\otimes cm})^{\hat{H_u}}$ and $$\oplus_{m=0}^\infty H^0(X,L_{m\chi}^{\otimes cm})^{\hat{H}} = (\oplus_{m=0}^\infty H^0(X,L_{m\chi}^{\otimes cm})^{\hat{H_u}})^{H_r}$$ are 
finitely generated.   
Moreover the enveloping quotient $X\env \hat{H_u}$ is the projective variety associated to the algebra of invariants $\oplus_{m=0}^\infty H^0(X,L_{m\chi}^{\otimes cm})^{\hat{H_u}}$ and is a geometric quotient of the open subset $X^{s,{\hat{H_u}}}_{\min+}$ of $X$ by $\hat{H_u}$, while the enveloping quotient $X\env \hat{H}$ is the projective variety associated to the algebra of invariants $\oplus_{m=0}^\infty H^0(X,L_{m\chi}^{\otimes cm})^{\hat{H}}$ and is the reductive GIT quotient of $X\env \hat{H_u}$ by the induced action of the reductive group ${\hat{H}}/\hat{H_u} \cong H_r$ with respect to the linearisation induced by a sufficiently divisible tensor power of $L$. 
\end{thm}

Applying this result with $X$ replaced by $X \times \PP^1$, with respect to the tensor power of the linearisation $L$ (over $X$) with $\mathcal{O}_{\PP^1}(M)$ (over $\PP^1$) for $M>>1$, gives us a  projective variety $(X \times \PP^1) \env \hat{H}$ which is a categorical quotient by $\hat{H}$ of an $\hat{H}$-invariant open subset of $X \times \kk$. This open subset is the inverse image in $(X \times \PP^1)^{s,{\hat{H_u}}}_{\min+}$ of the $H_r$-semistable subset $((X\times \PP^1) \env \hat{H_u})^{ss,H_r}$ of $(X\times \PP^1) \env \hat{H_u} = (X \times \PP^1)^{s,{\hat{H_u}}}_{\min+}/{\hat{H_u}}$,
and contains as an open subset a geometric quotient by $H$ of an $H$-invariant open subset $X^{\hat{s},H}$ of $X$. Here $X^{\hat{s},H}$ can be identified in the obvious way with $X^{\hat{s},H}  \times \{ [1:1]\} $ which is the intersection with $X \times \{ [1:1]\}$ of the inverse image in $(X \times \PP^1)^{\rms,{\hat{H_u}}}_{\min+} = (X \times \PP^1)^{\ssfg,{\hat{H_u}}}_{\min+}$ of the $H_r$-stable subset $((X\times \PP^1) \env \hat{H_u})^{s,H_r}$ of 
$$(X\times \PP^1) \env \hat{H_u} =  ((X^0_{\min} \times \kk^*) \sqcup (X^{s,{\hat{H_u}}}_{\min+}  \times \{0\}))/{\hat{H_u}} \cong 
(X^0_{\min}/H_u) \sqcup (X^{s,{\hat{H_u}}}_{\min+}  /{\hat{H_u}}).$$
Furthermore the geometric quotient $X^{\hat{s},H}/H$ and its projective completion $(X \times \PP^1) \env \hat{H}$ can be described using Hilbert--Mumford-like criteria, by combining the description of 
$(X\times \PP^1) \env \hat{H_u}$ as the geometric quotient $ (X \times \PP^1)^{s,{\hat{H_u}}}_{\min+}/{\hat{H_u}}$ with reductive GIT for the induced linear action of the reductive group $H_r=H/H_u$ on $(X\times \PP^1) \env \hat{U}$.

In  \cite{bedhk16b} it is also shown that when the condition ($\mf{C}^{*}$) is not satisfied, but is replaced with the much weaker condition
\begin{equation}  
\text{$\min_{x \in X} \, \dim(\stab_{H_u}(x)) = 0,$} \tag{$\mf{C}^{**}$}
\end{equation}
then
there is a  sequence of blow-ups of $X$ along $\hat{H}$-invariant subvarieties  (analogous to that of \cite{kir85} in the reductive case) resulting in a projective variety $\hat{X}$ with an induced linear action of $\hat{H}$ satisfying the  condition ($\mf{C}^{*}$). In this way we obtain a projective variety $\widehat{X \times \PP^1} \env \hat{H}$ which is a categorical quotient by $\hat{H}$ of a $\hat{H}$-invariant open subset of a blow-up of ${X} \times \kk$ and contains as an open subset a geometric quotient of an $H$-invariant open subset $X^{\hat{s},H}$ of $X$ by $H$, where the geometric quotient $X^{\hat{s},H}/H$ and its projective completion $\widehat{X \times \PP^1} \env \hat{H}$ have descriptions in terms of Hilbert--Mumford-like criteria and  the explicit blow-up construction. 

\begin{rmk}
It is observed in \cite{bedhk16b} that, at least when $H_u$ is abelian, most of these conclusions hold even when the condition ($\mf{C}^{**}$) is not satisfied.
The case when \\ $\min_{x \in X} \,\, \dim(\stab_{H_u}(x))>0$ is studied in \cite{bejk16, behjk16}.
\end{rmk}

\subsection{Automorphism groups of complete simplicial toric varieties}
\label{subsec:automorphism groups}

The automorphism group of the weighted projective plane $\PP(1,1,2) = (\kk^3 \setminus \{ 0 \})/ \GG_m$, for $\GG_m$ acting linearly on $\kk^3$ with weights $1,1,2$,  is given by 
$$\mbox{Aut}(\PP(1,1,2)) \cong R \ltimes U$$
where $R \cong \GL(2,\kk)$ is reductive and 
$U \cong (\kk)^3$  is unipotent. Here   $(\lambda,\mu,\nu) \in (\kk)^3$ acts on $\PP(1,1,2)$ as 
$$[x,y,z]  \mapsto [x,y,z+\lambda x^2 + \mu xy + \nu y^2].$$
The central one-parameter subgroup $\GG_m$ of $R \cong \GL(2,\kk)$ acts on the Lie algebra of 
$H_u$ with all positive weights, and the
associated semi-direct product
$$\hat{U} = U \rtimes \GG_m$$
can be identified with a subgroup of $\mbox{Aut}(\PP(1,1,2))$. Thus the results discussed in $\S$\ref{subsec:graded unipotent group actions} have an immediate application to linear actions of $\mbox{Aut}(\PP(1,1,2))$.

\begin{cor}
 Suppose that $H=\mbox{Aut}(\PP(1,1,2))$ acts linearly on a projective
variety $X$. If the linearisation is replaced with a suitable positive tensor power and twisted by an appropriate character of $H$, then when condition ($\mf{C}^*$) holds the enveloping quotient
$X\env H$
is the projective variety associated to the $H$-invariants on $X$, and is a categorical quotient by $H$ of an $H$-invariant open subset of $X$ which can be described using Hilbert--Mumford-like criteria. Even if ($\mf{C}^*$) fails, provided that the weaker condition ($\mf{C}^{**}$) holds there is a geometric quotient by $H$ of an open subset of $X$ described by Hilbert--Mumford-like criteria, with a projective completion  which is a categorical quotient  of an open subset  of an $H$-equivariant blow-up $\tilde{X}$ of $X$ and coincides with $X\env H$  when  ($\mf{C}^*$) holds.
\end{cor}

In fact the same is true for the automorphism group of { any} complete simplicial toric variety.
For it was observed in \cite{bedhk16a} using the description in \cite{c95} that
the automorphism group $H$ of
any complete simplicial toric variety is a linear algebraic group with a graded unipotent radical $U$ such that the grading is defined by a one parameter subgroup $\GG_m$ of $H$ acting by conjugation on the Lie algebra of $U$ with all weights strictly positive, and inducing a central one-parameter subgroup of $R=H/U$. Thus the results of  $\S$\ref{subsec:graded unipotent group actions} can be applied to any linear action of $H$ on an irreducible projective variety with respect to an ample linearisation.

\subsection{Groups of $k$-jets of holomorphic reparametrisations of $(\CC^p,0)$}
\label{subsec:reparametrisation groups}

Suppose now that $\kk = \CC$ and consider $k$-jets at 0 of holomorphic maps from $\CC^p$ to a complex manifold $Y$  for any $k, p \geq 1$. It was observed in \cite{bk15} that 
 the group $\GG_{k,p}$ of $k$-jets of holomorphic reparametrisations of $(\CC^p,0)$  has 
a graded unipotent radical $\UU_{k,p}$ such that the grading is defined by a one-parameter subgroup of $\GG_{k,p}$ acting by conjugation on the Lie algebra of $\UU_{k,p}$ with all weights strictly positive, and inducing a central one-parameter subgroup of the reductive group $\GG_{k,p}/\UU_{k,p}$. 
So the results  discussed in $\S$\ref{subsec:graded unipotent group actions} apply
to any linear action of the
reparametrisation group  $\GG_{k,p}$.

Here $\GG_k = \GG_{k,1}$ is the group of $k$-jets of germs of biholomorphisms of $(\CC,0)$
 given by
$$ t \mapsto \phi(t) = a_1 t + a_2 t^2 + \ldots + a_k t^k, $$ 
for $a_1,\ldots , a_k \in \CC$ and $ a_1 \neq 0$, 
under composition modulo $t^{k+1}$. It is isomorphic to the group of matrices  
 $$ \GG_k \cong \left\{ \left( \begin{array}{cccc} a_1 & a_2 & \ldots & a_k\\
 0 & a_1^2 & \ldots & \\
 & & \ldots & \\
 0 & 0 & \ldots & a_1^k \end{array} \right) : a_1 \in \CC^*, a_2,\ldots a_k \in \CC \right\}$$
and hence is a linear algebraic group.
$\GG_k$ has a subgroup $\CC^*$ (represented by $\phi(t)=a_1 t$) and a unipotent subgroup
$\UU_k$ (represented by $\phi(t)= t + a_2 t^2 + \ldots + a_k t^k$) which is its unipotent radical, with
$$ \GG_k \cong \UU_k \rtimes \CC^*.$$ 
If $Y$ is a complex manifold 
 then $\GG_k$ acts fibrewise on the bundle 
$J_k  \to Y$  of $k$-jets at $0$ of holomorphic curves $f:\CC \to Y$ 
by reparametrising $k$-jets.
 Similarly the group $\GG_{k,p}$ of $k$-jets of germs of biholomorphisms of $(\CC^p,0)$ acts fibrewise on the bundle $J_{k,p} \to Y$ of $k$-jets at the origin of holomorphic maps $f:\CC^p \to X$, and
$$ \GG_{k,p} \cong \UU_{k,p} \rtimes \GL(p,\CC)$$
where $\UU_{k,p}$ is the unipotent radical of $\GG_{k,p}$, and the central
one-parameter subgroup $\CC^*$ of $\GL(p,\CC)$ acts on the Lie algebra of $\UU_{k,p}$ with
all weights strictly positive. Thus $\GG_{k,p}$ has the structure required in $\S$\ref{subsec:graded unipotent group actions}.

\subsection{Unstable strata for linear actions of reductive groups}
\label{subsec:unstable strata}

Now let $G$ be a reductive group over an algebraically closed field $\kk$ of characteristic zero, acting linearly on a projective variety $X$ with respect to an ample line bundle $L$. Associated to this linear $G$-action and an invariant inner product on the Lie algebra of $G$, there is a stratification 
$$ X = \bigsqcup_{\beta \in \mathcal{B}} S_\beta$$ of $X$ by locally closed subvarieties $S_\beta$, 
indexed by a partially ordered finite subset $\mathcal{B}$ of a positive Weyl chamber for the reductive group $G$,  such that 

 (i) $S_0 = X^{ss}$, 

\noindent and for each $\beta \in \mathcal{B}$

 (ii) the closure of $S_\beta$ is contained in $\bigcup_{\gamma \geqslant \beta} S_\gamma$, and

 (iii) $S_\beta \cong G \times_{P_\beta} Y_\beta^{ss}$

\noindent where
$P_\beta$ is a parabolic subgroup of $G$ acting on  a projective subvariety $\overline{Y}_\beta$ of $X$ with an open subset $Y_\beta^{ss}$ which is determined by the action of the Levi subgroup of $P_\beta$ with respect to a suitably twisted linearisation \cite{kir84}.

Here the original linearisation for the action of $G$ on $L \to X$ is restricted to the action of the parabolic subgroup $P_\beta$ over $\overline{Y}_\beta$, and then twisted by a rational character of $P_\beta$ which is well adapted in the sense of  $\S$\ref{subsec:graded unipotent group actions} for a central one-parameter subgroup of the Levi subgroup of $P_\beta$ acting with all weights strictly positive on the Lie algebra of the unipotent radical of $P_\beta$.
Thus to construct a quotient by $G$ of (an open subset of) an unstable stratum $S_\beta$, we can study the linear action on $\overline{Y}_\beta$ of the parabolic subgroup $P_\beta$, twisted appropriately, and apply the results discussed in  $\S$\ref{subsec:graded unipotent group actions}.

In particular we can consider moduli spaces of sheaves of fixed Harder--Narasimhan type over a nonsingular projective variety $W$ (cf.  \cite{hok12}). 
There are  well known constructions going back to Simpson \cite{s94}
of the  moduli spaces of semistable pure sheaves on $W$ of fixed  Hilbert polynomial as GIT
quotients of linear actions of suitable  special linear groups $G$ on  schemes $Q$ (closely related to 
quot-schemes) which are $G$-equivariantly embedded in projective spaces.
 These constructions can
be chosen so that elements of $Q$ which parametrise sheaves of a fixed Harder--Narasimhan
type  form a stratum in the stratification of $Q$ associated to the linear action of $G$ (at least modulo
taking connected components of strata) \cite{hok12}. $\S$\ref{subsec:graded unipotent group actions} can be applied to  the associated linear actions of  parabolic subgroups of these special linear groups $G$, appropriately twisted, to construct and study moduli spaces of  sheaves
of fixed Harder--Narasimhan type over $W$ \cite{behjk16}.
The simplest non-trivial case is that of unstable vector bundles of rank 2 and fixed Harder--Narasimhan type over a nonsingular projective curve $W$ (cf. \cite{bramn09}).


\appendix


\section{Appendix: Linearisations of Products of Reductive Groups}
\label{subsec:ApLiProducts}

We discuss GIT quotients of direct products
of reductive groups. For this section, suppose $G_1$ and
$G_2$ are reductive groups and $X$ is a $G_1 \times
G_2$-variety equipped with a $G_1 \times G_2$-linearisation $L \to
X$. Via the natural embeddings $G_i \hookrightarrow G_1 \times G_2$, $i=1,2$,
this data is equivalent to saying that the variety $X$ and the line bundle $L$ are
equipped with two commuting linearisations $G_i \act L \to X$. In
particular, it makes sense to consider the semistable loci $X^{\rmss(G_1)}$ and
$X^{\rmss(G_1 \times G_2)}$ with
respect to the linearisations $G_1 \act L \to X$ and $G_1 \times G_2 \act L \to
X$ respectively, together with their reductive GIT quotients
\begin{align*}
  \pi_{G_1}&:X^{\rmss(G_1)} \to X \dblslash G_1, \\
\pi_{G_1 \times G_2}&: X^{\rmss(G_1 \times G_2)} \to X \dblslash (G_1 \times G_2).
\end{align*}

In the case where $L \to X$ is ample and
$X$ is projective over an affine variety the following result is well known
(cf. \cite{ost99} and \cite[Section 1.5.3]{sch08} for the case $X=\PP^n$ and
also \cite{tha96}), though proofs in the general case are hard to come by. For
the reader's convenience, we
include here a proof for the 
more general case of when $X$ is any variety and $L \to X$ is any linearisation.  

\begin{propn} \label{prop:Ap1} Retain the notation above. 
\begin{enumerate}
\item \label{itm:Ap1-1} The set $X^{\rmss(G_1)}$ is stable under the $G_2$-action on 
$X$ and there is
  a canonical action of $G_2$ on $X \dblslash G_1$ such that $\pi_{G_1}$ is
  $G_2$-equivariant.
\item \label{itm:Ap1-2} There is a natural ample $G_2$-linearisation $M \to X
  \dblslash G_1$ 
  such that, for some $n>0$, we have $\pi_{G_1}^*M = L^{\ten n}|_{X^{\rmss(G_1)}}$ as
  $G_2$-linearisations and $X^{\rmss(G_1 \times G_2)} \subseteq
  \pi_{G_1}^{-1}((X \dblslash G_1)^{\rmss(M)})$. Letting  
\[
\ol{\pi}_{G_2}:(X \dblslash G_1)^{\rmss(M)} \to (X \dblslash G_1)
\dblslash_{M} G_2
\]
denote the reductive GIT quotient with respect to this linearisation, there is a
canonical open immersion $\psi:X\dblslash (G_1 \times G_2) \hookrightarrow (X
\dblslash G_1) \dblslash_{M} G_2$ such that the following diagram commutes:
\[
\begin{diagram}
 & & X^{\rmss(G_1 \times G_2)} \\
 & \ldTo(2,4)^{pi_{G_1 \times G_2}} & \dTo_{\pi_{G_1}} \\
 & & (X \dblslash G_1)^{\rmss(M)} \\
 & & \dTo_{\ol{\pi}_{G_2}} \\
X \dblslash (G_1 \times G_2) & \rInto^{\psi} & (X \dblslash G_1) \dblslash_{M} G_2 \\ 
\end{diagram}
\]
\item \label{itm:Ap1-3} If $X$ is further assumed to be projective, then
  $X^{\rmss(G_1 \times G_2)}=\pi^{-1}((X \dblslash G_1)^{\rmss(M)})$ and $\psi$ is an
  isomorphism. 
\end{enumerate}   
\end{propn}

\begin{proof}
(Proof of \ref{itm:Ap1-1}.) Suppose $f \in H^0(X,L^{\ten r})^{G_1}$, for some
$r>0$, such that $X_f$ is affine. For any $g_2 \in G_2$ the section $g_2 \cdot
f$ is again $G_1$-invariant, 
and acting on $X$ by $g_2$ induces an isomorphism (with inverse given by
$g_2^{-1}$) $X_f \to X_{g_2 \cdot f}$, so that $X_{g_2 \cdot f}$ is also
affine. Hence the $G_2$-action on $X$ restricts to define an action $\sigma:G_2 \times
X^{\rmss(G_1)} \to X^{\rmss(G_1)}$. Recall that the GIT quotient $
\pi_{G_1}:X^{\rmss(G_1)} \to X
\dblslash G_1$ is a categorical quotient for the action of $G_1$ on
$X^{\rmss(G_1)}$. Let $G_1$ act on $G_2 \times X^{\rmss(G_1)}$ by demanding that $G_1$ 
acts
trivially on $G_2$. Then the composition 
\[
G_2 \times X^{\rmss(G_1)} \overset{\sigma}{\longrightarrow} X^{\rmss(G_1)}
\overset{\pi_{G_1}}{\longrightarrow} X \dblslash G_1 
\]
is $G_1$-invariant by virtue of the fact that $G_1$ is normal in $G_1 \times
G_2$, so there is a canonical map $\ol{\sigma}:G_2 \times X
\dblslash G_1 \to X \dblslash G_1$ such that the diagram 
\[
\begin{diagram}
G_2 \times X^{\rmss(G_1)} & \rTo^{\sigma}  &
X^{\rmss(G_1)} \\
\dTo_{\id_{G_2} \times \pi_{G_1}} & & \dTo_{\pi_{G_1}} \\
G_2 \times (X \dblslash G_1) &\rTo^{\ol{\sigma}} & X \dblslash G_1 \\
\end{diagram}
\]   
commutes. Using the universal property of
categorical quotients it is easy to verify that $\ol{\sigma}$ defines an action
of $G_2$ on $X \dblslash G_1$---we omit the details. 

(Proof of \ref{itm:Ap1-2}.) The construction of the GIT quotient $X \dblslash
G_1$ comes with an ample line bundle $M \to X
\dblslash G_1$ such that $\pi_{G_1}^*M = L^{\ten n}|_{X^{\rmss(G_1)}}$, for some
$n>0$ \cite[Theorem 1.10]{mfk94}. In fact, the natural map $L^{\ten
  n}|_{X^{\rmss(G_1)}} \to M$ thus arising is a good categorical quotient of the
action of $G_1$ on $L^{\ten n}|_{X^{\rmss(G_1)}}$. (This can be shown by following
through the proof of the following more general statement \cite[Proposition
3.12]{new78}: if $G$ is a reductive group acting on varieties $X$ and $Y$, if $X \to
Y$ is an affine $G$-equivariant morphism and $Y$ possesses a good categorical
quotient by $G$, then so too does $X$.) Following an argument similar to that in
the proof of \ref{itm:Ap1-1}, one sees that there is a canonical
$G_2$-action on $M$ such that $L^{\ten n}|_{X^{\rmss(G_1)}} \to M$ is
$G_2$-equivariant and the line bundle projection $M \to X \dblslash G_1$ is
equivariant. 

We next show that $X^{\rmss(G_1 \times G_2)} \subseteq \pi_{G_1}^{-1}((X \dblslash
G_1)^{\rmss(M)})$. Let $x \in X^{\rmss(G_1 \times G_2)}$. Then without loss of
generality there is an invariant section $f \in H^0(X,L^{\ten mn})^{G_1 \times
  G_2}$ with $m>0$ such that $x \in X_f$ and $X_f$ is affine. Clearly 
$\pi_{G_1}$ is defined at $x$. Because both $\pi_{G_1}:X^{\rmss(G_1)} \to X \dblslash
G_1$ and $L^{\ten mn}|_{X^{\rmss(G_1)}} \to M^{\ten m}$ are $G_2$-equivariant maps 
that are
categorical quotients for the $G_1$-actions, pulling back along $\pi_{G_1}$ defines a 
canonical
$G_2$-equivariant isomorphism 
\[
\pi_{G_1}^*:H^0(X \dblslash G_1,M^{\ten m}) \overset{\cong}{\longrightarrow}
H^0(X^{\rmss(G_1)},L^{\ten mn})^{G_1}.
\]
Hence there is $F \in H^0(X \dblslash G_1,M^{\ten m})^{G_2}$ such that
$\pi_{G_1}^{-1}((X \dblslash G_1)_F)=X_f$. The map $\pi_{G_1}$ restricts to a
good categorical quotient $\pi_{G_1}:X_f \to (X \dblslash G_1)_F$ for the
$G_1$-action on $X_f$, and since $X_f$ is affine so too is $(X \dblslash
G_1)_F$ by Theorem \ref{thm:BgRe1}, \ref{itm:BgRe1-1}. Thus $(X \dblslash G_1)_F 
\subseteq (X \dblslash G_1)^{\rmss(M)}$ and
$\pi_{G_1}(x) \in (X \dblslash G_1)^{\rmss(M)}$. 

The composition $\ol{\pi}_{G_2} \circ \pi_{G_1}:X^{\rmss(G_1 \times G_2)} \to (X
\dblslash G_1) \dblslash_M G_2$ is $G_1 \times G_2$-invariant, so induces a
unique morphism $\psi:X \dblslash(G_1 \times G_2) \to (X \dblslash G_1)
\dblslash_M G_2$ making the required diagram commute. Recall from the construction of 
the GIT
quotient that $X \dblslash (G_1 \times G_2)$ is covered by affine open subsets
$\pi_{G_1 \times G_2}(X_f) = \spec(\OO(X_f)^{G_1 \times G_2})$, for $f \in
H^0(X,L^{\ten mn})^{G_1 \times G_2}$ with $m>0$. The morphism $\psi$ 
maps $\pi_{G_1 \times G_2}(X_f)$ to the affine open subset $\ol{\pi}_{G_2}((X
\dblslash G_1)_F)$ of $(X \dblslash G_1)
\dblslash_M G_2$, where as above $F$ is a
$G_2$-invariant section such that $\pi_{G_1}^*F=f|_{X^{\rmss(G_1)}}$; this map
corresponds to the isomorphism of rings 
\[
\OO(\ol{\pi}_{G_2}((X \dblslash G_1)_F))
\overset{\ol{\pi}_{G_2}^*}{\longrightarrow} \OO((X \dblslash G_1)_F)^{G_2}
\overset{\pi_{G_1}^*}{\longrightarrow} \OO(X_f)^{G_1 \times G_2}.
\]
Hence $\psi$ restricts to an isomorphism $\pi_{G_1 \times G_2}(X_f) \to
\ol{\pi}_{G_2}((X \dblslash G_1)_F)$. Patching over all such $\pi_{G_1 \times
  G_2}(X_f)$ shows that $\psi$ is an open immersion.

(Proof of \ref{itm:Ap1-3}.) Suppose now that $X$ is projective and $L$ is ample. Then
the GIT quotient $X \dblslash G_1$ is canonically isomorphic to
$\proj(\kk[X,L^{\ten n}]^{G_1})$, 
with $\kk[X,L^{\ten n}]^{G_1}$ finitely generated and $M \to X \dblslash G_1$
corresponding to the twisting sheaf $\OO(1)$ on $\proj(\kk[X,L^{\ten
  n}]^{G_1})$ \cite[Page 40]{mfk94}. The GIT quotient $\pi_{G_1}:X^{\rmss(G_1)} \to X
\dblslash G_1$ is the morphism defined by the inclusion $\kk[X,L^{\ten n}]^{G_1}
\hookrightarrow \kk[X,L^{\ten n}]$. Moreover, by Serre vanishing \cite[Chapter
3, Proposition 5.3]{har77}, for sufficiently large $m>0$ the natural map
$H^0(X,L^{\ten mn})^{G_1} \to H^0(X \dblslash G_1,M^{\ten m})$ is
surjective. Now suppose $x \in X^{\rmss(G_1)}$ maps to $(X \dblslash G_1)^{\rmss(M)}$
under $\pi_{G_1}$. Then there is $F \in H^0(X \dblslash G_1,M^{\ten m})^{G_2}$
such that $F(\pi_{G_1}(x)) \neq 0$, with $m$ sufficiently large so that
$\pi_{G_1}^*F = f|_{X^{\rmss(G_1)}}$ for some global invariant section $f \in
H^0(X,L^{\ten mn})^{G_1 \times G_2}$, so that $x \in X_f \subseteq X^{\rmss(G_1 \times
  G_2)}$. Thus $X^{\rmss(G_1 \times G_2)}=\pi_{G_1}^{-1}((X \dblslash
G_1))$. The induced map $\pi_{G_1}:X^{\rmss(G_1 \times G_2)} \to (X \dblslash
G_1)^{\rmss(M)}$ is therefore a categorical quotient for the $G_1$-action on
$X^{\rmss(G_1 \times G_2)}$, and so its composition with the categorical
$G_2$-quotient $\ol{\pi}_{G_2}: (X \dblslash G_1)^{\rmss(M)} \to (X \dblslash G_1)
\dblslash_M G_2$ is a categorical quotient for the full $G_1 \times G_2$-action
on $X^{\rmss(G_1 \times G_2)}$. It follows that the canonically induced map $\psi:X
\dblslash(G_1 \times G_2) \to (X \dblslash G_1) \dblslash_M G_2$ is an
isomorphism.                               
\end{proof}


\bibliographystyle{alpha}  

\end{document}